\documentclass{amsart}
\usepackage{etex}
\usepackage[bookmarksopen,bookmarksdepth=2]{hyperref}
\RequirePackage{amsmath}
\RequirePackage{amssymb}
\RequirePackage{amsxtra}
\RequirePackage{amsfonts}
\RequirePackage{latexsym}
\RequirePackage{euscript}
\RequirePackage{amscd}
\RequirePackage{amsthm}
\RequirePackage{xypic}
\usepackage{stmaryrd}
\usepackage{mathtools}
\usepackage{enumitem}\usepackage{longtable}
\xyoption{all}
\newif\iffinalrun
\finalruntrue
\iffinalrun
  \newcommand{\need}[1]{}
  \newcommand{\mar}[1]{}
\else
  \newcommand{\need}[1]{{\tiny *** #1}}
  \newcommand{\mar}[1]{\marginpar{\raggedright\tiny squirrel: #1}}
\fi

\newcommand{\Xreg}{X_{\mathrm{reg}}}
\newcommand{\uT}{\underline{T}}
\newcommand{\uB}{\underline{B}}
\newcommand{\uG}{\underline{G}}
\newcommand{\Stab}{\operatorname{Stab}}
\newcommand{\ch}{\operatorname{ch}}
\newcommand{\cXbar}{\overline{\cX}}

\newcommand{\WBM}{\mathrm{W}_{\mathrm{BM}}}
\newcommand{\WBT}{\mathrm{W}_{\mathrm{BT}}}
\newcommand{\Wexpl}{\mathrm{W}_{\mathrm{expl}}}
\newcommand{\Wobv}{\mathrm{W}_{\mathrm{obv}}}
\newcommand{\Wcrisexists}{\mathrm{W}^\exists_{\mathrm{cris}}}
\newcommand{\Wcrisforall}{\mathrm{W}^\forall_{\mathrm{cris}}}

\def\Lbar{\overline{L}}

\newcommand{\Xbar}{\overline{X}}
\newcommand{\Fpbarx}[1]{\Fbar{}_p^{#1}}
\newcommand{\Qpbarx}[1]{\Qbar{}_p^{#1}}
\newcommand{\Zpbarx}[1]{\Zbar{}_p^{#1}}

\usepackage{dsfont}

\def\jbar{\overline{j}}

\def\dotw{\dot {w}}

\def\Bhat{\widehat{B}}
\def\Ghat{\widehat{G}}
\def\Hhat{\widehat{H}}
\def\That{\widehat{T}}
\def\Uhat{\widehat{U}}
\def\Mhat{\widehat{M}}
\def\Delhat{\widehat{\Delta}}

\def\can{\mathrm{can}}
\def\expll{\mathrm{expl}}

\def\ve{{\varepsilon}}

\newcommand{\isorho}{\stackrel{\rho}{\To}}
\newcommand{\isorhoK}{\stackrel{\rho_K}{\To}}
\newcommand{\isotau}{\stackrel{\tau}{\To}}
\newcommand{\isotauK}{\stackrel{\tau_K}{\To}}

\let\emptyset\varnothing

\newcommand{\JH}{\operatorname{JH}}

\def\A{\mathbb A}

\def\C{\mathbb C}
\def\F{\mathbb F}

\def\Q{\mathbb{Q}}
\def\R{\mathbb{R}}
\def\T{\mathbb{T}}

\def\Z{\mathbb{Z}}
\def\Fbar{\overline{\F}}

\def\Qbar{\overline{\Q}}

\def\Zbar{\overline{\Z}}

\def\Pr{\mathbb{P}}

\def\m{\mathfrak m}

\newcommand{\wh}{\widehat}

\def\Vbar{\overline{V}}

\def\chibar{\overline{\chi}}

\def\id{\mathrm{id}}

\def\ab{\mathrm{ab}}

\def\ss{\mathrm{ss}}

\def\GL{\mathrm{GL}}

\def\Gal{\mathrm{Gal}}

\def\Aut{\mathrm{Aut}}
\DeclareMathOperator{\Sym}{Sym}
\def\Ext{\mathrm{Ext}}

\def\Art{\mathop{\mathrm{Art}}\nolimits}

\def\Hom{\mathop{\mathrm{Hom}}\nolimits}

\def\Spf{\mathop{\mathrm{Spf}}\nolimits}
\def\Frob{\mathop{\mathrm{Frob}}\nolimits}

\def\Ind{\mathop{\mathrm{Ind}}\nolimits}

\def\rhobar{\overline{\rho}}

\def\W{\mathrm{W}}
\def\WD{\mathrm{WD}}

\def\m{\mathfrak{m}}

\def\plim#1{\displaystyle \lim_{\myatop{\longleftarrow}{#1}}}

\def\ev{\mathrm{ev}}

\newcommand{\onto}{\twoheadrightarrow}

\newcommand{\into}{\hookrightarrow}

\newcommand{\ua}{\uparrow}

\newcommand{\To}{\longrightarrow}
\newcommand{\isoto}{\stackrel{\sim}{\To}}

\newlength{\ownl}

\newcommand{\diag}{{\operatorname{diag}}}

\newcommand{\Fr}{{\operatorname{Fr}}}

\newcommand{\gr}{{\operatorname{gr}\,}}

\newcommand{\rec}{{\operatorname{rec}}}

\newcommand{\Res}{{\operatorname{Res}}}

\newcommand{\wt}[1]{\widetilde{#1}}
\newcommand{\GSp}{\operatorname{GSp}}

\newcommand{\LG}{{{}^{L}G}}
\newcommand{\LH}{{{}^{L}H}}

\newcommand{\der}{{\operatorname{der}}}

\newcommand{\reg}{{\operatorname{reg}}}
\newcommand{\semis}{{\operatorname{ss}}}

\newcommand{\G}{{\mathbb{G}}}

\newcommand{\CC}{{\mathcal{C}}}

\newcommand{\CH}{{\mathcal{H}}}

\newcommand{\CS}{{\mathcal{S}}}

\newcommand{\cA}{\mathcal{A}}

\newcommand{\cC}{\mathcal{C}}

\renewcommand{\cH}{\mathcal{H}}

\newcommand{\cO}{\mathcal{O}}
\renewcommand{\O}{\cO}

\newcommand{\cS}{\mathcal{S}}

\newcommand{\cV}{\mathcal{V}}
\newcommand{\cW}{\mathcal{W}}
\newcommand{\cX}{\mathcal{X}}

\newcommand{\cZ}{\mathcal{Z}}

\newcommand{\barK}{\overline{{K}}}

\newcommand{\tR}{\widetilde{{R}}}

\newcommand{\tW}{\widetilde{{W}}}

\newcommand{\ta}{\widetilde{{a}}}

\newcommand{\gammabar   }{\overline{\gamma}}

\newcommand{\varepsilonbar  }{\overline{\varepsilon}}

 \newcommand{\sigmabar   }{\overline{\sigma}}

 \newcommand{\psibar   }{\overline{\psi}}

\DeclareMathOperator{\conjj}{conj}

\def\RCS$#1: #2 ${\expandafter\def\csname RCS#1\endcsname{#2}}
\RCS$Revision: 85 $
\RCS$Date: 2011-12-16 18:54:17 -0600 (Fri, 16 Dec 2011) $

\DeclareMathOperator{\sgn}{sgn}

\newcommand{\p}{\mathfrak{p}}

\newcommand{\bb}{\mathbb}

\DeclareMathOperator{\FL}{FL}

\renewcommand{\o}[1]{\overline{#1}}
\newcommand{\rbar}{{\bar{r}}}

\newcommand{\HT}{\operatorname{HT}}

 \newcommand{\Qp}{{\Q_p}}

\newcommand{\Zp}{{\Z_p}}
 
\newcommand{\Qpbar}{\overline{\Q}_p}
\newcommand{\Qpbartimes}{{\overline{\Q}_p^\times}}
\newcommand{\Zpbar}{\overline{\Z}_p}

\newcommand{\Fpbar}{\overline{\F}_p}

\newcommand{\Fp}{{\F_p}}

\newcommand{\emb}{\kappa}
\newcommand{\embb}{\overline{\emb}}

\DeclareMathOperator{\BM}{BM}

\newcommand{\dual}{^\vee}
\newcommand{\nua}{\!\not\,\ua}
\newcommand{\uua}{\ua\!\ua}
\renewcommand{\AA}{\mathcal A}
\newcommand{\upl}{\textup{(}\hskip -0.0833em\nolinebreak}
\newcommand{\upr}{\nolinebreak\hskip 0.0833em\textup{)}}
\newcommand{\HC}{\cC}
\newcommand{\HA}{\cA}

\newcommand{\SK}{S_K}
\newcommand{\Sk}{S_k}

\DeclareMathOperator{\obv}{obv} 
\DeclareMathOperator{\expl}{expl}

\theoremstyle{plain} 
\newtheorem{lem}[equation]{Lemma}
\newtheorem{lemma}[equation]{Lemma}
\newtheorem{prop}[equation]{Proposition}
\newtheorem{thm}[equation]{Theorem}
\newtheorem{cor}[equation]{Corollary}
\newtheorem{conj}[equation]{Conjecture}
\newtheorem{question}[equation]{Question}

\theoremstyle{definition}
\newtheorem{defn}[equation]{Definition}
\newtheorem{hyp}[equation]{Hypothesis}

\theoremstyle{remark}
\newtheorem{rem}[equation]{Remark}
\newtheorem{remark}[equation]{Remark}
\newtheorem{ex}[equation]{Example}
\newtheorem{example}[equation]{Example}

\numberwithin{equation}{subsection}
\numberwithin{figure}{subsection}

\newcommand{\SW}{\W}
\newcommand{\Xn}{X^{(n)}_1}
\newcommand{\XnSk}{(\Xn)^{S_k}}
\newcommand{\SWX}{\XnSk/{\sim}}

\newcommand{\Cr}{\mathrm{cr}}
\newcommand{\myatop}[2]{\genfrac{}{}{0pt}{}{#1}{#2}}

\begin{document}

\title{General Serre weight conjectures}
\author{Toby Gee} \address{Department of
  Mathematics, Imperial College London}\email{toby.gee@imperial.ac.uk}
\thanks{The first author was 
  partially supported by a Leverhulme Prize, EPSRC grant EP/L025485/1, Marie Curie Career
  Integration Grant 303605, and by
  ERC Starting Grant 306326.}
\author{Florian Herzig}
\address{Department of Mathematics, University of Toronto}
\email{herzig@math.toronto.edu}
\thanks{The second author was partially supported by NSF grants
  DMS-0902044 and DMS-0932078,
  a Sloan Fellowship, and an NSERC grant.}
\author{David Savitt} 
\address{Department of Mathematics, Johns Hopkins University}\email{savitt@math.jhu.edu}
\thanks{The third author was partially
  supported by NSF  grant DMS-0901049 and NSF CAREER grant
  DMS-1054032.}

\begin{abstract} We formulate a number of related generalisations of the weight
  part of Serre's conjecture to the case of $\GL_n$ over an arbitrary number
  field, motivated by the formalism of the Breuil--M\'ezard conjecture. We
  give evidence for these conjectures, and discuss their relationship to
  previous work. We generalise one of these conjectures to the case of
  connected reductive groups which are unramified over $\Qp$, and we also
  generalise the second author's previous conjecture for $\GL_n/\Q$ to this
  setting, and show that the two conjectures are generically in agreement.
\end{abstract}

\maketitle

\setcounter{tocdepth}{1}
\tableofcontents

\section{Introduction}
\label{sec:introduction}

The goal of this paper is to formulate a number of related generalisations of the weight
  part of Serre's conjecture to the case of $\GL_n$ over an arbitrary number
  field. Since this is a problem with a long and involved history and
  since we work in significant generality in this paper, we
  begin with an extended introduction, in which we try to summarize
  this history (Sections~\ref{sec:serr-conj-gl_2} to~\ref{sec:shad-weights-cryst}) and give a detailed overview of the approach that we
  have taken (Sections~\ref{sec:this-paper} and~\ref{sec:summary-our-weight}). 

\subsection{Serre's conjecture for \texorpdfstring{$\GL_2$}{GL(2)} over \texorpdfstring{$\Q$}{Q}}
\label{sec:serr-conj-gl_2}

Let $p$ be a prime.
 Serre's conjecture, as originally formulated in 1973 (see \cite[p.\
 9]{MR0406931} and~\cite[\S3]{MR0382173}), 
predicted that every odd irreducible continuous representation 
$\rbar: G_{\mathbb Q} \to \GL_2(\Fbar_p)$ which is unramified outside
$p$ has a twist by a power of the mod $p$ cyclotomic character 
which arises from a cuspidal modular Hecke
eigenform of level $1$ and weight at most $p+1$. The theory of the
$\theta$-operator then implies that $\rbar$ itself is modular of weight at
most $p^2-1$.
This was a bold conjecture, for at the time there was little evidence 
outside of the cases $p=2,3$. In those cases, since there are no
cusp forms of level~$1$ and weight less than $12$ the conjecture simply
predicts that there are no such representations. This can be established via
discriminant bounds, as in~\cite{MR1299740} and~\cite[p.\
710]{MR926691}.

Serre later formulated (\cite{MR885783}) a version of the conjecture with no
restriction on the ramification of $\rbar$, 
which included a precise recipe for both the weight and 
the level of a modular eigenform giving rise to $\rbar$. In this way
the conjecture became computationally verifiable, and was tested in a
number of cases in which $\rbar$ has small image.

At least as far back as~\cite{MR0382173}, it had been known that in the theory
of mod~$p$ modular forms, one can trade off the weight and level (and
Nebentypus) at $p$. For this reason Serre restricted his attention to modular forms of level
prime to $p$. He conjectured that the minimal possible level of the
candidate eigenform giving rise to $\rbar$ could be taken
to be the prime-to-$p$ Artin conductor of $\rbar$, while
his conjectural recipe for the minimal possible weight
of the eigenform (in prime-to-$p$-level) was more
intricate, and depended on the ramification behaviour of $\rbar$
at~$p$.



The part of Serre's conjecture which predicts that every odd
irreducible continuous representation $\rbar: G_{\Q} \to
\GL_2(\Fpbar)$ arises from some modular eigenform
is often referred to
as ``the weak form of Serre's conjecture'', while the form
of the conjecture that  includes the precise recipes for
the minimal weight and level is called ``the strong
form of Serre's conjecture''.
Much of the early work concerning Serre's conjecture was focussed 
on proving that the weak form implies the strong form, and it is natural to
expect that work on generalisations of Serre's conjecture will follow the same
pattern. (Indeed, the eventual proof of Serre's conjecture
\cite{MR2551763,MR2551764,MR2551765} relied on the work
that had been done to prove the equivalence of the weak and strong forms.)

Serre's conjectural recipe for the minimal weight of an eigenform of
prime-to-$p$ level giving rise to
$\rbar$ was more subtle than the recipe for the level, but essentially amounted to providing the minimal
weight $k$ that was consistent with the known properties of the restriction to a
decomposition group at $p$ of the Galois representations associated to
eigenforms. To make this precise one nowadays uses the language of $p$-adic
Hodge theory. Given a modular eigenform $f$ of weight $k \ge 2$ and
prime-to-$p$ level, the
associated $p$-adic Galois representation $r_f : G_{\Q} \to
\GL_2(\Qpbar)$ has the property that the restriction $r_f|_{G_{\Qp}}$ to
a decomposition group at $p$ is crystalline with Hodge--Tate weights $k-1$
and $0$. Therefore any results on the reduction mod~$p$ of crystalline
representations of $G_{\Qp}$ with Hodge--Tate weights $k-1$ and $0$, such as the early
results of Deligne and Fontaine--Serre when $k \le p$, give
information (purely in 
terms of $\rbar|_{G_{\Qp}}$) about the possible weights $k$ of the modular eigenforms giving
rise to 
$\rbar$. 

To give a concrete example, let $\varepsilon$ denote the cyclotomic
character, and $\varepsilonbar$ its reduction mod $p$.  Suppose that \begin{equation}\label{eq:easy-rbar} \rbar|_{I_{\Qp}} \cong
\begin{pmatrix}
  \varepsilonbar^{k-1} & * \\ 0 &  1
\end{pmatrix}\end{equation}
where $I_{\Qp}$
is the inertia group at $p$, and $2 < k
< p+1$. Then the minimal weight predicted by Serre's recipe is
$k$. Indeed, it is known that
 any crystalline representation $\rho : G_{\Qp} \to
\GL_2(\Qpbar)$ with Hodge--Tate weights $k-1$ and $0$ (with $k$ in the
given range) and whose reduction mod $p$ is reducible must be an
extension of an unramified character by an unramified twist of
$\varepsilon^{k-1}$, and therefore the shape of $\rhobar|_{I_{\Qp}}$ must be 
as on the right-hand side of \eqref{eq:easy-rbar}.

We make one further remark about the above example. Suppose that the
extension class $*$ vanishes, and assume for simplicity that $k < p-1$. Serre observed that \begin{equation}\label{eq:companion} (\rbar \otimes
\varepsilonbar^{1-k})|_{I_{\Qp}} \cong \begin{pmatrix}
  \varepsilonbar^{p-k} & 0 \\ 0 &  1
\end{pmatrix}\end{equation} and therefore has minimal weight $p+1-k$. Thus,
although Serre's conjecture predicts that any $\rbar$ has a twist
which is modular with weight at most $p+1$, in this split case there are
actually two such twists. This is the so-called  ``companion forms'' phenomenon.

\subsection{Serre weights}
\label{sec:serre-weights} 
We now explain a representation-theoretic reformulation of the weight
$k$ in Serre's conjecture. This optic first appears
in the work of Ash--Stevens~\cite{ashstevens}, and both simplifies the original weight recipe for
$\GL_2$ over $\Q$ and has proved crucial for formulating the weight part of
Serre's conjecture for other groups and over other fields. 

The Eichler--Shimura isomorphism allows one to reinterpret Serre's
conjecture in terms of the cohomology of arithmetic groups. 
If $V$ is an $\Fpbar$-representation of
$\GL_2(\Fp)$ and $N$ is prime to $p$, then we have a natural action of the Hecke algebra of
$\Gamma_1(N)$ on $H^1(\Gamma_1(N),V)$, and so it makes sense to speak
of a continuous representation $\rbar :  G_{\Q} \to \GL_2(\Fpbar)$
being associated to an eigenclass in that cohomology group. If $\rbar$
is odd and irreducible, then the Eichler--Shimura isomorphism implies that $\rbar$ is
modular of weight $k$ and prime-to-$p$ level $N$ if and only if $\rbar$ is
associated to an eigenclass in $H^1(\Gamma_1(N),\Sym^{k-2} \Fpbarx{2})$, where
$\Sym^{k-2} \Fpbarx{2}$ is the $(k-2)$th symmetric power of the standard
representation of $\GL_2(\Fp)$ on $\Fpbarx{2}$. 
 By d\'evissage one deduces that
$\rbar$ is modular of weight $k$ and prime-to-$p$ level $N$ if and
only if $\rbar$ is associated to an eigenclass in $H^1(\Gamma_1(N),V)$
for some Jordan--H\"older factor $V$ of $\Sym^{k-2}
\Fpbarx{2}$. (Recall that the representation $\Sym^{k-2} \Fpbarx{2}$ is
reducible as soon as $k > p+1$.) 

It is then natural to associate to $\rbar$ the set
$W(\rbar)$ of irreducible $\Fpbar$-representations $V$ of $\GL_2(\Fp)$
such that $\rbar$ is associated to an eigenclass in
$H^1(\Gamma_1(N),V)$ for some prime-to-$p$ level $N$. Thanks to the
argument in the previous paragraph, the (finite) set $W(\rbar)$ determines all weights in which $\rbar$
occurs in prime-to-$p$ level, and not just the minimal such weight. 
For this reason such representations of $\GL_2(\Fp)$
are now often referred to as \emph{Serre weights}, or even simply
\emph{weights}.

To illustrate, suppose once again that $\rbar$ is as in
\eqref{eq:easy-rbar}, with $2 < k < p-1$. If the extension class $*$ is non-split, then
we have $W(\rbar) = \{ \Sym^{k-2} \Fpbarx{2}\}$. However, in the
companion forms case where the
extension class $*$ is split, we have \[ W(\rbar) = \{ \Sym^{k-2}
\Fpbarx{2}, \det\nolimits^{k-1} \otimes \Sym^{p-1-k} \Fpbarx{2}  \}\]
instead. Here the second weight comes from observing via
\eqref{eq:companion} that the weight $\Sym^{p-1-k} \Fpbarx{2}$ should
lie in $W(\rbar \otimes \varepsilonbar^{1-k})$, and then undoing the twist.


 Serre in fact  asked~\cite[\S3.4]{MR885783} whether a ``mod-$p$ Langlands philosophy'' exists
which would give a more natural definition of the weight, and which would
allow for generalisations of the conjecture to other groups and number
fields. This is now known to be true for $\GL_2$ over $\Q$
(\cite{MR2642409,emerton2010local}) and the set $W(\rbar)$
intervenes naturally from this point of view (see for example
\cite{MR2827792}). There is considerable evidence that such a
philosophy remains true in more
general settings, although it is far from completely developed at this
point. Indeed the results to date on
generalisations of the weight part of Serre's conjecture have been a major
guiding influence on the development of the mod $p$ Langlands program, rather
than a consequence of it.

\subsection{Early generalisations}
\label{sec:early-gener}

Formulations of very general versions of the weak conjecture have been known to the
experts for many years; the main issue is to define the correct generalisation
of ``odd'', for which see for example~\cite{grossodd}
and~\cite[\S6]{MR3028790}. (If one wishes to consider automorphic forms or
cohomology classes for groups which are not quasi-split, it is also necessary to
impose conditions on the ramification of
$\rbar$ at places at which the underlying group is ramified; see~\cite[Def.\
4.5.3]{geekisin} for the case of quaternion algebras.) Moreover,
granting an understanding of classical local Langlands and its
relationship to local-global compatibility, it is reasonably
straightforward to generalise the definition of the (prime-to-$p$) level in terms of the
prime-to-$p$ ramification of $\rbar$. For example, for generalisations to
$\GL_n$ over arbitrary number fields, one again expects to take the level to be
the prime-to-$p$ Artin conductor of $\rbar$; see e.g.~\cite[\S2.2]{bib:ADP}
for the case that the number field is~$\Q$. 

However, formulating the weight part of the conjecture in any generality
has proved difficult.  We stress at the outset that, in keeping with
the mod $p$ Langlands philosophy, one conjectures that the set of Serre weights associated
to $\rbar$ depends only on the restrictions of $\rbar$ to
decomposition groups at places dividing $p$. For this reason all of
the 
weight predictions that we discuss in this paper are formulated in
terms of local Galois representations.

For Hilbert modular forms over a totally real
field $F$ in which~$p$ is unramified, a precise conjecture was formulated by
Buzzard--Diamond--Jarvis in~\cite{bdj}. It was essential for
\cite{bdj} to use the
``Serre weight'' point of view, since weights of Hilbert modular forms
are $[F:\Q]$-tuples of integers and so  there isn't a natural notion of
minimal weight of $\rbar$. In this context a Serre weight
is an irreducible $\Fpbar$-representation of $\prod_{v \mid p}
\GL_2(k_v)$, where $k_v$ is the residue field of the completion
$F_v$. The recipe of \cite{bdj} predicts the set of weights
$W(\rbar)$ in terms of the Hodge--Tate weights of crystalline lifts of
$\rbar|_{G_{F_v}}$ for $v \mid p$, in line with the discussion at the
end of Section~\ref{sec:serr-conj-gl_2}. The prediction of \cite{bdj}
is now known to be correct \cite{gls12,geekisin}.

In another direction, the study of the weight part of Serre's
conjecture for $\GL_n$ over $\Q$ was initiated by Ash and his
collaborators \cite{bib:ASinn,bib:ADP}, with a particular focus on $\GL_3$. 
They gave a combinatorial recipe for a predicted set of weights, in the spirit of Serre's original
recipe but using the language of Serre weights. The combinatorial
recipe takes as input the tame inertia weights of $\rbar|_{I_{\Qp}}$
(the base $p$ ``digits'' of the exponents when
$\rbar|_{I_{\Qp}}$ is written as a successive extension of powers of
fundamental characters),
much as in the examples~\eqref{eq:easy-rbar}, \eqref{eq:companion} and their
reformulations in Section~\ref{sec:serre-weights}.

In the case where $\rbar|_{G_{\Qp}}$ is semisimple, the thesis
\cite{herzigthesis} of the second-named author gave a
representation-theoretic recipe for a predicted set of weights, which
for generic $\rbar|_{G_{\Qp}}$ should be the full set of weights.  The
prediction is made in terms of the reduction mod~$p$ of
Deligne--Lusztig representations, and involves a mysterious
involution $\mathcal{R}$ on the set of Serre weights. In particular
\cite{herzigthesis} predicts some weights for $\GL_3$ that are not
predicted by \cite{bib:ADP}, and that were subsequently
computationally confirmed (in some concrete cases) by Doud and Pollack.
(We stress that \cite{bib:ASinn,bib:ADP} did not claim to predict the
full set of weights for~$\rbar$.)

\subsection{The Breuil--M\'ezard conjecture}
\label{sec:breu-mezard-conj}
We now turn to the Breuil--M\'ezard conjecture, which gives a new way
of looking at the weight part of Serre's conjecture. 

Originally the
Breuil--M\'ezard conjecture arose in the context of attempts to
generalise the Taylor--Wiles method \cite{Taylor--Wiles}, and was also one of the starting points of the $p$-adic Langlands
program. It was clear early on that understanding the geometry of deformation spaces
of local mod $p$ Galois representations with prescribed $p$-adic
Hodge-theoretic conditions was essential for proving automorphy
lifting theorems; the earliest automorphy lifting theorems required the smoothness
of such deformation spaces. The Breuil--M\'ezard conjecture gives a
measure of the complexity of these deformation spaces, in terms of the modular
representation theory of $\GL_2$. 

We state a version of this conjecture for $\GL_n$ over $\Qp$,
following~\cite{emertongeerefinedBM}. We need the following
data and terminology:\
\begin{itemize}
\item a continuous representation $\rhobar : G_{\Qp} \to
  \GL_n(\Fpbar)$,
\item a \emph{Hodge type} $\lambda$, which in this setting is an $n$-tuple 
  $\lambda = (\lambda_1,\ldots,\lambda_n)$ of integers with $\lambda_1 \ge \dots 
  \ge \lambda_n$, and
\item an \emph{inertial type} $\tau$, i.e.\ a representation $I_{\Qp} \to
  \GL_n(\Qpbar)$ with open kernel and that can be extended to a
  representation of $G_{\Qp}$. 
\end{itemize}
Kisin \cite{kisindefrings} associates to this data a lifting ring $R^{\lambda,\tau}_{\rhobar}$
whose characteristic $0$ points parameterise the lifts of $\rhobar$
that are potentially crystalline with  type $\tau$ and Hodge--Tate
weights \begin{equation}\label{eq:shift}\lambda_1 + n-
  1,\dots,\lambda_{n-1}+1,\lambda_n.\end{equation} The Breuil--M\'ezard conjecture predicts the Hilbert--Samuel
multiplicity $e(R^{\lambda,\tau}_{\rhobar} \otimes_{\Zpbar} \Fpbar)$
of $R^{\lambda,\tau}_{\rhobar} \otimes_{\Zpbar} \Fpbar$, as follows. 

The inertial local Langlands correspondence (cf.\ Henniart's appendix
to \cite{MR1944572}) associates to $\tau$ a finite-dimensional
smooth $\Qpbar$-representation $\sigma(\tau)$ of $\GL_n(\Zp)$. On the other
hand associated to $\lambda$ is the irreducible algebraic
representation $W(\lambda)$ of $\GL_n(\Qp)$ of highest weight
$\lambda$.

\begin{conj}[The generalised Breuil--M\'ezard conjecture]\label{conj:intro-bm}
  There exist non-negative integers $\mu_V(\rhobar)$, indexed by Serre
  weights $V$, such that for all Hodge types $\lambda$ and inertial
  types $\tau$ we have
\[ e(R^{\lambda,\tau}_{\rhobar} \otimes_{\Zpbar} \Fpbar) = \sum_{V}  n_{\lambda,\tau}(V)
\mu_V(\rhobar)\]
where $ n_{\lambda,\tau}(V)$ is the multiplicity of $V$ in the
reduction modulo $p$ of $W(\lambda) \otimes_{\Qpbar} \sigma(\tau)$ \upl as
a $\GL_n(\Zp)$-representation\upr.
\end{conj}

This conjecture was first formulated by
Breuil--M\'ezard~\cite{MR1944572} for $\GL_2$
with certain restrictions on $\lambda$ and $\tau$. In the special case
where the Serre weight $V$ is actually a Weyl module, and therefore lifts to
some $W(\lambda)$ in characteristic zero, taking $\tau$ trivial in
Conjecture~\ref{conj:intro-bm} gives an equality 
$\mu_V(\rhobar) = e(R^{\lambda,\mathrm{triv}}_{\rhobar}
\otimes_{\Zpbar} \Fpbar)$. For this reason we typically refer to the integers
$\mu_V(\rhobar)$ as multiplicities.

Based on some explicit
calculations, Breuil and M\'ezard
furthermore gave predictions for the multiplicities $\mu_V(\rhobar)$,
and observed a close connection between these multiplicities and the
weight part of Serre's original modularity conjecture:\ namely, that
one appeared to have $\mu_V(\rbar|_{G_{\Qp}}) > 0$ if and only if~$V$ was a
predicted Serre weight for $\rbar$. 

More recently, the first-named author and Kisin \cite{geekisin}
suggested that one could turn this around and use the Breuil--M\'ezard
conjecture to define the set of Serre weights $W_{\BM}(\rhobar) = \{ V
: \mu_V(\rhobar) > 0 \}$ associated to a local Galois
representation~$\rhobar$. One would then conjecture that $W(\rbar) =
W_{\BM}(\rbar|_{G_{\Qp}})$.   

Note that this prediction for the set of Serre weights associated to 
$\rbar$, while very general, is contingent on the truth of the 
Breuil--M\'ezard conjecture.  In fact, what \cite{geekisin} actually
do is prove the Breuil--M\'ezard conjecture for $\GL_2$ and
$\lambda=0$ (for arbitrary $K/\Qp$), which allows them
unconditionally to define a set of weights
$W_{\mathrm{BT}}(\rhobar)$. Here $\mathrm{BT}$ stands for Barsotti--Tate. 
This description of the set of weights turns out to be extremely useful, and was an important
part of the resolution in~\cite{geekisin} and~\cite{gls13} of the conjectures of~\cite{bdj} and their
generalisations to arbitrary totally real fields. 

The key technique used by \cite{geekisin} is the method of
Taylor--Wiles--Kisin patching. One first constructs a globalisation
$\rbar$ of $\rhobar$. 
 Write $X_{\infty} = \Spf
R_{\rhobar}[[x_1,\ldots,x_h]]$, with $R_{\rhobar}$ the universal
lifting ring of $\rhobar$  and $h \ge 0$ a certain integer.
Similarly write $X_{\infty}^{\tau} = \Spf 
  R_{\rhobar}^{0,\tau}[[x_1,\ldots,x_h]]$, which if non-empty is of
  dimension $d+1$ for some $d$ independent of $\tau$.
In the
context of \cite{geekisin} a \emph{patching functor} is a non-zero
covariant exact functor $M_{\infty}$ from the category of finitely
generated $\Zpbar$-modules with a continuous action of $\GL_2(\cO_K)$,
to the category of coherent sheaves on $X_{\infty}$, with the
properties that:
\begin{itemize}
\item for all inertial types $\tau$ the sheaf 
  $M_{\infty}(\sigma(\tau))$ is $p$-torsion free and has scheme-theoretic support~$X_{\infty}^{\tau}$,
  and in fact is maximal Cohen--Macaulay over $X_{\infty}^{\tau}$;
\item the (maximal Cohen--Macaulay over a regular scheme, so) locally free sheaf $M_{\infty}(\sigma(\tau))[1/p]$ has rank
  one over the generic fibre of~$X_{\infty}^{\tau}$, and
\item for all Serre weights $V$, the support of the sheaf
  $M_{\infty}(V)$ either has dimension $d$ or is empty.
\end{itemize}
This is an abstraction of the output of the Taylor--Wiles--Kisin
patching method applied to spaces of automorphic forms. The existence
of a patching functor can be shown to imply that the Breuil--M\'ezard
conjecture holds (in the cases under consideration in
\cite{geekisin}), and moreover that $W_{\mathrm{BT}}(\rhobar)$ is
precisely the set of weights $V$ for which $M_{\infty}(V) \neq 0$. 
On the other hand, \cite{geekisin} construct such a functor, and the
construction implies that $M_{\infty}(V) \neq 0$ if and only if
$\rbar$ is automorphic of weight $V$. Putting these together, \cite{geekisin} conclude that the set
$W_{\mathrm{BT}}(\rhobar)$ is indeed the correct weight set for
$\rbar$. 


\subsection{Shadow weights and the crystalline lifts conjecture}
\label{sec:shad-weights-cryst}

One of the features of the weight part of Serre's conjecture for
$\GL_n$ ($n \ge 3$) that distinguishes it from the $\GL_2$ case is
that there exist Serre weights that do not lift to 
characteristic $0$. For example, for $\GL_3$ over $\Q$ roughly half the Serre
weights are so-called ``upper alcove weights''. These have the
property that if $W$
is the irreducible representation in characteristic~$0$ with the same highest
weight as an upper alcove weight $U$, then the reduction mod $p$ of $W$ is not irreducible but
rather has two Jordan--H\"older factors, one of which is $U$ and 
another which we denote by $L$ (for ``lower alcove'').

It was observed in the conjecture of \cite{herzigthesis} (in the
semisimple case) as well as in the computations of~\cite{bib:ADP}
(including some non-semisimple examples) that whenever $L$ was a predicted Serre weight for some
$\rbar$, so also was $U$. For this reason one began to refer to
the weight $U$ as a \emph{shadow} of the weight $L$. The
conjecture that $U$ occurs in the set of Serre weights of $\rbar$
whenever $L$ does (as well as its natural generalisation to the
$\GL_n$ setting) became known as the shadow weight conjecture.

In the optic of the Breuil--M\'ezard conjecture, the shadow weight
conjecture  says that if $\mu_L(\rhobar) > 0$ for some local Galois
representation $\rhobar$ and if $U$ is a shadow of $L$, then also $\mu_U(\rhobar) > 0$. The
Breuil--M\'ezard conjecture itself implies that if $\rhobar$ has Serre
weight $U$ then $\rhobar$ has a crystalline lift  with
Hodge--Tate weights corresponding to the highest weight of $U$, or
equivalently to the highest weight of $W$; and conversely that if
$\rhobar$ has such a crystalline lift, then $\rhobar$ has at least one
Serre weight that occurs in the reduction of $W$. In combination with the shadow
weight conjecture, this is elevated to an if-and-only-if:\ that $\rhobar$ has Serre
weight $U$ if and only if a crystalline lift of~$\rhobar$ with Hodge--Tate weights as above exists.

This attractive picture (as well as its generalisation to $\GL_n$ over
more general number fields) was known as the crystalline lifts
version of the weight part of Serre's conjecture, and was
widely believed for a number of years. For the sake of historical
accuracy, we should remark that the crystalline lifts version of the
weight part of Serre's conjecture emerged \cite[\S4]{gee061} before the Breuil--M\'ezard
optic, motivated by its evident parallels with the $\GL_2$ case (both
Serre's original conjecture and the conjecture of \cite{bdj}) and its
compatibility with the conjectures of \cite{herzigthesis}.

The crystalline lifts version of the weight part of Serre's conjecture was contained in drafts of the present paper as recently as
2014. We had a narrow escape, then, when (prior to the completion of
this paper) Le, Le Hung, Levin, and Morra \cite{LLLM} produced
counterexamples to the shadow weight conjecture for $\GL_3$ over $\Q$ (in the
non-semisimple case), thus also disproving the crystalline lifts
version of the weight part of Serre's conjecture for $\GL_3$.

The geometric explanation seems to be as follows. The
papers~\cite{emertongeeproper,emertongeepicture}
construct a finite type equidimensional Artin stack $\cXbar$ over
$\Fp$ whose $\Fpbar$-points naturally correspond to the isomorphism
classes of representations $\rhobar : G_{\Qp} \to
\GL_3(\Fpbar)$. The stack $\cXbar$ should have among its irreducible components
$\cXbar(U)$ and $\cXbar(L)$, whose $\Fpbar$-points are precisely the
representations $\rhobar$ for which $U$ and $L$ respectively are Serre
weights, and these components appear to intersect in codimension
one. Since~\cite{bib:ADP} make computations for representations
$\rbar$ which by construction have small image, it is not surprising
in hindsight that those representations might lie in special loci of $\cXbar$.

\subsection{This paper}
\label{sec:this-paper} In this paper, we explain a general formulation 
of the weight part of Serre's conjecture (Conjecture~\ref{conj:BM-version}) in terms of the
Breuil--M\'ezard conjecture, based on the philosophy outlined in
Section~\ref{sec:breu-mezard-conj}.  
Moreover, there are compelling reasons (coming from the
Fontaine--Mazur conjecture and the Taylor--Wiles method) to believe that this
recipe gives the correct weights in full generality. In particular, in
Proposition~\ref{prop: patching functors imply BM weights} we
prove that the existence of a suitable patching functor would on the one hand imply the
Breuil--M\'ezard conjecture, and would on the other hand imply that the
set $W_{\BM}(\rhobar)$ is the set of Serre weights of 
globalisations of $\rhobar$.

Although we believe this description of the weights is the ``correct'' one, and
it seems likely that any proof of the weight part of Serre's conjecture in
general situations will need to make use of this formulation, it is of interest
to have more explicit descriptions of the set of weights. 
For a variety of
reasons (which we discuss in the body of the paper), it seems unlikely
that in general
there will be explicit and complete descriptions of the sort that one
finds  for $\GL_2$  in \cite{gls12,gls13,ddr,cegm}, but it does
seem reasonable to hope for something more concrete in the case that
$\rhobar$ is semisimple and suitably generic.

Indeed, it remains plausible that the crystalline lifts version of the weight
part of Serre's conjecture, despite being false in general, is
nevertheless true in the case where $\rbar$ is semisimple locally at places
above $p$. For instance, there is considerable evidence in the $3$-dimensional
case over $\Q$:\ when $\rbar|_{G_{\Qp}}$ is suitably generic many
cases of the conjecture are proved in \cite{MR3079258} and
\cite{LLLM}, and for some non-generic $\rbar|_{G_{\Qp}}$ there is
computational evidence due to \cite{bib:ADP}.  The more recent papers
\cite{LLLM2,LLL}  extend the results of \cite{LLLM} to the case of totally real fields in
which $p$ is unramified, and establish \emph{weight elimination} (that
the set of modular weights is a subset of the set of predicted
weights) in arbitrary dimension in this setting, again with a
genericity hypothesis on $\rbar$ locally above $p$.

In Conjecture~\ref{conj:crystalline-version} we formulate the
crystalline lifts version of the weight part of Serre's conjecture for 
Galois representations that are semisimple locally at primes above
$p$.  We remark that when the extension $K/\Qp$ is ramified, the
definition of the weight set in terms of crystalline lifts involves a
choice of lifting $K \into \Qpbar$ for each embedding $k \into \Fpbar$
of the residue field $k$ of $K$. 
(The former embeddings index Hodge--Tate weights, the latter are used to
parameterise Serre weights.)
This leads us to define two weight
sets, $\Wcrisexists(\rhobar)$ and $\Wcrisforall(\rhobar)$. The former
is the set of weights obtained by taking the union over all such
choices of liftings, and the latter is the set of weights obtained by
taking the intersection. We conjecture that these two sets are in fact equal.


Section~\ref{sec: the picture} contains a brief and informal
discussion of some intuition for Serre weight conjectures that is
suggested to us by
the
Galois moduli stacks of~\cite{emertongeeproper,emertongeepicture}.

We next explore the possibility of making the
conjectures of Section~\ref{sec:cryst-lifts-serre} explicit. Our basic idea is that in the case when
$\rhobar$ is semisimple, we can explicitly construct many crystalline
lifts of $\rhobar$ by lifting each irreducible factor of $\rhobar$
separately (and this comes down to constructing crystalline lifts of characters, since
each irreducible mod $p$ representation of $G_K$ is induced from a
character of an unramified extension of $K$). We call a crystalline
lift obtained in this way an \emph{obvious crystalline lift}, and
correspondingly we obtain a set of weights $\Wobv(\rhobar)$. (In fact
this is not quite accurate:\ the set $\Wobv(\rhobar)$ also takes into
account our expectation that the set of Serre weights of $\rhobar$
should depend only on $\rhobar|_{I_K}$;\ see Definition~\ref{defn:obvious-lift-ss}.)
To illustrate, 
the representation \[\rhobar \cong \begin{pmatrix}
  \varepsilonbar^{k-1}\o\chi_1 & 0 \\ 0 &  \o\chi_2
\end{pmatrix}\] of $G_{\Qp}$ with $\o\chi_i$ unramified has a crystalline lift $\rho$
of the form \[\rho \cong \begin{pmatrix}
  \varepsilon^{k-1}\chi_1 & 0 \\ 0 &  \chi_2
\end{pmatrix}\] where each $\chi_i$ is unramified and lifts $\o\chi_i$
($i = 1,2$). The representation $\rho$
has Hodge--Tate weights $k-1$ and $0$;\ taking into account the shift by $(1,0)$ as 
in \eqref{eq:shift}, we predict that the Serre weight
$\Sym^{k-2} \Fpbarx{2}$, which is described by the highest weight $(k-2,0)$, is
contained in $\Wobv(\rhobar)$. Similarly, $\rhobar$ has
a crystalline lift with Hodge--Tate weights $p-1$, $k-1$ obtained by
lifting $\o\chi_2$ instead to $\varepsilon^{p-1} \chi_2$, leading to
the inclusion
$\det\nolimits^{k-1} \otimes \Sym^{p-1-k} \Fpbarx{2} \in \Wobv(\rhobar)$,
in accordance with Section~\ref{sec:serre-weights}.

As will be discussed in Remark~\ref{rem: equivalence of crystalline lifts and shadow weights}, under the assumption of the
generalised Breuil--M\'ezard conjecture, the shadow weight conjecture and the crystalline lifts conjecture
are equivalent. Therefore, our explicit weight set for $\rhobar$ needs to be closed under the consequences of the 
shadow weight conjecture. We denote the smallest set of Serre weights that satisfies this requirement and that contains
$\Wobv(\rhobar)$ by $\cC(\Wobv(\rhobar))$ (see Section~\ref{sec:shad-weights:-inert}). We call the weights that lie
in the complement $\cC(\Wobv(\rhobar)) \setminus \Wobv(\rhobar)$ \emph{shadow weights}. The simplest example occurs
for $\GL_3$ over~$\Qp$, as explained in Section~\ref{sec:shad-weights-cryst}.

For a period of time, we hoped that the set $\cC(\Wobv(\rhobar))$
might
 explain the full set of weights of
$\rhobar$ arising from crystalline lifts; unfortunately, this cannot always
be the case. Again this phenomenon first occurs for $\GL_3$ over $\Qp$.  In some cases we can
inductively construct further crystalline lifts of $\rhobar$ coming from Levi subgroups. The idea is that we write $\rhobar
= \oplus_i\, \rhobar_i$ and take the direct sum of certain crystalline lifts
$\rho_i$ of~$\rhobar_i$ whose existence would be implied by the generalised
Breuil--M\'ezard conjecture in combination with the explicitly
constructed weight set for~$\rhobar_i$. 
 The weight set resulting 
from these (hypothetical) crystalline lifts is denoted by $\Wexpl(\rhobar)$. It contains in fact all shadow weights. We call
the weights in the complement $\Wexpl(\rhobar) \setminus \cC(\Wobv(\rhobar))$ \emph{obscure weights}. See
Example~\ref{ex:exceptional-weight-gl3} for examples of such weights in the case of $\GL_3$ over $\Qp$. In Section~\ref{sec:shifted-weights}
we ask furthermore whether the weight set of $\rhobar$ should be closed under certain ``weight shifts'', and we give a limited
amount of evidence for a positive answer.

In general we do not know how close our explicit weight set
$\Wexpl(\rhobar)$ is to the actual set of weights of $\rhobar$. 
In Section~\ref{sec:evidence} we compare this predicted weight set to all
existing conjectures and computational evidence that we are aware
of. Then, in the final part of our paper, we give strong evidence that
in case $K/\Qp$ is unramified and $\rhobar$ is sufficiently
generic (a genericity condition on the tame inertia weights of
$\rhobar|_{I_{\Qp}}$) we are not missing any weights.
It turns out that it is most natural to work in the setting of unramified groups $G$ over $\Qp$, considering $\GL_n$ over $K$
as the restriction of scalars $\Res_{K/\Qp} \GL_n$ to $\Qp$. (For our precise conditions on $G$, see Hypothesis~\ref{hyp:gp}.)
We extend both our explicit weight set $\Wexpl(\rhobar)$ as well as the weight set $W^?(\rhobar)$ of~\cite{herzigthesis} to 
this general setting and then prove the following theorem.
\begin{thm}[Theorem~\ref{thm:main-result}]
  If $\rhobar|_{I_{\Qp}}$ is semisimple and sufficiently generic, then $W^?(\rhobar) = \Wexpl(\rhobar) = \cC(\Wobv(\rhobar))$.
\end{thm}
In particular, there are no obscure weights in this generic setting. The proof is not immediate but requires some subtle
modular representation theory. We thus see this result as an encouraging sign that $\Wexpl(\rhobar)$ is correct
in the generic unramified case.

This paper has two appendices. Appendix~\ref{sec:wangs-result-uparrow}
contains the proof of a theorem of J.\ C.\ Ye and J.\ P.\ Wang on alcove geometry that is needed in
Section~\ref{sec:herzig-comparison}; as far as we know the only
published proofs are in Chinese. In
Appendix~\ref{sec:wobvrhbar-nonempty} we prove by 
combinatorial arguments that the explicit set
of Serre weights defined in Section~\ref{sec:expl-weight-conj} is
always non-empty.

The only part of the paper that we have not yet mentioned is
Section~\ref{sec:global framework}, in which we
describe 
a global framework for formulations of generalisations of Serre's conjecture, in
terms of the cohomology of arithmetic quotients of ad\`ele
groups.  Although it is of course necessary to have chosen such a framework
before one can begin to speak about Serre's conjecture (e.g.\ in order
to define what one means when one says that $\rbar$ to be modular of a given weight!),
this discussion is in some sense secondary to the rest of the paper,
which is entirely local except for parts of
Section~\ref{sec:patching, BM and serre weights} on patching. 

\subsection{Index to the weight sets defined in this paper}\hfill
\label{sec:summary-our-weight}

\medskip

Associated to a global  Galois
representation $\rbar : G_F \to \GL_n(\Fpbar)$:\
\begin{itemize}
\item $\W(\rbar)$, the Serre weights of $\rbar$: Definition~\ref{defn: weights of a modular Galois
    representation};
\item $\W_v(\rbar)$, the local Serre weights of $\rbar$ at a place
  $v\mid p$: Conjecture~\ref{conj:basic-conjecture}.
\end{itemize}

Associated to a local Galois representation $\rhobar : G_K \to
\GL_n(\Fpbar)$:\
\begin{itemize}
\item $\W_{\BM}(\rhobar)$, the Breuil--M\'ezard predicted weights for
  $\rhobar$: Definition~\ref{defn:Serre weights predicted by BM};
\item $\W_{\cS}(\rhobar)$, the $\cS$-Breuil--M\'ezard predicted
  weights for $\rhobar$: Definition~\ref{defn:Serre weights predicted
    by BM of type S};
\item $\Wcrisexists(\rhobar)$ and $\Wcrisforall(\rhobar)$, the crystalline weights for $\rhobar$: Definition~\ref{defn:crystalline-weight-sets}.
\end{itemize}

Associated to $\rhobar : G_K\to \GL_n(\Fpbar)$ such that
$\rhobar|_{I_K}$ is semisimple:
\begin{itemize}
\item $\Wobv(\rhobar)$, the obvious weights for $\rhobar$:
  Definition~\ref{defn:obvious-lift-ss};
\item $\cC(\Wobv(\rhobar))$, the obvious and shadow weights for $\rhobar$: Definition~\ref{defn:closure-operation};
\item $\Wexpl(\rhobar)$, the explicit predicted weights for $\rhobar$:
  Definition~\ref{defn:explicit-weights-ss};
\item $\W^?(\rhobar)$ ($n=3$), the weights predicted by
  \cite{herzigthesis}: Proposition~\ref{prop:def-of-W-herzig}.
\end{itemize}

See Section~\ref{sec:summary} for a summary of our conjectures about
these weight sets.

\medskip

Associated to a tame inertial $L$-parameter $\tau : I_{\Qp} \to
\Ghat(\Fbar)$, for a group $G$ as in Hypothesis~\ref{hyp:gp}
(generalising the corresponding definitions for $\rhobar$ such that
$\rhobar|_{I_K}$ is semisimple):
\begin{itemize}
\item $\Wobv(\tau)$, the obvious weights for $\tau$: Definition~\ref{defn:obv-tau};
\item $\cC(\Wobv(\tau))$, the obvious and shadow weights for $\tau$: Definition~\ref{def:closure-C-for-general-G};
\item $\Wexpl(\tau)$, the explicit predicted weights for $\tau$: Definition~\ref{def:wexpl-general};
\item $\W^?(\tau)$,\ the weights for $\tau$ predicted in the manner of 
  \cite{herzigthesis}:\ Definition~\ref{defn:herzigwts}.
\end{itemize}

The latter three of these sets coincide for sufficiently generic
$\tau$ (Theorem~\ref{thm:main-result}).

\subsection{Acknowledgments}\label{subsec:thanks} The point of view
that we adopt in this paper owes a considerable debt to the 
ideas of Matthew Emerton; he has declined to be listed as a coauthor,
but we hope that his influence on the paper is clear. We would like to thank Kevin
Buzzard, Fred Diamond, Brandon Levin and L\^e H\`ung Vi\^et Bao for many helpful conversations. F.H.\ would like to thank
Ida Bulat for her assistance with typing, as well as the MSRI for the excellent working conditions it
provided. We are very grateful to
Mehmet Haluk \c{S}eng\"{u}n for repeating the calculations of~\cite{MR2887610}
for us, and to Darrin Doud for extending some of his calculations for us.  We thank Chuangxun (Allen) Cheng for his translation of parts of \cite{wang} and Jim Humphreys for inspiring correspondence related to Corollary~\ref{cor:ye-wang}.
We are thankful to the referee for helpful comments.

\subsection{Notation and conventions}\label{subsec:notation}

We
fix a prime $p$.  If $K$ is any field, we let $\o K$ be a separable
closure of $K$, and let $G_K = \Gal(\o K/K)$; nothing we do will
depend on the choice of $\o K$, and in particular we will sometimes
consider $G_L$ to be a subgroup of $G_K$ when $K\subset L$. 
In Section~\ref{sec:unramified groups} we will
instead denote $\Gal(\o K/K)$  by~$\Gamma_K$ to avoid a conflict of
notation. All Galois representations
are assumed to be
continuous with respect to the profinite topology on the Galois group 
and the natural topology on the coefficients (which will usually be
either the $p$-adic topology or the discrete topology).

If $K$ is a finite
extension of $\bb{Q}_p$, we write $\cO_K$ and $k$ respectively for the ring of
integers and residue field of $K$, $I_K$ for the inertia
subgroup of~$G_K$, and
$\Frob_K$ for a geometric Frobenius element of $G_K$.
If $F$ is a number field and $v$ is a
finite place of $F$ then we let $\Frob_v$ denote a geometric Frobenius
element of $G_{F_v}$ and we write $k_v$ for the residue field of the
ring of integers of $F_v$.

We will use $E$ to denote our coefficient field, a finite
extension of $\Qp$ contained in $\Qpbar$.   We
write $\cO=\cO_E$ for the ring of integers of $E$ and $\F$ for its
residue field.  When we are working with representations of the
absolute Galois group of a finite extension $K/\Qp$, we will often
assume that $E$ is \emph{sufficiently large}, by which we mean that 
the images of all embeddings $K\into \Qpbar$ are
contained in $E$.
We also let $\Zpbar$ denote the ring of integers of $\Qpbar$ and $\Fpbar$ its
residue field (it is thus our fixed choice of algebraic closure of $\Fp$).

Let $K$ be a finite extension of $\Qp$, and let $\Art_K:K^\times\to W_K^{\ab}$ be the isomorphism
provided by local class field theory, which we normalise so that
uniformisers correspond to geometric Frobenius elements. Let $\rec$ denote the local
Langlands correspondence from isomorphism classes of irreducible
smooth representations of $\GL_n(K)$ over $\C$ to isomorphism classes
of $n$-dimensional Frobenius semisimple Weil--Deligne representations
of $W_K$ as in the introduction to \cite{ht}, so that when $n = 1$ we have
$\rec(\pi) = \pi \circ \Art_K^{-1}$.
 We fix an isomorphism $\imath:\Qpbar\to\C$ and define the
local Langlands correspondence $\rec_p$ over $\Qpbar$ by $\imath \circ
\rec_p = \rec \circ \imath$. This depends only on
$\imath^{-1}(\sqrt{p})$ (and the only ambiguity is a quadratic
unramified twist, so that in particular $\rec_p|_{I_K}$ does not depend on any choices).

Assume for the rest of this
section that~$E$ is sufficiently large. For the purposes of defining
the notation below, we also allow $E=\Qpbar$, $\cO = \Zpbar$,
$\F=\Fpbar$ in what follows in this section.
Define $S_k = \{ \sigma : k \into \F
\}$ and $S_K = \{ \emb : K
\into E\}$.  If $\emb \in S_{K}$, we let $\embb$ be
the induced element of $S_{k}$.  Let
$\varepsilon$ denote the $p$-adic cyclotomic character, and
$\varepsilonbar$ the mod $p$ cyclotomic character.
For each $\sigma\in S_{k}$ we define the
fundamental character $\omega_{\sigma}$ corresponding to $\sigma$ to be
the
composite  $$\xymatrix{I_{K}\ar@{->>}[r] & \cO_{K}^{\times}\ar[r]
  & k^{\times}\ar[r]^{\sigma} & \F^{\times},}$$ where the first map
is induced by $\Art_K^{-1}$.  In particular $\left(\prod_{\sigma \in S_k}
\omega_\sigma\right)^{e(K/\Qp)} = \varepsilonbar|_{I_K}$. When $k =
\Fp$ and $\sigma : k \to \F$ is the unique embedding, we will often
write~$\omega$ in place of $\omega_\sigma$.    If $\chi$ is a
character of $G_{K}$ or $I_{K}$, we denote its reduction mod $p$ by
$\overline{\chi}$.

   If $W$ is a de Rham
representation of $G_K$ over $E$,
then for each $\emb \in S_K$ we will write $\HT_\emb(W)$ for the multiset of Hodge--Tate
weights of $W$ with respect to $\emb$. By definition this set contains
the integer $-i$ with multiplicity
$\dim_{E} (W \otimes_{\emb,K} \widehat{\overline{K}}(i))^{G_K}
$. Thus for example $\HT_\emb(\varepsilon)=\{ 1\}$.   The set $\HT_{\emb}(W)$ is invariant under
extensions of the coefficient field, and so also makes sense for de Rham
representations over~$\Qpbar$ (and embeddings $\emb: K \into
\Qpbar$).

We say that $W$ has \emph{regular} Hodge--Tate weights if for each
$\emb$, the elements of $\HT_\emb(W)$ are pairwise distinct. Let $\Z^n_+$
denote the set of tuples $(\lambda_1,\dots,\lambda_n)$ of integers
with $\lambda_1\ge \lambda_2\ge\dots\ge \lambda_n$. A \emph{Hodge
  type} is an element of $(\Z^n_+)^{\SK}$. Then if $W$ has
regular Hodge--Tate weights, there is a Hodge type $\lambda$ such that for each $\emb \in
S_K$ we have \[\HT_{\emb}(\rho)=\{\lambda_{\emb,1}+n-1,\lambda_{\emb,2}+n-2,\dots,\lambda_{\emb,n}\},\]
and we say that $W$ is regular of weight $\lambda$ (or Hodge type
$\lambda$).

An \emph{inertial type} is a representation
$\tau:I_K\to\GL_n(E)$  with open kernel and which extends
to the Weil group $W_K$. Then we say that a de Rham representation
$\rho:G_K\to\GL_n(E)$ has inertial type $\tau$ and Hodge
type~$\lambda$, or more briefly that $\rho$ has type $(\lambda,\tau)$,  if $\rho$ is regular of weight $\lambda$, and the
restriction to $I_K$ of the Weil--Deligne representation $\WD(\rho)$ associated to
$\rho$ is equivalent to $\tau$.

For any
$\lambda\in\Z^n_+$, view $\lambda$ as a dominant weight of the
algebraic group $\GL_{n/\cO_K}$ in the usual way, and let
$M'_\lambda$
be the algebraic $\cO_K$-representation of $\GL_n$ given
by \[M'_\lambda:=\Ind_{B_n}^{\GL_n}(w_0\lambda)_{/\cO_K}\] where $B_n$
is the Borel subgroup of upper-triangular matrices of $\GL_n$, and
$w_0$ is the longest element of the Weyl group. (This representation
is denoted by $H^0_{\cO_K}(\lambda)$ in \cite[\S II.8]{MR2015057}. Note
that its generic fibre is irreducible with 
highest weight $\lambda$ by \cite[II.5.6]{MR2015057}.) Write $M_\lambda$ for the
$\cO_K$-representation of $\GL_n(\cO_K)$ obtained by evaluating
$M'_\lambda$ on $\cO_K$. For any $\lambda\in(\Z^n_+)^{\SK}$
we write $L_{\lambda,\cO}$ for the $\cO$-representation of $\GL_n(\cO_K)$
defined by \[\otimes_{\emb \in
  S_K} (M_{\lambda_\emb}\otimes_{\cO_K,\emb}\cO),\]
although when $\cO$ is clear from the context we will suppress it and write simply $L_{\lambda}$.

We remark that the sets $S_K$ for varying (sufficiently large) coefficient fields $E$ can
be naturally identified, and we will freely do so; and similarly for the sets
$S_k$, $(\Z^n_+)^{S_K}$, and $(\Z^n_+)^{S_k}$.

If $A$ is a Noetherian local ring with maximal ideal $\m$ of dimension~$d$,
and $M$ is a finite $A$-module, then there is polynomial
$P_M^A(X)$ of degree at most $d$ (the Hilbert--Samuel polynomial of $M$),
uniquely determined by the requirement 
that for $n\gg 0$, the value $P_M^A(n)$ is equal to the length of
$M/\m^{n+1}M$ as an $A$-module. Then the Hilbert--Samuel multiplicity
$e(M,A)$ is defined to be $d!$ times the
coefficient of $X^d$ in $P_M^A(X)$, and we write $e(A)$ for $e(A,A)$.
\section{A global setting}
\label{sec:global framework}\subsection{\texorpdfstring{$\GL_{n}$}{GL(n)} over a number field}\label{subsec: GLn framework}In this section we briefly explain a
possible global setup in which we can formulate the weight part of Serre's
conjecture for $\GL_n$ over a number field $F$. It is presumably possible to
formulate such conjectures for a general connected reductive group over a number
field (by a characteristic $p$ analogue of the conjectures
of~\cite{BuzzardGee}), but this would entail developing a great deal of material
of no relevance to the bulk of this paper.

The point of this paper is to formulate and study only the weight part
of Serre's conjecture, and this is expected to be a purely local question (see
Conjecture~\ref{conj:basic-conjecture} below). Indeed, essentially everything in
this paper (other than various comparisons to results in the literature, giving
evidence for our conjectures) after the present section is purely local. On the
other hand, to make a general global formulation requires a careful discussion
of various technical issues (such as:\ the association of mod $p$ Satake
parameters to characteristic polynomials of Frobenii; characteristic $p$
analogues of the various considerations of~\cite{BuzzardGee}, such as the
$C$-group; variants using the Galois action on the \'etale cohomology of Shimura
varieties, rather than the Hecke action on the Betti cohomology; and so on).

As all of these points are orthogonal to our goals in the remainder of
the paper, we have restricted ourselves to a brief description of the case of
$\GL_n$, as this suffices for the bulk of the paper, and for much of the
computational evidence to date. (The remaining computations concern forms of
$\GL_2$.)

Let $\A_F$ denote the ad\`eles of $F$, and let $\A_F^\infty$ denote the finite
ad\`eles of $F$. Let $U=U^pU_p$ be a compact open subgroup of
$\GL_n(\A_F^\infty)$, where $U^p$ is a compact open subgroup of
$\GL_n(\A_F^{\infty,p})$, assumed to be sufficiently small, and $U_p=\GL_n(\cO_{F}\otimes_\Z\Zp)$. Let
$A_\infty^\circ=\R^\times_{>0}$, embedded diagonally in
$\prod_{v|\infty}\GL_n(F_v)$, and write
$U_\infty^\circ=\prod_{v|\infty}U_v^\circ\subset \prod_{v|\infty}\GL_n(F_v)$, where
$U_v^\circ=\mathrm{SO}_n(\R)$ if $v$ is real and $U_v^\circ=\mathrm{U}_n(\R)$ if $v$ is
complex. Set
\[Y(U):=\GL_n(F)\backslash \GL_n(\A_F)/UA_\infty^\circ U_\infty^\circ.\]

Let $W$ be an irreducible smooth $\Fpbar$-representation of
$U_p$; the action of $U_p$ on $W$ necessarily factors through
$\prod_{v|p}\GL_n(k_v)$, and we write $W\cong\otimes_{v|p}W_v$, where $W_v$ is an
irreducible $\Fpbar$-representation of $\GL_n(k_v)$.  We can define a local
system $\cW$ of $\Fpbar$-vector spaces on $Y(U)$ via 
\begin{equation}
\cW:=
\bigl((\GL_n(F)\backslash\GL_n(\A_F)/U^pA_\infty^\circ U_\infty^\circ)\times W\bigr)/U_p.\label{eq:9}
\end{equation}

By shrinking $U^p$ we are free to
assume that it is a product $U^p=\prod_{v\nmid p}U_v$. 
There is a finite set $\Sigma_0$ of finite places of $F$ (dependent on $U$)
which contains all places dividing $p$, and which has the property that if $v\notin\Sigma_0$ is a finite
place of $F$, then $U_v=\GL_n(\cO_{F_v})$. For each $v\notin\Sigma_0$, the spherical
Hecke algebra $\cH_{v}:=\cH(\GL_n(\cO_{F_v})\backslash\GL_n(F_v)/\GL_n(\cO_{F_v}),\Zpbar)$
(cf.~\cite{MR1696481}) with
coefficients in $\Zpbar$ acts naturally on each cohomology group
$H^i(Y(U),\cW)$. Indeed, $\cH_v$ is identified with the subalgebra $\cH(U\backslash U \GL_n(F_v) U/U,\Zpbar)$
of the usual adelic Hecke algebra $\cH(U\backslash\GL_n(\A_F^{\infty})/U,\Zpbar)$, and 
the bigger subalgebra $\cH(U\backslash U\GL_n(\A_F^{\infty,p}) U/U,\Zpbar)$ acts naturally on each cohomology group
$H^i(Y(U),\cW)$ (the prime-to-$p$ condition being relevant, as $W$ may be non-trivial).

Let $\rbar:G_F\to\GL_n(\Fpbar)$ be an  irreducible
representation. For any $U$ as above, and for any finite set 
$\Sigma$ of places of $F$ containing $\Sigma_0$, all the
finite places at which $\rbar$ is ramified, and all the infinite
places of $F$, we may define  a maximal ideal $\m=\m(\rbar,U,\Sigma)$
of $\T_\Sigma :=\otimes'_{v\notin \Sigma}\cH_v$ with
residue field $\Fpbar$ by demanding that for all places $v\notin\Sigma$, the
semisimple part of $\rbar(\Frob_v^{-1})$ is conjugate to the class defined by
the $\cH_v$-eigenvalues determined by $\m$ under the (suitably twisted) Satake
isomorphism (cf.~\cite[\S 17]{MR1729443}). (Of course, since we are
working with $\GL_n$, this just amounts to specifying the
characteristic polynomial of~$\rbar(\Frob_v)$, as in~\cite[Prop.\
3.4.4(2)]{cht}, but the formulation we have used here generalises more
easily to more general groups.)

\begin{defn}\label{defn: modular Galois representation}
  We say that $\rbar$ is \emph{automorphic} if there is some $W, U,\Sigma$ as
  above such that $H^i(Y(U),\cW)_\m\ne 0$ for some $i\ge 0$.
  \end{defn}

\begin{defn}
  \label{defn: weights of a modular Galois representation}Suppose that $\rbar$ is
  automorphic. Let $\W(\rbar)$ denote the set of isomorphism classes of
  irreducible representations $W$
  of $\prod_{v|p}\GL_n(k_v)$ for which
  $H^i(Y(U),\cW)_\m\ne 0$ for some $i\ge 0$. We refer to $\W(\rbar)$ as the set
  of \emph{Serre weights} of~$\rbar$.
\end{defn}

\begin{rem}
  \label{rem: trivial coefficients}Let $U_p(1)$ be the kernel of the
  homomorphism $U_p\to\prod_{v|p}\GL_n(k_v)$. A natural variant of the
  definition of the Serre weights of $\rbar$ would be
  to ask that $\Hom_{U_p}(W^\vee,H^i(Y(U_p(1)),\Fpbar)_\m)\ne 0$ for some $i\ge
  0$. We do not know how to show unconditionally that the two definitions always give the same set
  of weights, but this would follow from conjectures in the
  literature, as we now explain.

  Let $r_1$ be the number of real places of $F$, and let $r_2$ be the
  number of complex-conjugate pairs of complex places. Set
  $q_0=r_1\lfloor n^2/4\rfloor+r_2n(n-1)/2$; this is the minimal
  degree of cohomology to which tempered cohomological automorphic
  representations of $\GL_n(\A_F)$ will contribute. According to the
  conjectures of~\cite{MR2905536}, as expanded upon
  in~\cite[\S3.1.1]{emertonICM}, it is expected that if some
  $H^i(Y(U),\Fpbar)_\m$ is non-zero (where now $U_p$ can be arbitrarily
  small), then in fact $H^{q_0}(Y(U),\Fpbar)_\m\ne 0$, while
  $H^i(Y(U),\Fpbar)_\m=0$ for $i< q_0$; there is a similar expectation
  for the $H^i(Y(U),\cW)_\m$. If this conjecture holds, then it
  follows easily from the Hochschild--Serre spectral sequence that
  this variant definition gives the same set of Serre weights as
  Definition~\ref{defn: weights of a modular Galois representation}.
\end{rem}

We now have the following general formulation of a weak version of the
weight part of Serre's conjecture.
\begin{conj}
  \label{conj:basic-conjecture} Suppose that $\rbar$ is automorphic. Then
  we may write
  $\W(\rbar) = \bigotimes_{v|p} \W_v(\rbar)$, where  $\W_v(\rbar)$ is a set of
  isomorphism classes of irreducible representations of $\GL_n(k_v)$, which
  depends only on $\rbar|_{G_{F_v}}$.
\end{conj}

In fact one expects something more, namely that the set $\W_v(\rbar)$ should
depend only on $\rbar|_{I_{F_v}}$, a point that will be important for
making explicit Serre weight conjectures later in the paper
(see especially Section~\ref{sec:obvious-semisimple}).

Much of the rest of the paper will be occupied with the question of making
Conjecture~\ref{conj:basic-conjecture} more precise (and for giving evidence for
the more precise conjectures) in the sense of giving conjectural descriptions of
the sets $\W_v(\rbar)$ in terms of $\rbar|_{G_{F_v}}$.

\subsection{Groups which are compact modulo centre at infinity}\label{subsec:discrete
  series} While it is natural to work with the group $\GL_n/F$, just as in the
characteristic $0$ Langlands program it is often advantageous to work with other
choices of group, in particular those that admit discrete series. From the point
of view of Serre's conjecture, it is particularly advantageous to work with
groups which are compact mod centre at infinity; the associated arithmetic
quotients only admit cohomology in degree $0$, which facilitates an easy
exchange between information in characteristic $0$ and characteristic $p$. (In
the more general context of groups that admit discrete series, there is an
expectation that after localising at a maximal ideal $\m$ as above which is
``non-Eisenstein'' in the sense that it corresponds to an irreducible Galois
representation, cohomology should only occur in a single degree; however there
are at present only fragmentary results in this direction, beyond the case of
groups of semisimple rank $1$.)

In particular, in the papers~\cite{MR1729443}, \cite{MR1643625} and~\cite{grossodd} Gross
considers questions relating to the weak form of Serre's
conjecture for certain groups over $\Q$ which are compact mod centre at
infinity.
While he does not consider the weight part
of Serre's conjecture in his setting (although the discussion
of~\cite[\S4]{MR1729443} could be viewed as a starting point in this
direction), the conjectures we will make in this paper, especially
those for more general reductive groups,
could be used to make explicit Serre weight
conjectures for   ``algebraic modular forms'' (in
Gross's terminology). It seems likely that computations with these
automorphic forms would be a
good way to
investigate our general Serre weight conjectures.

A great deal of progress has been made on these and related questions for a
particular class of such groups, namely unitary groups or quaternion algebras over
totally real and CM fields, 
which are compact mod centre at
infinity. (In the case of quaternion algebras over a totally real field, it is
also possible to allow the quaternion algebra to split at a single infinite
place:\ in that case the semisimple rank is $1$, and it is easy to show that
the cohomology of the associated Shimura curves vanishes outside of degree $1$
after localising at a non-Eisenstein maximal ideal, for example
via~\cite[Lem.\ 2.2]{bdj}).

In particular, for these groups the association of Galois representations
(valued in $\GL_n(\Qpbar)$) to
automorphic representations is well-understood (see~\cite{shin} and the
references therein), and the Taylor--Wiles machinery is also well-developed
(\cite{cht}) and has been successfully applied to the problem of the weight part
of Serre's conjecture (see for example~\cite{geekisin}). The relevance of these
results  (which address characteristic $0$ Galois representations and how
characteristic $p$ Galois representations deform to characteristic $0$) to the
weight part of Serre's conjecture is the following simple principle, which
underlies the proofs of most of what is known about the weight part of Serre's
conjecture to date, and also motivates much of the material in the following
sections. 

Let us abusively adopt the notation of Section~\ref{subsec: GLn framework}
above, although the groups we are considering are now (say) unitary groups which
are compact mod centre at infinity. Let $V$ be a finite free $\Zpbar$-module
with a continuous action of $U_p$, and let
$\Vbar=V\otimes_{\Zpbar}\Fpbar$. Then (recalling that $U$ is
sufficiently small)
we can define a local system of $\Zpbar$-modules $\cV$ on $Y(U)$ as in (\ref{eq:9}),
and since $Y(U)$ only has cohomology in degree $0$,  we see that
$H^0(Y(U),\cV)_\m\ne 0$ if and only if $H^0(Y(U),\overline{\cV})_\m\ne 0$
if and only if $H^0(Y(U),{\cW})_\m\ne 0$ for some Jordan--H\"older factor $W$ of $\Vbar$. 

Now, if $H^0(Y(U),\cV)_\m\ne 0$ then we may consider the $p$-adic Galois
representations attached to the automorphic representations contributing to
$H^0(Y(U),\cV)_\m$; these will lift our representation $\rbar$, and in
particular for places $v|p$ the restrictions to $G_{F_v}$ of these
representations will lift
$\rbar|_{G_{F_v}}$. The known $p$-adic Hodge-theoretic properties of the
Galois representations associated to these automorphic representations then
prescribe non-trivial relationships between the $\rbar|_{G_{F_v}}$ (for $v|p$) and $V$, and
thus between the $\rbar|_{G_{F_v}}$ and $\Vbar$. In particular, by considering
the Jordan--H\"older factors $W$ of $\Vbar$, we
obtain necessary conditions in terms of the $\rbar|_{G_{F_v}}$ for $\rbar$ to be
automorphic of Serre weight $W$. The basic perspective of this paper (which was
perhaps first considered in Section 4 of~\cite{gee061}, and was refined in~\cite{geekisin}) is that these
necessary conditions are often also sufficient.

\begin{example}
  \label{example: crystalline lifts for modular forms and more generally}As a
  specific example of these considerations, consider the case of a definite
  quaternion algebra over $\Q$ that is split at $p$. Up to twist, an irreducible
  $\Fpbar$-representation of $\GL_2(\Fp)$ is of the form $W=\Sym^{k-2}\Fpbarx{2}$
  for some $2\le k\le p+1$. Taking $V=\Sym^{k-2}\Zpbarx2$ in the above
  discussion, and using the Jacquet--Langlands correspondence, we find that if
  $\rbar$ is automorphic of Serre weight $W$, then $\rbar$ can be lifted to the Galois
  representation attached to a newform of weight $k$ and level prime to~$p$. By
  local-global compatibility, this means that $\rbar|_{G_{\Qp}}$ has a lift to
  a crystalline representation with Hodge--Tate weights $\{k-1,0\}$. If
  one assumes that conversely the only obstruction to $\rbar$ being
  automorphic of Serre weight~$W$ is this property of $\rbar|_{G_{\Qp}}$ having
  a crystalline lift  with Hodge--Tate weights $\{k-1,0\}$, then an
  examination of the possible reductions modulo~$p$ of such
  crystalline representations recovers Serre's original
  conjecture~\cite{MR885783} (or rather, the specialisation of the
  conjecture of~\cite{bdj} to the case of modular forms over~$\Q$,
  which implies Serre's original conjecture by, for example, an
  explicit comparison of Serre's original recipe for a minimal weight
  with the explicit list of Serre weights; see the proof of~\cite[Thm.\ 3.17]{bdj}).

This example generalises in an obvious fashion to the case of forms of $\mathrm
U(2)$ over totally real fields which are compact at infinity, and allows one to recover the Serre weight conjecture
of~\cite{bdj}. (In the case of quaternion algebras over totally real fields
there is a parity obstruction to finding lifts to characteristic zero, coming
from the global units. However, in line with Remark~\ref{rem:mod p Langlands
  philosophy and different groups} below, the weight part of Serre's conjecture
is known for both quaternion
algebras and compact forms of $\mathrm U(2)$ over totally real fields, and the sets of Serre weights are the same in
both cases. We will elaborate on this point, and in particular say a
few words about its proof, in Remark~\ref{rem: patching functors and GK} below.)

More generally one can work over a totally real field with a form of $\mathrm
U(n)$ which is compact at infinity, and employ similar considerations;
the general theory of ``change of weight'' for Galois representations
developed in~\cite{BLGGT} (which generalises an argument of
Khare--Wintenberger) shows that it is reasonable to expect that the
only obstructions to producing automorphic lifts of particular weights
will be the local ones prescribed by $p$-adic Hodge theory. However,
for most choices of $W$ it is no longer possible to find a representation $V$
for which $\overline{V}\cong W$, and it is far from clear how to
extract complete information in characteristic~$p$ from information in
characteristic zero, and accordingly far from clear how to
generalise the description of the weight part of Serre's conjecture for
$\GL_2$.
However, we do still obtain information (for example, that being
automorphic of some Serre weight implies the existence of a crystalline lift of some
specific Hodge--Tate weights), and much of this paper is devoted to exploring
the relationship between the weight part of Serre's conjecture and 
$p$-adic Hodge theory.
In particular, a consequence of the
philosophy of the paper~\cite{geekisin} is that information about 
potentially semistable lifts
 is sufficient to determine the set of Serre weights
in general; we explain this in Sections~\ref{sec:BM} and~\ref{sec:patching, BM
  and serre weights} below.
\end{example}

\begin{rem}
  \label{rem:mod p Langlands philosophy and different groups}It is generally
  expected that there is a mod $p$ Langlands correspondence satisfying
  local-global compatibility at places dividing $p$; this is known for
  $\GL_2/\Q$ by the results of~\cite{emerton2010local}. A consequence of such a
  compatibility would be that the sets $\W_v(\rbar)$ would only depend on the
  reductive group over $F_v$. It therefore seems reasonable to use
  considerations from groups which are compact modulo centre at infinity to make conjectures about
  $\W_v(\rbar)$ for more general groups; in particular, one can use
  considerations about unitary groups which split at places above $p$ (as in
  Example~\ref{example: crystalline lifts for modular forms and more generally})
  to make predictions about the weight part of Serre's conjecture for $\GL_n$.
\end{rem}

\begin{rem}\label{rem:other things we won't do: level and oddness}
One could consider the question of the relationship of the ramification of the Galois
  representation away from $p$ to the tame level (``the level in Serre's
  conjecture''), and the question of sufficient conditions for a mod $p$ Galois
  representation to correspond to a Hecke eigenclass in the first place (``the
  weak form of Serre's conjecture'' which
  should correspond to an oddness condition at infinite places, see for
  example~\cite{grossodd} and~\cite[\S 6]{MR3028790}). Again, these questions lie in
  a rather different direction to the concerns of this paper, and we will not
  address them here.
\end{rem}
\section[Alt]{The Breuil--M\'ezard formalism for \texorpdfstring{$\GL_{n}$}{GL(n)} and Serre weights}
\label{sec:BM}
In this section we will recall the formalism of the general
Breuil--M\'ezard conjecture for $\GL_n$,
following~\cite{emertongeerefinedBM}, and then explain how the
formalism leads to a Serre weight conjecture. As
in~\cite{emertongeerefinedBM}, we will only formulate the potentially
crystalline (as opposed to potentially semistable) version of the
conjecture, as this is all that we will need. We expect an analogous
conjecture to hold in the potentially semistable case, and we refer
the reader to Section~1.1.4 of~\cite{kisinfmc} for a discussion of the
differences between the potentially crystalline and potentially
semistable versions of the conjecture in the case of $\GL_2/\Qp$. (See also
Lemma 5.2 of~\cite{gee-geraghty-quaternion}, which shows for $\GL_2$ that the
potentially crystalline and potentially semistable conjectures predict the same
multiplicities; we anticipate that the proof will extend to $\GL_n$.)
\subsection{Serre weights}\label{subsec:Serre weights} Let $K/\Qp$ be a finite
extension, and assume throughout this section that the field $E/\Qp$
(our field of coefficients) is sufficiently large. Recall that $k$ and $\F$ denote the residue fields of
$K$ and $E$ respectively. Fix a  representation
$\rhobar:G_K\to\GL_n(\F)$. (We will use $\rhobar$ to denote a local
Galois representation, typically of the group $G_K$, in contrast to $\rbar$
which we reserve for a global Galois representation, typically of the
group $G_F$.)

\begin{defn}
  \label{defn:serre-weight}
  A \emph{Serre weight} is an isomorphism class of irreducible
  $\F$-represen\-tations of $\GL_n(k)$. (This definition will be
  extended to more general reductive groups in Definition~\ref{defn:
    Serre weight general group}.)
\end{defn}

We will sometimes (slightly abusively) refer to an individual irreducible
representation as a Serre weight.

\begin{remark}
  \label{rem:coeff-irrelevant-serre-wt}
From the results recalled below, it follows that all Serre weights 
 can be defined over $k$ 
 (note that our running assumptions imply in particular that $\F$ contains the 
 images of all embeddings $k\into\Fpbar$). Hence the choice of coefficient field 
 is irrelevant, will occasionally be elided below, and will be 
 taken to be $\Fpbar$ from Section~\ref{sec:cryst-lifts-serre} onwards. 
  \end{remark}

Let $\SW(k,n)$ denote the
set of Serre weights for our fixed $k$ and $n$.
In the following paragraphs we give an explicit
description of this set.

Write $\Xn$
for the subset of $\Z^n_{+}$ consisting of tuples
$(a_i)$ such that $p-1 \ge a_i - a_{i+1}$ for all $1\le i \le n-1$.
If $a = (a_{\sigma,i}) \in (\Z^n_+)^{S_k}$, write $a_{\sigma}$ for the
component of $a$ indexed by $\sigma \in S_k$. Set $f = [k:\Fp]$, and
let $\sim$ denote the equivalence relation on $(\Z^n_+)^{S_k}$ in
which $a \sim a'$ if and only if there
exist integers $x_{\sigma}$ such that $a_{\sigma,i} - a'_{\sigma,i} =
x_{\sigma}$ for all $\sigma,i$ and for any
labeling $\sigma_j$ of the elements of $S_k$ such that $\sigma_j^p =
\sigma_{j+1}$ we have $\sum_{j=0}^{f-1} p^j x_{\sigma_j} \equiv 0
\pmod{p^f-1}$.
When
$k=\Fp$ we can omit the subscript $\sigma$, and the above equivalence
relation amounts to $a_i - a'_i = (p-1)y$ for some integer $y$, independent of $i$.

Given any $a\in \Xn$, we define the $k$-representation $P_a$ of $\GL_n(k)$ to be the
representation obtained by evaluating $\Ind_{B_n}^{\GL_n}(w_0 a)_{/\cO_K}$
on $k$ (so we have a natural $\GL_n(\cO_K)$-equivariant isomorphism $M_a \otimes_{\cO_K} k \isoto P_a$), 
and let $N_a$ be the irreducible sub-$k$-representation of $P_a$
generated by the highest weight vector (that this is indeed
irreducible follows from the analogous result for the algebraic group $\GL_n$,
cf.\ II.2.2--II.2.6 in \cite{MR2015057}, and the appendix to \cite{herzigthesis}).

If $a\in \XnSk$ then we define an
irreducible $\F$-representation $F_a$ of $\GL_n(k)$
by \[F_a:=\otimes_{\sigma\in\Sk}(N_{a_\sigma}\otimes_{k,\sigma}\F).\]
The representations $F_a$ are absolutely irreducible,  and every irreducible $\F$-represen\-tation of $\GL_n(k)$ is
of the form $F_a$ for some $a$ (see for example the appendix
to~\cite{herzigthesis}). Furthermore $F_{a} \cong
F_{a'}$ if and only if $a \sim a'$, and so the map $a \mapsto F_a$
gives a bijection from $\SWX$ to the set of Serre weights. We
identify the two sets $\SW(k,n)$ and $\SWX$ under this bijection,
and we 
refer to the elements of $\SWX$ as Serre weights. (We will abuse this
terminology in two specific ways: if $a \in \XnSk$ and $\W$ is a set
of weights, we may write $a
\in \W$ when literally we mean $F_a \in \W$, and we may
write ``the weight $a$'' when literally we mean ``the Serre weight represented
by $a$.'')

If $k = \Fp$ we will also write $F(a_{\sigma,1},\dots,a_{\sigma,n})$ for $F_a$, where $S_k = \{\sigma\}$.

\subsection{The Breuil--M\'ezard conjecture}\label{sec:bmc} By the main results of~\cite{kisindefrings}, for each Hodge type
$\lambda$ and inertial type
$\tau$ there is a unique reduced and $p$-torsion free quotient
  $R^{\lambda,\tau}_{\rhobar,\cO}$ of the universal lifting
  $\cO$-algebra $R_{\rhobar,\cO}$ which is characterised by the
  property that if $E'/E$ is a finite extension of fields, then an $\cO$-algebra homomorphism
    $R_{\rhobar,\cO}\to E'$ factors through
    $R^{\lambda,\tau}_{\rhobar,\cO}$ if and only if the
    corresponding representation $G_K\to\GL_n(E')$ is potentially
    crystalline of Hodge type $\lambda$ and inertial type $\tau$.  The
    ring  $R^{\lambda,\tau}_{\rhobar,\cO}[1/p]$ is regular by \cite[Thm.~3.3.8]{kisindefrings}.
    When $\cO$ is clear from the context, we will suppress it and
    write simply $R^{\lambda,\tau}_{\rhobar}$.  If $\tau$ is trivial we will write
$R^\lambda_{\rhobar}$ for
$R^{\lambda,\tau}_{\rhobar}$.

Given an inertial type $\tau$, there is a finite-dimensional smooth
irreducible $\Qpbar$-representation $\sigma(\tau)$ of $\GL_n(\cO_K)$
associated to $\tau$ by the ``inertial local Langlands
correspondence'', as in the following consequence of the results
of~\cite{MR1728541}, which is Theorem~3.7 of~\cite{EGPatching}.

\begin{thm}\label{thm: inertial local Langlands, N=0}If $\tau$ is an
  inertial type, then there is a finite-dimensional smooth irreducible
  $\Qpbar$-representation $\sigma(\tau)$ of $\GL_n(\cO_K)$  such that
  if $\pi$ is any irreducible smooth
  $\Qpbar$-representation of $G$, 
  then the restriction of $\pi$ to $\GL_n(\cO_K)$ contains \mbox{\upl an}
  isomorphic copy of\upr\ $\sigma(\tau)$ as a subrepresentation if
  and only if $\rec_p(\pi)|_{I_K}\sim\tau$ and $N=0$ on
  $\rec_p(\pi)$. Furthermore, in this case the restriction of $\pi$ to
  $\GL_n(\cO_K)$ contains a unique copy of $\sigma(\tau)$. 
  \end{thm}

\begin{rem}
  \label{rem:trivial type}In particular, if $\tau$ is the trivial
  inertial type, then $\sigma(\tau) \cong \Qpbar$ is the trivial
  one-dimensional representation of~$\GL_n(\cO_K)$. 
\end{rem}
  \begin{rem}
    \label{rem: types aren't unique}In general the type $\sigma(\tau)$ need not
    be unique, although it is a folklore conjecture (which is known for $n=2$,
    see Henniart's appendix to~\cite{MR1944572}) that $\sigma(\tau)$ is unique
    if $p>n$. The Breuil--M\'ezard conjecture, as formulated below, should hold
    for any choice of $\sigma(\tau)$; indeed it seems plausible that the
    semisimplification
    of the reduction mod $p$ of $\sigma(\tau)$ does not depend on any choices
    (this is the case when $n=2$ by
    Proposition~4.2 of~\cite{breuildiamond}).
  \end{rem}
  Enlarging $E$ if necessary, we may assume that $\sigma(\tau)$ is
  defined over $E$. Since it is a finite-dimensional representation of
  the compact group $\GL_n(\cO_K)$, it contains a
  $\GL_n(\cO_K)$-stable $\cO$-lattice $L_\tau$.  Set
  $L_{\lambda,\tau}:=L_\tau\otimes_\cO L_\lambda$, a finite free
  $\cO$-module with an action of $\GL_n(\cO_K)$. Then we may
  write \[(L_{\lambda,\tau}\otimes_\cO\F)^{\semis}\isoto\oplus_a
  F_a^{n_{\lambda,\tau}(a)},\]where the sum runs over Serre
  weights $a \in \SW(k,n)$, and 
 the $n_{\lambda,\tau}(a)$ are non-negative
  integers. Then we have the following conjecture.

\begin{conj}[The generalised Breuil--M\'ezard conjecture]
  \label{conj: generalised BM conjecture}There exist non-negative integers
  $\mu_a(\rhobar)$ depending only on $\rhobar$ and $a$ such that for
  all Hodge types $\lambda$ and inertial types $\tau$ we have $e(R^{\lambda,\tau}_{\rhobar}/\varpi)=\sum_an_{\lambda,\tau}(a)\mu_a(\rhobar)$.
\end{conj}

Here $\varpi$ is a uniformiser of $\cO$. The finitely many integers $\mu_a(\rhobar)$ are in fact hugely
overdetermined by the infinitely many equations
$e(R^{\lambda,\tau}_{\rhobar}/\varpi)=\sum_a n_{\lambda,\tau}(a)\mu_a(\rhobar)$.
 We will return to this point in the next subsection.

\begin{rem}
  \label{rem:enlarge-coefficients}
 The multiplicities $\mu_a(\rhobar)$ in Conjecture~\ref{conj: generalised BM
    conjecture} will be
  independent of the coefficient field $E$, in the following sense.
 Let $E'$ be a finite extension of $E$, with ring of integers $\cO'$
  and residue field $\F'$.   Write $\rhobar' = \rhobar \otimes_{\F}
  \F'$ and $\tau' = \tau \otimes_{E} E'$.   One knows 
  (\cite[Lem.~1.2.1 and \S1.4]{BLGGT}) that there is an isomorphism
  $R^{\lambda',\tau'}_{\rhobar',\cO'}  \cong
  R^{\lambda,\tau}_{\rhobar,\cO} \otimes_{\cO} \cO'$.  It
  follows that if Conjecture~\ref{conj: generalised BM
    conjecture} holds then $\mu_a (\rhobar') = \mu_a(\rhobar)$ for all
  $a \in \SW(k,n)$.
\end{rem}

The generalised Breuil--M\'ezard conjecture is almost completely
understood when $n=2$ and $K=\Qp$ i.e.\ in the setting originally
studied and conjectured by Breuil and M\'ezard \cite{MR1944572}\footnote{Breuil and M\'ezard restricted their original
conjecture to the case of Hodge types $\lambda = (r,0)$ with $0 \le r
\le p-3$, due to the lack of a suitable integral $p$-adic Hodge theory
at the time, and considered potentially semistable deformation rings. The conjecture was later
extended to
arbitrary Hodge types and adapted to the potentially crystalline setting by Kisin~\cite{kisinICM}.}. In
fact, it is completely understood in this setting when $p>3$
\cite{kisinfmc,PaskunasBM,HuTan}; when 
$p=2,3$ it is known in all cases except when the representation~$\rhobar$ is
  reducible and the  characters on the diagonal of $\rhobar$ have
  ratio~$\varepsilonbar$ 
($= \varepsilonbar^{-1}$ when $p \le 3$)    \cite{PaskunasBMatTwo, kisinfmc,
    SanderScalar}.   The multiplicities $\mu_a(\rhobar)$ are
described in most cases (those for which $\rhobar$ has only scalar
endomorphisms) in~\cite[\S 2.1.2]{MR1944572}, and in general in~\cite[\S
1.1]{kisinfmc} together with~\cite{sandermultiplicities}.

Assuming Conjecture~\ref{conj: generalised BM conjecture}, we make the
following definition and conjecture.
\begin{defn}
  \label{defn:Serre weights predicted by BM}We define $\WBM(\rhobar)$,
  the \emph{Breuil--M\'ezard predicted weights for} $\rhobar$, to be
  the set of Serre weights $a$ such that $\mu_a(\rhobar)>0$.
\end{defn}

\begin{conj}
  \label{conj:BM-version} In the weight part of Serre's conjecture
  \upl Conj.~\ref{conj:basic-conjecture}\upr,
  we may take $\W_v(\rbar) = \WBM(\rbar|_{G_{F_v}})$.
\end{conj}

In Section~\ref{sec:patching, BM and serre weights} below we will
explain how the formalism of the Taylor--Wiles--Kisin patching method
shows that this is a natural definition for the set of predicted Serre
weights.

\subsection{Breuil--M\'ezard systems}
\label{subsec:known cases of BM}

We now describe a family of variants of
Conjectures~\ref{conj: generalised BM conjecture}
and~\ref{conj:BM-version}.
Let $\CS$ be a set of pairs $(\lambda,\tau)$ such that
$\lambda$ is a Hodge type and $\tau$ is an inertial type (both for our
fixed $K$ and $n$).  
We say that $\CS$ is a \emph{Breuil--M\'ezard system} if the map
$\Z^{\SW(k,n)} \to \Z^{\CS}$ given by the formula 
$$(x_a)_{a \in \SW(k,n)} \mapsto \left(\sum_a n_{\lambda,\tau}(a)
x_a\right)_{(\lambda,\tau) \in \cS}$$
is injective; in particular, if $\CS$ is a Breuil--M\'ezard system then for each representation
$\rhobar$ the equations $e(R^{\lambda,\tau}_{\rhobar}/\varpi) = \sum_a
n_{\lambda,\tau}(a) \mu_a(\rhobar)$, regarded as a system of linear
equations in the variables $\mu_a(\rhobar)$, can have at most one solution.

We remark that $\CS$ is a Breuil--M\'ezard system if and only if the image of the
map $\Z[\CS] \to K_0(\mathrm{Rep}_{\F}(\GL_n(k)))$ sending  $(\lambda,\tau) \mapsto
[L_{\lambda,\tau} \otimes_{\cO} \F]$ has finite index. (Here we write
$K_0(\mathrm{Rep}_{\F}(\GL_n(k)))$  for the Grothendieck group of
finite-dimensional $\F[\GL_n(k)]$-modules.)  Indeed, if
$\CS$ is finite then this is precisely the dual of the definition in
the previous paragraph; in general, for the `only if' direction
one simply notes that any
Breuil--M\'ezard system contains a finite Breuil--M\'ezard system, and
similarly for the `if' direction.

\begin{ex}
  \label{ex:barsotti-tate-system}
  Take $n=2$ and let $\mathrm{BT}$ be the set of pairs $(0,\tau)$, so that
  $\mathrm{BT}$ is the set of potentially Barsotti--Tate types.  Then
  \cite[Lem.~3.5.2]{geekisin} shows that $\mathrm{BT}$ is a
    Breuil--M\'ezard system, and indeed that this is true even if we restrict
    to types $\tau$ such that $\det \tau$ is tame.
\end{ex}

To give another example, we make the
following definition.

\begin{defn}
  \label{defn:lift of a serre weight}
We say that an element
  $\lambda\in(\Z^n_+)^{\SK}$ is a \emph{lift} of an element
  $a\in(\Z^n_+)^{\Sk}$ if for each $\sigma \in \Sk$ there exists
  $\emb_\sigma \in \SK$ lifting $\sigma$ 
  such that $\lambda_{\emb_\sigma}=a_\sigma$, and $\lambda_{\emb'} =
  0$ for all other $\emb'\neq \emb_\sigma$
  in $\SK$ lifting $\sigma$.  In that case we may say that the lift
  $\lambda$  is taken with respect to the choice of embeddings $(\emb_\sigma)$.
When $a \in (\Xn)^{S_k}$  we will also say that $\lambda$ is a
  \emph{lift}   (with respect to the choice of embeddings $(\emb_{\sigma})$) of  the Serre weight represented by $a$.
\end{defn}

\begin{ex}
  \label{ex:crystalline-systems}
  Fix a lift $\lambda_b$ for each Serre weight $b$, and let
  $\widetilde{\Cr}$ be the set of pairs $(\lambda_b,\mathrm{triv})$,
  where $\mathrm{triv}$ denotes the trivial type.  Then
  $\widetilde{\Cr}$ is a Breuil--M\'ezard system, because Lemma~\ref{lem:correction to EG 4.1.1.} below shows
  (inductively) that the natural map $\Z[\widetilde{\Cr}] \to
  K_0(\mathrm{Rep}_{\F}(\GL_n(k)))$ is surjective.

\end{ex}

\begin{defn}\label{defn: norm on weights}
  For $a\in(\Z^n_+)^{S_k}$, let
  $\|a\|:=\sum_{i,\sigma}(n+1-2i)a_{\sigma,i}\in\Z_{\ge 0}$.
\end{defn}

\begin{lem}
  \label{lem:correction to EG 4.1.1.}If $\lambda$ is a lift of $a\in (\Xn)^{S_k}$, then
  $L_\lambda\otimes_{\O}\F$ has socle $F_a$, and every other Jordan--H\"older factor
  of $L_\lambda\otimes_{\O}\F$ is of the form $F_b$ with $b\in (X_1^{(n)})^{S_k}$ and $\|b\|<\|a\|$.
\end{lem}

To prove Lemma~\ref{lem:correction to EG 4.1.1.} it is best to work
not with the group $\GL_n$ over $k$ but rather its restriction of
scalars to $\Fp$. For this reason we defer the proof until
Section~\ref{sec:proof-of-correction-to-EG}. However, we stress that
Lemma~\ref{lem:correction to EG 4.1.1.} will only be used in our
discussion of the Breuil--M\'ezard system $\widetilde{\Cr}$.

In the following conjectures and definition, we let $\CS$ be a Breuil--M\'ezard system.

\begin{conj}[The Breuil--M\'ezard conjecture for representations of
  type $\CS$]
  \label{conj: BM conjecture of type S} There exist non-negative integers
  $\mu_a(\rhobar)$ depending only on $\rhobar$ and $a$ such that for
  all $(\lambda,\tau) \in \CS$ we have $e(R^{\lambda,\tau}_{\rhobar}/\varpi)=\sum_an_{\lambda,\tau}(a)\mu_a(\rhobar)$.
\end{conj}

\begin{defn}
  \label{defn:Serre weights predicted by BM of type S}
  Suppose that the Breuil--M\'ezard
  conjecture for representations of type $\CS$
  is true for $\rhobar$.  We define
  $\W_{\CS}(\rhobar)$ to be
  the set of Serre weights $a$ such that $\mu_a(\rhobar)>0$.
\end{defn}

\begin{conj}[The $\CS$-weight part of Serre's conjecture]
  \label{conj:BM-type-S-version}
  Suppose that the Breuil--M\'ezard
  conjecture for representations of type $\CS$ \upl Conj.~\ref{conj: BM
    conjecture of type S}\upr\ is true.  Then the weight part of
  Serre's conjecture \upl Conj.~\ref{conj:basic-conjecture}\upr\
  holds with $\W_v(\rbar) = \W_{\CS}(\rbar|_{G_{F_v}})$.
\end{conj}

Of course if the generalised Breuil--M\'ezard conjecture (Conj.~\ref{conj: generalised BM conjecture}) holds,
then so does the Breuil--M\'ezard conjecture for representations of
any type $\CS$, and
in that case we must always have $\WBM(\rhobar) =
\W_{\CS}(\rhobar)$. In particular, if we believe Conjecture~\ref{conj:
  generalised BM conjecture} (and as we explain in
Section~\ref{sec:patching, BM and serre weights} below, we certainly
\emph{should}  believe Conjecture~\ref{conj: generalised BM
  conjecture}!), then the Breuil--M\'ezard predicted weights for~$\rhobar$ are completely
determined by information about the crystalline lifts of~$\rhobar$ of
bounded Hodge--Tate weights.

\begin{ex}
  \label{ex:pot-BT-known}    Assume that $p > 2$.  Gee and Kisin \cite[Cor.~3.5.6]{geekisin}
  have established the Breuil--M\'ezard conjecture for potentially
  Barsotti--Tate representations; that is, they
  have shown 
  that Conjecture~\ref{conj: BM conjecture of type S} holds for the
  system $\mathrm{BT}$ of Example~\ref{ex:barsotti-tate-system}.  In
  fact they also prove (subject to a Taylor--Wiles-type hypothesis)
  that the $\mathrm{BT}$-weight part of Serre's conjecture 
  holds in this setting, i.e.\ that the
  analogue of Conjecture~\ref{conj:basic-conjecture} for quaternion algebras or
  forms of $\mathrm U(2)$ over totally real fields holds with
   $\W_v(\rbar) = \WBT(\rbar|_{G_{F_v}})$ (\cite[Cor.~4.5.4]{geekisin});
  see
 also  the discussion in
  Section~\ref{ss:relationship} of this paper.
\end{ex}

\begin{ex}
  \label{ex:crystalline-known}
  Let $\widetilde{\Cr}$ be one of the Breuil--M\'ezard
  systems of Example~\ref{ex:crystalline-systems}.  Then a weak version of
  the Breuil--M\'ezard conjecture for representations of type
  $\widetilde{\Cr}$ is trivially true; namely, there are uniquely
  determined integers~$\mu_a(\rhobar)$ satisfying the required
  equations, but it is not clear that these integers are non-negative. (Since the
  system~$\widetilde{\Cr}$ is in bijection with the set of Serre
  weights, it is immediate that there are  uniquely
  determined rational numbers~$\mu_a(\rhobar)$ satisfying the required
  equations, and that they are in fact integers follows easily from Lemma~\ref{lem:correction to EG 4.1.1.}.) 
  
If $n=2$, it follows trivially that the $\mu_a(\rhobar)$ are indeed
non-negative integers (so that the Breuil--M\'ezard conjecture holds for
representations of type  $\widetilde{\Cr}$), and that $a \in \W_{\widetilde{\Cr}}(\rhobar)$ if and
  only if $\rhobar$ has a crystalline lift of Hodge type $\lambda_a$
  (for our chosen lift $\lambda_a$ of $a$).  

The set
  $\W_{\widetilde{\Cr}}(\rhobar)$ \emph{a priori} could depend on the
  choice of lifts in the construction of $\widetilde{\Cr}$.  However,
  if $n=2$ and $p >
  2$, it is proved in \cite[Thm.~6.1.8]{gls13} that the set
  $\W_{\widetilde{\Cr}}(\rhobar)$ is independent of these choices, and indeed is
  equal to $\WBT(\rhobar)$ and thus (as explained in
  Example~\ref{ex:pot-BT-known}) under a mild Taylor--Wiles hypothesis the
  analogue of Conjecture~\ref{conj:basic-conjecture} for quaternion algebras or
  forms of $\mathrm U(2)$ over totally real fields holds with
  $\W_v(\rbar) = \W_{\widetilde{\Cr}}(\rbar|_{G_{F_v}})$.
\end{ex}

\section{Patching functors and the Breuil--M\'ezard
  formalism}\label{sec:patching, BM and serre weights}\subsection{Patching functors}The most general results
available to date on the weight part of Serre's conjecture have been based on
the method of Taylor--Wiles patching (see, for example,~\cite{blggU2}
and~\cite{geekisin}). In this section, we give a general formalism for
these arguments, and we  explain how the resolution of the weight part of
Serre's conjecture for Hilbert modular forms in~\cite{geekisin,blggU2,gls13}
fits into this framework.

The formalism we have in mind is a generalisation of the one employed
in~\cite{geekisin}, which in turn is based on Kisin's work on the
Breuil--M\'ezard conjecture~\cite{kisinfmc}. Since our aim in this paper is not
to prove new global theorems, but rather to explain what we believe should be true, we
avoid making specific Taylor--Wiles patching arguments, and instead use the
abstract language of patching functors, originally introduced for $\GL_2$
in~\cite{emertongeesavitt}. Our patching functors will be for $\GL_n$, and will
satisfy slightly different axioms from those in~\cite{emertongeesavitt}, but are
motivated by the same idea, which is to abstract the objects produced by
Taylor--Wiles patching. In practice one often wants to consider all places above $p$
at once, but for simplicity of notation we will work at a single place in this
section.

Continue to work in the context of Section~\ref{sec:BM}, so that we have a fixed
 representation $\rhobar:G_K\to\GL_n(\F)$. Fix some $h\ge 0$, and write
$R_\infty:=R_{\rhobar}[[x_1,\dots,x_h]]$, $X_\infty:=\Spf R_\infty$. (In
applications, the $x_i$ will be the auxiliary variables that arise in the
Taylor--Wiles method; they will be unimportant in our discussion, and the reader
unfamiliar with the details of the Taylor--Wiles method will lose nothing by
assuming that $h=0$.) We write
$R_{\infty}^{\lambda,\tau}:=R^{\lambda,\tau}_{\rhobar}[[x_1,\dots,x_h]]$ and
$X_\infty(\lambda,\tau):=\Spf R_{\infty}^{\lambda,\tau} $. Write $\Xbar_\infty$ and
$\Xbar_\infty(\lambda,\tau)$ for the special fibres of $X_\infty$ and
$X_\infty(\lambda,\tau)$ respectively. Write $d+1$ for the dimension of the non-zero
$X_\infty(\lambda,\tau)$ (which is independent of the choice of $\lambda,\tau$).

Let $\cC$ denote the category of finitely generated $\cO$-modules with a
continuous action of $\GL_n(\cO_K)$; in particular, we have $L_{\lambda,\tau}\in\cC$ for any
$\lambda,\tau$. Fix a Breuil--M\'ezard system $\cS$ in the sense of Section~\ref{subsec:known cases of BM}.

\begin{defn}\label{defn:patching functor}
  A \emph{patching functor} for $\cS$ is a non-zero covariant exact functor $M_\infty$ from $\cC$
  to the category of coherent sheaves on $X_\infty$, with the properties that:
  \begin{itemize}
  \item for all pairs $(\lambda,\tau)\in\cS$, the sheaf $M_\infty(L_{\lambda,\tau})$ is
    $p$-torsion free and has scheme-theoretic support $X_\infty(\lambda,\tau)$, and in fact is
    maximal Cohen--Macaulay over $X_\infty(\lambda,\tau)$;
  \item for all Serre weights $F_a$, the support $\Xbar_\infty(F_a)$ of the sheaf $M_\infty(F_a)$ either
    has dimension $d$ or is empty;
\item the (maximal Cohen--Macaulay over a regular scheme, so) locally
  free sheaf $M_\infty(L_{\lambda,\tau})[1/p]$  has rank one over the generic
  fibre of $X_\infty(\lambda,\tau)$.
  \end{itemize}

\end{defn}
\begin{rem}\label{rem: patching functors from patching}
      In practice, examples of patching functors $M_\infty$ come from the
      Taylor--Wiles--Kisin patching method applied to spaces of automorphic
      forms, localised at a maximal ideal of a Hecke algebra which corresponds
      to a global Galois representation $\rbar$ which locally at some place
      above $p$ restricts to give $\rhobar$. For example, the functor
      $\sigma^\circ\mapsto M_\infty(\sigma^\circ)$ defined
      in~\cite[\S4]{EGPatching} is conjecturally a patching functor; the only
      difficulty in verifying this is that the usual Auslander--Buchsbaum
      argument only shows that $M_\infty(L_{\lambda,\tau})$ is maximal
      Cohen--Macaulay over its support, which is a union of
      irreducible components of the generic fibre of $X_\infty(\lambda,\tau)$.

Showing that this support is in fact the whole of $X_\infty(\lambda,\tau)$ is
one of the major open problems in the field; it is closely related to the Fontaine--Mazur conjecture, and is
therefore strongly believed to hold in general. By the main results
of~\cite{BLGGT}, this is known whenever all potentially crystalline
representations of Hodge type $\lambda$ and inertial type $\tau$ are potentially
diagonalisable, but this condition seems to be hard to verify in
practice.

  \end{rem}
  \begin{rem}
    \label{rem:mult one minimal patching functor}The assumption that
    $M_\infty(L_{\lambda,\tau})[1/p]$ has rank one corresponds to the
    notion of a minimal patching functor
    in~\cite{emertongeesavitt}. The following arguments go through
    straightforwardly if one allows the rank to be higher, and in
    applications coming from Taylor--Wiles--Kisin patching, it is
    occasionally necessary to allow this (due to the need to ensure
    that the tame level is sufficiently small when the image of the
    global Galois representation is also small), but it makes no
    essential difference to the discussion below. However, these cases
    are rare, and in particular the patching constructions
    of~\cite{EGPatching} give examples where the rank is one.
  \end{rem}

\subsection{The relationship to the Breuil--M\'ezard conjecture} \label{ss:relationship}The connection between patching functors and Serre weights is the following
result, which is an abstraction of one of the main ideas of~\cite{geekisin}.
\begin{prop}\label{prop: patching functors imply BM weights}If a patching
  functor for $\cS$ exists, then the Breuil--M\'ezard conjecture for
  representations of type $\CS$ \upl Conj.~\ref{conj: BM conjecture of type S}\upr\
  holds, and the set $\W_{\cS}(\rhobar)$ is precisely the set of weights $\sigmabar$
  for which $M_\infty(\sigmabar)\ne 0$.
  \end{prop}

  \begin{proof}
Let $M_\infty$ be a patching functor for $\cS$. The  $\Xbar_\infty(\lambda,\tau)$ are all equidimensional of 
dimension $d$ by~\cite[Lem.\ 2.1]{BreuilMezardRaffinee}. (Strictly speaking the 
context of \emph{loc.\ cit.} has $n=2$, but its
proof is completely general.)  By \cite[Thm.\ 14.6]{MR1011461} we have 
\begin{equation}\label{eq:matsumura} 
e(M_\infty(L_{\lambda,\tau}\otimes_\cO\F),\Xbar_\infty(\lambda,\tau))=\sum_a
n_{\lambda,\tau}(a)e(M_\infty(F_a),\Xbar_\infty(\lambda,\tau))
\end{equation}
(noting that $M_\infty(F_a)$ is supported on $\Xbar_\infty(\lambda,\tau)$ whenever $n_{\lambda,\tau}(a) > 0$).
Now from \cite[Thm.\ 14.7]{MR1011461} 
it follows that if $A \onto B$ is a surjection of
Noetherian local rings of the same dimension and $M$ is a finitely
generated $B$-module, then $e(B,M) = e(A,M)$, where on the right-hand
side $M$ is regarded as an $A$-module via the given map. If
$\Xbar_\infty(F_a)$ is non-empty, then from the definition of a
patching functor it has dimension $d$, and it follows that
\begin{equation}\label{eq:mult-component}
n_{\lambda,\tau}(a)e(M_\infty(F_a),\Xbar_\infty(\lambda,\tau)) =
n_{\lambda,\tau}(a)e(M_\infty(F_a),\Xbar_\infty(F_a)).   
\end{equation}
If we make the convention that $e(0, \varnothing) = 0$, then 
\eqref{eq:mult-component} holds in
general, since the left-hand side is $0$ when $\Xbar_\infty(F_a)$ is empty.

From the third bullet point in the definition of a patching functor,
we know that $M_\infty(L_{\lambda,\tau})_{\p}$ is free of rank $1$ for
any minimal prime $\p$ of $R_\infty^{\lambda,\tau}$ (note that the
latter ring has no $p$-torsion).  Let $S := R_\infty^{\lambda,\tau}\setminus \bigcup_{\p} \p$,
the union being taken over all minimal primes of
$R_\infty^{\lambda,\tau}$. We have $S^{-1} M_\infty(L_{\lambda,\tau})
\cong \prod_{\p} M_\infty(L_{\lambda,\tau})_{\p}$ as an $S^{-1}
R_\infty^{\lambda,\tau} \cong \prod_{\p}
  (R_\infty^{\lambda,\tau})_{\p}$-module, the products again being taken over all minimal primes of
$R_\infty^{\lambda,\tau}$. Hence we can find find $m' \in
M_\infty(L_{\lambda,\tau})$ such that for any such $\p$, the image of
$m'$ in $ M_\infty(L_{\lambda,\tau})_{\p}$ is a basis as a
$(R_\infty^{\lambda,\tau})_{\p}$-module.  It follows
that~\cite[Proposition 1.3.4]{kisinfmc}(2) applies with $A = M =
R_\infty^{\lambda,\tau}$, $M' = M_\infty(L_{\lambda,\tau})$, $G = 1$,
$x = \varpi$, and $f : M \to M'$ the map sending $1 \mapsto m'$, from
which we find that 
 \begin{equation}\label{eq:mult-compare}
e(M_\infty(L_{\lambda,\tau}\otimes_\cO\F),\Xbar_\infty(\lambda,\tau))=e(R^{\lambda,\tau}_{\rhobar}/\varpi).\end{equation}
(Note
that
$e(\Xbar_\infty(\lambda,\tau))=e(R^{\lambda,\tau}_\infty/\varpi)=e(R^{\lambda,\tau}_{\rhobar}/\varpi)$.)

Putting together equations \eqref{eq:matsumura}, \eqref{eq:mult-component}, and
\eqref{eq:mult-compare}, we find that Conjecture~\ref{conj: BM conjecture of type S}
  holds with \[\mu_a(\rhobar):=e(M_\infty(F_a),\Xbar_\infty(F_a)).\] By the
  definition of a Breuil--M\'ezard system, the $\mu_a(\rhobar)$ are uniquely
  determined.
 Finally, it follows from~\cite[Formula~14.2]{MR1011461} that $e(M_\infty(F_a),\Xbar_\infty(F_a))>0$ if and only if
  $M_\infty(F_a)\ne 0$, and the result follows.
      \end{proof}
      \begin{rem}\label{rem: patching functors and GK}
        In the cases that $M_\infty$ arises from the
        Taylor--Wiles--Kisin patching construction, $M_\infty(F_a)$
        corresponds to (patched) spaces of mod $p$
        automorphic forms of weight~$a$, and it is immediate from the definition
        that $M_\infty(F_a)\ne 0$ if and only if $\rbar$ is
        automorphic of Serre weight
        $a$. Thus in cases where it can be shown that the Taylor--Wiles--Kisin
        method gives a patching functor for $\cS$ (which, as explained in Remark~\ref{rem:
          patching functors from patching}, amounts to showing that the support
        of the $M_\infty(L_{\lambda,\tau})$ is as large as possible),
        the $\CS$-weight part of Serre's conjecture 
        (Conj.~\ref{conj:BM-type-S-version}) follows from Proposition~\ref{prop: patching functors imply BM weights}.

As explained in Remark~\ref{rem: patching functors from patching},
        it is not in general known that potentially crystalline representations
        are potentially diagonalisable, which limits the supply of patching
        functors for general Breuil--M\'ezard systems.

The situation is better when $n=2$, and indeed as a result of the papers~\cite{geekisin}, \cite{blggU2}
and~\cite{gls13}, it is now known that if $p>2$, then
$\WBT(\rhobar)=\W_{\widetilde{\Cr}}(\rhobar)$, where $\mathrm{BT}$ is the 
Breuil--M\'ezard system of Example~\ref{ex:barsotti-tate-system}, and
$\widetilde{\Cr}$ is any of the
Breuil--M\'ezard systems of
Example~\ref{ex:crystalline-systems}; and it is known that the analogue of Conjecture~\ref{conj:basic-conjecture} for quaternion algebras or forms of
$\mathrm U(2)$ over totally real
fields holds for this set of weights.

We briefly recall the argument. By the results of~\cite{kis04,MR2280776}
potential diagonalisability is known for the system $\mathrm{BT}$, and
Proposition~\ref{prop: patching functors imply BM weights} (applied to the
Taylor--Wiles--Kisin patching method for automorphic forms on suitable
quaternion algebras or forms of $\mathrm U(2)$) then implies the
result of~\cite{geekisin} discussed in Example~\ref{ex:pot-BT-known}. Indeed, as
we have already explained, Proposition~\ref{prop: patching functors imply BM
  weights} is an abstraction of the arguments of~\cite{geekisin}. It remains to
show that $\WBT(\rhobar)=\W_{\widetilde{\Cr}}(\rhobar)$. 

Since this question is purely local, it suffices to work in the $\mathrm U(2)$
setting, where it is essentially immediate (by the considerations explained in
Example~\ref{example: crystalline lifts for modular forms and more generally})
that $\WBT(\rhobar)\subset\W_{\widetilde{\Cr}}(\rhobar)$
(note that by the previous paragraph, the left hand side is known at this
point in the argument to be the set of weights that occur globally). The purely local
results of~\cite{gls13}, coming from a detailed study of the underlying integral
$p$-adic Hodge theory, show that if $F_a\in\W_{\widetilde{\Cr}}(\rhobar)$, then
$\rhobar$ necessarily has a potentially diagonalisable crystalline lift of Hodge
type $\lambda_{\ta}$. The above machinery then shows that
$\W_{\widetilde{\Cr}}(\rhobar)\subset\WBT(\rhobar)$ (again using that the right hand
side is the set of weights that occur globally; this part of the argument is
carried out in~\cite{blggU2}), as required.
      \end{rem}

\section{Crystalline lifts and Serre weights}
\label{sec:cryst-lifts-serre}

The Breuil--M\'ezard version of the weight part of Serre's conjecture (Conjecture~\ref{conj:BM-version})
has the obvious drawback that even the definition of the conjectural set of weights
$\WBM(\rhobar)$ is contingent on Conjecture~\ref{conj: generalised
  BM conjecture}.  (Of course, in theory it is possible to determine the
conjectural values of the $\mu_a(\rhobar)$'s without proving
the generalised Breuil--M\'ezard conjecture first, by computing
$e(R^{\lambda,\tau}_{\rhobar}/\varpi)$ for enough choices of
$\lambda$ and $\tau$, but in practise this seems to be very
difficult.)  In this section we will, under the assumption that
$\rhobar|_{I_K}$ is semisimple, define another 
conjectural set of Serre weights in terms of crystalline lifts.
Although this set of weights may not be any more computable than
$\WBM(\rhobar)$, its definition will not depend on any unproven conjectures,
and perhaps more importantly it will provide a bridge between the
Breuil--M\'ezard description of the set of Serre weights and a much more
explicit set of Serre weights to be defined in
Section~\ref{sec:obvious-semisimple}.

It is perhaps also worth recalling that, although we have emphasized
the Breuil--M\'ezard perspective in this article, the crystalline
lifts perspective historically came first. Indeed, the original explicit description of weights given in
\cite{MR885783} can in retrospect be understood as the most optimistic
conjecture that one could make given the constraints provided by known results
on the reduction mod $p$ of the crystalline representations associated to
modular forms, and similarly the conjecture of~\cite{bdj} arose from the
consideration of crystalline lifts via Fontaine--Laffaille
theory. Unfortunately, when $n>2$ it seems that (contrary to the
conjectures made in~\cite{gee061}) the obvious extension of these
conjectures to the general case that $\rhobar|_{I_K}$ is not semisimple is false, and 
it now seems likely that a precise description of the sets of
weights in general will be extremely complicated; see
Sections~\ref{subsec: crystalline lifts in the picture}
and~\ref{sec:obvious-general} below.

\subsection{Crystalline lifts}\label{subsec:crystalline} 

We fix a finite extension $K/\Qp$ and a representation
$\rhobar : G_K \to \GL_n(\Fpbar)$.

\begin{remark}
  \label{rem:coefficient-switch}
    Note that we have now switched (for the remainder of the paper) to
  working with $\rhobar$ whose coefficients are algebraically closed. 
  By Remark~\ref{rem:enlarge-coefficients} it
  still makes sense to speak of $\WBM(\rhobar)$:\ choose any
  sufficiently large finite extension $\F/\Fp$ such that $\rhobar$ has
  a model $\rhobar_{\F}$ over $\F$, set $\mu_a(\rhobar) = \mu_a(\rhobar_{\F})$, and take
  $\WBM(\rhobar) = \{ a : \mu_a(\rhobar) > 0\}$ as usual. (Recall that
  Serre
  weights can equally well be taken to be defined over $\Fpbar$; cf.~Remark~\ref{rem:coeff-irrelevant-serre-wt}.)
Similarly for any
  Hodge type $\lambda$ we can write $R^{\lambda}_{\rhobar} =
  R^{\lambda}_{\rhobar,\cO} \otimes_{\cO} \Zpbar$ for any sufficiently
  large $\cO$, and
  Remark~\ref{rem:enlarge-coefficients} again shows that this is
  well-defined. Correspondingly, in this section $L_\lambda$ will mean
$L_{\lambda,\cO} \otimes_{\cO} \Zpbar$ for any sufficiently
  large $\cO$.
\end{remark}

\begin{defn}
  \label{defn:crystalline-lift}
  Suppose that $\lambda \in (\Z_+^n)^{S_K}$.  A \emph{crystalline lift}
  of $\rhobar$ of Hodge type $\lambda$ is a representation
  $\rho : G_K \to \GL_n(\Zpbar)$ such that
  \begin{itemize}
  \item $\rho \otimes_{\Zpbar} \Fpbar \cong\rhobar$ and
  \item $\rho \otimes_{\Zpbar} \Qpbar$ is crystalline and regular of weight $\lambda$.
  \end{itemize}
\end{defn}

To motivate our reformulation of the weight part of Serre's conjecture
in terms of crystalline lifts, we consider the following lemma.

\begin{lemma}
  \label{lem:weight-implies-lift}
  Assume that the generalised Breuil--M\'ezard conjecture \upl
  Conj.~\ref{conj: generalised BM conjecture}\upr\ holds. Then
 $\rhobar$ has a crystalline lift of Hodge type $\lambda$ if and only if
 $\WBM(\rhobar) \cap \JH_{\GL_n(k)} (L_{\lambda} \otimes_{\Zpbar} \Fpbar) \ne \varnothing$.
\end{lemma}

\begin{proof}
Since a  representation $\rhobar : G_K \to \GL_n(\Zpbar)$
has image contained in $\GL_n(\cO')$ for some finite $\cO'/\Zp$, it
follows that  $\rhobar$ has a crystalline lift of Hodge type $\lambda$ if
and only if $R^\lambda_{\rhobar}\ne 0$.
Under the assumption of
Conjecture~\ref{conj: generalised BM conjecture}, this is equivalent to there
being  a Jordan--H\"older factor $F_a$  of
  $L_{\lambda} \otimes_{\Zpbar} \Fpbar$ such that 
$\mu_a(\rhobar) > 0$, which by definition is equivalent to $a \in \WBM(\rhobar)$.
\end{proof}

\begin{cor}
  \label{cor:weight-implies-lift}
  Assume the generalised Breuil--M\'ezard conjecture \upl
  \mbox{Conj.~\ref{conj: generalised BM conjecture}\upr} holds, and let $\lambda$ be a
  lift of the Serre weight $a$.  If  $a \in \WBM(\rhobar)$, then
 $\rhobar$ has a crystalline lift of Hodge type $\lambda$.
\end{cor}

\begin{proof}
Suppose that $\lambda$ is a lift of $a$ with respect to the lift
$(\emb_\sigma)$ of $\Sk$.
 From Definition~\ref{defn:lift of a serre weight} we see that
 $L_{\lambda} = 
\otimes_{\sigma \in S_k} \ M_{a_{\sigma}}
\otimes_{\cO_K,\emb_{\sigma}} \Zpbar$,
and so
$L_{\lambda} \otimes_{\Zpbar} \Fpbar \cong  \otimes_{\sigma \in S_k} \ P_{a_{\sigma}}
\otimes_{k,\sigma} \Fpbar$.
In
particular $L_{\lambda} \otimes_{\Zpbar} \Fpbar$ has $F_a$ as a
Jordan--H\"older factor and Lemma~\ref{lem:weight-implies-lift} applies.
\end{proof}

We are thus led to make the following definition.

\begin{defn}
  \label{defn:crystalline-weight-sets}
  We define $\Wcrisexists(\rhobar)$, the \emph{crystalline 
    weights for $\rhobar$}, to be the set of Serre weights $a$ such that
 the representation $\rhobar$ has a crystalline lift of Hodge type $\lambda$ for
  some lift $\lambda$ of $a$.  We further define $\Wcrisforall(\rhobar)$
  to be the set of Serre weights $a$ such
  that $\rhobar$ has a crystalline lift of Hodge type $\lambda$ for
  every lift $\lambda$ of $a$.  
\end{defn}

It is not difficult to check that this definition
is reasonable in the following sense:\
let $a \in \SWX$ be a Serre weight, and suppose that $\lambda =
(a_{\sigma,i})$ and $\lambda' = (a'_{\sigma,i})$ are two lifts of $a$
to $(\Z_+^n)^{\SK}$, 
each taken with respect to the same choice of embeddings
$(\emb_\sigma)$;\ then $\rhobar$ has a
crystalline lift of Hodge type $\lambda$ if and only if it has a
crystalline lift of Hodge type $\lambda'$.  To see this, we first
recall the following basic fact about crystalline characters and
their reductions modulo $p$. 

\begin{lem}\label{lem:existenceofcrystallinechars}
Let $\Lambda=\{\lambda_\emb\}_{\emb\in S_{K}}$ be a collection of integers. 
\begin{enumerate}
\item There is a crystalline character $\psi^{K}_{\Lambda}:G_{K}\to \Zpbarx{\times}$
  such that for each $\emb\in S_{K}$ we have
  $\HT_{\emb}(\psi^{K}_{\Lambda})=\lambda_{\emb}$; this character is uniquely
  determined up to unramified twists.

\item  We
  have $\overline{\psi}^{K}_{\Lambda} |_{I_K} =\prod_{\sigma\in
    S_{k}}\omega_{\sigma}^{b_{\sigma}}$, where $b_{\sigma}=\sum_{\emb\in
    S_{K}:\overline{\emb}=\sigma }\lambda_{\emb}.$
\end{enumerate}
\end{lem}

\begin{proof}\label{pf:existenceofcrystallinechars} Existence in (i)
  is well-known; see for instance \cite[\S2.3, Cor.~2]{MR563476} or
  \cite[Prop.~B.4]{conradlifting}. 
  If $\psi$ and
  $\psi'$ are crystalline characters of $G_{K}$ with the same labeled
  Hodge--Tate weights, then $\psi^{-1} \psi'$ is a crystalline
  representation all of whose Hodge--Tate weights are zero, and so is
  unramified.  This proves (i), while (ii) is a consequence of
  \cite[Prop.~B.3]{conradlifting} (see also the proof of \cite[Prop.~6.7]{gls12}).
\end{proof}

Now, to justify the claim preceding Lemma~\ref{lem:existenceofcrystallinechars}, write $a'_{\sigma,i}
- a_{\sigma,i} = x_{\sigma}$. Then the lift of Hodge type $\lambda'$ can be obtained by
twisting the lift of type $\lambda$
by a crystalline character with $\emb_\sigma$-labeled Hodge--Tate
weight $x_\sigma$ for each $\sigma\in \Sk$, $\emb'$-labeled
Hodge--Tate weights $0$ for all other $\emb' \in \SK$, and trivial
reduction; such a character exists by Lemma~\ref{lem:existenceofcrystallinechars}.

In general, we obviously have $\Wcrisforall(\rhobar) \subset
  \Wcrisexists(\rhobar)$, and assuming the generalised
  Breuil--M\'ezard conjecture we even have $\WBM(\rhobar) \subset
  \Wcrisforall(\rhobar) =  \Wcrisexists(\rhobar)$.  
(The equality follows from Lemma~\ref{lem:weight-implies-lift}, noting that
$L_{\lambda} \otimes_{\Zpbar} \Fpbar \cong L_{\lambda'} \otimes_{\Zpbar} \Fpbar$
for any two lifts $\lambda$, $\lambda'$ of the same Serre weight.)
If
  $\rhobar|_{I_K}$ is semisimple, then as in \cite[Conj.~4.2.1]{gee061}, we make the following conjecture. As
we have already remarked, we do not believe
that~\cite[Conj.~4.2.1]{gee061} is true without the semisimplicity
hypothesis that we impose here; even in the semisimple case, where
there is (as we will see below) considerable evidence in favour of the
conjecture, we do not have a fully satisfying reason to believe that it
holds in complete generality, in the sense that for instance we do not know
how to see that it would follow from other widely-believed conjectures.

\begin{conj}[The weight part of Serre's conjecture in terms of
  crystalline lifts]
  \label{conj:crystalline-version}  \hfill
  \begin{enumerate}
  \item  We have  $\Wcrisexists(\rhobar) = \Wcrisforall(\rhobar)$.  
  \item If $\rhobar|_{I_K}$ is semisimple, then in the context of the generalised Breuil--M\'ezard conjecture \upl
  Conj.~\ref{conj: generalised BM conjecture}\upr, one has
    $\WBM(\rhobar) = \Wcrisexists(\rhobar) = \Wcrisforall(\rhobar)$.  
  \item If $\rbar|_{I_{F_v}}$ is semisimple for all places $v|p$, then 
    the weight part of Serre's conjecture \upl 
    Conj.~\ref{conj:basic-conjecture}\upr\ holds with 
    $\W_v(\rbar) = \Wcrisexists(\rbar|_{G_{F_v}})$. 
  \end{enumerate}
\end{conj}

If one believes the $\CS$-weight part of Serre's conjecture
(Conj.~\ref{conj:BM-type-S-version}) --- and as explained in
Section~\ref{sec:patching, BM and serre weights} above, the Taylor--Wiles--Kisin
method strongly suggests that we \emph{should} believe
Conjecture~\ref{conj:BM-type-S-version} --- then the mysterious part
of Conjecture~\ref{conj:crystalline-version} is the assertion that 
$\Wcrisexists(\rhobar)$ is no larger than $\W_{\CS}(\rhobar)$. The evidence for this
conjecture is for the most part limited to the case $n \le 2$ (but see
Remark~\ref{rem:dims-1-2-weak} below) and the case of $\GL_3(\Qp)$, and from a
theoretical point of view the conjecture is rather mysterious; however, the
evidence for the case of $\GL_3(\Qp)$ is striking (see
Section~\ref{sec:evidence} for a detailed discussion of the theoretical and
computational evidence in this case), and makes the conjecture seem
plausible in general.

\begin{remark}
  \label{rem:dims-1-2}
 Considerable progress has been made on Conjecture~\ref{conj:crystalline-version} in the case
where $\rhobar$ is at most two-dimensional.  

 If $n=1$ then Conjecture~\ref{conj:crystalline-version} is a
 consequence of class field theory together with an analysis of
 crystalline characters and their reductions modulo~$p$. (For example,
 part (i) of
 the conjecture when $n=1$ follows from
 Lemma~\ref{lem:existenceofcrystallinechars}.)

  If $n=2$ and $p>2$ then,  as explained
  in Remark~\ref{rem: patching functors and GK}, part (i) of
  Conjecture~\ref{conj:crystalline-version} is known, and the
  analogue
  of part (iii)
  for quaternion algebras and forms of $\mathrm U(2)$ over totally real
  fields is also known. If $n=2$ and $K=\Qp$ then
  part~(ii) is known whenever the Breuil--M\'ezard conjecture is
  known; that is, it is known unless $p=2,3$, the representation~$\rhobar$ is
  reducible, and the
  characters on the diagonal of $\rhobar$ have ratio
  $\varepsilonbar$  ($= \varepsilonbar^{-1}$ when $p \le 3$)  \cite{PaskunasBM, HuTan, PaskunasBMatTwo,
    SanderScalar}. Indeed,
  all of these results hold without the assumption of semisimplicity.
\end{remark}

\begin{remark}
  \label{rem:dims-1-2-weak}
Again assuming the generalised Breuil--M\'ezard conjecture, we note that the weights in $\WBM(\rhobar)$ and $\Wcrisexists(\rhobar)$
  which are in the closure of the lowest alcove (i.e.\ the weights
  $a$ such that $a_{\sigma,1} - a_{\sigma,n} + (n-1)
  \le p$ for each $\sigma$)
   must always coincide:\
  this follows by considerations similar to those in the proofs of
  Lemma~\ref{lem:weight-implies-lift} and
  Corollary~\ref{cor:weight-implies-lift}, noting that if $\lambda$ is
  a lift of such a weight, the representation $L_{\lambda}
  \otimes_{\Zpbar}  \Fpbar$ is irreducible. In particular, when $n \le
  2$ all
  Serre weights are in the closure of the lowest alcove, so that the
  progress towards Conjecture~\ref{conj:crystalline-version} in the
  case $n \le 2$ should be regarded as relatively weak evidence for
  the general case.
\end{remark}

It is worth mentioning that while it is an open problem to prove
that~$\Wcrisexists(\rhobar)$ is non-empty in general, we strongly
believe that this is the case. Indeed, if~$\rhobar$ arises as the
local mod~$p$ representation associated to an automorphic
representation of some unitary group which is compact at infinity,
then this is automatic from the considerations explained in
Example~\ref{example: crystalline lifts for modular forms and more
  generally} (in brief: the corresponding system of Hecke eigenvalues
will show up in the cohomology associated to some Serre weight~$W$,
and lifting to characteristic $0$ gives a global Galois representation
which is crystalline of the appropriate Hodge--Tate weights).

While it might seem that this is a rather restrictive requirement
on~$\rhobar$, it is expected that such an automorphic representation
exists for every choice of~$\rhobar$ (of course, one has to allow
unitary groups associated to arbitrary CM fields). Indeed, as
explained in~\cite[App.\ A]{emertongeerefinedBM}, the methods
of~\cite{frankII} allow one to globalise~$\rhobar$ to a representation
which should (under the assumption of a weak version of Serre's
conjecture for unitary groups) correspond to an automorphic
representation on some unitary group. Furthermore, even without
knowing weak Serre, under the
assumptions that~$p\nmid 2n$ and that~$\rhobar$ admits a potentially
diagonalisable lift with regular Hodge--Tate weights,
the potential automorphy results of~\cite{BLGGT} imply that~$\rhobar$
can indeed be globalised to an automorphic Galois
representation~\cite[Cor.\ A.7]{emertongeerefinedBM}, so
that~$\Wcrisexists(\rhobar)$ is provably non-empty for such
representations. It is widely expected that every~$\rhobar$ admits such a
potentially diagonalisable lift, and this is known if~$\rhobar$ is
semisimple by~\cite[Lem.\ 2.2]{EGPatching}. (These considerations are
expanded upon in~\cite[\S3]{gs-fllifts}.)

 We close this section with the observation that
the generalised Breuil--M\'ezard
conjecture and the crystalline lifts version of the weight part of Serre's conjecture (Conjectures~\ref{conj:
  generalised BM conjecture} and~\ref{conj:crystalline-version})
together with Lemma~\ref{lem:weight-implies-lift} entail the following conjecture.

\begin{conj}\label{conj:closure-operation} Suppose that
  $\rhobar|_{I_K}$ is semisimple.  If $\Wcrisexists(\rhobar) \cap \JH_{\GL_n(k)} (L_{\lambda} \otimes_{\Zpbar} \Fpbar) \ne \varnothing$ for some lift $\lambda$ of the
    Serre weight $a$, then $a \in \Wcrisexists(\rhobar)$.
\end{conj}

It is possible to use (global) potential automorphy techniques to prove
Conjecture~\ref{conj:closure-operation} in certain special cases;
 see
\cite[\S3]{gs-fllifts} for details.

\begin{rem}
  \label{rem: equivalence of crystalline lifts and shadow
    weights}Assume that the generalised Breuil--M\'ezard conjecture
  holds so that,
  as we have already observed, $\WBM(\rhobar) \subset
  \Wcrisforall(\rhobar) =  \Wcrisexists(\rhobar)$. Then
  Conjecture~\ref{conj:crystalline-version}(ii) is equivalent to
  the variant of Conjecture~\ref{conj:closure-operation} where
  $\Wcrisexists$ is replaced with $\WBM$ (both times). Indeed, this variant is equivalent to
  $\Wcrisexists(\rhobar) \subset \WBM(\rhobar)$ by Lemma~\ref{lem:weight-implies-lift}.
\end{rem}

\section{The picture}\label{sec: the picture}
\subsection{A geometric perspective}\label{subsec:the picture}We now
explain a geometric perspective (``the picture'') on the weight part of Serre's
conjecture. Full details will appear in the
papers~\cite{emertongeeproper,emertongeepicture,CEGSKisinwithdd}. Continue
to fix a finite extension~$K/\Qp$ and an integer $n\ge 1$. Assume
that~$p$ is odd. Then the
papers~\cite{emertongeeproper,emertongeepicture,CEGSKisinwithdd}
construct a finite type equidimensional Artin stack~$\cXbar$
over~$\Fp$ (of dimension $[K:\Qp]\binom{n}{2}$), whose
$\Fpbar$-points naturally correspond to the isomorphism classes of those
 representations $\rhobar:G_K\to\GL_n(\Fpbar)$ that admit a de Rham
lift to $\GL_n(\Zpbar)$ (of course, these are conjecturally all the~$\rhobar$, but as far as we are aware this is only known
if~$n\le 3$; the case $n=3$ is due to Muller \cite{MullerThesis}). 

Furthermore, for each pair $(\lambda,\tau)$ consisting of a Hodge type
$\lambda$ and an inertial type $\tau$, there is a finite type
formal Artin stack $\cX_{\lambda,\tau}$ over~$\Spf\Zp$, whose~$\Zpbar$-points
are in natural bijection with the isomorphism classes of de
Rham representations $\rho:G_K\to\GL_n(\Zpbar)$ of
type~$(\lambda,\tau)$. There is a specialisation morphism
$\pi:\cX_{\lambda,\tau}\to\cXbar$, which on points just sends $\rho$
to its reduction modulo~$p$. The underlying reduced substack
of~$\pi(\cX_{\lambda,\tau})$ is a union of irreducible components
of~$\cXbar$. 

Each irreducible component of~$\cXbar$ has a dense open subset of
closed points that lie only on that component, and which correspond to
certain maximally non-split upper-triangular
representations with characters $\chi_1,\ldots,\chi_n$ on the diagonal
such 
that the
characters~$\chi_i|_{I_K}$ are fixed.  We refer to these points as the
\emph{generic $\Fpbar$-points} of the component. 

Suppose for example that $n=2$, and
fix characters $\psi_i : I_K \to \Fpbarx{\times}$ for $i=1,2$ that
extend to $G_K$.  Then
whenever $\psi_1\psi_2^{-1}\ne\varepsilonbar$, there is a unique component whose
generic $\Fpbar$-points correspond to extensions of $\chi_2$ by $\chi_1$ with
$\chi_i |_{I_K} \cong \psi_i$, and these representations have a unique
Serre weight.
We label the irreducible component by the
corresponding Serre weight. Note that this Serre weight can be read
off directly from an expression of the $\chi_i$ in terms of fundamental
characters (that is, from the tame inertial weights).

To illustrate what happens when 
$\psi_1\psi_2^{-1}=\varepsilonbar$, suppose further that $K=\Qp$. 
There is one component of~$\cXbar$ whose generic
$\Fpbar$-points are tr\`es ramifi\'ee extensions of $\chi$ by $\chi\varepsilonbar$, 
where $\chi$ is any unramified character,
and also another component whose
generic~$\Fpbar$-points are extensions of $\chi_2$ by $\chi_1\varepsilonbar$,
where $\chi_1 \ne \chi_2$ are any unramified characters.
The peu
ramifi\'ee extensions of $\chi$ by $\chi\varepsilonbar$
lie on both components (and so are not generic $\Fpbar$-points on
either of them).
 We label the first component by the Serre
weight~$\Sym^{p-1}\Fpbarx{2}$, while the second is labeled by both
$1$ and $\Sym^{p-1}\Fpbarx{2}$, the two Serre weights of a generic
$\Fpbar$-point on that component.  In particular every component of~$\cXbar$
labeled by~$1$ is also labeled by~$\Sym^{p-1}\Fpbarx{2}$. All other
components of~$\cXbar$ are labeled by a single Serre weight, as in the
previous paragraph, and in
fact each other Serre weight
is the label for  a unique irreducible
component of $\cXbar$.

More generally, we expect that when $n>2$ there will be a set of weights associated to
each component, and the Serre weights of any~$\rhobar$ will be
precisely the union of the sets of weights associated to the components that it
lies on.  In particular the labels of a component must therefore be the Serre
weights of its generic $\Fpbar$-points. This structure, with the set
of Serre weights for~$\rhobar$ being the set~$\W_{\cS}(\rhobar)$ for a
Breuil--M\'ezard system $\cS$, should be a consequence of the Breuil--M\'ezard
conjecture for representations of type $\cS$. Indeed for $n=2$ (with~$K$ arbitrary)
and $\CS = \mathrm{BT}$
this can be proved, as a consequence of
the results of~\cite{geekisin}
(see~\cite{CEGSKisinwithdd}).

Accordingly, understanding the weight
part of Serre's conjecture should
reduce to understanding the components
of~$\cXbar$ on which a given representation lies, and understanding what the Serre weights are for
maximally non-split upper-triangular representations (that are generic
enough to lie on a single component).

While this structure is already (at least to us) very
attractive, we expect that the picture is both simpler and more
structured than what is entailed by the Breuil--M\'ezard
conjecture. Specifically, we expect that most components are labeled
by a single weight, and that in the cases where there are multiple  weights labeling a
component, they are frequently related in a simple way 
(see Section~\ref{sec:shifted-weights}). 
For example, if $K=\Qp$ and a component has
$F(a_1,\dots,a_n)$ as a label, then the generic representations on the
component are of the form \[ \begin{pmatrix}
    \chi_1  & * &\dots & *\\
& \chi_2 &\dots &*\\
& & \ddots &\vdots\\
&&&\chi_n
  \end{pmatrix}\] where $\chi_i|_{I_{\Qp}}=\omega^{a_i+n-i}$. Furthermore, if none of the $a_i-a_{i+1}$ are equal to $0$
  or $p-1$, then we expect there to be a unique component labeled by
  this weight, and this
  component should be labeled only by $F(a_1,\dots,a_n)$. We will discuss the case where some $a_i-a_{i+1}$ are equal
  to~$0$ or~$p-1$ in Section~\ref{sec:shifted-weights}.

\subsection{Crystalline lifts}\label{subsec: crystalline lifts in the
  picture}We briefly explain what light the geometric perspective of
Section~\ref{subsec:the picture} sheds on the crystalline lifts
conjectures of Section~\ref{sec:cryst-lifts-serre}, and on their
expected failure to extend to the case of non-semisimple representations.

Let $a$ be a Serre weight, and let~$\lambda$ be a lift of~$a$. As
explained in Section~\ref{subsec:the picture}, there is a specialisation morphism
$\pi:\cX_{\lambda,\mathrm{triv}}\to\cXbar$, which on points just sends a 
crystalline representation~$\rho$ of weight~$\lambda$
to its reduction modulo~$p$. The underlying reduced substack
of~$\pi(\cX_{\lambda,\mathrm{triv}})$ is a union of irreducible components
of~$\cXbar$, and the geometrisation of the Breuil--M\'ezard conjecture
of~\cite{BreuilMezardRaffinee,emertongeerefinedBM} strongly suggests
that these irreducible components should be precisely the ones that
have some Jordan--H\"older factor of $L_\lambda\otimes\Fpbar$ among
their labels.

If the conjectures of Section~\ref{sec:cryst-lifts-serre} held for
arbitrary (not necessarily semisimple) $\rhobar$, then we would be
forced to conclude that the Serre weight~$F_a$ is a label of each
of the above components. However, work of L\^e H\`ung
Vi\^et Bao, Brandon Levin, Dan Le and Stefano Morra \cite{LLLM} contradicts this
conclusion; instead, their calculations indicate that already for
$n=3$ and $K=\Qp$, if $a$ is in the upper alcove and is suitably
generic, then the two Jordan--H\"older factors $F_a$, $F_b$ of
$L_\lambda\otimes\Fpbar$ correspond to two components of $\cXbar$,
labeled by the single weight $F_a$ (resp.\ $F_b$),
which meet in a codimension one substack. Thus the generic $\rhobar$
on the component labeled by the weight in the lower alcove do not
satisfy the conjectures of Section~\ref{sec:cryst-lifts-serre}; it is
only those $\rhobar$ which lie in a special position which do so. 
(The limited evidence available in the cases $n=2,3$ suggests that it
is possible that two components labeled by suitably generic weights
$F, F'$ meet in codimension~$i$, where $i$ is minimal such that
$\Ext^i_{\GL_n(\Fp)}(F,F')\ne 0$, but we do not know if it is
reasonable to expect this to be true in general.)

Of course, the most special position is that occupied by
semisimple~$\rhobar$, which agrees (in the case that $\rhobar$ is a
sum of characters) with the conjectures of
Section~\ref{sec:cryst-lifts-serre}. Note also that in general it
seems reasonable to expect that any component containing $\rhobar$
also contains $\rhobar^{\semis}$, which is consistent with the
folklore belief that the set of Serre weights for $\rhobar$ should be
a subset of those for $\rhobar^{\semis}$.

\section{Explicit weight conjectures in the semisimple case}
\label{sec:expl-weight-conj}

Once again assume that $\rhobar|_{I_K}$ is semisimple. The set $\Wcrisexists(\rhobar)$ is, in general, very badly understood:
for instance at the time of writing we do not know how to prove, in general,
that it is non-empty!  (Though we do when $\rhobar$ is semisimple; see Appendix~\ref{sec:wobvrhbar-nonempty}.)
We would therefore like to have a version of the weight part of Serre's
conjecture that is more explicit than conjectures we have already
described, such as Conjecture~\ref{conj:crystalline-version} in terms of
crystalline lifts.

In Sections~\ref{sec:obvious-semisimple}
and~\ref{sec:shad-weights:-inert} we construct various sets of weights that we have
good reason (e.g.\ as a consequence of the generalised
Breuil--M\'ezard conjecture) to believe are contained in
   $\Wcrisexists(\rhobar)$.  On the other hand there is no reason
to think that in general any of these sets are actually equal to
$\Wcrisexists(\rhobar)$; to the contrary, we explain in
Example~\ref{ex:FH-ramified-example} and
Section~\ref{sec:shifted-weights} why we believe that this should
\emph{not} be the case. These examples  illustrate the difficulty in making a general explicit conjecture.

However, we do expect that at least  for generic
$\rhobar$ and unramified $K/\Qp$, the set $\cC(\Wobv(\rhobar))$
defined below in Section~\ref{sec:shad-weights:-inert} \emph{is} equal to
$\Wcrisexists(\rhobar)$, motivated by a comparison with the
conjectures of~\cite{herzigthesis}; this will be explained in Section~\ref{sec:herzig-comparison}.

\subsection{Obvious lifts}\label{sec:obvious-semisimple}

Recall
from Section~\ref{sec:global framework} that if $\rbar : G_F \to
\GL_n(\Fpbar)$ is automorphic, one may hope that the set
$\W_v(\rbar)$ depends only on $\rbar|_{I_{F_v}}$.  We do not understand
this as well as we would like; for instance, it appears to be somewhat more than can be deduced
easily from the Breuil--M\'ezard formalism, because even in the case
of $\GL_2(\Qp)$, the quantities $\mu_a(\rhobar)$ do not depend only
on $\rhobar|_{I_{\Qp}}$ (see \cite[Thm.~1]{sandermultiplicities}).
However, since
$p$-adic Hodge theoretic conditions are fundamentally conditions about
ramification, it is not unreasonable to imagine that the set
$\Wcrisexists(\rbar|_{G_{F_v}})$ depends exclusively on
$\rbar|_{I_{F_v}}$ and not on the
image of $\Frob_v$ under $\rbar$.  For instance this is known to be
true when $n=2$ and $p > 2$  by \cite[Prop.~6.3.1]{gls13}.

To make Conjecture~\ref{conj:crystalline-version}
explicit, one can imagine trying to exhibit specific elements of
$\Wcrisexists(\rhobar)$ by constructing crystalline lifts of $\rhobar$ of various
Hodge types, for instance by taking sums $\rho'$ of inductions of
crystalline characters.  This is essentially what we
will do; however, an immediate flaw with this plan is that for such $\rho'$ one has
limited control over the image of $\Frob_K$ under $\rhobar'$, and in particular one may not be able to match the image of
$\Frob_K$ under $\rhobar$. (This will however be possible in generic situations.)  Guided
by the expectation that $\Wcrisexists(\rhobar)$ should depend only on
$\rhobar|_{I_K}$, we will be satisfied with constructing certain
crystalline representations $\rho'$ (that we call \emph{obvious lifts} of $\rhobar$) with the property that
$\rhobar'|_{I_K} \cong \rhobar|_{I_K}$. In particular we caution that an obvious lift of
$\rhobar$ need not literally be a lift of $\rhobar$.  When an obvious
lift $\rho'$ of $\rhobar$ has Hodge
type $\lambda_a$, with $\lambda_a$ a lift of a Serre weight $a$, we will call $a$ an
\emph{obvious weight} of $\rhobar$.

We now set up some basic results about crystalline characters.
For each integer $n\geq 1$, let $K_{n}$ be the unique extension of $K$
inside $\overline{K}$ which is unramified of degree $n$. We denote the
residue field of $K_{n}$ by $k_{n}$.
Given a character $\chi:G_{K_{n}}\to \Zpbarx{\times}$, we define a
character $\chi^{(r)}$ by $$\chi^{(r)}(g)=\chi(\Frob_{K}^{r}\cdot
g\cdot \Frob_{K}^{-r}).$$  Note that this character does not depend on
the choice of $\Frob_K$.  The following lemma is elementary (see also
Lemma~\ref{lem:HT}(iii) for a generalisation).
	
\begin{lem}\label{lem:HT wts under Frobenius conjugation} If
  $\chi:G_{K_{n}}\to \Zpbarx{\times}$ is crystalline, then so is $\chi^{(r)}$ and for
any $\emb' \in S_{K_n}$ we have $\HT_{\emb'}(\chi^{(r)})=\HT_{\emb'\circ\Frob_{K}^{-r}}(\chi)$.
\end{lem}

Recall from Lemma~\ref{lem:existenceofcrystallinechars} that for any collection of integers $\Lambda = \{ \lambda_{\emb}\}_{\emb \in S_{K}}$ there exists a crystalline character
 $\psi^{K}_{\Lambda}:G_{K}\to \Zpbarx{\times}$ 
  such that for each $\emb\in S_{K}$ we have
  $\HT_{\emb}(\psi^{K}_{\Lambda})=\lambda_{\emb}$, and that this character is uniquely
  determined up to unramified twists.

\begin{cor}\label{cor:HT wts of induction} Let $\Lambda = \{
  \lambda_{\emb'}\}_{\emb' \in S_{K_n}}$ be a collection of integers.  The representation
  $\rho^K_\Lambda := \Ind_{G_{K_{n}}}^{G_{K}} \Zpbar(\psi^{K_{n}}_{\Lambda})$ is crystalline, and for
  each $\emb \in S_K$ we have
  $$\HT_{\emb}(\rho^K_\Lambda)=\left\{\lambda_{\emb'}
    \, : \, \emb' \in S_{K_n} \text{ such that } \emb'|_{K}=\emb\right\}.$$
Moreover we have $$ \rhobar^K_\Lambda |_{I_K} \cong \bigoplus_{i=0}^{n-1}
\left(\prod_{\sigma \in S_{k_n}} \omega_{\sigma}^{b_{\sigma}}
\right)^{q^i}$$ where $b_{\sigma} = \sum_{\emb' \in S_{K_n} : \embb' =
  \sigma} \lambda_{\kappa'}$ and $q = \#k$.
\end{cor}
\begin{proof}\label{pf:HT wts of induction} If $\rho$ is a Hodge--Tate
  representation of $G_K$ and $L$ is a finite extension of $K$, then
  $\gr^{-i} (D_{\HT}(\rho|_{G_{L}})) =  \gr^{-i}( D_{\HT}(\rho))
  \otimes_{K} L$.
From this we deduce that if $\emb' \in S_{L}$ then
  $\HT_{\emb'}(\rho|_{G_{L}}) = \HT_{\emb'|_{K}}(\rho)$.   Applying
  this statement for $L = K_n$, the Corollary now follows from Lemma
  \ref{lem:HT wts under Frobenius conjugation}, the fact that 
$\Ind_{G_{K_{n}}}^{G_{K}}(\psi^{K_{n}}_{\Lambda})|_{G_{K_{n}}} \cong
\oplus_{r=0}^{n-1} (\psi^{K_{n}}_{\Lambda})^{(r)}$, and (for the first part
of the statement) the fact that the
property of being crystalline only depends on the restriction to
inertia.  The formula for $\rhobar^K_\Lambda |_{I_K}$ follows from
Lemma~\ref{lem:existenceofcrystallinechars}(ii).
\end{proof}

\begin{defn}
  \label{defn:obvious-lift-ss} 
 Suppose that $\rhobar|_{I_K}$ is semisimple.  We define an \emph{obvious lift} of $\rhobar$ to be a representation
  of the form $\rho' = \rho^K_{\Lambda_1} \oplus \cdots \oplus \rho^K_{\Lambda_d}$ (for
  some partition $n_1 + \cdots + n_d = n$ of $n$) such that
  $\rhobar'|_{I_K} \cong \rhobar|_{I_K}$.

We define
$\Wobv(\rhobar)$ to be the set of Serre weights $a$ such that
$\rhobar$ has an obvious lift $\rho'$ of Hodge type $\lambda$ for some
lift $\lambda$ of $a$.  (In this case we say that the
lift $\rho'$ \emph{witnesses} the obvious weight $a$.)
\end{defn}

It is essential in this definition that we have required $\rhobar'|_{I_K} \cong
  \rhobar|_{I_K}$ rather than $\rhobar' \cong
  \rhobar$: as we will see in  Example~\ref{ex:unramified-3-example}
  making the latter definition would sometimes have produced a
  different (too small) set of weights.   We note that if $\rhobar$ has an obvious lift of some Hodge type
lifting the Serre weight $a$, then it has a lift of any Hodge type lifting $a$:\ this follows
from Corollary~\ref{cor:HT wts of induction} (specifically, the fact
that $\rhobar^K_\Lambda|_{I_K}$ only depends on the multisets $\{\lambda_{\emb'}
  : \emb' \in S_{K_n} \textrm{ lifting } \sigma\in S_{k_n}\}$)  and an argument as in the paragraph following
Definition~\ref{defn:crystalline-weight-sets}. 

\begin{rem}\label{rem:wobv-nonempty}
The set $\Wobv(\rhobar)$ is always
non-empty. This is not at all immediate from the definitions, and
unfortunately the only proof we have been able to find proceeds via a direct and
somewhat painful combinatorial argument; for this reason we have
deferred the proof to Appendix~\ref{sec:wobvrhbar-nonempty}.
\end{rem}

 Since we expect that
the possible Hodge types of the crystalline lifts of $\rhobar$ depend
only on $\rhobar|_{I_K}$, we make the following conjecture.

\begin{conj}
  We have $\Wobv(\rhobar) \subset \Wcrisforall(\rhobar)$.
\end{conj}

We consider several illustrative examples.

\begin{example}
  \label{ex:1-dimensional} 
When $n=1$, any obvious lift of $\rhobar$ is an unramified twist of a
genuine crystalline lift of $\rhobar$, from which it follows that
$\Wobv(\rhobar) = \Wcrisforall(\rhobar) = \Wcrisexists(\rhobar)$.
\end{example}

\begin{example}[Comparison with Schein's conjecture]
  \label{ex:2-dimensional}
We determine the obvious weights of a representation $\rhobar : G_K
\to \GL_2(\Fpbar)$ such that $\rhobar|_{I_K}$ is semisimple.  Let $e$ be
the absolute ramification index of $K$.  The weight part of Serre's
conjecture has been formulated in this context by Schein
\cite{MR2430440}.

Suppose  first that $\rhobar$ is irreducible.  Consider a Serre
weight $a$ represented by $ (x_{\sigma},y_{\sigma})_{\sigma \in S_k}
\in (\Z_+^2)^{\Sk}$, and let $\lambda \in (\Z_+^2)^{\SK}$ be a lift of $ (x_{\sigma},y_{\sigma})$.  An obvious
lift $\rho'$ of $\rhobar$ must be
of the form
$\Ind_{G_{K_2}}^{G_K} \psi_{\Lambda}^{K_2}$.  Suppose that $\rho'$ witnesses
$a$. We may take $\rho'$ to have Hodge type $\lambda$, so that
the Hodge--Tate weights $\lambda_{\emb'}$ of $\psi_{\Lambda}^{K_2}$ are as follows.  For each $\sigma \in S_k$, there is a pair
$(\emb'_{\sigma,1},\emb'_{\sigma,2})$ of $K_2/K$-conjugate embeddings $K_2 \into \Qpbar$ such that
$\embb'_{\sigma,1}$, $\embb'_{\sigma,2} : k_2 \into \Fpbar$ extend
$\sigma$, and $\lambda_{\emb'_{\sigma,1}} = x_\sigma+1$, $\lambda_{\emb'_{\sigma,2}} =
y_\sigma$.   For the remaining $e-1$ pairs $(\emb'_{1},\emb'_{2})$ of
$K_2/K$-conjugate embeddings $K_2 \into \Qpbar$ such that
$\embb'_{1}$, $\embb'_{2} : k_2 \into \Fpbar$ extend
$\sigma$, we have $\{\lambda_{\emb'_1},\lambda_{\emb'_2}\} = \{1,0\}$. Write
$\sigma_1$ for $\embb'_{\sigma,1} \in S_{k_2}$ and $\sigma_2$ for its
$k_2/k$-conjugate.    Let
$0 \le m_{\sigma} \le e-1$ be the number of embeddings $\emb'_1 \neq \emb'_{\sigma,1}$ with
$\embb'_1 = \sigma_1$ and $\lambda_{\emb'_1}=1$.  Then we see from Corollary~\ref{cor:HT wts of induction}
that 
 $$\rhobar|_{I_{K}} \cong \left(\begin{array}{cc}\prod_{\sigma\in
       S_{k}}\omega_{\sigma_{1}}^{x_{\sigma}+1+m_{\sigma}}\omega_{\sigma_{2}}^{y_{\sigma}+e-1-m_{\sigma}}
     & 0\\ 0 &\prod_{\sigma\in
       S_{k}}\omega_{\sigma_{2}}^{x_{\sigma}+1+m_{\sigma}}\omega_{\sigma_1}^{y_{\sigma}+e-1-m_{\sigma}}\end{array}\right).$$
In other words, we have $a \in \Wobv(\rhobar)$ if and only if for
each $\sigma \in S_k$ we can write the elements of $S_{k_2}$ extending
$\sigma$ as $\sigma_1,\sigma_2$ so that the above formula holds for
some choice of integers  $0 \le m_{\sigma} \le e-1$.  Observe that these are
precisely the Serre weights predicted for $\rhobar$ in \cite[Thm.~2.4]{MR2430440}.

Next suppose that $\rhobar$ is reducible, and let $a$ be a
Serre weight as in the previous paragraph.  If $\rhobar|_{I_K}$ is
non-scalar then every obvious lift of $\rhobar$ must be a sum of two
characters, but if $\rhobar|_{I_K}$ is scalar then $\rhobar$ may also
have obvious lifts whose generic fibres are irreducible.  Consider first the
obvious lifts $\rho'$  of $\rhobar$ that have Hodge type $\lambda$  (hence witness $a$) and that are sums
of two characters.  Say $\rho' = \psi^K_{\Lambda} \oplus \psi^K_{\Lambda'}$ with
$\Lambda = \{\lambda_{\emb}\}$ and $\Lambda' = \{\lambda'_{\emb}\}$. For
each $\sigma \in S_k$, there is an embedding $\emb_{\sigma} \in S_K$ lifting
$\sigma$ such that $\{ \lambda_{\emb_\sigma},\lambda'_{\emb_{\sigma}}\} =
 \{ x_{\sigma}+1,y_{\sigma} \}$.  For the remaining $e-1$ embeddings
 $\emb \in S_K$ extending $\sigma$, we have $\{\lambda_{\emb},\lambda'_{\emb} \} =
 \{1,0\}$.  Define $J = \{ \sigma \in S_k \, : \, \lambda_{\emb_{\sigma}} =
   x_{\sigma} + 1\}$.  If $\sigma \in J$ we let  $0 \le m_{\sigma} \le
   e-1$ be the
   number of embeddings $\emb \neq \emb_{\sigma}$ extending $\sigma$
   such that $\lambda_\emb = 1$, while if $\sigma \not\in J$ we let $0 \le
   m_{\sigma} \le e-1$ be the number of embeddings $\emb \neq
   \emb_{\sigma}$ extending $\sigma$ such that $\lambda'_{\emb} = 1$.   Then
   we see from
 Corollary~\ref{cor:HT wts of induction}
that $\rhobar|_{I_{K}}$ is isomorphic to
$$\left(\begin{array}{cc}\prod_{\sigma\in J}\omega_{\sigma}^{x_{\sigma}+1+m_{\sigma}}\prod_{\sigma\notin J}\omega_{\sigma}^{y_{\sigma}+e-1-m_{\sigma}} & 0\\ 0 &\prod_{\sigma\notin J}\omega_{\sigma}^{x_{\sigma}+1+m_{\sigma}}\prod_{\sigma\in J}\omega_{\sigma}^{y_{\sigma}+e-1-m_{\sigma}}\end{array}\right).$$
In other words the weight $a \in \Wobv(\rhobar)$ is witnessed by an
obvious lift whose generic fibre is reducible if and only if the above formula holds for some
subset $J
\subset S_k$ and a choice of integers $0 \le m_{\sigma} \le e-1$.

In fact if $\rhobar|_{I_K}$ is scalar, then it turns out that every weight $a \in
\Wobv(\rhobar)$ that is witnessed by an obvious lift
whose generic fibre is irreducible is also witnessed by an obvious lift 
whose generic fibre is reducible, so that the previous paragraph still describes
the whole set $\Wobv(\rhobar)$ in this case.
 This observation
is an elementary but not necessarily straightforward exercise that we
leave  to the reader. (One first reduces to the case $e \le p-1$ by noting
that if $e \ge p$ then \emph{every} weight $a$ whose central character
is compatible with $\det(\rhobar)|_{I_K}$ lies in $\Wobv(\rhobar)$
and is  witnessed by an
obvious lift whose generic fibre is reducible.  Alternately, if $p\ge 3$ the observation can
be deduced from the local results in \cite{gls13}, specifically
Theorem~4.1.6 of \emph{loc.\ cit.}, while for $p=2$ one reduces to the
case $e=1$ as above. But the case $e=1$ is straightforward:\  after
twisting one may suppose that $y_{\sigma}=0$ for all $\sigma$;\ then
$x_{\sigma}+1 \in \{1,p-1,p\}$ for all $\sigma$ (see the last
paragraph of the proof of \cite[Thm.~10.1]{gls12} for a more precise statement), and one checks that
$\rhobar|_{I_K}$ has the above shape with $J = \{ \sigma : x_\sigma = 0\}$.)
  Observe that these are
precisely the Serre weights predicted for $\rhobar$ in \cite[Thm.~2.5]{MR2430440}.

\end{example}

\begin{example}
  \label{ex:niveau-one-3-dimensional}
  Consider $(a,b,c)\in\Z^3_+$ with $a-b,b-c
  > 1$ and $a-c<p-2$.  We determine the obvious
  weights of a representation $\rhobar : G_{\Qp} \to \GL_3(\Fpbar)$ such
  that $\rhobar|_{I_{\Qp}} = \omega^a \oplus \omega^b \oplus
  \omega^c$. 
  Any obvious lift $\rho'$ must be a sum of characters.  In
  particular $\rho'$ has the form $\psi^{\Qp}_{\{x\}} \oplus \psi^{\Qp}_{\{y\}} \oplus
  \psi^{\Qp}_{\{z\}}$ where $(x-2,y-1,z)\in X_1^{(3)}$, and $\{x,y,z\}$ and
  $\{a,b,c\}$ reduce to the same subset of $\Z/(p-1)$.
  It follows
  that the only possibilities for  $(x,y,z)$ (up to translation by  $\Z(p-1,p-1,p-1)$)
  are 
 $$(a,b,c),\,(b,c,a-p+1),\,(c+p-1,a,b) ,$$ $$ (c+p-1,b,a-p+1),\, (a,c,b-p+1),\,(b+p-1,a,c)$$
and therefore $$\Wobv(\rhobar) = \{F(a-2,b-1,c),\,F(b-2,c-1,a-p+1),\,F(c+p-3,a-1,b) ,$$ $$
 F(c+p-3,b-1,a-p+1),\, F(a-2,c-1,b-p+1),\,F(b+p-3,a-1,c)\}.$$
We see from this example that we \emph{cannot} expect to have
$\Wobv(\rhobar) = \Wcrisexists(\rhobar)$: this is because (at least if
$\rhobar$ is semisimple) the weights 
$F(c+p-2,b-1,a-p)$, $F(a-1,c-1,b-p)$, and $F(b+p-2,a-1,c-1)$ also belong
to $\Wcrisexists(\rhobar)$.  We explain this for 
$F(c+p-2,b-1,a-p)$; the others are similar.  We need to exhibit a lift $\rho'$
of $\rhobar$ with Hodge--Tate weights $\{c+p, b, a-p\}$.
Since $p+1 < 2p-(a-c) < 2p$, for example by
\cite[Thm.~3.2.1(3)]{ber05}, we
can take $\rho'$ to be the sum of an unramified twist of
$\varepsilon^{b}$ and an unramified twist of $\varepsilon^{a-p}
\otimes V$ where $V$ is a suitable 
crystalline representation with irreducible generic fibre and Hodge--Tate weights $\{2p-(a-c),0\}$
(one of the representations $V_{2p+1-(a-c),a_p}$ considered in \cite{ber05}).
\end{example} 

\begin{example}
  \label{ex:unramified-3-example}  Next, we determine the obvious
  weights of an unramified representation $\rhobar : G_{\Qp} \to
  \GL_3(\Fpbar)$.
  The reader can
  verify that the family of obvious lifts $\psi^{\Qp}_{\{p-1\}} \oplus
  \psi^{\Qp}_{\{0\}} \oplus \psi^{\Qp}_{\{-p+1\}}$ witness the weight
  $F(p-3,-1,-p+1)$; the obvious lifts $\rho^{\Qp}_{\{-1,p\}} \oplus
  \psi^{\Qp}_{\{0\}}$ witness the weight $F(p-2,-1,-1)$; the
  obvious lifts $\rho^{\Qp}_{\{-1,p\}} \oplus
  \psi^{\Qp}_{\{p-1\}}$ witness the weight $F(p-2,p-2,-1)$; and
  that these are the only weights in $\Wobv(\rhobar)$ when $p > 2$.
  When $p = 2$, it is easy to check that we have $\Wobv(\rhobar) = W(\F_2,3)$ (so there are four
  weights in this case).
  This
  example illustrates two points.  First, although $\rhobar$ is a
  sum of characters, there are obvious weights of $\rhobar$ that 
  cannot be witnessed by sums of characters.  Second,
  we remark that many unramified representations $\rhobar : G_{\Qp} \to
  \GL_3(\Fpbar)$ do not have literal lifts of the form  $\psi^{\Qp}_{\{p-1\}} \oplus
  \psi^{\Qp}_{\{0\}} \oplus \psi^{\Qp}_{\{-p+1\}}$  (or of the other
  two shapes above).  For instance if $\rhobar : G_{\Qp} \to
  \GL_3(\Fpbar)$ has a lift of the form $\psi^{\Qp}_{\{p-1\}} \oplus
  \psi^{\Qp}_{\{0\}} \oplus \psi^{\Qp}_{\{-p+1\}}$ then
  $\rhobar(\Frob_{\Qp})$ will be semisimple.  Similarly, it may be 
  the case that $\rhobar^{\ss}$
  may not have a literal lift of the form $\rho^{\Qp}_{\{-1,p\}} \oplus
  \psi^{\Qp}_{\{0\}}$ or $\rho^{\Qp}_{\{-1,p\}} \oplus
  \psi^{\Qp}_{\{p-1\}}$, since possessing such a lift imposes
  restrictions on the eigenvalues of   $\rhobar(\Frob_{\Qp})$.  
 \end{example}

\subsection{Shadow and obscure weights}
\label{sec:shad-weights:-inert}

Now we would like to address the observation (from
Example~\ref{ex:niveau-one-3-dimensional}) that in general we need not
have $\Wobv(\rhobar) = \Wcrisexists(\rhobar)$.  To begin to account
for this,
Conjecture~\ref{conj:closure-operation} motivates the following definition.

\begin{defn}
  \label{defn:closure-operation}
  If $\W$ is a set of Serre weights, we define $\cC(\W)$ to be the
  smallest set of weights with the properties:
  \begin{itemize}
  \item  $\W \subset \cC(\W)$, and
  \item if $\cC(\W) \cap \JH_{\GL_n(k)} (L_{\lambda} \otimes_{\Zpbar} \Fpbar) \ne \varnothing$ for some lift $\lambda$ of the
    Serre weight $a$, then $a \in \cC(\W)$.
  \end{itemize}
For instance, Conjecture~\ref{conj:closure-operation} asserts that we
should have $\Wcrisexists(\rhobar) = \cC(\Wcrisexists(\rhobar) )$.
\end{defn}

\begin{example} 
  \label{ex:closure-for-niveau-1-gl3}
Return to the case of $\GL_3$ over $\Qp$. 
  If $F(x,y,z)$ is a Serre weight such that $\cC(\{F(x,y,z)\})
  \supsetneq \{F(x,y,z)\}$, then $x-z < p-2$ and $\cC(\{F(x,y,z)\}) = \{F(x,y,z),F(z+p-2,y,x-p+2)\}$.
Indeed, if $x-z < p-2$  then by Proposition 3.18 of~\cite{herzigthesis} there is a short exact sequence
$$ 0 \to F(\lambda)  \to L_{\lambda} \otimes_{\Zpbar} \Fpbar \to
F(x,y,z) \to 0$$
where $\lambda = (z+p-2,y,x-p+2)$, so that $F(\lambda) \in
\cC(\{F(x,y,z)\})$, and these give all the instances of reducible $L_{\lambda} \otimes_{\Zpbar} \Fpbar$ with $\lambda \in X_1^{(3)}$.

For instance, in the setting of
Example~\ref{ex:niveau-one-3-dimensional} we see that $F(c+p-2,b-1,a-p) \in \cC(\{F(a-2,b-1,c)\})$, that 
$F(a-1,c-1,b-p) \in \cC(\{F(b-2,c-1,a-p+1)\})$, and that
$F(b+p-2,a-1,c-1) \in \cC(\{F(c+p-3,a-1,b)\})$.  In fact one can check in
this setting that $\cC(\Wobv(\rhobar))$ is precisely
$\Wobv(\rhobar)$ together with these three extra weights.  (We note
that this same prediction can be found in the discussion immediately following
\cite[Def.~3.5]{bib:ADP}.)
\end{example}

If one believes Conjecture~\ref{conj:closure-operation}, then one
might hope that also $\cC(\Wobv(\rhobar)) = \Wcrisexists(\rhobar)$, and
indeed we will show in Section~\ref{sec:herzig-comparison} that this
is a reasonable expectation when $K/\Qp$ is unramified and $\rhobar$ is sufficiently
generic in a precise sense. However, the following generalisation of
the principle behind Conjecture~\ref{conj:closure-operation} will show
that this cannot be true in all cases.

Suppose that $\rhobar|_{I_K} \cong (\oplus_{j=1}^r
\rhobar^{(j)})|_{I_K}$ with $\rhobar^{(j)} : G_K \to \GL_{n_j}(\Fpbar)$ not necessarily irreducible. Write
$\eta_{m} = (m-1,\ldots,1,0)$ for any $m \ge 1$. Let $a$ be a Serre
weight, and suppose that $\lambda$ is some lift of $a$.
Suppose that $\lambda^{(j)}$ (for each $1\le j \le r$) are Hodge types in
$(\Z_+^{n_j})^{\SK}$ such that the $\lambda^{(j)}_\emb + \eta_{n_j}$ for
each $\emb$ are obtained by partitioning $\lambda_{\emb}+\eta_n$ into
$r$ decreasing subsequences of length $n_j$. (We will say that the
$\lambda^{(j)}$ are an \emph{$\eta$-partition} of~$a$.)
If $\Wcrisexists(\rhobar^{(j)}) \cap \JH_{\GL_{n_j}(k)} (L_{\lambda^{(j)}} \otimes_{\Zpbar} \Fpbar) \ne \varnothing$,
then Conjecture~\ref{conj:closure-operation} entails that $\rhobar^{(j)}$ has a
    crystalline lift of Hodge type $\lambda^{(j)}$. The direct sum of
    these lifts would be a crystalline lift of $\oplus_j
    \rhobar^{(j)}$ of Hodge type $\lambda$, in which case $a\in \Wcrisexists(\oplus_j \rhobar^{(j)})$. Since we
    expect that $\Wcrisexists(\rhobar)$ depends only on
    $\rhobar|_{I_K}$, we then also expect to have $a \in
    \Wcrisexists(\rhobar)$.

We are thus led to the following definition. 
 
 \begin{defn} 
   \label{defn:explicit-weights-ss} 
   Suppose that $\rhobar|_{I_K}$ is semisimple.   
 We recursively define 
   $\Wexpl(\rhobar)$, the \emph{explicit predicted weights for $\rhobar$}, to be the smallest set 
   containing $\Wobv(\rhobar)$ and satisfying the expectation described 
   in the previous paragraph:\ that is,
   $a \in \Wexpl(\rhobar)$ for any  
  Serre weight $a$ such that there exists 
  a decomposition  
    $\rhobar|_{I_K} \cong \oplus_{j=1}^r \rhobar^{(j)}|_{I_K}$
 and 
 an $\eta$-partition $\lambda^{(j)}$ of $a$ such that
   $\Wexpl(\rhobar^{(j)}) \cap \JH_{\GL_{n_j}(k)} (L_{\lambda^{(j)}} \otimes_{\Zpbar} \Fpbar) \ne \varnothing$
   for each $j$. 
 
   Taking $r = 1$ in this definition we see that $\cC(\Wobv(\rhobar)) \subset \Wexpl(\rhobar)$.
  We say that 
   an element of $\cC(\Wobv(\rhobar))\setminus \Wobv(\rhobar)$ is a 
   \emph{shadow weight}, while an element of $\Wexpl(\rhobar) \setminus 
   \cC(\Wobv(\rhobar))$ is an \emph{obscure weight}.
 \end{defn}

\begin{example} 
   \label{ex:n2operatortrivial} 
   If $n\le 2$ it is easily checked that $\Wexpl(\rhobar) =
   \Wobv(\rhobar)$. It is shown in \cite[Thm.~4.1.6]{gls13} that when
   $n=2$ and $ p > 2$ we have $\Wcrisforall(\rhobar) =
   \Wcrisexists(\rhobar)$, and that these sets agree with the
   prediction of Schein~\cite{MR2430440}.
It follows that if $n = 1$, or $n = 2$ with $p > 2$, then $\Wexpl(\rhobar) = \Wobv(\rhobar) = \Wcrisforall(\rhobar) = \Wcrisexists(\rhobar)$.

   Explicitly, if $n = 1$, Lemma~\ref{lem:existenceofcrystallinechars} implies that $a \in \Wexpl(\rhobar)$
   for a Serre weight $a$ if and only if $\rhobar|_{I_K} =\prod_{\sigma\in S_{k}}\omega_{\sigma}^{a_{\sigma}}$.
   If $n = 2$, Example~\ref{ex:2-dimensional} shows that
   $\Wobv(\rhobar)$  coincides with the
   set of weights predicted by Schein~\cite{MR2430440}. 
   Since
   $\Wobv(\rhobar)=\Wexpl(\rhobar)$ in this setting, the claim follows
   from the above  results of \cite{gls13}.

 \end{example}

\begin{example}
  \label{ex:exceptional-weight-gl3} The existence of shadow weights
  in the case of $\GL_3$ over $\Qp$ was discussed in
  Example~\ref{ex:closure-for-niveau-1-gl3}. We now classify the
  obscure weights in this case (showing, in particular, that they
  sometimes exist). We will repeatedly make use of our knowledge
  of $\Wexpl(\rhobar)$ for $n \le 2$, see Example~\ref{ex:n2operatortrivial}. Suppose that $\rhobar :
  G_{\Qp} \to \GL_3(\Fpbar)$ is a representation such that
  $\rhobar|_{I_{\Qp}}$ is semisimple.

 Since $\Sym^r\,
   \Fpbarx{2}$ is irreducible as a $\GL_2(\Fp)$-representation for $r
   \le p-1$, it is straightforward to see that 
   the weight $F(x,y,z)$ can be obscure for $\rhobar$ only 
   if  we have:
   \begin{itemize}
   \item     $\rhobar|_{I_\Qp} \cong (\rhobar^{(1)} \oplus \rhobar^{(2)})|_{I_\Qp}$ 
    with $\dim \rhobar^{(i)} = i$, 
\item $F(y+1) \in \Wexpl(\rhobar^{(1)})$, i.e.\ $\rhobar^{(1)}|_{I_\Qp} \cong \omega^{y+1}$, and
\item $\Wexpl(\rhobar^{(2)}) \cap \JH_{\GL_{2}(\Fp)} (L_{(x+1,z)} \otimes_{\Zpbar} \Fpbar) \ne \varnothing$.
   \end{itemize}
Moreover, as $F(x,y,z)$ is not obvious, the restriction
   $\rhobar^{(2)}|_{I_{\Qp}}$ does not have the form $\omega^{x+2}
   \oplus \omega^{z}$ or $\omega_{\sigma_1}^{(x+2)+pz} \oplus \omega_{\sigma_2}^{(x+2)+pz}$, where
   $S_{\F_{p^2}} = \{\sigma_1,\sigma_2\}$. Hence $x-z \ge p-1$. A calculation shows that the irreducible constituents of 
   $L_{(x+1,z)} \otimes_{\Zpbar} \Fpbar$ are $F(x-p+2,z)$, $F(x-p+1,z+1)$, $F(z+p-1,x-p+2)$ if $p-1 \le x-z < 2p-2$
   (where the second weight is omitted if $x-z = p-1$)
   and $F(z+p-1,z+1)$ (twice), $F(z+1,z)$ if $x-z = 2p-2$. Hence $\rhobar^{(2)}|_{I_{\Qp}}$ is either $\omega^{x+1}
   \oplus \omega^{z+1}$ or $\omega_{\sigma_1}^{(x+3)+p(z-1)} \oplus \omega_{\sigma_2}^{(x+3)+p(z-1)}$ (the latter only if $x-z\neq 2p-2$).
   (This is of course compatible with \cite[Thm.~3.2.1]{ber05}
   computing the reduction of crystalline representations
   with Hodge--Tate weights $\{x+2,z\}$.)

In the first case, one finds that $\rhobar|_{I_\Qp} \cong \omega^{x+1}
\oplus \omega^{y+1} \oplus \omega^{z+1}$. If $x-y,y-z < p-1$, then
$F(z+p-2,y,x-p+2)$ is an obvious weight for $\rhobar$, and $F(x,y,z)$ is
its shadow (so in particular is not obscure). Suppose on the other
hand that $x-y=p-1$ or $y-z=p-1$.  Once again $F(x,y,z)$ cannot be a shadow
weight (as a shadow weight $F(x,y,z)$ always has $x-y,y-z < p-1$) but
sometimes it is an obvious weight. 

The weight $F(y+p-1,y,z)$ is straightforwardly checked to be obvious precisely when $p=2$, or else $p > 2$ and $y-z \in
\{0,p-2\}$. (When $y-z=p-2$ the obvious lift is a sum of characters,
while when $y=z$ the obvious lift has the shape $\rho^{\Qp}_{\{y+p+1,z\}}
\oplus \psi^{\Qp}_{\{y+1\}}$.)  Thus the weight $F(y+p-1,y,z)$ is an
obscure weight for $\rhobar|_{I_\Qp} \cong \omega^{y+1}
\oplus \omega^{y+1} \oplus \omega^{z+1}$ exactly when $p > 2$ and $y-z
\not\in \{0,p-2\}$. By a similar analysis the weight $F(x,y,y-p+1)$ is
an obscure weight for $\rhobar|_{I_\Qp} \cong \omega^{x+1}
\oplus \omega^{y+1} \oplus \omega^{y+1}$ exactly when $p > 2$ and $x-y
\not\in \{0,p-2\}$.

Now suppose instead that  $x-z \neq 2p-2$ and $\rhobar|_{I_{\Qp}} \cong
\omega^{y+1} \oplus \omega_{\sigma_1}^{(x+3)+p(z-1)} \oplus
\omega_{\sigma_2}^{(x+3)+p(z-1)}.$  If $x-y,y-z < p-1$ then again the
weight $F(z+p-2,y,x-p+2)$ is obvious (the obvious lift has the shape
$\psi^{\Qp}_{\{y+1\}} \oplus \rho^{\Qp}_{\{x-p+2,z+p\}}$) and
$F(x,y,z)$ is its shadow.  Suppose on the other
hand that $x-y=p-1$ or $y-z=p-1$.  Then once again $F(x,y,z)$ cannot be a shadow
weight, while sometimes it is an obvious weight. 

The weight $F(y+p-1,y,z)$ with $y-z\neq p-1$  can be checked to be obvious precisely when $y-z=p-2$; in this
case the
obvious lift has the shape $\psi^{\Qp}_{\{z\}} \oplus
\rho^{\Qp}_{\{y+p+1,y+1\}}$.   (Note that $y-z=p-1$ is excluded because $x-z\neq
2p-2$.)   Thus the weight $F(y+p-1,y,z)$ is an
obscure weight for $\rhobar|_{I_{\Qp}} \cong
\omega^{y+1} \oplus \omega_{\sigma_1}^{(y+2)+pz} \oplus
\omega_{\sigma_2}^{(y+2)+pz}$ precisely when $y-z
\not\in \{p-2,p-1\}$. By a similar argument the weight $F(x,y,y-p+1)$ is an obscure
weight for $\rhobar|_{I_{\Qp}} \cong
\omega^{y+1} \oplus \omega_{\sigma_1}^{(x+2)+py} \oplus
\omega_{\sigma_2}^{(x+2)+py}$  exactly when  $x-y
\not\in \{p-2,p-1\}$. This completes our analysis of obscure weights
for $\GL_3(\Qp)$.

\end{example}

One might optimistically hope that there is an equality
$\Wexpl(\rhobar) = \Wcrisexists(\rhobar)$;
for example this is known
to be the case when $n \le 2$ (except for $n=2$ and $p=2$)  thanks to \cite{gls13}. Unfortunately
we do not expect this to be true in general; for example, in 
Example~\ref{ex::gl3-shifted-weights} we give some explicit
examples in the case of~$\GL_3(\Qp)$ of weights  which are not
in~$\Wexpl(\rhobar)$ but which we suspect are in~$\Wcrisexists(\rhobar)$. 
 Furthermore we remark that the sets $\Wcrisexists(\rhobar)$ must also be  
   compatible  with other functorial operations, such as suitable tensor products 
   and inductions, and it is far from clear whether or not the sets 
   $\Wexpl(\rhobar)$ satisfy these compatibilities.
  
 On the other hand, in the unramified setting we are prepared to conjecture that 
 these two weight sets are equal at least for sufficiently generic
 $\rhobar$.

\begin{conj} 
   \label{conj:Wexpl-equals-Wcris-ss} Suppose that $K/\Qp$ is
   unramified and that $\rhobar|_{I_K}$ is 
   semisimple and sufficiently generic. Then $\Wexpl(\rhobar) = \Wcrisexists(\rhobar)$.
 \end{conj} 
 
 \begin{conj} 
   \label{conj:explicit-version-ss} Suppose that for each $v|p$, 
   the extension $F_v/\Qp$ is unramified and $\rbar|_{I_{F_v}}$ is 
   semisimple and sufficiently generic. Then the weight part of Serre's 
   conjecture \upl Conj.~\ref{conj:basic-conjecture}\upr\ holds with 
   $\W_v(\rbar) = \Wexpl(\rbar|_{G_{F_v}})$. 
 \end{conj} 
 
The general definition of ``sufficiently generic'' will be given in
Definition~\ref{defn:suff-generic}, but to give the reader a sense of
the meaning of this term, we spell it out in the case where
$\rhobar$ is a direct sum of characters.

\begin{example}
  \label{sufficiently-generic-niveau-1}
Suppose that $K$ is an unramified extension of $\Qp$ and that
$\rhobar$ is a sum of characters, so that $\rhobar|_{I_K} \cong \oplus_{i=1}^n \prod_{\sigma \in
  S_k} \omega_{\sigma}^{\mu_{\sigma,i}}$ for integers $\mu_{\sigma,i}$
(very much not uniquely defined).
Fix $\delta > 0$. We say that $\rhobar$ is $\delta$-generic if it is
possible to choose the integers $\mu_{\sigma,i}$ such that
$\mu_{\sigma,i}-\mu_{\sigma,i+1} \ge \delta$ for all $1 \le i < n$ and
all $\sigma$,
and furthermore $\mu_{\sigma,1}-\mu_{\sigma,n} \le p-n-\delta$.
We say that a statement is true for sufficiently generic $\rhobar$
if there exists $\delta > 0$ (independent of $p$) such that the statement is true for all
$\delta$-generic $\rhobar$.
\end{example}

 We 
 will prove in Theorem~\ref{thm:main-result} that Conjecture~\ref{conj:explicit-version-ss} agrees with all 
 other conjectures in the literature, in particular that of 
 \cite{herzigthesis} (hence our willingness to make the conjecture,
 even though it is stronger than what is entailed by the
 generalised Breuil--M\'ezard formalism and by Conjecture~\ref{conj:crystalline-version}). 
 In fact we   
  will show in Theorem~\ref{thm:main-result} that for sufficiently generic 
  $\rhobar|_{I_K}$ and $K/\Qp$ unramified we have
   $\Wexpl(\rhobar) = \cC(\Wobv(\rhobar))$ (that is, 
   there are no obscure weights), so that   
  in the context of Conjecture~\ref{conj:explicit-version-ss} the  construction of the set $\Wexpl(\rbar|_{G_{F_v}})$ is somewhat simplified.  
 
 We stress that for any fixed $\rhobar$ the set of weights $\Wexpl(\rhobar)$ is quite 
 explicit in principle, at least for $p$ large:\ the calculation of $\Wobv(\rhobar)$ is a 
 combinatorial exercise (as in 
 Examples~\ref{ex:niveau-one-3-dimensional} 
 and~\ref{ex:unramified-3-example}), and then  
 the shadow and obscure weights are determined by the Jordan--H\"older 
 decompositions of the representations $L_\lambda \otimes_{\Zpbar} 
 \Fpbar$. As for the computability of those decompositions, consider
 first the case $k=\Fp$. One needs to decompose $\GL_n$-modules $L_\lambda \otimes_{\Zpbar} 
 \Fpbar$ with $\lambda$ dominant and
 $\|\lambda\| < Np$ (for some $N$ independent of $p$, and with $\| \cdot
 \|$ as in Definition~\ref{defn: norm on weights}) into simple
 $\GL_n(\Fp)$-modules. For $p \gg 0$, Lusztig's conjecture allows one
 to recursively decompose $L_\lambda \otimes_{\Zpbar} 
 \Fpbar$ into simple $\GL_n$-modules when
 $\lambda$ is $p$-regular (\cite[II.8.22]{MR2015057},
 \cite{MR2999126}).  For the remaining $\lambda$ one uses
 \cite[II.7.17(b)]{MR2015057}. For decomposing simple $\GL_n$-modules as
 representations of $\GL_n(\Fp)$, see for example
 \cite[\S1.5]{MR0933356}. For general $k$ one follows the same strategy, replacing $\GL_n$ with the algebraic group 
$G = \Res_{W(k)/\Zp} \GL_n$ and $L_\lambda \otimes_{\Zpbar} \Fpbar$
with the dual Weyl module $W(\lambda)$ as defined in Sections~\ref{sec:unramified
   groups}--\ref{sec:herzig-comparison}.

\begin{example}
  \label{ex:FH-ramified-example}
As remarked above,
 we   
  will show in Theorem~\ref{thm:main-result} that for sufficiently generic 
  $\rhobar|_{I_K}$ and $K/\Qp$ unramified we have
   $\Wexpl(\rhobar) = \cC(\Wobv(\rhobar))$. In this example, we show
   that this statement does not extend to the case where $K/\Qp$ is
   ramified.

Suppose that $K/\Qp$ is ramified quadratic and $\rhobar : G_K \to
\GL_3(\Fpbar)$ is such that $\rhobar|_{I_K} \cong\omega^{a+3} \oplus
\omega^{b+2} \oplus \omega^{c+1}$, where $a > b > c$ and $a-c <
p-4$. We claim that $F(a,b,c)$ is an obscure weight of $\rhobar$.

If we had $F(a,b,c) \in \cC(\Wobv(\rhobar))$, then $F(a,b,c) \in
\Wobv(\rhobar)$, as $(a,b,c)$ lies in the lowest alcove. As
$\rhobar|_{I_K}$ is a sum of distinct characters, any obvious
crystalline lift of $\rhobar|_{I_K}$ is a sum of characters. From
Lemma~\ref{lem:existenceofcrystallinechars} we would get that $\rhobar|_{I_K}
\cong \omega^r \oplus \omega^s \oplus \omega^t$ with $(r,s,t) =
(a+2,b+1,c) + w(2,1,0)$ for some permutation $w\in S_3$.
By the bounds on $(a,b,c)$ we get a contradiction.

To show that in fact $F(a,b,c) \in \Wexpl(\rhobar)$, note that we can
find $\rhobar^{(i)} : G_K \to \GL_i(\Fpbar)$ ($i=1,2$) with
$\rhobar^{(1)}|_{I_K} \cong \omega^{b+2}$ and $\rhobar^{(2)}|_{I_K}
\cong \omega^{a+3} \oplus \omega^{c+1}$. By
Lemma~\ref{lem:existenceofcrystallinechars} we have $F(b+2) \in
\Wobv(\rhobar^{(1)})$ and $F(a+2,c) \in \Wobv(\rhobar^{(2)})$. 
Let $S_K
= \{\sigma_1,\sigma_2\}$. We define an $\eta$-partition of $F(a,b,c)$
as follows:
\begin{align*}
\lambda_{\sigma_1} &= (a,b,c),  &\lambda_{\sigma_2} &= 0,\\
\lambda_{\sigma_1}^{(1)}& = (b+1), & \lambda_{\sigma_2}^{(1)} &= (1),\\
\lambda_{\sigma_1}^{(2)} &= (a+1,c), & \lambda_{\sigma_2}^{(2)} &= (1,0).
\end{align*}
Then $L_{\lambda^{(1)}} \otimes \Fpbar \cong F(b+2)$ and
$L_{\lambda^{(2)}} \otimes \Fpbar \cong \Sym^{a-c+1} \Fpbarx{2} \otimes
\Sym^1 \Fpbarx{2} \otimes \det^c.$
We see that $F(a+2,c)$ is a Jordan--H\"older factor of
$L_{\lambda^{(2)}} \otimes \Fpbar$, for example by Brauer's
formula~\cite[II.5.8(b)]{MR2015057}. (The only other factor is
$F(a+1,c+1)$.) From Definition~\ref{defn:explicit-weights-ss} we see
that indeed $F(a,b,c) \in \Wexpl(\rhobar)$.
\end{example}

\subsection{Remarks on the general (non-semisimple) case}\label{sec:obvious-general}

Now let us drop our assumption that $\rhobar|_{I_K}$ is semisimple,
and consider what we might say about explicit weights for $\rhobar$.  
As mentioned in Section~\ref{subsec: crystalline lifts in the
  picture}, one expects that the Serre weights of $\rhobar$ should be a
subset of the Serre weights of $\rhobar^{\ss}$.
However, we hesitate to make any sort of precise conjecture:\ evidence is
scant beyond the two-dimensional case, and the limited information
that we do possess suggests that there are serious
complications that arise already in the three-dimensional case.

We begin with a brief review of  the two-dimensional case (for
$p > 2$ and general $K/\Qp$) as studied in \cite{gls13}.  It is shown
that $\Wcrisexists(\rhobar)$ depends only on
$\rhobar|_{I_K}$ (\cite[Prop.~6.3.1]{gls13}), and that
$\Wcrisforall(\rhobar) = \Wcrisexists(\rhobar) \subset
\Wcrisexists(\rhobar^{\ss})$.
Suppose now that $\rhobar : G_K \to \GL_2(\Fpbar)$ is an extension of
characters $\chi_1$ by $\chi_2$. By Example~\ref{ex:2-dimensional} any weight $a \in
\Wcrisexists(\rhobar^{\ss})$ is witnessed by a sum of
characters.
Let
$L(\chi_1,\chi_2,a)$ be the subset of $H^1(G_K,\chi_2 \chi_1^{-1})$
obtained by taking the union, over all literal lifts
$\psi_1,\psi_2$ of $\chi_1,\chi_2$ such that $\psi_1 \oplus \psi_2$
witnesses $a\in
\Wcrisexists(\rhobar^{\ss})$, of the image in $H^1(G_K,\chi_2 \chi_1^{-1})$ of
$H^1_f(G_K,\Zpbar(\psi_2 \psi_1^{-1}))$. Then $a \in
\Wcrisexists(\rhobar)$ if and only if the extension class corresponding
to $\rhobar$ lies in $L(\chi_1,\chi_2,a)$.

In fact it is almost always true that  if $\psi_1,\psi_2$ as above  are chosen
so that the dimension of  $H^1_f(G_K,\Zpbar(\psi_2 \psi_1^{-1}))$ is as large as possible, then
the image of that space in $H^1(G_K,\chi_2
\chi_1^{-1})$ is actually equal to $L(\chi_1,\chi_2,a)$.  The lone
exception occurs when $\chi_2 \chi_1^{-1} $ is the cyclotomic character
and $a$ is represented by $(x_{\sigma},y_{\sigma})_{\sigma \in S_k}$ with $x_{\sigma}-y_\sigma = p-1$ for all $\sigma \in S_k$. In that
case, if the $\psi_i$ as above are chosen so that the dimension of
$H^1_f(G_K,\Zpbar(\psi_2 \psi_1^{-1}))$ is as large as possible, then
the images of the spaces $H^1_f(G_K,\Zpbar(\lambda \psi_2
\psi_1^{-1}))$ cover $L(\chi_1,\chi_2,a)$ as $\lambda$ varies over all
unramified characters with trivial reduction mod $p$ (cf.\
\cite[Thm.~5.4.1, Thm.~6.1.8]{gls13} and their proofs).

In three dimensions, the situation appears to be
considerably more complicated.  In addition to the discussion of Section~\ref{subsec: crystalline lifts in the
  picture}, we have the following example.

\begin{example}\label{ex:herzig-morra}
  Suppose that $\rhobar : G_{\Qp} \to \GL_3(\Fpbar)$ is such that
  \begin{eqnarray*}
    \rhobar \sim
    \begin{pmatrix}
      \chi_1 & * & * \\ &\chi_2 &* \\ && \chi_3
    \end{pmatrix}.
  \end{eqnarray*}
  Suppose moreover that $\chi_1|_{I_{\Qp}} = \omega^{a+1}$, $\chi_2|_{I_{\Qp}} = \omega^{b+1}$, $\chi_3|_{I_{\Qp}} =
  \omega^{c+1}$ with integers $a > b > c > a-(p-1)$, where all gaps in the inequalities are at least
  3, and that $\rhobar$ is maximally non-split (i.e.\ $\chi_1$ is the unique subrepresentation and
  $\chi_3$ the unique quotient representation). When the $\chi_i$ are fixed, the isomorphism class
  of $\rhobar$ is determined by an invariant $\FL(\rhobar) \in \Pr^1(\Fpbar) \setminus \{\chi_2(p)\}$.
  In the global setting of a suitable compact unitary group the Serre weights of $\rhobar$ are almost
  completely determined in \cite{HerzigMorra}: with the possible addition of the shadow weight $F(c+p-1,b,a-p+1)$,
  the set of Serre weights equals
  \begin{equation*}
    \begin{cases}
      \{F(a-1,b,c+1)\} & \text{if $\FL(\rhobar) \not\in \{0,\infty\}$}, \\
      \{F(a-1,b,c+1), F(b+p-1,a,c)\} & \text{if $\FL(\rhobar) = 0$}, \\
      \{F(a-1,b,c+1), F(a,c,b-p+1)\} & \text{if $\FL(\rhobar) = \infty$}. \\
    \end{cases}
  \end{equation*}
 That is, the set of Serre weights consists of one element of
 $\Wobv(\rhobar^{\ss})$, namely the obvious weight coming from the
 diagonal characters of $\rhobar$ in their given order, together with
 a set of shadow weight(s) depending on the parameter $\FL(\rhobar)$.
  The occasional presence of the weights $F(b+p-1,a,c)$ and $F(a,c,b-p+1)$ suggests that there is no naive explicit conjecture for
  non-semisimple $\rhobar$.  We make two further remarks. First,
  \cite{HerzigMorra} verify that in
  this setting there exists an ordinary crystalline lift of $\rhobar$
  that witnesses the containment $F(a-1,b,c+1) \in \Wcrisexists(\rhobar)$.
  Second, when the maximal non-splitness assumption above is dropped, an upper bound on the set
  of Serre weights of $\rhobar$ was obtained by Morra--Park \cite{MorraPark}.
\end{example}

\subsection{Shifted weights}
 \label{sec:shifted-weights}

We continue to consider $\rhobar$ such that $\rhobar|_{I_K}$ may not
be semisimple. 
Recall from Section~\ref{subsec:the picture} that when $n=2$ and $K=\Qp$, every component
of~$\cXbar$ labeled by $1$ is also labeled by $\Sym^{p-1}\Fpbarx{2}$;
equivalently, every $\rhobar$ with $1$ as a Serre weight also has
$\Sym^{p-1}\Fpbarx{2}$ as a Serre weight. This can be viewed as the
first instance of the following more general question:\ for which pairs of
Serre weights $F,F'$ does $F \in \W_{\BM}(\rhobar)$ imply that one
must have $F' \in
\W_{\BM}(\rhobar)$ as well?  In this case we say that the weight
$F$ {\em entails} the weight~$F'$.

The geometric perspective explained in Section~\ref{subsec:the picture} (combined with the
Breuil--M\'ezard conjecture) allows a significant
reduction to this question. The weight $F$ will entail the weight
$F'$ if and only if every component of~$\cXbar$ labeled by $F$ is also
labeled by $F'$; to check the latter it suffices to check that every
generic $\Fpbar$-point (of some component) that has $F$ as a Serre
weight also has $F'$ as a Serre weight. In particular, if one believes
that the Breuil--M\'ezard conjecture holds, then one should believe that $F$ entails $F'$
for arbitrary $\rhobar$ as long as the same holds for maximally
non-split upper-triangular~$\rhobar$ (or even those that are
sufficiently generic to lie on just one component of~$\cXbar$).  

In the remainder of this section we will discuss the following
specific instance of the weight entailment question.

\begin{defn}
  If $a,b$ are Serre weights, we say that $b$ is a
\emph{shift} of $a$ if there exists $1 \le i_0  < n$ such that
$$ b_{\sigma,i} - a_{\sigma,i} = \begin{cases} p-1 & \text{if $i \le i_0$} \\ 0 &
  \text{if $i > i_0$}\end{cases}$$ for all $\sigma \in S_k$. 

Note that this definition only depends on $i_0$ but not on the choice of representative $a \in
(X_1^{(n)})^{\Sk}/\sim$, and that we must have
$a_{\sigma,i_0} = a_{\sigma,i_0+1}$ for all $\sigma\in S_k$ in order
for any shift of $a$ to exist.
\end{defn}

\begin{question}\label{weight-shift-Q}
 If the weight $b$ is a shift
 of the weight $a$, does $a \in \W_{\BM}(\rhobar)$ entail 
$b \in \W_{\BM}(\rhobar)$ for representations $\rhobar : G_K \to
 \GL_n(\Fpbar)$?

We equally well ask the same question with $\W_{\BM}(\rhobar)$
replaced by any set that is conjecturally the same as it, such as
$\Wcrisforall(\rhobar)$, $\Wcrisexists(\rhobar)$, or $\W_{\cS}(\rhobar)$ for any Breuil--M\'ezard system $\cS$.
 \end{question}

 \begin{remark}
   This question was suggested to us by the work of Ash--Pollack--Soares \cite{MR2103328} and
   Doud \cite{doud-supersingular}:\ the weight sets conjectured for $\rhobar : G_{\Qp} \to \GL_3(\o\F_2)$ in \cite[\S2]{MR2103328},
   resp.\ for irreducible $\rhobar : G_{\Qp} \to \GL_n(\Fpbar)$ in 
   \cite[Conj.~2.10]{doud-supersingular} are by definition closed under
   shifts (cf.\ also \cite[Def.~2.7]{doud-supersingular}).
 \end{remark}

 \begin{example}\ Suppose that $n=2$ and $p > 2$. Twisting by a suitable character,
   Question~\ref{weight-shift-Q} when $n=2$ can be reduced to the case
  where $a=0$ and $b_{\sigma} = (p-1,0)$ for all $\sigma \in
  \Sk$. Since one knows (even if $\rhobar|_{I_K}$ is not semisimple) that $\WBT(\rhobar) = \Wcrisexists(\rhobar) =
  \Wcrisforall(\rhobar)$ in this setting by the work of \cite{gls13}, an affirmative
  answer to Question~\ref{weight-shift-Q} for any of these sets is
  equivalent to the statement that if $\rhobar:
  G_K \to \GL_2(\Fpbar)$ has a regular Barsotti--Tate lift
  then it also has a
  crystalline lift with Hodge type some lift of $b$, which is well
  known (and can be proved for example via the techniques
  of~\cite{gls13}, or by using the corresponding fact for automorphic
  forms and the potential modularity techniques
  of~\cite[App.\ A]{geekisin}).
 \end{example}

 \begin{example}\label{ex::gl3-shifted-weights}
   We now give an extended discussion of the case $\GL_3/\Qp$ which
   suggests to us that Question~\ref{weight-shift-Q} may have an
   affirmative answer in this setting as well.  Computational evidence
   for this
   (due to \cite{bib:ADP,MR2103328,doud-supersingular})  will be
   reviewed in Section~\ref{sec:computational}. Our discussion
    will be heuristic; in particular we will assume  the
    Breuil--M\'ezard conjecture, and will extrapolate the labelling of
    the irreducible components of~$\cXbar$ from the case
    $n=2$ in a speculative fashion. In particular, note that for
    $n=2$, the labeling of the irreducible components of~$\cXbar$ is
    dictated by the restrictions to inertia of the characters of
    generic reducible~$\rhobar$ on those components, with the subtlety 
that in the ambiguous case that these weights could either be one-dimensional
or twists of the Steinberg representation, we always predict the
twist of the Steinberg representation, and only predict the
one-dimensional representation in the case that (twists of) these
generic~$\rhobar$ admit a crystalline lift of Hodge type~$0$. 

In particular, every component labeled by a one-dimensional weight is
also labeled by the corresponding twist of the Steinberg
representation, and this fact is reflected by the fact that a generic
reducible representation admitting a crystalline lift of Hodge type 0
also necessarily admits one of Hodge type corresponding to the
Steinberg representation. We will now assume that similar
considerations apply for $n=3$, and see what is implied.

  We first suppose that $\rhobar|_{I_\Qp}$ is semisimple and observe that
   the set $\Wexpl(\rhobar)$
   described in Section~\ref{sec:shad-weights:-inert} is \emph{not} necessarily
   closed under shifts, so that a positive answer to
   Question~\ref{weight-shift-Q} means that $\Wexpl(\rhobar)$ is at
   best a proper subset of $\Wcrisexists(\rhobar)$. We leave it as an
   exercise to the reader to check the following. If
   $\rhobar$ is reducible, then
   $\Wexpl(\rhobar)$ is closed under shifts. (Use that weights
   $F(x,y,y)$ or $F(y,y,z)$ are either obvious or obscure.)  On the
   other hand if 
  $$ \rhobar|_{I_{\Qp}} \cong \oplus_{\sigma \in S_{\F_{p^3}}}
  \omega_{\sigma}^{(y+2) + p(y+1) + p^2 z} $$ with $0 \le y-z \le p-2$
  then $F(y,y,z) \in \Wexpl(\rhobar)$ but $F(y+p-1,y,z) \not\in
  \Wexpl(\rhobar)$ (this can be checked by hand, or seen 
from the tables in Proposition~\ref{prop:gl3-non-regular}), and dually if
 $$ \rhobar|_{I_{\Qp}} \cong \oplus_{\sigma \in S_{\F_{p^3}}}
   \omega_{\sigma}^{y + p(y+1) + p^2(x+2)} $$ with $0 \le x-y\le p-2$
   then $F(x,y,y) \in \Wexpl(\rhobar)$ but $F(x,y,y-p+1) \not\in
   \Wexpl(\rhobar)$; and moreover these are the only shifts missing from
   $\Wexpl(\rhobar)$ for irreducible $\rhobar$. (It is perhaps worth
      pointing out that shifts do not account for all of the obscure
      weights of  Example~\ref{ex:exceptional-weight-gl3}, so that 
 neither shifts nor obscure weights alone can
   account for the difference between $\cC(\Wobv(\rhobar))$ and the
   full set of weights.)

    Let us now consider the weight entailment problem for weights of
    the form $F = F(y,y,z)$ and $F'=F(y+p-1,y,z)$; the case of $F(x,y,y)$ and
    $F(x,y,y-p+1)$ will be dual.  Recall (e.g.\ from
    Example~\ref{ex:closure-for-niveau-1-gl3}) that $L_{\lambda}
    \otimes_{\Zpbar} \Fpbar = F(\lambda)$ for both $\lambda = (y,y,z)$
    and $\lambda=(y+p-1,y,z)$, so that we expect that $F(y,y,z)$
    (resp.\ $F(y+p-1,y,z)$) is a
    weight for $\rhobar$ if and only if $\rhobar$ has a crystalline
    lift of Hodge type $(y,y,z)$ (resp.\ $(y+p-1,y,z)$).
   Suppose that a component $\cZ$
    of $\cXbar$ has $F(y,y,z)$ among its labels, so that a generic $\Fpbar$-point on
    $\cZ$ corresponds to $\rhobar$ that
    has a crystalline lift of Hodge type $(y,y,z)$. We
    wish to know whether $\rhobar$ also has a crystalline lift of Hodge type
    $(y+p-1,y,z)$.

  If $y-z \le p-3$, Fontaine--Laffaille theory
  implies that a generic $\Fpbar$-point on $\cZ$ corresponds to
  $\rhobar$ having  the shape
  \begin{eqnarray}\label{eq:rhobar-shape-shift}
    \rhobar|_{I_{\Qp}} \sim
    \begin{pmatrix}
      \omega^{y+2} & * & * \\ &\omega^{y+1} &* \\ && \omega^{z}
    \end{pmatrix}.
  \end{eqnarray}
The same conclusion seems likely to hold if $y-z=p-2$:\ an argument as
in \cite[Prop.~7.8]{gls12} shows at least  that $\rhobar$ has the same
semisimplification as the representation \eqref{eq:rhobar-shape-shift}, and it seems quite plausible that the
order of the characters on the diagonal will be correct. Suppose this
is so.

Let $\chi_1,\chi_2,\chi_3$ be the characters on the diagonal of $\rhobar$
(in the same order as given in \eqref{eq:rhobar-shape-shift}). Then as long as none of $\chi_i/\chi_j$ with $i < j$ are cyclotomic it is
straightforward to show that a crystalline lift of $\rhobar$ with Hodge type
$(y+p-1,y,z)$ exists. One can even take this lift to be
upper-triangular;  see for example \cite[Lem.~3.1.5]{gg}. Even if
some $\chi_i/\chi_j$ is cyclotomic, it is reasonable to imagine that the
same conclusion holds; e.g.\ when $y-z \le p-3$ this is immediate
from~\cite[Cor.\ 2.3.5]{gs-fllifts}.

Alternately, it is plausible
that any $\rhobar$ having the shape \eqref{eq:rhobar-shape-shift} and
having a crystalline lift of Hodge type $(y,y,z)$ has an
\emph{ordinary} such lift, with characters down the diagonal having
Hodge--Tate weights $y+2$, $y+1$, $z$ (in that order); cf.\ the first
remark at the end of Example~\ref{ex:herzig-morra}, as well as the
discussion of the case $n=2$ and $K/\Qp$ arbitrary in Section~\ref{sec:obvious-general}.
Write this lift
as an extension of a two-dimensional crystalline representation $V$
(with Hodge--Tate weights $\{y+1, z\}$) by a character $W$. One may then hope
to produce the desired lift of $\rhobar$ of Hodge type $(y+p-1,y,z)$ by considering extensions of $V$ by
unramified twists of $W \otimes \varepsilon^{p-1}$.

We remark that the above arguments are agnostic regarding the
case $y-z=p-1$. However, it is at least the case for
$\rhobar|_{I_\Qp}$  semisimple that the set $\Wexpl(\rhobar)$ contains
$F(y+p-1,y,y-p+1)$ whenever it contains $F(y+p-1,y,y)$.
 \end{example}

When $n > 3$ the heuristic arguments in
Example~\ref{ex::gl3-shifted-weights} at least make it plausible that
Question~\ref{weight-shift-Q} has an affirmative answer for
shifts of weights $F = F(a_1,\ldots,a_n)$ with $a_1-a_n$ small (e.g.\ when
$a_1-a_n \le p-n$, so that Fontaine--Laffaille theory still determines
the shape of $\rhobar|_{I_{\Qp}}$ for $\rhobar$ corresponding to a
generic $\Fpbar$-point on a component of $\cXbar$ labeled by $F$).

\subsection{Summary}
\label{sec:summary}

We briefly summarize the Serre weight conjectures that we have
explained in this section.

\begin{defn}
  \label{defn:various-serre-weight-sets} Let $\rhobar : G_K \to
  \GL_n(\Fpbar)$ be a  representation.
\begin{itemize}
\item If the generalised Breuil--M\'ezard conjecture holds, we define
  $\WBM(\rhobar)$ to be the set of Serre weights $a$ such that
  $\mu_a(\rhobar) > 0$.

\item We define $\Wcrisexists(\rhobar)$ to be the set of Serre weights $a$ such
  that $\rhobar$ has a crystalline lift of Hodge type $\lambda_a$ for
  some lift $\lambda_a$ of $a$.
\item We define $\Wcrisforall(\rhobar)$
  to be the set of Serre weights $a$ such
  that $\rhobar$ has a crystalline lift of Hodge type $\lambda_a$ for
 every lift $\lambda_a$ of $a$.  
\item If $\rhobar|_{I_K}$ is semisimple, we define a non-empty set of 
obvious
  weights $\Wobv(\rhobar)$ in
  Definition~\ref{defn:obvious-lift-ss}, and a set of explicit weights
  $\Wexpl(\rhobar) \supset \cC(\Wobv(\rhobar))$ in Definition~\ref{defn:explicit-weights-ss}.
\end{itemize}
\end{defn}

\begin{conj}
  \label{conj:various-serre-weight-conjectures-local}  Let $\rhobar : G_K \to
  \GL_n(\Fpbar)$ be a  representation. Assume that
  $\rhobar|_{I_K}$ is semisimple.
  \begin{enumerate}
  \item We have $\cC(\Wcrisexists(\rhobar)) = \Wcrisexists(\rhobar)$.
  \item The sets $\Wcrisexists(\rhobar)$ and $\Wcrisforall(\rhobar)$ depend
    only on $\rhobar|_{I_K}$, as does 
    $\WBM(\rhobar)$ if it is defined \upl i.e.\ if the generalised
    Breuil--M\'ezard conjecture holds\textup).
  \item We have $\Wexpl(\rhobar) \subset \Wcrisexists(\rhobar) = \Wcrisforall(\rhobar)$.
  \item If the generalised Breuil--M\'ezard conjecture holds then
    $\WBM(\rhobar) = \Wcrisexists(\rhobar)$.
  \item If $K/\Qp$ is unramified, and $\rhobar|_{I_K}$ is
    sufficiently generic,
    then  $    \Wcrisexists(\rhobar)=\Wexpl(\rhobar) = \cC(\Wobv(\rhobar))$.
  \end{enumerate}
\end{conj}

\begin{conj}
  \label{conj:serre-weight-conjectures-global}
  If each $\rbar|_{I_{F_v}}$ is semisimple, then the weight part of Serre's conjecture \upl
  Conj.~\ref{conj:basic-conjecture}\upr\ holds with $\W_v(\rbar) = \Wcrisexists(\rbar|_{G_{F_v}})$.
\end{conj}

Finally (assuming again that~$\rhobar|_{I_K}$ is semisimple), we recall that $\Wcrisforall(\rhobar) \subset
\Wcrisexists(\rhobar)$ by definition; that if the generalised
Breuil--M\'ezard conjecture holds then we have $\WBM(\rhobar)
\subset \Wcrisforall(\rhobar) = \Wcrisexists(\rhobar)$
(cf.~Lemma~\ref{lem:weight-implies-lift}); and that if
the set
$\Wcrisexists(\rhobar)$ depends only on $\rhobar|_{I_K}$, then
$ \Wobv(\rhobar) \subset \Wcrisexists(\rhobar)$ (and similarly for
$\Wcrisforall(\rhobar)$).

\section{Existing  conjectures in the literature}\label{sec:evidence} In this
section we review the theoretical and computational evidence
for our conjectures, beyond the case $n=2$ which was discussed in detail
above. We also make comparisons with other conjectures in the literature.

\subsection{The case of \texorpdfstring{$\GL_3(\Qp)$}{GL(3)/Qp}}\label{subsec: GL3}

Take $K = \Qp$ and fix an odd and irreducible representation $\rbar : G_{\Q} \to \GL_n(\Fpbar)$ such
that $\rbar|_{I_{\Qp}}$ is semisimple.  The first Serre weight conjectures in this context were made by
Ash, Doud, Pollack, and Sinnott \cite{bib:ASinn}, \cite{bib:ADP}. We will discuss their work in Section~\ref{sec:conjecture-ash-doud} below.
  Later in the paper (Section~\ref{sec:herzig-comparison}), we will show that
Conjecture~\ref{conj:explicit-version-ss} agrees 
with the Serre weight conjecture made by the second author
in \cite{herzigthesis}.  (In fact we will ultimately work in a somewhat
more general context than this.)  Recall, however, that in
Conjecture~\ref{conj:explicit-version-ss} the representation $\rbar|_{I_{\Qp}}$
is assumed to be sufficiently generic. In this next section we will check
that in the $3$-dimensional case the conjecture of \cite{herzigthesis} 
agrees \emph{completely} with the explicit set of weights described
in the previous section.

\subsection{The conjecture of \texorpdfstring{\cite{herzigthesis}}{[Her09]} for \texorpdfstring{$\GL_3$}{GL(3)}
  over \texorpdfstring{$\Qp$}{Q\_p}}   

Recall that \cite[Conj.~6.9]{herzigthesis} predicts the 
set of \emph{regular} Serre weights for which a given irreducible, odd representation $\rbar : G_{\Q} \to
\GL_n(\Fpbar)$ is automorphic.  Regular Serre weights are defined as
follows.\footnote{We caution the reader that the term regular as
  applied to Serre weights is unrelated to the term
  regular as applied to Hodge--Tate weights in Section~\ref{subsec:notation}.}

\begin{defn}
  \label{defn:regular-weight}
  A weight $a \in \SW(k,n)$ is said to be \emph{regular}
if $a_{\sigma,i} - a_{\sigma,i+1} < p-1$ for all $\sigma,i$, and
\emph{irregular} otherwise.  Let $\SW_{\reg}
\subset \SW(k,n)$ be the set of regular weights.
\end{defn}

The set of Serre weights predicted in \cite{herzigthesis} is denoted $\W^{?}(\rbar|_{I_{\Qp}})$,
so we want to check  that
$\W^{?}(\rhobar|_{I_{\Qp}}) = \Wexpl(\rhobar) \cap \SW_{\reg}$ for a
local representation $\rhobar : G_{\Qp} \to \GL_3(\Fpbar)$ such that
$\rhobar|_{I_{\Qp}}$ is semisimple.
To describe the set $\W^{?}(\rhobar|_{I_{\Qp}})$, we begin with the following
definitions.

\begin{defn}
  \label{defn:tau-and-good}
  Suppose that $(w,\mu) \in S_n \times \Z^n$. Let $w = w_1\cdots w_m$
  be the unique decomposition of the permutation $w$ into
  disjoint cycles (including trivial cycles), and write $\mu = (\mu_1,\ldots,\mu_n)$. 
  \begin{enumerate}

\item If $w_i =
  (c_0\,\cdots\, c_{d_i-1})$ we set $N_i =
  \sum_{j=0}^{d_i-1} p^j \mu_{c_j}$, write $k_i = \F_{p^{d_i}}$,
 and define $\tau_{d_i}(w_i,\mu)$ to be the isomorphism class of the
 inertial Galois representation
  $\oplus_{\sigma \in S_{k_i}} \omega_{\sigma}^{N_i}$ of dimension $d_i$.

\item We define $\tau(w,\mu)$ to be the isomorphism class of the
  inertial Galois representation $\oplus_{i=1}^m \tau_{d_i}(w_i,\mu)$ of dimension $n$. 

\item
 We say that the pair $(w,\mu)$ is \emph{good} if for all $1 \le i \le
 m$ and  for all $d \mid
  d_i$, $d \neq d_i$
  we have $(p^d - 1)N_i \not\equiv 0 \pmod{p^{d_i}-1}$.
  \end{enumerate}
It is straightforward to verify that these definitions do not depend on
any of the choices involved.
\end{defn}

\begin{remark}
  \label{rem:canonical-defs}
  The above definitions are concrete instances of the more general and
  more canonical \cite[(6.15)]{herzigthesis} and
  \cite[Def.~6.19]{herzigthesis}. We will recall this more canonical
  definition, and extend it to other groups, in   Proposition~\ref{prop:l-parameter}.
\end{remark}

If $\rhobar|_{I_{\Qp}} \cong \tau(w,\mu)$, the condition
that the pair $(w,\mu)$ is good means, concretely, that the dimensions of the Jordan--H\"older
factors of $\rhobar$ correspond to the cycle type of $w$.

\begin{example}\label{ex:w13}
 Suppose that $w$ is the transposition swapping $i$ and $j$. Then the pair $(w,\mu)$
 is good if and only if $p+1 \nmid \mu_i + p\mu_j$, or equivalently $p+1
 \nmid \mu_i - \mu_j$. In particular this is
 always the case if $|j-i|=1$ and $\mu = \lambda + (n-1,\ldots,1,0)$ where
 $\lambda$ is the lift of a Serre weight.
\end{example}

\begin{example}\label{ex:w123}
  Suppose that $n=3$ and $w = (i\, j\, k)$ is a $3$-cycle. If $\mu =
  \lambda + (2,1,0)$ where $\lambda$ represents a regular Serre
  weight, it is a straightforward exercise to verify that the pair
  $(w,\mu)$ is always good.
\end{example}

Suppose for the remainder of
this section that $n=3$, so that $\W_{\reg}$ refers to the regular
weights in $\W(\Fp,3)$, and write $\eta = (2,1,0)$.

\begin{defn}
  \label{defn:Xsets}  We define $X_{\reg}^{(3)} \subset X_1^{(3)}$ to
  be the set of triples such that
$a\ge b \ge c$ and $a-b,b-c \le p-2$.
Note that $X_{\reg}^{(3)}/{\sim} \cong \W_{\reg}$.

If $(x,y,z) \in \Z^3$, let $\reg(x,y,z)$ be the unique
element of $\W_{\reg}$ represented by some  $(x',y',z') \in X_{\reg}^{(3)}$ with $(x',y',z')
\equiv (x,y,z)$ modulo $(p-1)\Z^3$.
\end{defn}

\begin{defn}
  \label{defn:A-sets}
Following \cite[Prop.~3.18]{herzigthesis} (see also 
 Example~\ref{ex:closure-for-niveau-1-gl3}), we define the function
 (from weights to sets of weights)
 $$r\big(F(x,y,z)\big) =  \begin{cases} \{F(x,y,z)\} & \text{if } x-z \ge 
   p-2, \\ 
 \{F(x,y,z),\ F(z+p-2,y,x-p+2) \} & \text{if } x-z <
   p-2. \end{cases} $$ 
Following \cite[Def.~7.3]{herzigthesis} we set $\HA(\mu) = r(\reg(\mu-\eta))$
for each $\mu   \in X_1^{(3)}$. Note that if $F \in \W_{\reg}$ then
$r(F) \subset \W_{\reg}$, and so $\HA(\mu) \subset \W_{\reg}$
for any $\mu \in X_1^{(3)}$.
\end{defn}

The set  $\W^{?}(\rhobar|_{I_{\Qp}})$ in \cite{herzigthesis} is
defined in terms of a Deligne--Lusztig representation associated to
$\rhobar|_{I_{\Qp}}$.  In the three-dimensional case we have the
following explicit description of this set, which is all we will need for the purposes of this section.

\begin{prop} \textup{(\cite[Prop.~7.4]{herzigthesis})}
  \label{prop:def-of-W-herzig}  Let  $\rhobar: G_{\Qp} \to \GL_3(\Fpbar)$ be a  representation
  such that $\rhobar|_{I_{\Qp}}$ is semisimple. 
  Set $$\HC(\rhobar|_{I_{\Qp}}) = \{ \mu \in X_1^{(3)} : \textrm{there
    exists } w \in S_3 \textrm{ with } (w,\mu) \textrm{ good and }
  \rhobar|_{I_{\Qp}} \cong \tau(w,\mu)\}.$$
Then $$\W^{?}(\rhobar|_{I_{\Qp}}) = \bigcup_{\mu \in
  \HC(\rhobar|_{I_{\Qp}})} \HA(\mu).$$
  \end{prop}

\begin{ex}
  \label{ex:unramified-herzig}
  Suppose that $\rhobar|_{I_{\Qp}}$ is unramified.  Then 
$\Wexpl(\rhobar)$ consists of the four weights $F(p-3,-1,-p+1)$, $F(p-2,-1,-1)$,
  $F(p-2,p-2,-1)$, $F(p-2,-1,-p)$
  by Examples~\ref{ex:unramified-3-example} and \ref{ex:exceptional-weight-gl3}.
  On the other hand,   
  $\HC(\rhobar|_{I_{\Qp}})$ consists of all $\mu \in X_1^{(3)} \cap (p-1) \Z^3$, and  $\reg(\mu-\eta) =
  F(p-3,-1,-p+1)$ for all $\mu \in \HC(\rhobar|_{I_{\Qp}})$.  We
  therefore have
  $\W^{?}(\rhobar|_{I_{\Qp}}) = \{ F(p-3,-1,-p+1) \}$.  Hence we confirm that
  $\W^{?}(\rhobar|_{I_{\Qp}}) = \Wexpl(\rhobar) \cap \SW_{\reg}$,
  since the other three weights in $\Wexpl(\rhobar)$ are irregular.
\end{ex}

We can now prove the main result of this section.

\begin{prop}
  \label{prop:gl3-comparison}
  Let  $\rhobar: G_{\Qp} \to \GL_3(\Fpbar)$ be a  representation
  such that $\rhobar|_{I_{\Qp}}$ is semisimple.  Then we have
  $\Wexpl(\rhobar) \cap \SW_{\reg} = \W^{?}(\rhobar|_{I_{\Qp}})$. 
\end{prop}

\begin{rem}\label{rem:non-regular} In
  Proposition~\ref{prop:gl3-non-regular} we will describe the
  irregular weights in $\Wexpl(\rhobar)$.

\end{rem}

\begin{proof}  Note that obscure weights for $\GL_3/\Qp$ were analysed
  completely in 
  Example~\ref{ex:exceptional-weight-gl3}, and were all found to be
  irregular, so that $\Wexpl(\rhobar) \cap \W_\reg$ consists entirely
  of obvious and shadow weights. It is
  then easy to see from the definition of
  $\Wexpl(\rhobar)$ (or alternatively from
  Proposition~\ref{prop: obvious weights of tau} in the next section), together with the discussion of
  Examples~\ref{ex:closure-for-niveau-1-gl3}, that 
$$ \Wexpl(\rhobar) \cap \SW_{\reg} = \bigcup_{\mu \in
  \HC'(\rhobar|_{I_{\Qp}})} \HA(\mu)$$
where
$$\HC'(\rhobar|_{I_{\Qp}}) = \{ \mu \in X_{\reg}^{(3)} + \eta : \textrm{there
    exists } w \in S_3 \textrm{ with }  \rhobar|_{I_{\Qp}} \cong \tau(w,\mu)\}.$$
To show the inclusion ``$\subset$'' in the Proposition, we have to
consider $\mu \in X_{\reg}^{(3)}+\eta$ such that there exists $w \in S_3$ with
$(w,\mu)$ not good and $\rhobar|_{I_{\Qp}} \cong \tau(w,\mu)$.  By
Examples~\ref{ex:w13} and~\ref{ex:w123} this only happens when $w =
(1\, 3)$, $\mu = (a,b,a-p-1)$. The condition $\mu \in X_{\reg}^{(3)}+\eta$ forces
$2 \le a-b \le p-1$. Then $\tau(w,\mu) \cong \omega^{a-1} \oplus
\omega^{a-1} \oplus \omega^b$ and so $\tau(w,\mu) \cong \tau(1,\mu')$
with $\mu' = (a-1,b,a-p)$. Since $(1,\mu')$ is good and $\mu' \in
X_1^{(3)}$, we have $\mu' \in \HC(\rhobar|_{I_{\Qp}})$.  Directly from
Definition~\ref{defn:A-sets} one calculates that $\HA(\mu) =
\{F(a-2,b-1,a-p-1)\}$ and $\HA(\mu') =
\{F(a-3,b-1,a-p),F(a-2,b-1,a-p-1)\}$. In particular $\HA(\mu) \subset
\HA(\mu') \subset  \W^{?}(\rhobar|_{I_{\Qp}})$, as required.

For the reverse inclusion, one must consider $\mu \in X_1^{(3)} \setminus
(X_{\reg}^{(3)}+\eta)$ such that there exists $w\in S_3$ with $(w,\mu)$ good and
$\rhobar|_{I_{\Qp}}\cong \tau(w,\mu)$.  There are three cases. First,
if $\mu=(a,a,a)$, then $w = 1$ and $\rhobar|_{I_{\Qp}}$ is a sum of
three copies of $\omega^a$.  After a twist we can reduce to the
unramified case, which we have already considered in
Example~\ref{ex:unramified-herzig}. (Alternately, just note that $\mu'
= (a+p-1,a,a-p+1) \in X_{\reg}^{(3)}+\eta$ with $\tau(1,\mu') \cong \tau(1,\mu)$
and $\HA(\mu') = \HA(\mu)$.)  

Second, suppose $\mu = (a,a,c)$ with $0 < a-c
\le p-1$, so that $\HA(\mu) =\{ F(a+p-3,a-1,c)\}$. Without loss of generality we may assume $w \in \{1,(1\,
3),(1\, 2 \, 3)\}$ (note that the pair $((1\, 2),\mu)$ is not
good). For each of these three possibilities it is easy to check that there exists $w' \in S_3$ such
that $\tau(w,\mu) \cong \tau(w',\mu')$ and $\mu' \in \{ (a+p-1,a,c),
(c+p,a,a-1)\} \subset X_{\reg}^{(3)}+\eta$; for instance if $w = (1\, 2\, 3)$ we take $w' = (1\,
3\, 2)$ and $\mu' = (c+p,a,a-1)$.  In particular $\mu' \in
\HC'(\rhobar|_{I_{\Qp}})$ and either $\HA(\mu) = \HA(\mu')$ (if $\mu'
= (a+p-1,a,c)$) or $\HA(\mu) \subset \HA(\mu')$ (if $\mu' =
(c+p,a,a-1)$), as required.

Finally, if $\mu = (a,c,c)$ with $0 < a-c \le p-1$ then one can
argue as in the previous case; alternately we can reduce
to the previous case by duality using the following lemma, valid for
$\GL_n$ (cf.~\cite[Prop.~6.23(ii)]{herzigthesis}), whose
proof is straightforward.
\end{proof}   

\begin{lemma}
  \label{lem:weights-dual} We have $\W_{\expl}(\rhobar^{\vee}) = \{
  F^\vee \otimes \det^{1-n} : F \in \W_{\expl}(\rhobar)\}.$
 \end{lemma}

We now describe the irregular weights in
$\Wexpl(\rhobar)$. As a preliminary, we observe that the possibilities
for $\rhobar|_{I_{\Qp}}$ are given by the following alternatives.

\begin{lemma}
  \label{lem:shape-alternatives}
  Suppose that $\tau : I_{\Qp} \to \GL_3(\Fpbar)$ is semisimple and
  extends to a  representation of $G_{\Qp}$. Then precisely
  one of the following alternatives holds:
  \begin{enumerate}
  \item $\tau \cong \tau(1,(a,b,c))$ where $a \ge b \ge c$ and $a-c
    \le p-1$,
\item $\tau \cong \tau((2\, 3),(a,b,c))$ where $a \ge b > c$ and $a-c
  \le p-1$,
\item $\tau \cong \tau((1\, 2\, 3),(a,b,c))$ where $a > b \ge c$ and
  $a-c \le p$,
\item $\tau^{\vee} \cong \tau((1\, 2\, 3),(a,b,c))$ where $a > b \ge c$ and
  $a-c \le p$.
  \end{enumerate}
Moreover, in (i) the triple $(a,b,c)$ is unique up to the equivalence
relation generated by $(a,b,c)\sim(c+p-1,a,b)$; in (ii) the triple
$(a,b,c)$ is unique up to translation by $(p-1,p-1,p-1)\Z$; in (iii) and
(iv), the triple $(a,b,c)$ is unique up to the equivalence relation
generated by $(a,b,c)\sim (c+p,a-1,b)$.
\end{lemma}

\begin{proof}
Parts (iii) and (iv) follow from \cite[Lem.~5.2.2]{MR3079258}. The rest
of the proof is left to the reader.
\end{proof}

\begin{prop}
  \label{prop:gl3-non-regular}
  Let  $\rhobar: G_{\Qp} \to \GL_3(\Fpbar)$ be a  representation
  such that $\rhobar|_{I_{\Qp}}$ is semisimple.  Then the
  weights in 
  $\Wexpl(\rhobar) \setminus \SW_{\reg}$ are described as
  follows.
  \begin{enumerate}\setlength{\itemsep}{0.2cm}
  \item Suppose that $\rhobar|_{I_{\Qp}} $ is as in
    Lemma~\ref{lem:shape-alternatives}(i).
    The set
    $\Wexpl(\rhobar) \setminus \SW_{\reg}$ consists of the weights $F(\mu-\eta)$ for
    triples $\mu$ as in the second column of the following table,
    under the conditions as in the first column.
    \begin{longtable}[c]{l|l}
      condition & $\mu$\\
\hline 
\endhead 
$a-b=1$, $b-c\neq 0$ & $(b+p,b,c)$ \\ 
$a-b=1$, $b-c \le 1$ & $(b+p,b,c-p+1)$\\
$b-c=1$, $a-c \neq p-1$ & $(c+p,c,a-p+1)$\\
$b-c=1$, $a-c \ge p-2$ & $(c+p,c,a-2p+2)$\\
$a-c=p-2$, $a-b\neq 0$ & $(a+p,a,b)$ \\
$a-c=p-2$, $a-b \le 1$ & $(a+p,a,b-p+1)$\\
$b-c=1$, $a-b \neq 0$ & $(a,b,b-p)$\\
$b-c=1$, $a-b \le 1$ & $(a+p-1,b,b-p)$\\
$a-b=1$, $a-c \neq p-1$ & $(c+p-1,a,a-p)$ \\
$a-b=1$, $a-c \ge p-2$ & $(c+2p-2,a,a-p)$ \\
$a-c=p-2$, $b-c\neq 0$ & $(b,c,c-p)$ \\
$a-c=p-2$, $b-c \le 1$ & $(b+p-1,c,c-p)$\\
$a-b = 0$
& $(b+p,b,c-1)$, $(c+p,a,a-p)$\\
$b-c = 0$
& $(c+p,c,a-p)$,  $(a+1,b,b-p)$\\
$a-c=p-1$
& $(a+p,a,b-1)$, $(b+1,c,c-p)$\\

\end{longtable}

\item   Suppose that $\rhobar|_{I_{\Qp}} $ is as in
    Lemma~\ref{lem:shape-alternatives}(ii).  The set
    $\Wexpl(\rhobar) \setminus \SW_{\reg}$ consists of the weights $F(\mu-\eta)$ for
    triples $\mu$ as in the second column of the following table,
    under the conditions as in the first column.
  \begin{longtable}[c]{l|l}
      condition & $\mu$ \\
\hline 
\endhead 
 $b-c=1$, $a-c \neq p-1$ & $(c+p,c,a-p+1)$\\
 $a-b=1$ & $(b+p,b,c)$\\
 $a-c=2$ & $(c+p+1,c+1,b-p)$\\
 $a-b=0$ & $(b+p,b,c-1)$, $(c+p,a,a-p)$\\
 $b-c=1$, $a-b\neq 0$ & $(a,b,b-p)$ \\
 $a-c=p-2$ & $(b,c,c-p)$\\
 $a-b=p-3$ & $(c+p,b-1,b-p-1)$\\
 $a-c=p-1$ & $(b+1,c,c-p)$, $(a+p,a,b-1)$  \\
\end{longtable}

\item   Suppose that $\rhobar|_{I_{\Qp}} $ is as in
    Lemma~\ref{lem:shape-alternatives}(iii). The set
    $\Wexpl(\rhobar) \setminus \SW_{\reg}$ consists of the weights $F(\mu-\eta)$ for
    triples $\mu$ as in the second column of the following table,
    under the conditions as in the first column.
 \begin{longtable}[c]{l|l}
      condition & $\mu$\\
\hline 
\endhead 
$a-b=2$ & $(c+p,a-1,a-p-1)$\\
 $a-b=1$, $a-c\neq 1$ & $(c+p+1,a-1,a-p-1)$\\
 $a-c=p-1$ & $(b+1,c,c-p)$ \\
 $a-c=p$, $b-c \neq p-1$ & $(b+2,c,c-p)$ \\
 $b-c=1$ & $(a,b,b-p)$ \\
 $b-c=0$, $a-b \neq p$ & $(a+1,b,b-p)$
    \end{longtable}
  \end{enumerate}
\end{prop}

\begin{remark}
  \label{rem:case4}
   Suppose that $\rhobar|_{I_{\Qp}} $ is as in
    Lemma~\ref{lem:shape-alternatives}(iv).  The description of set
    $\Wexpl(\rhobar) \setminus \SW_{\reg}$ can be extracted from
    Proposition~\ref{prop:gl3-non-regular}(iii) by duality using
    Lemma~\ref{lem:weights-dual}.  

Also, we remind the reader that in part (iii), the list of
weights given here does not include the shifted weights described at the
beginning of Example~\ref{ex::gl3-shifted-weights}.
\end{remark}

\begin{proof}
 Note that if $F(x,y,z)$ is a weight with $x-z < p-2$, then the weight
 $F(z+p-2,y,x-p+2)$ as in Example~\ref{ex:closure-for-niveau-1-gl3} is
 regular, so that all the weights in these tables must either be
 obvious or obscure.  The obscure weights are analyzed in
 Example~\ref{ex:exceptional-weight-gl3}, and are listed in the final three
 rows of the table in (i), and the second half of rows 4 and 8 of the table in
 (ii). (These rows also contain some obvious weights.)

Suppose, then, that the irregular weight $F\big((b+p,b,c)-\eta\big)$ lies in
$\Wobv(\rhobar)$, with $0 < b-c \le p$.  It follows from the definitions that
  $\rhobar|_{I_{\Qp}} \cong \tau(w,(b+p,b,c))$ for
  some $w \in S_3$.
  For each $w$ one then expresses $\rhobar|_{I_{\Qp}}$ in the
  form of Lemma~\ref{lem:shape-alternatives} (in all possible ways) to generate the lines of
  the above tables containing triples of the form $\mu =
  (\mu_1,\mu_2,\mu_3)$ in the second column with
  $\mu_1-\mu_2=p$ (relabelling as necessary, as well as keeping in
  mind the equivalence relations in
  Lemma~\ref{lem:shape-alternatives}). For example, if $p > 2$ then rewriting
  $\tau(1,(b+p,b,c))$ as $\tau(1,(b+1,b,c))$ when $0 < b-c \le p-2$
  and as $\tau(1,(b+1,b,c+p-1))$ when $p-1 \le b-c \le p$ gives the
  first two lines of the table in (i); then the next four lines come from
  the first two lines via the equivalence relation of
  Lemma~\ref{lem:shape-alternatives}.
 We leave the rest of the details as an exercise for the
  reader (for which we suggest considering the case $p=2$
  separately, at least in the cases when $\rhobar$ is reducible). Dualising using Lemma~\ref{lem:weights-dual}  we obtain a
  similar list for $\rhobar|_{I_{\Qp}}$ for the weight
  $F\big((a,b,b-p)-\eta\big)$.  
\end{proof}

\subsection{The results of \texorpdfstring{\cite{MR3079258}}{[EGH13]}}
\label{sec:results-egh}
The paper~\cite{MR3079258} considers the weight part of Serre's conjecture for
Galois representations $\rbar:G_F\to\GL_3(\Fpbar)$, where $F$ is a totally
real field in which $p$ splits completely, and $\rbar|_{G_{F_v}}$ is
irreducible for each $v|p$. The main results are proved with respect to some
abstract axioms, which are in particular satisfied for the cohomology of forms
of $\mathrm U(3)$ which are compact at infinity, and show that if each
$\rbar|_{G_{F_v}}$ satisfies a mild genericity condition, then the set of
weights in which $\rbar$ is automorphic contains the set of weights predicted
by the conjecture of~\cite{herzigthesis}, and that any other weight for which
$\rbar$ is automorphic is non-generic. (Here a weight is generic if it is
sufficiently far away from the walls of any alcove. As with the
definition of genericity
for a Galois representation, this will be made precise in
Section~\ref{sec:herzig-comparison} below.)

In the light of the discussion of Section~\ref{subsec: GL3}, these results are completely
consistent with our conjectures.

\subsection{The conjecture of Ash, Doud, Pollack, and Sinnott for \texorpdfstring{$\GL_n$}{GL(n)} over \texorpdfstring{$\Q_p$}{Qp}}
\label{sec:conjecture-ash-doud}

Let $\rbar : G_{\Q} \to
\GL_n(\Fpbar)$ be odd and irreducible. The first Serre
weight conjectures for such $\rbar$ with $n > 2$ were made by Ash, Doud, Pollack, and Sinnott \cite{bib:ASinn},
\cite{bib:ADP}.  When $n = 3$ a detailed comparison between their
conjecture and  the conjecture of \cite{herzigthesis} can
be found in \textit{ibid.}, \S7.  The purpose of this section is
to note the following result.

\begin{prop}\label{prop:adp-comparison}
Let $\rbar$ be as above and suppose that $\rbar|_{I_{\Qp}}$ is
semisimple and sufficiently generic. Then the Serre weights
predicted in \cite[Conj.~3.1]{bib:ADP} are a subset of $\Wexpl(\rbar|_{G_{\Qp}})$.
\end{prop}

For the term ``sufficiently generic'' we once again refer the reader to 
Definition~\ref{defn:suff-generic} (but see also 
Example~\ref{sufficiently-generic-niveau-1} for the case when
$\rbar|_{G_{\Qp}}$ is a sum of characters). 

\begin{rem}\label{rem:proper}
It turns out that the subset of Serre weights predicted in \cite[Conj.~3.1]{bib:ADP} consists of a mix of some (but not all) obvious weights
for $\rbar|_{G_{\Qp}}$ and some (but very far from all) shadow
weights. In any case we stress that \cite{bib:ADP} do not claim
to predict the full set of weights for $\rbar$. 
\end{rem}

Since the proof of
Proposition~\ref{prop:adp-comparison} will make use of terminology and
results from Sections~\ref{sec:unramified
  groups}--\ref{sec:herzig-comparison} we defer the proof until
Section~\ref{sec:adp-comp}. (We may safely do so because nothing in the paper depends logically on Proposition~\ref{prop:adp-comparison}.)

We make two further remarks. First, \cite{bib:ADP} still give a conjectural
set of Serre weights even when $\rbar|_{I_{\Qp}}$ is not semisimple, a
context in which we do not make an explicit prediction; we have no heuristic by
which to predict whether or not all the Serre weights conjectured by
\cite{bib:ADP} in this context are indeed weights of $\rbar$. Second,
we note that when $\rbar|_{G_{\Qp}}$ is irreducible, Doud
\cite[Conj.~2.10]{doud-supersingular} predicts precisely the set
$\W_{\obv}(\rbar|_{G_{\Qp}})$ together with all weight shifts as
described in Section~\ref{sec:shifted-weights}.

\subsection{The results of \texorpdfstring{\cite{blggUn}}{[BLGG14]}}The
article~\cite{blggUn} applies the machinery of the paper~\cite{BLGGT} to the problem of the
weight part of Serre's conjecture for unitary groups over CM fields. As
explained in Section~\ref{sec:patching, BM and serre weights}, a lack of general
results on the potential diagonalisability of crystalline representations limits
the scope for proving general comprehensive results. However, under mild
conditions the paper shows that when $n=3$, $p$ splits completely in
an imaginary CM field $F$, and
$\rbar:G_F\to\GL_3(\Fpbar)$ is such that $\rbar|_{G_{F_v}}$ is semisimple
for each $v|p$, then $\rbar$ is automorphic for every obvious predicted weight in
the sense of Definition~\ref{defn:obvious-lift-ss} above. This is, of course,
consistent with our conjectures.

\subsection{\texorpdfstring{$\GSp_4$}{GSp(4)}}\label{subsec:GSp_4} 
The paper~\cite{herzigtilouine} formulates a
version of the weight part of Serre's conjecture for irreducible representations
$\rbar:G_\Q\to\GSp_4(\Fpbar)$, under the assumption that $\rbar|_{G_\Qp}$ is
a sum of characters, and under a mild regularity condition on the weights. The
formulation follows that of~\cite{herzigthesis}, and is a special case of the
more general conjectures that we formulate in Sections~\ref{sec:unramified
  groups} and~\ref{sec:herzig-comparison}, which show that these conjectures are
also consistent with the philosophy of this paper.
\subsection{Computational evidence}\label{sec:computational} The paper~\cite{MR2887610} carried out
computations for the weight part of Serre's conjecture (Conj.~\ref{conj:basic-conjecture}) in the case that
$F=\Q(i)$ and $n=2$. In this setting the Taylor--Wiles method is not available
in anything like the generality required to make arguments along the lines of
those explained in Remark~\ref{rem: patching functors and GK} for totally real
fields, and so there are no theoretical results on the weight part of Serre's
conjecture. However, the computations of~\cite{MR2887610} are all consistent
with the expectation that the weight part of Serre's conjecture will behave
identically in this case, and are thus completely consistent with the
conjectures of this paper. (It is explained in Section 7 of~\cite{MR2887610}
that there was one example where two expected weights were not found; however
Mehmet Haluk \c{S}eng\"{u}n has independently reproduced the calculations
of~\cite{MR2887610} in unpublished work, and has found complete agreement,
\emph{except} that the two ``missing'' weights were also obtained.)

The paper~\cite{bib:ADP} explicitly carries out calculations for
the weight part of Serre's conjecture (Conj.~\ref{conj:basic-conjecture}) for the case $F=\Q$ and $n=3$. These
calculations, and some additional calculations that Doud and Pollack carried out
at the request of the second author, are all consistent with the conjecture
of~\cite{herzigthesis}, and thus with our conjectures; see Section 8
of~\cite{herzigthesis} for more details of this.

\begin{remark}
  \label{rem:special-position}
In fact, the calculations in~\cite{bib:ADP} would have been consistent
with Conjectures~\ref{conj:crystalline-version}
and~\ref{conj:closure-operation} even without the hypothesis that
$\rhobar|_{I_K}$ is semisimple. To be precise, set $F = F(x,y,z)$ with
$x-z < p-2$, and set $F' = F(z+p-2,y,x-p+2)$; then in the calculations
of~\cite{bib:ADP}, one finds that whenever $F$ is a Serre weight for some
$\rhobar$, so is $F'$. 

On the other hand, as we have explained in
Section~\ref{subsec: crystalline lifts in the
  picture}, Conjectures~\ref{conj:crystalline-version}
and~\ref{conj:closure-operation} are now known to be false if one omits the hypothesis that
$\rhobar|_{I_K}$ is semisimple. One thus expects that every $\rhobar$
with Serre weight $F$ considered by~\cite{bib:ADP} happens to lie on
the (codimension one) intersection between the two components of the stack $\cXbar$ labeled by
the Serre weights $F$ and $F'$ (cf.\ the discussion in
Section~\ref{sec: the picture}); while at the time of writing we do
not know for certain that this is the case, an examination of the
explicit representations considered by~\cite{bib:ADP} suggests that
they are indeed in a rather special position.
\end{remark}

\subsection{Computational evidence for irregular Serre
  weights}\label{subsec:irregular weight computations} Consider 
a representation $\rbar:G_{\Q}\to\GL_3(\Fpbar)$ that is odd and irreducible, and such that
$\rbar|_{I_p}$ is semisimple. 
Since our conjectures cover more weights than those of~\cite{herzigthesis}
(namely, the weights which are irregular), we now give computational evidence
for such weights.

\subsubsection{Examples from \cite{bib:ADP}} The weight predictions in~\cite[Conj.\
3.1]{bib:ADP} are ambiguous for irregular weights: if $x\equiv y$ or
$y\equiv z\pmod{p-1}$, then their weight prediction of $F(x,y,z)'$
means that $\rbar$ occurs in \emph{at least one} weight $F(x',y',z')$
such that $(x',y',z')\equiv(x,y,z)\pmod{(p-1)\Z^3}$ (so there are
either two or four such weights, the latter precisely when $x\equiv
y\equiv z\pmod{p-1}$).

In~\cite{bib:ADP} there are six examples with $\rbar|_{G_{\Qp}}$ of
length two (see~\cite[Table 10]{bib:ADP}) and two examples with
$\rbar|_{G_{\Qp}}$ irreducible (see~\cite[\S7.2]{bib:ADP}) where ambiguous weight predictions
occur. In each case, except for the second entry of \cite[Table 10]{bib:ADP} which lies outside the
scope of his program, Doud has checked for us that $\rbar$ appears in
\emph{both} weights implied by the ambiguous notation (testing Hecke
eigenvalues for all $l\le 47$, as in~\cite{bib:ADP}).
This is consistent with our conjecture, as all weights in
question are obvious.

\subsubsection{Examples from \cite{doud-supersingular}}The
paper~\cite{doud-supersingular} provided computational evidence for
several~$\rbar$ with~$\rbar|_{G_{\Qp}}$ occurring in irregular
weights. Recall that for such~$\rbar$ his predicted weight set is
obtained by adjoining all weight shifts
to~$\Wexpl(\rbar|_{G_{\Qp}})$. In most of his examples the irregular
weights are obvious for~$\rbar|_{G_{\Qp}}$; see Section~\ref{subsubsec:weight shifts} below for the
remaining cases.

\subsubsection{Obscure weights}Consider the irreducible polynomial
$f(x)=x^4-x^3+5x^2-4x+3$ over~$\Q$ with Galois group~$A_4$, as
in~\cite[Ex.\ 5.4]{bib:ADP}. By taking the unique 3-dimensional
irreducible representation of~$A_4$ over~$\Fbar_{13}$, we obtain a
Galois representation~$\rbar$ as above with
$\rbar|_{I_{13}}\cong\tau(1,(6,6,0))$. We have
$\Wexpl(\rbar|_{G_{\Qp}})=\{F(16,5,0),F(16,11,6),F(22,17,6),F(17,11,5),F(29,17,11),F(23,17,5)\}$,
where the last two weights are obscure. Doud could provide for us
computational evidence that~$\rbar$ is automorphic in each of these
weights, and showed that it doesn't occur in any other irregular
weight~$F(a,b,c)$ with $a-c\notin\{21,24\}$. (Note that the central
character forces $a-c\equiv 0\pmod{3}$ for irregular weights.)

\subsubsection{Weight shifts}\label{subsubsec:weight shifts}In the literature we found evidence for
shifted weights that are not contained in $\Wexpl(\rbar|_{G_{\Qp}})$
in the following cases. First, when $p=2$, \cite[Table 3]{MR2103328}
contains three examples where $\rbar|_{I_2}\cong\tau((1\, 2\, 3),(1,0,0))$
(or its dual). We have that
$\Wexpl(\rbar|_{G_{\Qp}})=\{F(0,0,0),F(1,1,0),F(2,1,0)\}$ (all are
obvious), and $F(1,0,0)$ is a shift of $F(0,0,0)$. In each
case~\cite{MR2103328} gives computational evidence that~$\rbar$ occurs
in all four weights.

Second, when $p=3$, \cite[\S5.3]{doud-supersingular} considers an
example in which $\rbar|_{I_3}\cong\tau((1\, 2\, 3),(2,0,0))$ (or its dual;
these are denoted by $m=2$, $m=8$ in~\cite[Table
2]{doud-supersingular}). In this case
$$\Wexpl(\rbar|_{G_{\Qp}})=\{F(1,1,1),F(1,0,0),F(3,2,0),F(3,3,1),F(5,3,1),F(2,1,0)\}$$
(all but the last weight being obvious), and $F(3,1,1)$ is a shift of $F(1,1,1)$.
Doud \cite{doud-supersingular} gives computational evidence
for all seven weights. (Note that $F(3,1,1)$ is missing
from~\cite[Table 2]{doud-supersingular}, but Doud confirmed to us that
this is just a typo.)

\section{Unramified groups}\label{sec:unramified groups}
We now explain how to extend the definition of
the set of weights $\Wexpl(\rhobar)$, as well as the set of weights
$\W^{?}(\rhobar)$ defined in \cite{herzigthesis}, to the more general
setting of unramified groups over $\Qp$.  In this section and the
next, we will use $\Gamma_K$ instead of $G_K$ to denote the absolute
Galois group of $K$, to avoid confusion with our notation for
algebraic groups.

\subsection{\texorpdfstring{$L$}{L}-groups and \texorpdfstring{$L$}{L}-parameters}
\label{sec:l-group-l}

Let $G$ be a connected reductive group over $\Zp$, i.e.\ a smooth affine algebraic
group whose geometric fibres are connected reductive.
Then $G\times \Qp$ is unramified (i.e.\ quasisplit and split over an unramified extension
of $\Qp$), and conversely every unramified group over $\Qp$ arises in this way (by choosing
a hyperspecial point in the building).
Let $B$ be a Borel subgroup of $G$ with Levi subgroup $T\subset B$, so $T$ is a maximal torus
of $G$.
Note that we have a canonical identification of character groups $X(T\times \Qpbar) \cong X(T\times \Fpbar)$, which is
compatible with the Galois action of $\Gamma_\Qp \onto \Gamma_\Fp$.
We sometimes write just $X(T)$ for this Galois module and similarly
$Y(T)$ for the co-character group
$Y(T\times\Qpbar) \cong Y(T\times\Fpbar)$.
Let $W := \big( N(T) / T\big) (\Qpbar) \cong \big( N(T) / T\big) (\Fpbar)$ denote the Weyl group.
Let $\Delta = \Delta (B,T) \subset X(T)$, respectively $\Delta^\vee = \Delta^\vee (B,T)\subset Y(T)$,
denote the simple roots (respectively coroots) defined by $B$.
Then $\Gamma_\Qp$ naturally acts on the based root datum $\Psi_0 (G,B,T) := \big( X(T) , \Delta , Y(T),\Delta^\vee
\big)$.
Let $L\subset \Qpbar$ denote the splitting field of $G$, i.e.\ the finite unramified extension of $\Qp$ cut out by
the $\Gamma_\Qp$-action on $\Psi_0 (G,B,T)$.

A \emph{dual group} of $G$ is a quadruple $(\Ghat, \Bhat, \That , \{ x_\alpha \}_{\alpha\in
\Delta (\Bhat,\That)})$, where $\Ghat$ is a split connected reductive group over $\Zp$, $\Bhat$ a Borel
of $\Ghat$, $\That\subset\Bhat$ a Levi subgroup, $x_\alpha : \G_a \isoto \Uhat_\alpha$ isomorphisms
of algebraic groups (where $\Delta(\Bhat,\That)$ is the set of simple roots
determined by $\Bhat$, and $\Uhat_\alpha$ is the root
subgroup of $\alpha$), together with an isomorphism
$\phi:\Psi_0 (\Ghat, \Bhat, \That) \isoto \Psi_0 (G,B,T)^\vee$.
This isomorphism induces an action  of $\Gamma_\Qp$ on $(\Ghat,\Bhat,\That,
\{ x_\alpha\})$ that factors through $\Gal (L/\Qp)$, and we define the $L$-group $\LG := \Ghat \rtimes \Gal (L/\Qp)$, a 
reductive group over $\Zp$.
The Weyl group of $\That$ is naturally identified with $W$ via the duality isomorphism.
We remark that any two pinnings $(\Ghat,\Bhat,\That,\{x_\alpha\})$
of $\Ghat$ are $\Ghat (\Zp)$-conjugate provided that $Z(\Ghat)$ is connected.
(This is equivalent to $G^\der$ being simply connected, which we
will assume in a moment.)
We also remark that our definition of the $L$-group is compatible
with that of \cite{MR757954}
and
\cite{BuzzardGee} who work with canonical based root data.
(The reason is that $(B,T)$ and $(\Bhat,\That)$ are defined over $\Zp$.)

From now on we suppose the following.

\begin{hyp}\label{hyp:gp}
Assume that the group $G^\der$ is simply connected, that $Z(G)$ is connected,
and that $G$ has a local twisting element $\eta$, which by definition means that 
$\eta\in X(T)^{\Gamma_\Qp}$
and $\langle \eta , \alpha^\vee \rangle = 1$ for all $\alpha \in 
\Delta$. (Twisting elements are defined in the same way for groups over number 
fields in~\cite[\S5.2]{BuzzardGee}; they are a key part of the general 
conjectures made in~\cite{BuzzardGee} on the association of Galois 
representations to automorphic representations.) 
 \end{hyp}

In the following definitions, $A$ is a topological $\Zp$-algebra, i.e.\ a $\Zp$-algebra
that is also a topological ring.

\begin{defn}
An \emph{$L$-parameter} is a continuous homomorphism $\Gamma_\Qp \to \LG (A)$
that is compatible with the projections to $\Gal (L/\Qp)$.
\end{defn}

\begin{defn}
An \emph{inertial $L$-parameter} is a continuous homomorphism $I_\Qp \to \Ghat (A)$ 
that admits an extension to an $L$-parameter $\Gamma_\Qp \to \LG (A)$.
\end{defn}

We say that (inertial) $L$-parameters $\rho_1$, $\rho_2$ are \emph{equivalent} if they are $\Ghat(A)$-conjugate,
and we write $\rho_1 \cong \rho_2$.

\begin{defn}\label{defn: Serre weight general group}
A \emph{Serre weight} is an isomorphism class of irreducible
$\Fpbar$-repres\-entations of $G(\Fp)$. (Just as for $\GL_n$, we will sometimes abuse terminology and refer to an individual irreducible
representation as a Serre weight.)
\end{defn}

Given a tamely ramified inertial $L$-parameter $\tau : I_\Qp \to \Ghat (\Fpbar)$
we will define below sets of Serre weights $\W^? (\tau)$ and $\Wexpl (\tau)$.
These generalise respectively the construction in \cite{herzigthesis} and the construction
in Section~\ref{sec:expl-weight-conj}. (To be precise, in the latter
case we will only
generalise the case of $\GL_n$ over unramified extensions of $\Qp$.)
Our main result, Theorem~\ref{thm:main-result}, will establish that the two sets are
equal for generic $\tau$.

\subsection{Definition of \texorpdfstring{$\W^? (\tau)$}{W\_?(tau)}}\label{sec:defin-W-questionmark}

In this section we generalise \cite[\S\S6.3--6.4]{herzigthesis}.
To simplify notation, let $(\uG, \uB, \uT) := (G,B,T) \times \Fpbar$ and
$(\uG^*,\uB^*, \uT^*) := (\Ghat,\Bhat,\That) \times\Fpbar$.
Let $F:\uG \to \uG$ denote the relative Frobenius, so
$\uG^F = G(\Fp)$.
Let $F^*:\uG^* \to \uG^*$ denote the composite $\Fr \circ \varphi = \varphi \circ \Fr$, where
$\Fr$ denotes the relative Frobenius on $\uG^*$ and $\varphi\in\Gamma_\Qp$ denotes from
now on a geometric Frobenius element.
Then $F^*$ is the relative Frobenius for a different $\Fp$-structure on $\uG^*$,
as $\varphi$ has finite order on $\uG^*$.

Recall that we fixed an isomorphism $\phi:\Psi_0 (\uG^*, \uB^*, \uT^*) \isoto \Psi_0
(\uG,\uB,\uT)^\vee$ above that is by definition $\Gamma_\Qp$-equivariant.
Our conventions are that $\Gamma_\Qp$ and the Weyl group act
on the left on $X(\uT)$ and $Y(\uT)$; 
so $\gamma(\mu) = \mu\circ \gamma^{-1}$ and $w (\mu) = \mu\circ w^{-1}$ for
$\mu\in X(\uT)$, $\gamma\in \Gamma_\Qp$, $w\in W$.
However, $F$ is not invertible, so we set $F(\mu) = \mu\circ F$.
Similar comments apply to $\uG^*$.
With these conventions we have $F\circ \phi = \phi\circ F^*$ on $Y(\uT^*)$, 
as\footnote{In order not to get confused about actions on $X(\uT)$, it helps to think in terms of the actions on $\uT$. For example, $F = p\varphi^{-1}$ on $\uT$.}
 $F=p\varphi$
on $X(\uT)$ and $F^*=p\varphi$ on $Y(\uT^*)$.
Thus $\phi$ is a duality between $(\uG,F)$ and $(\uG^*, F^*)$ in the
sense of Deligne--Lusztig~\cite{bib:DL}.
(Note once again that $\Psi_0 (\uG,\uB,\uT)$, $\Psi_0 (\uG^*,\uB^*,\uT^*)$
are canonically isomorphic to the canonical based data, as
$(\uB,\uT)$ is $F$-stable and $(\uB^*,\uT^*)$ is $F^*$-stable.)

Fix from now on a generator $(\zeta_{p^i-1}) \in \plim{i\ge 1} \F_{p^i}^*$.
Recall the following facts from \cite[\S5]{bib:DL}.

(i) The (canonical) Weyl group $W = N(\uT)/\uT$ is canonically identified with
$N(\uT^*) / \uT^*$ such that $w\circ \phi = \phi \circ w$ for all $w\in W$.
The actions of $F$ and $F^*$ on $W$ are inverse to each other.

(ii) There is a canonical bijection between $\uG^F$-conjugacy classes
of pairs $(\T,\theta)$ consisting of an $F$-stable maximal torus $\T\subset \uG$ and
a character $\theta:\T^F \to \Qpbarx\times$ and ${\uG^*}^{F^*}$-conjugacy classes of
pairs $(\T^*,s)$ consisting of an $F^*$-stable maximal torus $\T^* \subset \uG^*$ and
a semisimple element $s\in {\T^*}^{F^*}$.

(iii) If the classes of $(\T^*,s)$, $(\T,\theta)$ are in bijection in (ii), then they
are both said to be \emph{maximally split} if $\T^* \subset Z_{\uG^*} (s)$ is a
maximally split torus (i.e.\ contained in an $F^*$-stable Borel subgroup of $Z_{\uG^*}(s)$).

\begin{prop}\label{prop:cd}
We have the following commutative diagram
\begin{displaymath}
    \xymatrix{
        {\left\{ { \myatop{\text{\rm maximally split}}{(\T^* , s)}}\right\}_{/{\uG^*}^{F^*}}}  \ar@{<->}[d] \ar@{<->}[r]^{\text{\rm duality}} & {\left\{ \myatop{\text{\rm maximally split}}{(\T , \theta)}\right\}_{/\uG^F}} \ar@{^{(}->}[d] \\
        {\left\{ \myatop{\text{\rm tame inertial $L$-parameters}}{I_\Qp \to \Ghat (\Fpbar ) }\right\}_{/\cong}}  \ar@{-->}[r]^{\ \ V_\phi}    &  {\left\{ \myatop{\text{\rm representations of}}{ G(\Fp)\  \text{\rm over} \ \Qpbar }\right\}_{/\cong}} }
\end{displaymath}
defining the map $V_\phi$.
\end{prop}

\begin{proof}
We proceed as in the proof of \cite[Prop.~6.14]{herzigthesis}.
Recall that $(\zeta_{p^i-1})$ gives rise to a generator $g_\can\in I_\Qp / I_\Qp^w$, where $I_\Qp^w$ is the
wild ramification subgroup.
Since $\Gamma_\Qp / I_\Qp^w = I_\Qp / I_\Qp^w \rtimes \overline{\langle \varphi\rangle}$, where
$\varphi^{-1} g \varphi = g^p$ for $g\in I_\Qp / I_\Qp^w$, the bottom left-hand corner of the diagram
is in bijection with $\Ghat (\Fpbar)$-conjugacy classes of semisimple elements $s'\in \Ghat (\Fpbar)=
\uG^* (\Fpbar)$ satisfying that $\varphi^{-1} (s')$ is $\uG^* (\Fpbar)$-conjugate to $(s')^p$.
Equivalently, by conjugating $s'$ to $\uT^* (\Fpbar)$ and using that $F^* = p\varphi$ on $\uT^*$,
these are $F^*$-stable $\uG^*(\Fpbar)$-conjugacy classes of semisimple elements of $\uG^*(\Fpbar)$.
Using that $Z(\uG)$ is connected, every such conjugacy class has a representative in ${\uG^*}^{F^*}$,
unique up to ${\uG^*}^{F^*}$-conjugacy.
We then get the bijection on the left as in \cite{herzigthesis}.
The map on the right is given by $(\T,\theta)\mapsto \ve_{\uG} \, \ve_\T R_\T^\theta$,
with notation as in \cite[\S4.1]{herzigthesis} and \cite[Lem.~4.2]{herzigthesis}.
It is a genuine representation of $\uG^F = G(\Fp)$ by \cite[Prop.~10.10]{bib:DL}.
\end{proof}

The explicit description of $V_\phi$ in \cite{herzigthesis}
generalises, as we now explain.
Recall that for $w\in W$ we choose $g_w \in \uG (\Fpbar)$ such that $g_w^{-1} F (g_w)\in N(\uT)(\Fpbar)$
represents $w$ and define $\uT_w := g_w \uT g_w^{-1}$.
Then $\uT_w$ is an $F$-stable maximal torus.
Define $\theta_{w,\mu} : \uT_w^F \to \Qpbarx\times$ for $\mu\in X(\uT)$ by $\theta_{w,\mu} (t) :=
\tilde{\mu}(g_w^{-1} tg_w)$, where tilde denotes the Teichm\"uller lift.
Any pair $(\T, \theta)$ consisting of an $F$-stable maximal torus $\T$ and a
character $\theta : \T^F \to \Qpbarx\times$ is $\uG^F$-conjugate to $(\uT_w, \theta_{w,\mu})$ for
some $(w,\mu)$.

\begin{defn}\label{def:deligne-lusztig}
Let $R (w,\mu) := \ve_{\uG} \ve_{\uT_w} R_{\uT_w}^{\theta_{w,\mu}}$ be 
defined as in the proof of Proposition~\ref{prop:cd}. 
(It may be virtual if $(w,\mu)$ is not maximally split.) 
 \end{defn}

For $d\ge 1$ let $\omega_d : I_\Qp \to \Fpbarx{\times}$ be the character $\omega_{\sigma}$,
where $\sigma : \F_{p^d} \to \Fpbar$ denotes the inclusion of the unique subfield of $\Fpbar$ of degree $d$ over $\Fp$.
Let $\tau (w,\mu) : I_\Qp \to \That (\Fpbar)$ denote the tame representation
\[
\tau (w,\mu) := N_{(F^*\circ w^{-1})^d / F^* \circ w^{-1}} \big( \mu (\omega_d)\big) ,
\]
where $d\ge 1$ is chosen such that
$(F^* \circ w^{-1})^d = p^d$ on $Y(\uT^*)$, $\mu$ is considered as element of $Y(\uT^*)$ via $\phi$, and $N_{A^d / A} = \prod\limits_{i=0}^{d-1} A^i$.

\begin{prop}\label{prop:l-parameter}
The representation $\tau(w,\mu)$ is an inertial $L$-parameter. If $(\uT_w, \theta_{w,\mu})$
is maximally split then it corresponds to $\tau(w,\mu)$ under the bijections of Proposition~\ref{prop:cd}
and we have $V_\phi \big( \tau (w,\mu)\big) \cong R(w,\mu)$.
In particular, $V_\phi$ is independent of the choice of $(\zeta_{p^i-1})_i$.
\end{prop}

\begin{proof}
This is the same as in \cite[Prop.~6.14]{herzigthesis}.
\end{proof}

Let $X(\uT)_+$ denote the subset of $X(\uT)$ consisting of dominant weights, and let
$X_1 (\uT) := \{ \mu\in X (\uT): 0\le \langle
\mu,\alpha^\vee\rangle \le p-1\ \textrm{for all } \alpha\in\Delta\}$,
$X^0 (\uT) := \{ \mu\in X(\uT) : \langle \mu,\alpha^\vee\rangle = 0\ \textrm{for all }  \alpha \in \Delta\}$.
For $\mu\in X(\uT)_+$ let $F(\mu)$ denote the irreducible algebraic $\uG$-representation
of highest weight $\mu$. (As the referee points out, this notation is ambiguous, since $F$ is
also the Frobenius. However, below $F(\nu)$ for $\nu \in X(\uT)$ will always mean the
algebraic representation and never the weight $\nu \circ F$.)

\begin{lemma}\label{lem:bijection of weights}
The map \[\frac{X_1(\uT)}{(F-1)X^0(\uT)}\to \big\{ {\text{\rm Serre
    weights of}} \ G(\Fp)=\uG^F\big\}_{/\cong}\] \[\mu \mapsto   F(\mu) |_{\uG^F} \]
is a \upl well-defined\upr\ bijection.
\end{lemma}

\begin{proof}
We claim first that there exists a finite order automorphism $\pi$ of $(\uG,\uB,\uT)$
that commutes with $F$ and that induces the action $\varphi^{-1}$ on $\Psi_0 (\uG,\uB,\uT)$.
To see this, note first that we can choose a pinning $x_\alpha : \G_a \isoto U_\alpha$ for
$\alpha\in\Delta$ such that $F\circ x_{\varphi\alpha} = x_\alpha \circ F_a$ for all
$\alpha\in\Delta$, where $F_a$ is the relative Frobenius on $\G_a$ and
$U_\alpha$ is the root subgroup of $\alpha$.
Then let $\pi$ be the unique automorphism of $(\uG,\uB,\uT,\{ x_\alpha\}_\Delta)$ inducing
$\varphi^{-1}$ on $\Psi_0 (\uG,\uB,\uT)$.
To see that $F$ and $\pi$ commute, note that both maps send $U_\alpha$ onto
$U_{\varphi^{-1}\alpha}$.

Then $F\circ\pi^{-1}$ is the relative Frobenius of a \emph{split} $\Fp$-structure on $\uG$,
since $F\circ\pi^{-1} = F\varphi=p$ on $\uT$.
The lemma now follows from Proposition 1.3 in the appendix to \cite{herzigthesis}, noting that $F = p \pi^{-1}$ on $X(\uT)$.
\end{proof}

Let $X_\reg (\uT) := \{ \mu\in X(\uT): 0\le \langle
\mu,\alpha^\vee\rangle < p-1  \textrm{ for all }  \alpha\in\Delta\}
\subset X_1 (\uT)$.
Then $\mu\mapsto w_0 \cdot (\mu-p\eta)$ defines a self-bijection of $X_\reg (\uT)$ (where $w_0\in W$
is the longest element), which passes to the quotient
$\frac{X_\reg (\uT)}{(F-1)X^0(\uT)}$.
Via Lemma~\ref{lem:bijection of weights} this quotient is identified with a subset $\W_\reg$ of all (isomorphism classes of)
Serre weights of $G(\Fp)$.
We write
\[
\begin{matrix}
\mathcal R: & \W_\reg & \to &\W_\reg \\
     &F(\mu) &\mapsto &F\big( w_0 \cdot (\mu - p\eta)\big)
\end{matrix}
\]
for the induced bijection.

\begin{defn}\label{defn:herzigwts}
For a tame inertial $L$-parameter $\tau : I_\Qp \to \Ghat (\Fpbar)$ let
\[
\W^? (\tau) := \big\{ \mathcal R(F) :  \ F\in \W_\reg\ \text{an irreducible constituent of}\ \overline{V_\phi (\tau)}\big\} .
\]
\end{defn}

\begin{remark}
Note that this set does not agree completely with the set $W^?(\tau)$
defined in \cite[\S6]{herzigthesis} when $G = \GL_n$. The only 
discrepancy occurs for weights $F(\lambda)$ with $\langle 
\lambda,\alpha^\vee \rangle = p-2$ for some $\alpha \in \Delta$, which 
is irrelevant for our main result (Theorem~\ref{thm:main-result}).   
\end{remark}

\subsection{Definition of \texorpdfstring{$\W_\expll (\tau)$}{W\_expl(tau)}}\label{sec:defin-Wexpl}
We will now define $\W_\expll (\tau)$ for a tame inertial $L$-parameter
$\tau : I_{\Qp} \to \Ghat(\Fpbar)$.
When $G=\Res_{K/\Qp} \GL_n$ with $K/\Qp$ finite unramified we will
recover the set of weights given in
Definition~\ref{defn:explicit-weights-ss} (this will be Proposition~\ref{prop:Wexpl}).

\subsubsection{Hodge--Tate co-characters}\label{sec:hodge-tate-co}
Suppose for the moment that $K/\Qp$ is finite, or more generally that $K/\Qp$ is algebraic with finite
ramification index, and that $H$ is a (not necessarily connected)
algebraic group over $\Qpbar$.
Recall that a continuous homomorphism $\rho : \Gamma_K \to H(\Qpbar)$
is said to be Hodge--Tate (resp.\ crystalline) if for some faithful (and hence any) 
representation $H \to \GL_N$ over $\Qpbar$, the resulting
$N$-dimensional Galois representation is Hodge--Tate (resp.\ crystalline).
Given any $\rho : \Gamma_K \to H(\Qpbar)$ that is Hodge--Tate, and any
homomorphism $j : \barK \to \Qpbar$, it is explained in
\cite[\S2.4]{BuzzardGee} that there is an $H^\circ (\Qpbar)$-conjugacy class $\HT_j (\rho)$ of
co-characters $\G_m \to H$ over $\Qpbar$ (or equivalently, an element of $Y(\That)/W$, where $\That$
is a fixed maximal torus and $W$ its Weyl group in $H^\circ$). 
These classes satisfy the relation
\begin{equation}\label{eq:4}
\HT_{j\circ\gamma^{-1}} (\rho) = \rho(\gamma) \HT_j (\rho) \rho (\gamma)^{-1} \qquad
\textrm{for all } \gamma\in \Gamma_K .
\end{equation}
Our normalisation is such that $\HT_j(\varepsilon) = 1$ for all
$j$.
 If $\mu: \G_m \to H$ is a co-character we let $[\mu]$ denote its
class in $Y(\That)/W$. The following lemma is elementary.

\begin{lemma}\label{lem:HT}
Suppose $\rho : \Gamma_K \to H(\Qpbar)$ is Hodge--Tate.
\begin{enumerate}[label={\rm (\roman*)}]
\item
If $f : H\to H'$ is a map of algebraic groups over $\Qpbar$, then for $j:\barK\to \Qpbar$,
\[
f\circ \HT_j (\rho) = \HT_j (f\circ \rho) .
\]
\item
If $K'\subset \barK$ contains $K$, then for $j : \barK\to \Qpbar$,
\[
\HT_j (\rho |_{\Gamma_{K'}}) = \HT_j (\rho) .
\]
\item
Suppose that $L$ is another field and that $\gamma : \barK \isoto \Lbar$ is an
isomorphism sending $K$ onto $L$.
Then for $j : \Lbar \to \Qpbar$,
\[
\HT_{j\circ\gamma} (\rho) = \HT_j \big( \rho\circ (\gamma^{-1} (-) \gamma)\big) .
\]
\end{enumerate}
\end{lemma}

Now specialize to an $L$-parameter $\rho : \Gamma_{\Qp} \to \, {}^L G (\Qpbar)$ that is
crystalline.
Then $\HT_1 (\rho) \in Y(\That)/W = X(T)/W$ associated to $\id: \Qpbar \to \Qpbar$ determines
all Hodge--Tate co-characters by \eqref{eq:4}. Also, $\HT_j (\rho)$ depends only on $j|_L$.

Suppose that $\tau : I_\Qp \to \Ghat (\Fpbar)$ is a tame inertial $L$-parameter.

\begin{defn}\label{def5}
We say that an $L$-parameter $\rho : \Gamma_\Qp \to \, {}^L G(\Zpbar)$ is an \emph{obvious
crystalline lift of} $\tau$ if
\begin{enumerate}[label=(\roman*)]
\item $\overline{\rho |_{I_{\Qp}}}$ is $\Ghat (\Fpbar)$-conjugate to $\tau$;
\item $\rho$ is crystalline;
\item there is a maximal torus $T^* \subset \, {}^L G _{/\Zpbar}$ such that
\[
\rho (\Gamma_\Qp ) \subset N_{^L G_{/\Zpbar}} (T^*) (\Zpbar)\ \text{and}\ \rho (I_\Qp) \subset T^* (\Zpbar) .
\]
\end{enumerate}
\end{defn}

\begin{remark}
As any two maximal tori of ${}^L G_{/\Zpbar}$ are $\Ghat (\Zpbar)$-conjugate, we may assume
without loss of generality that $T^* = \That _{/\Zpbar}$ in (iii). We also remark that
$N_{^L G} (\That) = N_{\Ghat }(\That) \rtimes \Gal (L/\Qp)$.
\end{remark}

\begin{remark}
Note that (iii) implies that there exists $K\subset \Qpbar$ with $K/\Qp$ finite unramified
such that $\rho  (\Gamma_K) \subset T^* (\Zpbar)$.
\end{remark}

\begin{defn}\label{defn:obv-tau} We define 
\begin{align*}
\Wobv(\tau) := \big\{ F(\mu) :\ &\mu \in X_1 (T) ,\ \tau\ \text{admits an obvious crystalline lift}\ \rho \\
                                                       &\text{with}\ \HT_1 (\rho) = [\mu+\eta] \in X (T) / W \big\}.
\end{align*}
\end{defn}

\begin{prop}\label{prop: obvious weights of tau}
There holds the equality $\Wobv (\tau) = \{ F(\mu) : \mu\in X_1(T),\ \tau\cong \tau(w,\mu+\eta)\ \text{for some}\ w\in W\}$.
\end{prop}

\begin{proof}
More generally we will show that
\begin{align*}
\{ \HT_1 (\rho) &: \rho \ \text{an obvious crystalline lift of}\ \tau\} \\
&= \{ [\mu] : \mu\in X(T)\ \text{such that}\ \tau\cong \tau(w,\mu)\ \text{some}\ w\in W\} .
\end{align*}
(Note that $\tau (\sigma w F(\sigma)^{-1} , \sigma\mu)\cong \tau(w,\mu)$ for $\sigma\in W$.)

First suppose that $\rho : \Gamma_\Qp \to N_{^L G} (\That) (\Zpbar)$ is an obvious crystalline
lift of $\tau$.
Fix $K\subset\Qpbar$ with $K/\Qp$ finite unramified such that $\rho (\Gamma_K)\subset \That (\Zpbar)$.
Then for all $j : \Qpbar \to \Qpbar$, the co-character $\mu_j := \HT_j (\rho|_{\Gamma_K}) \in Y(\That)$
is a lift of $\HT_j (\rho) \in Y(\That)/W$ (by Lemma~\ref{lem:HT}).
Note that $\mu_j$ depends only on $j|_K$.
Also, $\conjj(\rho (\varphi))\circ \rho |_{\Gamma_K} = \rho |_{\Gamma_K}\circ \big( \varphi (-) \varphi^{-1}\big)$
(where $\conjj(g)$ denotes conjugation by $g$),
so by Lemma~\ref{lem:HT},
$\conjj(\rho (\varphi)) \circ \mu_j = \mu_{j\circ\varphi^{-1}}.$
Writing $\rho (\varphi) = \varphi \dotw^{-1} \in N_{^L G} (\That)(\Zpbar)$, for some
$\dotw \in N_{\Ghat} (\That) (\Zpbar)$ lifting $w\in W$, we get $\mu_{j\circ \varphi^{-1}} = \varphi w^{-1} \mu_j$.
For $\nu \in X(\That)$, $\nu\circ \rho |_{\Gamma_K} : \Gamma_K \to \Zpbarx\times$ is crystalline
with Hodge--Tate co-characters $\HT_j (\nu\circ\rho |_{\Gamma_K}) = \langle \nu , \mu_j \rangle \in \Z$,
hence by Lemma \ref{lem:existenceofcrystallinechars},
\[
\overline{\nu\circ\rho |_{I_\Qp}} = \prod_{j \in \Gamma_\Qp /\Gamma_K} \jbar (\omega_d)^{\langle\nu,\mu_j\rangle} ,
\]
where $d := [K: \Qp]$. It follows that
\begin{equation}\label{eq:1}
\overline{\rho |_{I_\Qp}} = \prod_{s=0}^{d-1} \mu_{\varphi^{-s}} (\omega_d^{p^s}) = \omega_d^
{\sum p^s (\varphi w^{-1})^s\mu_1} = \tau (w,\mu_1) ,
\end{equation}
as $F^* = p\varphi$ on $Y(\That) = X(T)$.

Conversely, given $(w,\mu)\in W\times X(T)$, choose $d\ge 1$ such that $(\varphi w^{-1})^d$
acts trivially on $\mu$.
Let $K\subset \Qpbar$ be unramified over $\Qp$ of degree $d$.
Let $\rho' : \Gamma_K \to \That (\Zpbar)$ be crystalline such that $\HT_{\varphi^{-s}} (\rho') = (\varphi w^{-1})^s
\mu \in Y(\That)$ for all $s\in\Z$.
(Note that $\HT_j (\rho')$ only depends on $j|_K$.
To construct $\rho'$ write $\That \cong \G_m^r$ and use Lemma~\ref{lem:existenceofcrystallinechars}.)
It follows by Lemma~\ref{lem:existenceofcrystallinechars} that $\varphi w^{-1} \circ \rho' = \rho'\circ
\big( \varphi (-) \varphi^{-1}\big)$ holds on $I_\Qp$.
Therefore we may define an $L$-parameter $\rho : \Gamma_\Qp \to \, {}^L G (\Zpbar)$ by
(i) $\rho |_{I_\Qp} = \rho' |_{I_\Qp}$ and (ii) $\rho (\varphi) = \varphi \dotw^{-1}$ where
$\dotw \in N_{\Ghat} (\That)(\Zpbar)$ is any fixed lift of $w$.
Then $\rho$ is crystalline, as $\rho |_{I_\Qp}$ is crystalline, and $\HT_1 (\rho) = [\mu]$ by Lemma~\ref{lem:HT}.
Also, $\overline{\rho |_{I_\Qp}} = \tau (w,\mu)$ by \eqref{eq:1}.
\end{proof}

For $\nu \in X(\uT)_+$ let $W(\nu)$ denote the $\uG$-module $\Ind_{\uB}^{\uG} (w_0 \nu)$ defined in \cite[II.2]{MR2015057}.
It has unique highest weight $\nu$ and $\uG$-socle $F(\nu)$.

\begin{defn}\label{def:closure-C-for-general-G}
If $\W$ is a set of Serre weights of $G(\Fp)$,  we define $\CC (\W)$ to be the smallest set of Serre weights
with the properties:
\begin{itemize}
\item $\W\subset\CC (\W)$, and
\item if $\W \cap \JH_{G(\Fp)} W(\nu) \ne \varnothing$, where $\nu\in X_1 (T)$, then $F(\nu)\in \CC (\W)$.
\end{itemize}
\end{defn}

\subsubsection{Levi predictions}
\label{sec:levi-predictions}

The Levi subgroups $M \subset G$ that contain $T$ are in bijection with
the $\Gamma_{\Qp}$-stable subsets of $\Delta$, by sending $M$ to
$\Delta_M$.  Note that each $M$ satisfies
Hypothesis~\ref{hyp:gp} with the same twisting element $\eta$ as
$G$. For each $M$ fix a dual group $(\Mhat, \Bhat_M, \That_M, \{
x_{\alpha,M}\}_{\alpha \in \Delhat_M})$. Then there is a
unique $\Gamma_{\Qp}$-equivariant homomorphism $i : (\Mhat,\Bhat_M,
\That_M) \to (\Ghat,\Bhat,\That)$ such that $i^* : X(\That) \to
X(\That_M)$ corresponds to $\id_{Y(T)}$ and $i \circ x_{\alpha,M} =
x_\alpha$ for all $\alpha \in \Delhat_M$. In fact, $i$ is a closed
immersion, so we can and will think of $\Mhat$ as the Levi subgroup of
$\Ghat$ containing $\That$ defined by $\Delta_M^{\vee}$, with induced
structures.

 We have variants of the definitions of
 Sections~\ref{sec:l-group-l}--\ref{sec:defin-Wexpl} with  $M$ replacing $G$, and we will indicate these by decorating notation with an
 $M$: in particular $X^M_1(\uT)$, $W^M(\mu)$ for $\mu \in
 X(\uT)_{+,M}$, $[\, \cdot \,]_M$, and
 in Section~\ref{sec:herzig-comparison} also $\Phi_M$, $\Phi_M^{+}$, $\| \cdot \|_M$, $\uparrow_M$, $\tau^M(w,\mu)$ for $w \in W_M, \mu\in X(\uT)$.

\begin{defn}\label{def:wexpl-general}
We recursively define $\Wexpl (\tau)$ to be the smallest set
containing $\Wobv(\tau)$ that is closed under the following operation:
whenever $\tau:I_{\Qp} \to \Ghat(\Fpbar)$ factors (perhaps after conjugation) through an inertial
$L$-parameter $\tau^M : I_{\Qp} \to \Mhat(\Fpbar)$ with $M$ as above,
and $\Wexpl(\tau^M)\cap \JH_{M(\Fp)} W^M(w\cdot \nu) \ne \varnothing$
(where $\nu \in X_1(\uT)$, $w\in W$ such that $w
\cdot \nu \in X(\uT)_{+,M})$, then $F(\nu) \in \Wexpl(\tau)$.
\end{defn}

\begin{remark}\label{rem:cl-wobv}
  If we take $M=G$ in the recursive step of this definition, we obtain
  that $\cC(\Wobv(\tau)) \subset \Wexpl(\tau)$.
\end{remark}

 The motivation for
  Definition~\ref{def:wexpl-general} is as in Section~\ref{sec:shad-weights:-inert}:\ we
  expect that $\tau^M$ has a crystalline lift of Hodge--Tate
  co-character $[w\cdot \nu + \eta]_M$, hence that $\tau$ has a
  crystalline lift of Hodge--Tate co-character $[\nu + \eta]$.

\begin{lem}\label{lem:wexpl-torus}
  If $G = T$ is a torus and $\tau$ a tame inertial $L$-parameter, then 
  $\Wexpl(\tau) = \Wobv(\tau) = \{ F(\mu-\eta) \}$ for any 
  $\mu \in X(\uT)$ such that $\tau \cong \tau(1,\mu)$.
\end{lem}

\begin{proof}
  The equality $\Wexpl(\tau) = \Wobv(\tau)$ is clear by Remark~\ref{rem:cl-wobv}.  By
  Proposition~\ref{prop:l-parameter}, there exists $\mu \in X(\uT)$ such that $\tau \cong
  \tau(1,\mu)$ (as $W = 1$), so $F(\mu-\eta) \in \Wobv(\tau)$ by Proposition~\ref{prop: obvious weights of tau}.
  By the same result, if we take any weight $F(\mu'-\eta) \in \Wobv(\tau)$, then $\tau \cong
  \tau(1,\mu')$. It follows that $(1+F^*+ \cdots + (F^*)^{d-1})(\mu-\mu') \equiv 0
  \pmod{(p^d-1) Y(\uT^*)}$, where $d$ is chosen as in
  \S\ref{sec:defin-W-questionmark}. Equivalently, $(1+F+ \cdots + F^{d-1})(\mu-\mu') \equiv 0
  \pmod{(p^d-1) X(\uT)}$. As $F^d-1 = p^d-1$ is injective on $X(\uT)$
  we see that $\mu \equiv \mu' \pmod{(F-1)X(\uT)}$, i.e.\ $F(\mu-\eta) \cong F(\mu'-\eta)$.
\end{proof}

\begin{remark}\label{rem:wexpl-torus}
  For general $G$ it follows from Lemma~\ref{lem:wexpl-torus} that if $M = T$ in the recursive step
  of Definition~\ref{def:wexpl-general}, then the non-emptiness of the intersection implies that
  $F(\nu) \in \Wobv(\tau)$, so no new weights are obtained.  (To see this, note that the
  non-emptiness implies that $\tau \cong \tau(1,w\cdot \nu + \eta)$ for some $w \in W$, by applying
  Proposition~\ref{prop: obvious weights of tau} to $\tau^T$. Hence $\tau \cong
  \tau(w^{-1} F(w), \nu + \eta)$, which implies the claim.)
\end{remark}

\subsection{Restriction of scalars}\label{sec:restriction-scalars}
Suppose for the rest of this section that $K\subset \Qpbar$ with $K/\Qp$ finite unramified.
In the following, if $X$ is a set (resp.\  group, resp.\  group scheme) with a smooth
left action of $\Gamma_K$, then we denote by $\Ind_{\Gamma_K}^{\Gamma_\Qp} X$ the induced
set (resp.\  group, resp.\  group scheme) consisting of functions $\Gamma_\Qp \to X$
that are $\Gamma_K$-equivariant.
For $\gamma\in\Gamma_\Qp$ let $\ev_\gamma : \Ind_{\Gamma_K}^{\Gamma_\Qp} X \to X$
denote the evaluation map at $\gamma$.
If $Y$ is a set of representatives of $\Gamma_K \backslash \Gamma_\Qp$, the $(\ev_y)_{y \in Y}$ provide
a non-canonical isomorphism $\Ind_{\Gamma_K}^{\Gamma_\Qp} X \isoto X^{[K:\Qp]}$ of sets
(resp.\  groups, resp.\  group schemes).

Suppose that $H$ is a connected reductive group over $\cO_K$ with Borel $B_H$, Levi $T_H\subset B_H$
and simple roots $\Delta_H \subset X (T_H)$.
Suppose that $H$ satisfies Hypothesis~\ref{hyp:gp} (or rather its analogue over $K$) with local twisting element $\eta_H\in X(T_H)^{\Gamma_K}$.
We may then obtain a group $G$ as in Section~\ref{sec:l-group-l} by restriction of scalars:
\[
(G,B,T) := \Res_{\cO_K / \Zp } (H,B_H,T_H) .
\]
Note that $G\times\Qpbar \cong \prod\limits_{\kappa:\, K\to\Qpbar} H\times_{K,\kappa} \Qpbar$, so
$G^\der$ is simply connected and $Z(G)$ is connected.
In particular, 
\[
X(T) \cong \bigoplus_{\kappa:\, K\to \Qpbar} X (T_H \times_{\kappa} \Qpbar) \cong
\Ind_{\Gamma_K}^{\Gamma_\Qp} X(T_H).
\]
It follows that $\Psi_0 (G,B,T) \cong \Ind_{\Gamma_K}^{\Gamma_\Qp} \Psi_0 (H,B_H,T_H)$
(where strictly speaking $\Delta$ consists of those functions in
$\Ind_{\Gamma_K}^{\Gamma_{\Qp}} \Delta_H$ that are supported on a single coset of $\Gamma_K$, and similarly
for $\Delta^\vee$) and that $\eta\in X(T)^{\Gamma_\Qp} \cong X(T_H)^{\Gamma_K}$ defined by $\eta_H$
is a local twisting element of $G$. Hence $G$ satisfies Hypothesis~\ref{hyp:gp}.
Let $L\subset \Qpbar$ denote the splitting field of $H$, so
$L/\Qp$ is finite unramified, $L \supset K$, and $G \times L$ is also split.

Let $(\Hhat,\Bhat_H ,\That_H , \{ x_{\alpha'}\} )$ be a dual group of $H$ as in Section~\ref{sec:l-group-l}, and define
\[
(\Ghat, \Bhat,\That) := \Ind_{\Gamma_K}^{\Gamma_{\Qp}} (\Hhat,\Bhat_H,\That_H) .
\]
Then $\Gamma_\Qp$ preserves a pinning $\{ x_\alpha\}$ of $\Ghat$ that is naturally induced from
$\{ x_{\alpha'}\}$.
We also see that $\Psi_0 (\Ghat, \Bhat,\That)\cong \Ind_{\Gamma_K}^{\Gamma_\Qp} \Psi_0
(\Hhat,\Bhat_H,\That_H)$ (with the same proviso as above), so via the induced isomorphism $\Psi_0 (\Ghat,\Bhat,\That)\isoto \Psi_0
(G,B,T)^\vee$ we can consider $(\Ghat,\Bhat,\That,\{ x_\alpha\})$ as a dual group of $G$. We let
$\LH := \Hhat \rtimes \Gal (L/K)$.

In the following, note that the notions of (inertial) $L$-parameter and obvious crystalline
lift carry over to representations of $\Gamma_K$ (repectively $I_K$). Also note that $\ev_1 : \Ghat (A) \to \Hhat (A)$ is
$\Gamma_K$-equivariant hence extends to a homomorphism $\ev_1 : \Ghat (A)\rtimes \Gal (L/K) \to \LH(A)$.

\begin{lemma}\label{lem:L-param}
Suppose that $A$ is a topological $\Zp$-algebra.
Then we have a bijection
\begin{equation}\label{eq:2}
\left\{ \myatop{ L\text{\rm -parameters} }{ \Gamma_\Qp \isorho \, {}^L G(A)} \right\}_{/\Ghat (A)} \isoto
\left\{ \myatop{L\text{\rm -parameters}}{\Gamma_K \isorhoK \, {}^L H(A)}\right\}_{/\Hhat (A)}.
\end{equation}
sending $\rho$ to $\rho_K = \ev_1 (\rho |_{\Gamma_K})$.
If $A=\Qpbar$ then $\rho$ is crystalline if and only if $\rho_K$ is crystalline.
\end{lemma}

\begin{proof}
For the moment let us consider $A$ with the discrete topology.
By writing $\rho (g) = \rho^0 (g)\rtimes g$ we see that $\rho^0$ defines a 1-cocycle $\Gamma_\Qp \to \Ghat (A)$,
and in this way we get a bijection between $\Ghat (A)$-conjugacy classes of $L$-parameters $\rho$
and the pointed set $H^1 \big( \Gamma_\Qp , \Ghat (A)\big)$.
As $\Ghat (A) \cong \Ind_{\Gamma_K}^{\Gamma_\Qp} \Hhat (A)$, the non-abelian Shapiro lemma
(\cite[Prop.~8]{MR2589844}) shows that $H^1\big( \Gamma_\Qp,\Ghat (A)\big) \cong H^1 \big( \Gamma_K , \Hhat (A)\big)$
where $\rho^0$ is sent to $\ev_1 (\rho^0 |_{\Gamma_K})$.
This proves \eqref{eq:2} if $A$ is discrete. From
the description $\rho_K = \ev_1 (\rho|_{\Gamma_K})$ and \eqref{eq:8} it follows in general that $\rho$ is continuous iff
$\rho |_{\Gamma_L}$ is continuous iff $\rho_K$ is
continuous, and 
similarly for the crystalline condition when $A = \Qpbar$.
For later reference we recall from \cite{MR2589844} a description of a representative
$\rho$ in the inverse image of $\rho_K$.
Let $Y$ be a set of representatives of $\Gamma_K \backslash \Gamma_\Qp$ with $1 \in Y$.
Then $\rho$ is defined by
\begin{equation}\label{eq:3}
\ev_\gamma \rho^0 (\gamma') = \rho_K^0 (\delta)^{-1}\cdot\rho_K^0 (\delta') \in \Hhat (A),
\end{equation}
where $\gamma = \delta y$, $\gamma \gamma' = \delta' y'$ with $\delta,\delta' \in \Gamma_K$,
$y,y'\in Y$.
\end{proof}

Note that in the context of Lemma~\ref{lem:L-param} for any $\gamma\in \Gamma_\Qp$ we have $\gamma\circ \rho |_{\Gamma_L}
\cong \rho |_{\Gamma_L}\circ \big( \gamma (-) \gamma^{-1}\big)$, so
\begin{equation}\label{eq:8}
  \ev_\gamma (\rho |_{\Gamma_L}) \cong \rho_K|_{\Gamma_L} \circ \big (\gamma (-) \gamma^{-1}\big) \ :\ \Gamma_L \to \Hhat (A).
\end{equation}

\begin{lemma}\label{lem:L-param2}
Suppose that $A$ is a topological $\Zp$-algebra.
Then we have a bijection
\[
\left\{ \myatop {\text{\rm inertial}\ L\text{\rm -parameters} }  { I_\Qp \isotau \Ghat (A)} \right\}_{/\Ghat (A)} \isoto
\left\{  \myatop {\text{\rm inertial}\ L\text{\rm -parameters}} { I_K \isotauK  \Hhat (A)}\right\}_{/\Hhat (A)}
\]
sending $\tau$ to $\tau_K = \ev_1 (\tau)$.
\end{lemma}

\begin{remark}
Note that $I_K = I_\Qp$.
\end{remark}

\begin{proof}
By Lemma~\ref{lem:L-param} the map is well defined and surjective.
Suppose that $\tau_1,\tau_2$ are inertial $L$-parameters such that $\tau_{1,K} \cong \tau_{2,K}$.
As the $\tau_i$ extend to $L$-parameters, from (\ref{eq:8}) we get that 
$\ev_\gamma (\tau_1) \cong \ev_\gamma (\tau_2)$ for any $\gamma\in \Gamma_\Qp$.
Let $Y$ be a set of representatives of $\Gamma_K \backslash \Gamma_\Qp$, and choose $h_y\in \Hhat (A)$ ($y\in Y$) such that
\[
\ev_y (\tau_1) = h_y \cdot \ev_y (\tau_2) \cdot h_y^{-1} ,
\qquad \textrm{for all } y \in Y .
\]
Then $\tau_1 = g\cdot \tau_2 \cdot g^{-1}$, where $g\in\Ghat (A)$ is defined by $g(y) = h_y$ for
$y\in Y$.
\end{proof}

\begin{lemma}\label{lem:obvious}
Suppose that $\tau : I_{\Qp} \to \Ghat (\Fpbar)$ is a tame inertial $L$-parameter and that
$\rho : \Gamma_\Qp \to \, {}^L G (\Zpbar)$ is an $L$-parameter.
Then $\rho$ is an obvious crystalline lift of $\tau$ iff $\rho_K$ is an obvious crystalline lift of $\tau_K$.
\end{lemma}

\begin{proof}
The ``only if'' implication is immediate from Lemma~\ref{lem:L-param}.
Conversely, if $\rho_K$ is an obvious crystalline lift of $\tau_K$, by Lemma~\ref{lem:L-param2}, $\overline{\rho |_{ I_\Qp}} \cong \tau$,
as it is true after evaluating at $1$.
Also, $\rho$ is crystalline by Lemma~\ref{lem:L-param}.
Now assume without loss of generality that $\rho_K$ takes values in
$N_{^L H} (\That_H)(\Zpbar)$. Then $\rho$ (or rather a representative
of $\rho$ in its conjugacy class) is obtained from $\rho_K$ by formula \eqref{eq:3},
which shows that $\rho$ takes values in $N_{^L G} (\That)
(\Zpbar)$. Now if $\gamma' \in I_\Qp$, then $\rho (\gamma')\in\That
(\Zpbar)$ follows from \eqref{eq:3}, as $y=y'$ and $\delta' \in \delta I_K$.
\end{proof}

\subsection{The case of \texorpdfstring{$\GL_n$}{GL(n)}}\label{sec:gln}
Suppose that $H=\GL_n$, $B_H$ the upper-triangular Borel subgroup and $T_H$ the diagonal torus, all
over $\cO_K$.
Define $(\Hhat, \Bhat_H, \That_H)$ likewise but over $\Zp$.
Identify $X(T_H)\isoto Y(\That_H)$ by sending $\diag (x_1 , \ldots, x_n)\mapsto \prod x_i^{a_i}$ to
$x\mapsto \diag (x^{a_1}, \ldots, x^{a_n})$.
Note that $L= K$ and $^L H = \Hhat$.
Let $\eta_H \in X(T_H)$ be the local twisting element $\diag (x_1,\ldots, x_n)\mapsto \prod x_i^{n-i}$.

\begin{lemma}
  \label{lem:obvious-lifts}
  The representation $\rho_K : \Gamma_K \to \GL_n (\Zpbar)$ is an obvious crystalline lift of $\tau_K$:
$I_K\to \GL_n (\Fpbar)$ in the sense of the current section if and only
if it is an obvious crystalline lift in the sense of Definition \ref{defn:obvious-lift-ss}.
\end{lemma}

\begin{proof}
It suffices to show that $\rho_K : \Gamma_K \to \GL_n (\Zpbar)$ satisfies condition (iii)
in Definition~\ref{def5} iff it is isomorphic to $\bigoplus \Ind_{\Gamma_{K_i}}^{\Gamma_K} \Zpbar (\chi_i)$ for some
$K_i \subset \Qpbar$ with $K_i / K$ finite unramified and characters $\chi_i : \Gamma_{K_i}\to
\Zpbarx\times$.
This is clear, as either condition is equivalent to $\Zpbarx n$ being a direct sum of $n$ rank $1$
free $\Zpbar$-submodules that are permuted by $\Gamma_K$, with $I_K$ preserving each summand.
\end{proof}

Combining the previous lemma with Lemma~\ref{lem:obvious}, we obtain
the following.

\begin{cor}
$\rho : \Gamma_\Qp \to \, {}^L G (\Zpbar)$ is an obvious crystalline lift of $\tau : I_\Qp \to \Ghat (\Fpbar)$
if and only if $\rho_K : \Gamma_K \to \GL_n (\Zpbar)$ is an obvious crystalline lift of $\tau_K$:
$I_K\to \GL_n (\Fpbar)$ in the sense of Definition \ref{defn:obvious-lift-ss}.
\end{cor}

We can now show that when $G = \Res_{K/\Qp} \GL_n$, the set
$\Wexpl(\tau)$ recovers the collection of Serre weights given by
Definition~\ref{defn:explicit-weights-ss}. More precisely, we have the following.

\begin{prop}\label{prop:Wexpl}
For $\tau : I_\Qp \to \Ghat (\Fpbar)$ a tame inertial $L$-parameter, $\Wexpl (\tau) = \Wexpl (\tau_K)$,
where the latter set is computed according to Definition~\ref{defn:explicit-weights-ss}.
\end{prop}

\begin{remark}
Note that we have canonical isomorphisms $G(\Fp) \cong H(k) = \GL_n (k)$, where $k$ is the residue field of $K$.
\end{remark}

\begin{proof}
Let $r : \cO_K \to k$ be the reduction map and $\sigma_0 : k\to \Fpbar$ the inclusion.
We have canonical isomorphisms
\begin{equation}\label{eq:5}
G \times \Fpbar \cong \prod_{\sigma: k\to \Fpbar} H \times_{\sigma r} \Fpbar ,
\end{equation}
\begin{equation}\label{eq:6}
X(T\times \Fpbar) \cong \bigoplus_{\sigma: k\to \Fpbar} X (T_H \times_{\sigma r} \Fpbar  ) \cong
\Ind_{\Gamma_K}^{\Gamma_\Qp} X (T_H \times_{\sigma_0 r} \Fpbar ) .
\end{equation}
For $\mu = (\mu_\sigma) \in X (T\times \Fpbar)$ we can write $\mu_\sigma = \nu_\sigma \times_\sigma \Fpbar$
with $\nu_\sigma \in \Hom_k (T_H \times_r k,\G_m)$, as $T_H$ is split.
Then in \eqref{eq:6}, for $\gamma\in \Gamma_\Qp$,
\[
\ev_{\gamma^{-1}} (\mu ) = \ev_1 (\gamma^{-1} \mu ) = (\gamma^{-1} \mu )_{\sigma_0} = \gamma^{-1}
\mu_{\gammabar \circ\sigma_0} = \nu_{\gammabar\circ\sigma_0} \times \Fpbar .
\]
In particular, as $\eta\in X(T)^{\Gamma_\Qp}$, $\ev_{\gamma^{-1}} (\eta)  = \eta_H = \eta_{H,0} \times \Fpbar$,
where $\eta_{H,0} = (n-1 , n-2 , \ldots, 0) \in \Hom_k (T_H \times_r k, \G_m )$.

Suppose that $\rho$ is an obvious crystalline lift of $\tau$ with $\HT_1 (\rho) = [\mu + \eta]$, where
$\mu\in X_1 (T)$.
Then from Lemma~\ref{lem:HT} and \eqref{eq:8} we get for any $\gamma\in \Gamma_\Qp$,
\begin{equation} \label{eq:7}
  \begin{split}
    \HT_\gamma (\rho_K) &= \HT_1 \big(\rho_K \circ (\gamma^{-1} (-) \gamma)\big) = \ev_{\gamma^{-1}} \big( \HT_1 (\rho)\big) \\ &= [\ev_{\gamma^{-1}} (\mu+\eta)]
    = \big[ (\nu_{\gammabar \circ\sigma_0} + \eta_{H,0}) \times \Fpbar \big] .
  \end{split}
\end{equation}
With respect to the decomposition \eqref{eq:5}, the Serre weight $F(\mu) \in \Wobv (\tau)$ is isomorphic to
$\bigotimes\nolimits_\sigma F (\mu_\sigma)$ as representations of $G(\Fp) \cong H(k)$, so
$F(\mu) \cong \bigotimes F_k (\nu_\sigma) \otimes_\sigma \Fpbar$, where $F_k (\nu_\sigma)$ denotes the irreducible
algebraic $H_{/k}$-representation with highest weight $\nu_\sigma$. Then $F_k (\nu_\sigma) \cong N_{\nu_\sigma}$ (cf.\ Section~\ref{subsec:Serre weights})
as representations of $G(\Fp) \cong H(k)$. By \eqref{eq:7}, $F(\mu)$ is the Serre weight
associated to the obvious crystalline lift $\rho_K$ by Definition \ref{defn:obvious-lift-ss}.
Hence $\Wobv (\tau) = \Wobv (\tau_K)$.

Similarly, with the above notation we get an isomorphism $W(\mu) \cong \bigotimes W_k(\nu_\sigma) \otimes_{\sigma}
\Fpbar$, where $W_k(\nu_\sigma)$ denotes the $H_{/k}$-module $\Ind_{{B_H}_{/k}}^{H_{/k}} (w_0 \nu_\sigma)$, so
$W_k(\nu_\sigma) \cong M'_{\nu_\sigma} \otimes_{\cO_K} k$ (cf.\ Section~\ref{subsec:notation}).
We deduce that $\cC(\Wobv(\tau)) = \cC(\Wobv(\tau_K))$, where the latter is
computed according to Definitions~\ref{defn:obvious-lift-ss}
and~\ref{defn:closure-operation}.

To compare explicit predicted weights, note first that any Levi $M$ of
$G$ containing $T$ is of the form $M = \Res_{\cO_K/\Zp} M_H$ with
$M_H$ a Levi of $H$ containing $T_H$, so  that $M_H \cong \prod_{j=1}^r
\GL_{n_j}$ for some $r$ and $n_j$'s; then $\tau$ factors through
$\tau^M$ if and only if there is an isomorphism $\tau_K \cong \oplus_j \tau_K^{(j)}$ with
$\dim \tau_K^{(j)} = n_j$ for all $j$.  In general, whenever $(G,B,T)
\cong \prod_j (G_j, B_j, T_j)$ factors as a product of pinned groups,
then $\eta = \sum_j \eta_j$, where $\eta_j$ is a local twisting element of
$G_j$, and $\mathrm{W}_{\mathrm{expl},\eta}(\tau_1 \times \cdots
\times \tau_r) = \{ \boxtimes_j F_j : F_j \in
\mathrm{W}_{\mathrm{expl},\eta_j}(\tau_j) \}$ where the subscripts
$\eta$ and $\eta_j$ indicate the dependence on the local twisting
element. Moreover, if $\mu = \sum \mu_j$ with $\mu_j \in X(T_j)_+$,
then $F(\mu) \cong \boxtimes_j F^{G_j}(\mu_j)$ and $W(\mu) \cong
\boxtimes_j W^{G_j}(\mu_j)$. From Proposition~\ref{prop: obvious weights
    of tau} we get that $\mathrm{W}_{\mathrm{expl},\eta}(\tau) =
  \mathrm{W}_{\mathrm{expl},\eta'}(\tau) \otimes F(\eta'-\eta)$
  whenever $\eta'$ is another local twisting element. Putting these
  observations together, we see that $\Wexpl(\tau) = \Wexpl(\tau_K)$,
  where the latter is computed according to Definition~\ref{defn:explicit-weights-ss}.
\end{proof}

\subsection{A unitary group example}\label{sec:unitary-example}

\makeatletter
\def\Ddots{\mathinner{\mkern1mu\raise\p@
\vbox{\kern7\p@\hbox{.}}\mkern2mu
\raise4\p@\hbox{.}\mkern2mu\raise7\p@\hbox{.}\mkern1mu}}
\makeatother

We work out another example of the constructions in this section. Suppose that $n \ge 1$ and let
$J =
\bigg(\begin{smallmatrix}
  && 1 \\ & \Ddots \\ 1
\end{smallmatrix}\bigg)$. For any $\Zp$-algebra $A$ define
$G(A) = \{ g \in \GL_n(A \otimes_{\Zp} \Z_{p^2}) : {}^t g \cdot J \cdot \o g = J \}$, where
conjugation $g \mapsto \o g$ is trivial on $A$ and the non-trivial Galois automorphism on
$\Z_{p^2}$.  Then $G$ is a connected reductive group over $\Zp$ with generic fibre the unramified
unitary group over $\Qp$ of absolute rank $n$. 

We consider the upper-triangular Borel $B$ and
diagonal maximal torus $T$. Then the splitting field is $\Q_{p^2}$ and $\varphi \in
\Gal(\Q_{p^2}/\Qp)$ acts as $(a_1,\dots,a_n) \mapsto -(a_n,\dots,a_1)$ on $X(T) \cong \Z^n$, i.e.\ as $-w_0$.  We
assume that $n$ is \emph{odd} so that $G$ has a local twisting element, namely $\eta =
(\frac{n-1}2,\frac{n-3}2,\dots,-\frac{n-1}2)$. As dual group we take $\wh G = \GL_n$ over $\Zp$
with upper-triangular Borel $\wh B$ and diagonal maximal torus $\wh T$, with pinning given by 
the isomorphisms sending $a \in \G_a$ to the upper-triangular unipotent matrix having unique
off-diagonal element $a$ in the $(i,i+1)$-entry (for $1 \le i < n$) and obvious identification $\phi$
(as in \S\ref{sec:gln}). With these choices, ${}^L G = \GL_n \rtimes \Gal(\Q_{p^2}/\Qp)$ with
$\varphi$ acting as $g \mapsto J' \cdot {}^t g^{-1} \cdot (J')^{-1}$, where
$J' = \bigg(\begin{smallmatrix}
  && 1 \\ & -1 \\ \Ddots
\end{smallmatrix}\bigg)$ (with alternating signs along the diagonal).

We can identify $\underline G$ with ${\GL_n}_{/\Fpbar}$ via the inclusion $\F_{p^2} \into
\Fpbar$, and then $F = -p w_0$ on $X(\underline T)$. Serre weights are identified with equivalence
classes $X_1^{(n)}/{\sim}$, where $a \sim a'$ if and only if $a-a' \in (p+1,\dots,p+1)\Z$.

To finish, here is an explicit example with $n = 3$. Consider $\mu = (a,b,c)$ sufficiently generic
in the lowest alcove and suppose that the inertial $L$-parameter $\tau : I_{\Qp} \to \GL_3(\Fpbar)$ is given by 
$\tau \cong \tau(1,\mu)$, i.e.\ $\tau \cong \bigg(\begin{smallmatrix}
  \omega_2^{a-pc} \\ & \omega_2^{b(1-p)} \\ && \omega_2^{c-pa}
\end{smallmatrix}\bigg)$. Then 
\begin{gather*}
  \Wexpl(\tau) = \{ F((a,b,c)-\eta), F((b,c-1,a-p)-\eta), F((c+p,a+1,b)-\eta),\\
  F((c+p-1,b,a-p+1)-\eta), F((a,c,b-p-1)-\eta), F((b+p+1,a,c)-\eta),\\
  F((c+p,b,a-p)-\eta), F((a,c-1,b-p)-\eta), F((b+p,a+1,c)-\eta)
  \},
\end{gather*}
where the first six weights are obvious and the last three are shadows.
For example, the second weight is obvious by Proposition~\ref{prop: obvious weights of tau}
since $\tau \cong \tau((1\,2\,3), (b,c-1,a-p))$. Note that there are no obscure
weights by Remarks~\ref{rem:cl-wobv} and \ref{rem:wexpl-torus}, as $G$ and $T$ are 
the only Levi subgroups of $G$ that contain $T$.

\section[sec8]{Comparison with \texorpdfstring{\cite{herzigthesis}}{[Her09]}}\label{sec:herzig-comparison} 
 
In this section we will prove that for $L$-parameters $\tau : I_{\Qp}
\to \Ghat(\Fpbar)$ that are sufficiently generic, the sets $\W^?(\tau)$
and $\Wexpl(\tau)$ are equal. This establishes in particular
that
Conjecture~\ref{conj:explicit-version-ss} is in
agreement with the Serre weight conjecture of \cite{herzigthesis}.

\subsection{The weight set \texorpdfstring{$\W^?(\tau)$}{W\_?(tau)} in the generic case}\label{subsec: weight
set in the generic case} We begin by giving an alternate
characterization of the set $\W^?(\tau)$ (for $\tau$ sufficiently
generic in a sense to be made precise below) in terms of the
$\uparrow$ relation on alcoves.  We refer the reader to
\cite[II.6.5]{MR2015057}
for the definition of the $\uparrow$
relation; see also \cite[Def.~3.15]{herzigthesis}.

We use the same notation as in
Section~\ref{sec:defin-W-questionmark}. Recall from the proof of Lemma~\ref{lem:bijection of weights}
that there is a finite order automorphism $\pi$ of $(\uG,\uB,\uT)$
that induces the action of $\varphi^{-1}$
on $\Psi_0(\uG,\uB,\uT)$. In particular, $F=p\pi^{-1}$ on $X(\uT)$.

Let $\Phi\subset X(\uT)$ denote the set of roots, $\Phi^+$ the subset of
positive roots, $W_p:=p\Z\Phi\rtimes W$ the affine Weyl group, and
$\tW_p:=pX(\uT)\rtimes W$ the extended affine Weyl group. We refer
to~\cite[II.6]{MR2015057} for the definition of alcoves and for the basic facts
about them.
We denote by $C_0$ the lowest (or fundamental) alcove. We say that a weight
$\lambda\in X(\uT)$ is \emph{$p$-regular} if it does not lie on any alcove walls; equivalently,  $\Stab_{W_p}(\lambda)=1$.

\begin{lem}
  \label{lem: stabiliser in alcove} Suppose $\lambda\in X(\uT)$ is $p$-regular. Then $\Stab_{\tW_p}(\lambda)=1$.
\end{lem}
\begin{proof}
  The proof of~\cite[Lemma 5.6]{herzigthesis} applies, as $Z(\uG)$ is connected.
\end{proof}

Recall from~\cite[\S 5.2]{herzigthesis} the definition of $\mu\in X(\uT)$ lying
\emph{$\delta$-deep} in an alcove. We say that a statement is true for
$\mu$ lying
\emph{sufficiently deep} in some alcove $C$ if there is a $\delta>0$ depending
only on the based root datum $\Psi_0(\uG,\uB,\uT)$ together with its automorphism $\pi$ (and in particular not on
$p$) such that the statement holds for all $\mu$ which are $\delta$-deep in $C$.

Recall from~\cite[II.9]{MR2015057} the definitions of $\uG_1\uT$-modules
$\widehat{Z}_1(\lambda)$ and $\widehat{L}_1(\lambda)$ for $\lambda\in X(\uT)$, where
the group scheme $\uG_1$ is the kernel of $F : \uG \to \uG$. Supposing
there exists $\mu\in C_0\cap X(\uT)$ (equivalently, $p > \langle\eta,\alpha^\vee\rangle 
\ \forall \alpha \in \Phi^+$), let \[D_1:=\{u\in \tW_p:u\cdot\mu\in
X_1(\uT)\}.\] This set is independent of the choice of $\mu$, and it is a finite union of $pX^0(\uT)$-cosets.

\begin{prop}
  \label{prop:reduction of R for mu sufficiently deep} For $\mu$ lying
  sufficiently
  deep in the alcove $C_0$, we have \[\overline{R(w,\mu+\eta)}=\sum_{u\in
    D_1/pX^0(\uT)}\sum_{\nu\in
    X(\uT)}[\widehat{Z}_1(\mu+p\eta):\widehat{L}_1(p\nu+u\cdot\mu)]\, F(u\cdot(\mu+w\pi\nu))\]
  in the Grothendieck group of finite-dimensional $\Fpbar[G(\Fp)]$-modules.
\end{prop}
\begin{remark}
  Recall that the Deligne--Lusztig representation $R(w,\mu)$ was defined in Definition~\ref{def:deligne-lusztig}.
  The notation $[\wh Z_1(\lambda): \wh L_1(\mu)]$ signifies the multiplicity of the simple
  $\uG_1\uT$-module $\wh L_1(\mu)$ as Jordan--H\"older factor of $\wh Z_1(\lambda)$.
\end{remark}
\begin{remark}
  \label{rem: inner sum depends only on coset of u}By Lemma~\ref{lem:bijection of
    weights} the inner sum depends only on the coset of $u$ in $D_1/pX^0(\uT)$.
\end{remark}
\begin{proof}
  This proposition is a generalisation of Jantzen's generic decomposition formula for
  Deligne--Lusztig representations~\cite[Satz~4.3]{bib:Jan-DL}. For a
  generalisation of~\cite[\S1--3]{bib:Jan-DL} to reductive groups with simply
  connected derived subgroups, see the appendix to~\cite{herzigthesis}. We now
  explain how~\cite[\S4]{bib:Jan-DL} generalises to the same context. We only
  leave aside the part of~\cite[\S4.1]{bib:Jan-DL}  that follows equation
  4.1(2). Without further comment, any reference in the remainder of this
  paragraph will be to~\cite{bib:Jan-DL}, and we keep the same notation and
  conventions as in the appendix of~\cite{herzigthesis}. For example, any
  occurrence of $\rho,\rho_w,\varepsilon_w,\gamma_{w_1,w_2}$ should be replaced by
  $\rho',\rho'_w,\varepsilon'_w,\gamma'_{w_1,w_2}$. In addition, in \S4 any 
  occurence of the term ``$\mu+\rho$'' should be replaced by
  ``$\mu+\pi\rho'$''. We let
  $h:=\max\{\langle\rho',\alpha^\vee\rangle+1:\alpha\in R^+\}$. Furthermore, any
  occurence of $D_n$ as the index of a sum should be replaced by (a fixed set of
  representatives of) $D_n/p^nX^0(T)$, which is finite. In particular, Satz 4.3
  says that \[\tR_w(n,\mu+\pi\rho')=\Psi\sum_{\genfrac{}{}{0pt}{}{u\in
    D_n/p^nX^0(T)}{\nu\in X(T)}}
  [\widehat{Z}(n,\mu+p^n\rho'):\widehat{L}(n,p^n\nu+u\cdot\mu)]\,
  \chi_p(u\cdot(\mu+w\pi\nu)).\] In \S 4.3 and \S 4.4, $\alpha_0^\vee$ denotes any
  choice of highest coroot. The inequality in line~$-3$ of page~472 is no longer true
  in general, but the following line still holds. In line~$-1$ of page~472 the second
  occurrence of $\nu_1$ should be $\nu_2$ (a typo). In the proof of Lemma~4.4,
  $-w'\varepsilon'_{w_0w'}=\rho'-\rho'_{w'}$ only holds modulo $X^0(T)$, but this
  is sufficient. The diagonal elements of the upper-triangular matrix now lie in
  $X^0(T)\subset\Z[X(T)]^W$, and so the terms $\chi$ likewise need to be
  multiplied by elements of $X^0(T)$. Similar comments apply to the following
  two displayed equations.

To deduce our proposition, we choose Jantzen's split $G/\Fp$ such that
$G\times\Fpbar\cong\uG$ with relative Frobenius $F\circ\pi^{-1}$ (see the proof
of Lemma~\ref{lem:bijection of weights}). We then choose $\rho'=\eta$ (noting that
$\pi\eta=\eta$), take $n=1$, and use Lemma~\ref{lem: stabiliser in alcove}.
\end{proof}
\begin{lem}
  \label{lem:condition for Z:L nonzero}Suppose $p\ge
  2\max\{\langle\eta,\alpha^\vee\rangle:\alpha\in \Phi^+\}$. For 
  weights $\lambda\in
  X(\uT)$, $\mu\in X_1(\uT)$ we
  have \begin{equation}\label{eqn:condition for Z:L
      nonzero}[\widehat{Z}_1(\lambda):\widehat{L}_1(\mu)]\ne 0\iff
  \sigma\cdot(\lambda-p\eta)\uparrow w_0\cdot(\mu-p\eta)  \textrm{ for all }  \sigma\in W.\end{equation}
\end{lem}
\begin{rem}
  Here we do not need to assume that $Z(\uG)$ is connected or that
  $\eta$ is $\Gal(L/\Qp)$-invariant. We remark that 
  \cite[Cor.~2.7]{DS} relies on Corollary~\ref{cor:ye-wang}(ii).
\end{rem}
\begin{proof}
  First suppose that $\uG=\uG^\der$. Then \cite[Cor.~2.7]{DS} shows that
  $[\widehat{Z}_1(\lambda):\widehat{L}_1(\mu)]\ne 0$ if and only if $\mu\in\cap_{y\in
    W_v}I^{-1}_{y,1}\cdot\operatorname{SL}(y\cdot\lambda)$ in their notation,
  where $v\in-\eta+pX(\uT)$ is arbitrary. Taking $v=-\eta$ and $y\in W_{-\eta} = W$, we find
  that this is equivalent to $\mu+p(y\eta-\eta)\uparrow y\cdot\lambda$
  for all
  $y\in W$, or equivalently $w_0y\cdot(\lambda-p\eta)\uparrow
  w_0\cdot(\mu-p\eta)$ for all $y\in W$.

In the general case, note that both sides of~(\ref{eqn:condition for Z:L
      nonzero}) imply that $\mu\in W_p\cdot\lambda$. For $\mu\in
    W_p\cdot\lambda$ we have
    $[\widehat{Z}_1(\lambda):\widehat{L}_1(\mu)]_{\uG}=[\widehat{Z}_1(\lambda):\widehat{L}_1(\mu)]_{\uG^\der}$
    and
    $\mu\uparrow\lambda$ if and only if
    $\mu|_{\uT\cap\uG^\der}\uparrow\lambda|_{\uT\cap\uG^\der}$. (Note that $W_p\cdot\lambda \subset \lambda + \Z \Phi$.
    The restriction map $X(\uT) \onto X(\uT\cap\uG^\der)$ induces a bijection $\lambda + \Z \Phi \isoto \lambda|_{\uT\cap\uG^\der} + \Z \Phi$,
    which identifies $\le$, $W_p$-actions, and hence $\uparrow$, on both sides. Also, $\widehat{Z}_1(\lambda)$, $\widehat{L}_1(\mu)$ restrict
    to corresponding objects for $\uG^\der$, see the proof of~\cite[Prop.~5.7]{herzigthesis}.)
    This reduces
    the claim to the case $\uG=\uG^\der$.
\end{proof}
\begin{prop}
  \label{prop: mu sufficiently deep condition on JH const of Rbar} For $\mu$ lying
  sufficiently deep in the alcove $C_0$, and for $\lambda\in
  X_1(\uT)$, we have that  $F(\lambda)$ is
  a Jordan--H\"older constituent of $\overline{R(w,\mu+\eta)}$ if and only if
  there exists $\nu\in X(\uT)$ such that \[\sigma\cdot(\mu+(w\pi-p)\nu)\uparrow
  w_0\cdot(\lambda-p\eta)\qquad  \textrm{for all } \sigma\in W.\]
\end{prop}
\begin{proof}
  For $\mu$  lying sufficiently deep in $C_0$ we have $p\ge
  2\max\{\langle\eta,\alpha^\vee\rangle:\alpha\in\Phi^+\}$, so we can
  (and do) assume this inequality. From
  Proposition~\ref{prop:reduction of R for mu sufficiently deep} we know that \[\overline{R(w,\mu+\eta)}=\sum_{u\in
    D_1/pX^0(\uT)}\sum_{\nu\in
    X(\uT)}[\widehat{Z}_1(\mu+p(\eta-\nu)):\widehat{L}_1(u\cdot\mu)]\, F(u\cdot(\mu+w\pi\nu)).\]By
  Lemma~\ref{lem:condition for Z:L nonzero}, the $(u,\nu)$ term of the
  double sum is non-zero if
  and only if
  \begin{equation}
    \label{eq:condfornonzero}
    \sigma\cdot(\mu-p\nu)\uparrow w_0\cdot(u\cdot\mu-p\eta)\qquad  \textrm{for all } \sigma\in W.
  \end{equation}
As in the proof of~\cite[Prop.~5.7]{herzigthesis}, for $\mu$ sufficiently
deep in $C_0$ we obtain $\sigma\cdot(\mu+(w\pi-p)\nu)\uparrow
w_0\cdot(u\cdot(\mu+w\pi\nu)-p\eta)$ for all $\sigma\in W$, which proves the ``only
if'' part of the proposition. (Note that in~(\ref{eq:condfornonzero}) there are
only finitely many possibilities for $\nu$ modulo $X^0(\uT)$, independent of
$\mu$.)

Conversely, if $\sigma\cdot(\mu+(w\pi-p)\nu)\uparrow
w_0\cdot(\lambda-p\eta)$ for all $\sigma\in W$, we may reverse the above argument,
as explained in the proof of~\cite[Prop.~5.7]{herzigthesis}.
\end{proof}
\begin{lem}
  \label{lem:suff deep implies maximally split} For $\mu$ lying sufficiently deep
  in the alcove $C_0$, we have that $(\uT_w,\theta_{w,\mu})$ is maximally split for all $w\in W$.
\end{lem}
\begin{proof}
  The dual pair is $(\T^*,s)$ with $\T^*=\uT^*_{F^*(w^{-1})}$ and
  $s=g^*_{F^*(w^{-1})}s'(g^*_{F^*(w^{-1})})^{-1}$, and where $s':=N_{(F^*\circ
    w^{-1})^d/F^*\circ w^{-1}}{\mu}(\zeta_{p^d-1})$ with $d > 0$ chosen such that $(F^*\circ
  w^{-1})^d=p^d$. We can define
  $\hat{s}:X(\T^*)\to\Qpbartimes$ by $\hat s(\mu) := \widetilde{\mu(s)}$. Then
  $\widehat{w(s)} = w(\hat s)$ for $w \in N(\T^*)/\T^*$, so 
  $\Stab_{N(\T^*)/\T^*}(\hat{s})\cong\Stab_{N(\T^*)/\T^*}(s)\cong\Stab_W(s')$. By \cite[Thm.~5.13]{bib:DL},
  this group is generated by reflections,
  as $Z(\uG)$ is connected. A reflection $s_\alpha\in W$ fixes $s'\in \uT^*$
  if and only if
  \begin{align*}
    & (1-s_\alpha)\left(\sum_{i=0}^{d-1}(F^*\circ
      w^{-1})^i\mu\right)(\zeta_{p^d-1})=1\\ \iff & (1-s_\alpha)\left(\sum_{i=0}^{d-1}(F^*\circ
      w^{-1})^i\mu\right)\equiv 0\pmod{(p^d-1)X(\uT)}\\ \iff & \biggr\langle \sum_{i=0}^{d-1}(p\pi^{-1}
      w^{-1})^i\mu,\alpha^\vee\biggr\rangle\equiv 0 \pmod{p^d-1},
  \end{align*}
where we used that $\langle\alpha,Y(\uT)\rangle=\Z$, as $Z(\uG)$ is
connected. Equivalently, \[
\sum_{i=0}^{d-1}p^i\langle\mu,(w\pi)^i\alpha^\vee\rangle\equiv 0 \pmod{p^d-1}.\]
If $\mu\in C_0$, the left-hand side has to be zero, so
$\langle\mu,\alpha^\vee\rangle\equiv 0\pmod{p}$, and this is impossible if $\mu$
lies $(h-1)$-deep in $C_0$, where
$h=\max\{\langle\eta,\beta^\vee\rangle+1:\beta\in\Phi^+\}$. Thus for $\mu$ lying
sufficiently deep in $C_0$, the Weyl group of $\T^*$ in the connected reductive
group $Z_{\uG^*}(s)$ is trivial, so $Z_{\uG^*}(s)=\T^*$, which implies that
$(\T^*,s)$ is maximally split.
\end{proof}
The group $X(\uT)\rtimes W$ acts on the set $W\times
X(\uT)$ by
\begin{equation}
  \label{eq:extendedaffineweylaction}
  {}^{(\nu,\sigma)}(w,\mu)=(\sigma w\pi\sigma^{-1}\pi^{-1},\sigma\mu+(p-\sigma w\pi\sigma^{-1})\nu),
\end{equation}see~\cite[\S3.1]{bib:Jan-DL}. This action has the same orbits as
the action considered in~\cite[\S4.1]{herzigthesis}, as $F=p\pi^{-1}$ on
$X(\uT)$ and $F(\sigma)=\pi\sigma\pi^{-1}$ in $\Aut(\uT)$ for $\sigma\in W$. In particular,
\cite[Lem.~4.2]{herzigthesis} still applies. Note also that $\tau(w,\mu)$
depends only on the orbit of $(w,\mu)$. 

\begin{defn}\label{defn:suff-generic}
We say that a tame inertial $L$-parameter $\tau$ is \emph{$\delta$-generic} 
if $\tau \cong \tau(w,\mu)$  for some $w\in W$ and $\mu$ lying
$\delta$-deep in $C_0$.  As in~\cite[\S6.5]{herzigthesis}  we say that a statement is true for all \emph{sufficiently generic} tame inertial
$L$-parameters $\tau$ if it holds for all $\tau\cong\tau(w,\mu)$ with $w\in W$
and for $\mu$ lying sufficiently deep in $C_0$; in other words,  if there exists $\delta > 0$ depending only in
the based root datum $\Psi_0(\uG,\uB,\uT)$ together with its
automorphism $\pi$ (and in particular not on $p$) such that the
statement holds for all $\tau$ that are $\delta$-generic.
\end{defn}

We remark that this  definition of $\delta$-generic differs slightly
from the one given at~\cite[Def.~6.27]{herzigthesis}, as we do not
require that the pair $(w,\mu)$ occurring in the definition be
``good'' (cf.\ \cite[Def.~6.19]{herzigthesis}). On the other
hand, this change does not affect what it means for a statement to be true
for all sufficiently generic $\tau$:\ indeed, \cite[Lem.~6.24]{herzigthesis} (whose analogue in this paper is
Lemma~\ref{lem:suff deep implies maximally split}) shows that the pair
$(w,\mu)$ is automatically good for $\mu$ lying sufficiently deep in
$C_0$.

\begin{prop}
  \label{prop: sufficiently generic tame inertial L-param}
  For 
  all sufficiently generic tame inertial $L$-parameters $\tau:I_{\Qp}\to \Ghat(\Fpbar)$
 and  for all $\lambda\in X_1(\uT)$, we have that $F(\lambda)\in
  \W^?(\tau)$ if and only if $\tau\cong\tau(w,\lambda'+\eta)$ for some dominant
  $\lambda'\uparrow\lambda$ and some $w\in W$.
\end{prop}
\begin{rem}
  \label{rem: equiv for lambda suff deep}Alternatively the equivalence holds for
  $\lambda$ lying sufficiently deep in a restricted alcove. (See~\cite[Prop.~6.28]{herzigthesis}.)
\end{rem}
\begin{proof}
  Write $\tau\cong\tau(w,\mu+\eta)$. For $\mu$ lying sufficiently deep in $C_0$,
  Proposition~\ref{prop: mu sufficiently deep condition on JH const of
    Rbar} gives that
  $\W^?(\tau)$ consists of the Serre weights $F(\lambda)$ for $\lambda\in
  \Xreg(\uT)$ such that there exists $\nu\in X(\uT)$
  with \[\sigma\cdot(\mu+(w\pi-p)\nu)\uparrow\lambda\qquad  \textrm{for all } \sigma\in W.\]
  Equivalently, this relation holds for the unique $\sigma$ making the left-hand side dominant. The
  proof concludes as in~\cite[Prop.~6.28]{herzigthesis}, using
  Lemma~\ref{lem:suff deep implies maximally split} and the formula~\eqref{eq:extendedaffineweylaction}.
\end{proof}

\subsection{The main result}\label{subsec: the main result}

The main result of this section is Theorem~\ref{thm:main-result}, which shows that for
all sufficiently generic tame inertial $L$-parameters $\tau$  the sets $W^?(\tau)$ and
$\Wexpl(\tau)$ coincide, and moreover that the Levi predictions of
Definition~\ref{def:wexpl-general} do not produce any new weights
beyond those already in $\cC(\Wobv(\tau))$.

Define $\|\lambda\|:=\sum_{\alpha>0}\langle 
\lambda,\alpha^\vee\rangle$ for $\lambda\in X(\uT)$. Then we have:
\begin{enumerate}
\item $\lambda<\mu\implies \|\lambda\|<\|\mu\|$. 
\item $\lambda\in X(\uT)_+\implies \|\lambda\|\ge 0$, with equality if and only if 
  $\lambda\in X^0(\uT)$. 
\item $\|\lambda\|=\|\pi(\lambda)\|$ for all $\lambda\in X(\uT)$. 
\end{enumerate}

\begin{lemma}\label{lem:fh-jan15-one}
Fix $N \in \Z_{> 0}$. Suppose that $\lambda \in X(\uT)_{+}$ with
$\|\lambda\| < Np$.
\begin{enumerate}
\item For all sufficiently generic $\tau$, if $\tau \cong
  \tau(w,\lambda)$ then $\lambda$ is as deep as we like in its
  alcove.

\item For $\lambda' \in X_1(\uT)$ lying sufficiently deep in a restricted
  alcove,  if $F(\lambda') \in \JH_{G(\Fp)} W(\lambda)$ then
  $\lambda$ is as deep as we like in its alcove.
\end{enumerate}
\end{lemma}
For instance, to be precise, the statement in (i) means that for each
fixed $\delta > 0$ and for all sufficiently generic $\tau$, if $\tau \cong
  \tau(w,\lambda)$ then $\lambda$ is $\delta$-deep in its  alcove; the
  meaning of (ii) is similar.

\begin{proof}
  (i) There is a finite collection of alcoves (independent of $p$) such that any
  $\lambda$ allowed by $\|\lambda\| < Np$ lies in the closure of one of them.
  Therefore, as explained after~\cite[Def.~6.27]{herzigthesis}, modulo
  $(p-1)X^0(\uT)$ there are only finitely many possible $\lambda$ (independent
  of $p$) and, for $\tau$ sufficiently generic, each one is $\delta$-deep in its alcove. (In the paragraph before~\cite[Prop.~6.28]{herzigthesis} note that
  $p-\pi$ is injective on the free abelian group $X(\uT)/\Z\Phi$, as $\pi$ has
  finite order.)

(ii)   In the argument that follows, if $\nu$ is an element of
$X(\uT)_+$ we will often write $\nu = \nu_0 + p \nu_1$
with $\nu_0\in X_1(\uT)$ and $\nu_1\in X(\uT)_+$ (so that $\nu_1$ is unique
modulo $X^0(\uT)$).   

Choose $\mu\in X(\uT)_+$ such that $F(\lambda')\in \JH_{G(\Fp)}
F(\mu)$ and $F(\mu) \in \JH_{\uG} W(\lambda)$. Then $\mu \uparrow
\lambda$, so $\|\mu\| < Np$ and $\mu$ lies as deep in its alcove as
$\lambda$.

If $\mu \in X_1(\uT)$, then $\mu \equiv \lambda'
\pmod {(p-\pi)X^0(\uT)}$, and we are done.  Otherwise, $F(\mu) \cong
F(\mu_0) \otimes F(\mu_1)^{(\pi)}$ as $G(\Fp)$-representations, so
there exists $\mu^{(1)} \in X(\uT)_+$ such that $F(\lambda') \in
\JH_{G(\Fp)} F(\mu^{(1)})$ and $F(\mu^{(1)}) \in \JH_{\uG}(F(\mu_0)
\otimes F(\mu_1)^{(\pi)})$. In particular, $\mu^{(1)} \le \mu_0 + \pi
\mu_1$, so as $\mu_1 \not\in X^0(\uT)$ we have 
\begin{equation}\label{eq:fh-star}
\|\mu^{(1)}\| \le \|\mu_0\| + \|\mu_1\| < \|\mu\|-(p-1) < \|\mu\|-
p/2.
\end{equation}
Iterating, we can find a sequence of dominant weights $\mu =
\mu^{(0)}, \mu^{(1)}, \ldots,\mu^{(r)}$ with 
\begin{itemize}
\item $F(\mu^{(i+1)}) \in \JH_{\uG}(F(\mu_0^{(i)}) \otimes
  F(\mu_1^{(i)})^{(\pi)})$ for all $0 \le i < r$,
\item $\mu^{(i)} \not\in X_1(\uT)$ for all $0 \le i < r$, but $\mu^{(r)} \in X_1(\uT)$,
\item $F(\mu^{(r)}) \cong F(\lambda')$ as $G(\Fp)$-representations.
\end{itemize}
Moreover, by \eqref{eq:fh-star}, we know that $r < 2N$.

On the other hand, as in the proof of \cite[Prop.~9.1]{herzigthesis}
  we can write $F(\mu_0) \otimes F(\mu_1)^{(\pi)} = \sum a_\varepsilon
  b_{\mu_0'} W(\mu_0' + \pi\varepsilon)$, where the sum runs over
  $\varepsilon\in X(\uT)$ such that $w\varepsilon \le \mu_1$ for all $w \in
  W$, and dominant $\mu_0' \uparrow \mu_0$. Hence $\mu^{(1)} \uparrow
  \sigma\cdot (\mu_0'+\pi\varepsilon)$ for some such $\varepsilon,\mu_0'$
  and some $\sigma\in W$. It follows that if $\mu^{(1)}$ is
  $(\delta+N)$-deep in its alcove, then $\mu_0'$ (and hence $\mu$)
  is $\delta$-deep in its alcove. Therefore, as $r < 2N$, if
  $\lambda'$ is $(\delta + 2N^2)$-deep in its alcove, then $\lambda$
  is $\delta$-deep in its alcove.
\end{proof}

\begin{lemma}\label{lem:fh-jan15-lemma2}
  Suppose $M \subset G$ is a Levi subgroup containing $T$, and that
  $\tau^M : I_{\Qp} \to \Mhat(\Fpbar)$ denotes a tame inertial
  $L$-parameter. Let $\tau : I_{\Qp} \to \Ghat(\Fpbar)$ denote the
  composite of $\tau^M$ and the inclusion $\Mhat \subset \Ghat$. Fix
  $\delta > 0$. For all $\tau$  sufficiently generic, $\tau^M$
  is $\delta$-generic.
\end{lemma}

\begin{proof}
  Write $\tau^M \cong \tau^M(w,\lambda)$ with $w\in W_M$,
  $\lambda \in X(\uT)$. Write $\lambda = \lambda_0 + p\lambda_1$ with
  $\lambda_0 \in X_1(\uT)$, $\lambda_1 \in X(\uT)_+$. Then
  $\tau^M(w,\lambda) \cong \tau^M(w,\lambda')$, where
  $\lambda' = \lambda_0 + w\pi \lambda_1$. For $\nu \in X(\uT)$ let
  $|\nu| := \sum_{\alpha > 0} |\langle v,\alpha^{\vee}\rangle|.$ Then
$$|\lambda'|  \le |\lambda_0| + |\lambda_1|  =  \frac{p-1}{p} |\lambda_0| +
  \frac{1}{p}|\lambda| \le \frac{(p-1)^2}{p} \sum_{\Delta} n_\beta + \frac{|\lambda|}{p},$$
where we write $\sum_{\alpha>0} \alpha^{\vee} = \sum_{\beta\in\Delta}
n_\beta \beta^{\vee}$.

Iterating, we deduce that $\tau^M \cong \tau^M(w,\mu)$ with $|\mu| < p
\sum_{\Delta} n_\beta$, so $\mu$ lies in the closure of a finite union
of alcoves (for $\uG$, hence also for $\underline{M}$).  A fortiori,
$\tau \cong \tau(w,\mu)$. As in the
proof of Lemma~\ref{lem:fh-jan15-one}(i), for $\tau$
sufficiently generic we have that $\mu$ lies as deep as we like in its
alcove (for $\uG$, hence also for $\underline{M}$). By reversing the
argument we deduce that $\tau^M$ is as generic as we like.
\end{proof}

\begin{lemma}\label{lem:fh-jan15-lemma3}
Suppose that $\lambda,\mu\in X(\uT)_{+,M}-\eta$. Choose $w,w' \in W$ such that
$w\cdot \lambda$, $w' \cdot \mu$ are in $X(\uT)_+-\eta$. Then $$\lambda
\uparrow_M \mu \qquad \implies \qquad w\cdot \lambda
\uparrow w' \cdot \mu.$$
\end{lemma}

\begin{proof}
Let $\mu' := w' \cdot \mu$, the unique element in $(X(\uT)_+-\eta) \cap W \cdot \mu$.
We may assume that $w' \in W$ has least possible length, i.e.\
$w'$ is a Kostant representative for $\Stab_W(\mu'+\eta)\backslash W$
(noting that the stabiliser is generated by simple reflections).

  First we claim that $w'(\Phi_M^+) \subset \Phi^+$, or equivalently $w'(\Delta_M) \subset \Phi^+$. 
  Suppose that $\alpha \in \Delta_M$. As $\mu \in X(\uT)_{+,M}-\eta$ we know that
  $\langle \mu'+\eta,w'(\alpha)^\vee\rangle \ge 0$. Hence if $w'(\alpha) \in \Phi^-$, then
  equality holds, i.e.\ $s_{w'(\alpha)} \in \Stab_W(\mu'+\eta)$. By our choice of $w'$
  it follows that $w' s_\alpha = s_{w'(\alpha)} w' > w'$, hence $w'(\alpha) \in \Phi^+$.
  This proves the claim.


By Corollary~\ref{cor:ye-wang} and induction, we may assume that
$\lambda=s_{\alpha,np} \cdot \mu$ for some $\alpha\in\Phi_M^+$, $n \in
\Z$ and that $\lambda \ne \mu$. As $\lambda\in X(\uT)_{+,M}-\eta$, we deduce that $\langle
\mu+\eta,\alpha^\vee\rangle > np >0.$ Hence $w' \cdot \lambda =
s_{w'(\alpha),np} w' \cdot \mu$ with $\langle w'\cdot \mu +
\eta,w'(\alpha)^\vee\rangle > np > 0$ and $w'(\alpha) \in \Phi^+$ by
the above. Then \cite[II.6.9]{MR2015057} shows that $$w'' w' \cdot
\lambda = w'' s_{w'(\alpha),np} w' \cdot \mu \uparrow w' \cdot \mu$$
for any $w'' \in W$ making $w'' w' (\lambda+\eta)$ dominant.
\end{proof}

\begin{lem}
\label{lem: lambda mu relation on bounded epsilon}Suppose that $\mu\in X(\uT)_+-\eta$
and that $\nu\in X(\uT)_+$. Then for $\lambda\in
X(\uT)_+-\eta$, we have $\lambda\uparrow\mu+p\nu$ if and only if
$\lambda=\sigma\cdot(\mu'+p\varepsilon)$ for some $\sigma\in W$, some 
$\mu'\uparrow\mu$ with $\mu' \in X(\uT)_+-\eta$
and some $\varepsilon\in X(\uT)$ such that
$w\varepsilon\le\nu$ for all $w\in W$.
\end{lem}
\begin{proof}
  Let $X(\mu,\nu)$ denote the subset of $\lambda\in X(\uT)_+-\eta$ defined by the
  right-hand side of the claimed equivalence.

For the ``if'' direction of the lemma note that
$w\cdot(\mu'+p\varepsilon)\uparrow\mu'+p\nu$ for all $w\in W$ by~\cite[Lem.~9.4]{herzigthesis} and that $\mu'+p\nu\uparrow\mu+p\nu$
by~\cite[II.6.4(4)]{MR2015057}. 
(We note that the proof of~\cite[Lem.~9.4]{herzigthesis} holds in our
more general context. The only necessary modifications are that in the
statement of that lemma the weights $\mu$, $\nu$ are to be taken in
$X(\uT)_+-\eta$ and in the proof of reduction step (R1)
we may assume $i > 0$ and then the first displayed inequality becomes
$0 < pi \le \langle \lambda'+\eta, w \alpha^\vee\rangle$.)

Conversely, suppose
$\lambda\uparrow\mu+p\nu$ with $\lambda$, $\mu$, $\nu$ as in the statement of the lemma.
By Corollary~\ref{cor:ye-wang} there is a sequence
$\lambda=\lambda_r\uparrow\lambda_{r-1}\uparrow\dots\uparrow\lambda_0=\mu+p\nu$,
where $\lambda_i\in X(\uT)_+-\eta$ and there exist affine reflections
$s_{\alpha_i,n_ip} =s_{\alpha_i}+n_i p\alpha_i \in W_p$ ($\alpha_i\in\Phi^+, n_i\in \Z$) such that
$\lambda_{i+1}=s_{\alpha_i,n_i p}\cdot\lambda_i$. Without loss of generality,
$\lambda_{i+1} < \lambda_i$ for all $i$.

We now show that $\lambda_i\in X(\mu,\nu)$ by induction on $i$. This is obvious
when $i=0$. For the induction step we are reduced to the following
statement. Given $\lambda,\lambda'$ in $X(\uT)_+-\eta$ such that
$\lambda=s_{\alpha,np}\cdot\lambda'$ with $\langle
\lambda'+\eta,\alpha^\vee\rangle > np$ and $\alpha\in \Phi^+$, then
$\lambda'\in X(\mu,\nu)$ implies $\lambda\in X(\mu,\nu)$. (Note that here
$\lambda$ no longer denotes the element $\lambda_r$ above.) Note that $n>0$, as
$np>\langle \lambda+\eta,\alpha^\vee\rangle\ge 0$. As $\lambda'\in X(\mu,\nu)$ we
can write $\lambda'=\sigma\cdot (\mu'+p\varepsilon)$ as in the statement of the
lemma. Then
\begin{equation}
  \label{eq:sactiononlambda}
  \lambda=s_{\alpha,np}\cdot\lambda'=s_\alpha\sigma\cdot(\mu'+p(\varepsilon-n\sigma^{-1}\alpha)).
\end{equation}

Case 1: Assume that $\langle\varepsilon,\sigma^{-1}\alpha^\vee\rangle\ge n$. 
To see that $\lambda\in X(\mu,\nu)$, by (\ref{eq:sactiononlambda}) it suffices to show that
$w(\varepsilon-n\sigma^{-1}\alpha)\le \nu$ for all $w\in W$. Let
$\varepsilon':=s_{\sigma^{-1}\alpha}\varepsilon=\varepsilon-\langle
\varepsilon,\sigma^{-1}\alpha^\vee\rangle\sigma^{-1}\alpha$. As
$\langle\varepsilon,\sigma^{-1}\alpha^\vee\rangle\ge n$, the sequence
$w\varepsilon$, $w(\varepsilon-n\sigma^{-1}\alpha)$, $w\varepsilon'$ is monotonic with respect
to $\le$ (i.e.\ either increasing or decreasing). As $w\varepsilon\le \nu$ and
$w\varepsilon'\le\nu$ by our assumption on $\varepsilon$, we conclude that $w(\varepsilon-n\sigma^{-1}\alpha)\le\nu$.

Case 2: Assume that $\langle\varepsilon,\sigma^{-1}\alpha^\vee\rangle=n-r$ for some
$r>0$. As
$\langle\mu'+\eta,\sigma^{-1}\alpha^\vee\rangle=\langle\lambda'+\eta,\alpha^\vee\rangle-p\langle\varepsilon,\sigma^{-1}\alpha^\vee\rangle$,
we see that $\langle \mu'+\eta,\sigma^{-1}\alpha^\vee\rangle > rp$. As $\mu' \in X(\uT)_+-\eta$
and $r>0$, we get $\sigma^{-1}\alpha\in \Phi^+$. Let $w\in W$ be such
that $\mu'':=w s_{\sigma^{-1}\alpha,rp}\cdot\mu' \in X(\uT)_+-\eta$. Then
$\mu''\uparrow\mu'$ by~\cite[II.6.9]{MR2015057} and
\begin{align*}
  \sigma
  w^{-1}\cdot(\mu''+pw\varepsilon)&=\sigma\cdot(s_{\sigma^{-1}\alpha}\cdot\mu'+rp\sigma^{-1}\alpha+p\varepsilon)\\
  &=s_\alpha\sigma\cdot(\mu'-rp\sigma^{-1}\alpha+p\varepsilon-p\langle\varepsilon,\sigma^{-1}\alpha^\vee\rangle\sigma^{-1}\alpha)\\
  &= s_\alpha\sigma\cdot (\mu'+p(\varepsilon-n\sigma^{-1}\alpha)),
\end{align*}which equals $\lambda$ by~(\ref{eq:sactiononlambda}). Hence
$\lambda\in X(\mu,\nu)$.
\end{proof}

Recall the definition of $d(C) \in \Z$ for an alcove $C$ (\cite{MR2015057}, II.6.6). For all $\alpha \in \Phi^+$ there is a unique $n_\alpha \in
\Z$ such that
\begin{equation}\label{eq:3app}
  n_{\alpha}p < \langle \lambda + \eta,\alpha\dual\rangle < (n_{\alpha}+1)p
\end{equation}
for all $\lambda \in C$. Then $d(C)
= \sum_{\Phi^+} n_\alpha$. If $C$ is dominant, then $d(C)$ is the number of affine root hyperplanes separating $C$ and the lowest alcove.
If $\lambda \in C$, then we set $d(\lambda) := d(C)$.
Note that if $\lambda$, $\mu \in X(\uT)$ are $p$-regular, then
$d(\mu)\le d(\lambda)$ for $\mu\uparrow\lambda$ and
$d(\mu)\le d(\mu+p\nu)$ for $\nu\in X(\uT)_+$, where equality holds only if
$\mu=\lambda$, respectively $\nu\in X^0(\uT)$ (\cite{MR2015057}, II.6.6).

\begin{prop}\label{prop:fh-jan15-4}
  Fix $N \in \Z_{> 0}$. Then for $\tau$ sufficiently generic and any
  $\lambda \in X(\uT)_+$ with $\|\lambda\| < Np$, the following are
  equivalent.
  \begin{enumerate}
  \item $W^?(\tau) \cap \JH_{G(\Fp)} W(\lambda) \neq \emptyset$.
\item $\tau \cong \tau(w,\lambda'+\eta)$ for some dominant $\lambda'
  \uparrow \lambda$ and some $w \in W$.
  \end{enumerate}
Moreover, if (ii) holds then $\lambda'=\lambda \in X_1(\uT)$ or there
exists $F(\nu)$ in  $W^?(\tau) \cap \JH_{G(\Fp)} W(\lambda) $ with
$d(\nu) < d(\lambda)$.
\end{prop}

\begin{proof}
 (i) $\Rightarrow$ (ii): Suppose $F(\lambda') \in W^?(\tau) \cap
 \JH_{G(\Fp)} W(\lambda)$ for some $\lambda' \in X_1(\uT)$. By
 Proposition~\ref{prop: sufficiently generic tame inertial L-param} we
 have $\tau \cong \tau(w,\lambda''+\eta)$ for some dominant $\lambda''
 \uparrow \lambda'$ and some $w \in W$. By
 Lemma~\ref{lem:fh-jan15-one} we see that $\lambda''$, $\lambda'$, and
 $\lambda$ are as deep in their respective alcoves as we like. To show
 (ii) we can now follow the proof of \cite[Prop.~9.1]{herzigthesis},
 noting that it never uses that $\lambda$ is restricted (as is assumed
 there) and making the following modifications: 
 $F(\mu_1)$ should be replaced by its $\pi$-twist $F(\mu_1)^{(\pi)}\cong
 F(\pi\mu_1)$ and $\rho$ by $\eta$. In the expressions
 $\mu_0+\varepsilon,\mu_0'+\varepsilon,\mu_0+w'\varepsilon$, $\varepsilon$
 should be replaced by $\pi\varepsilon$. Starting
 with~\cite[(9.2)]{herzigthesis}, the expression $pw^{-1}w'\varepsilon$ should
 be replaced by $p\pi^{-1}w^{-1}w'\pi\varepsilon$, as well as $\sigma w
 \sigma^{-1}$ by $\sigma w \pi\sigma^{-1}\pi^{-1}$. (Note also that $\pi\in W$
 in \emph{loc.\ cit.} is now a bad choice of letter.)

(ii) $\Rightarrow$ (i): Suppose that 
\begin{equation}
  \label{eq:taufromlambda'}
  \tau\cong\tau(w,\lambda'+\eta)\text{ for some dominant $\lambda'\uparrow\lambda$, some $w \in W$}.
\end{equation}
By Lemma~\ref{lem:fh-jan15-one}, we see that $\lambda'$ (and hence
$\lambda$) lie as deep in their respective alcoves as we like. If
$\lambda'=\lambda\in X_1(\uT)$, then $F(\lambda) \in W^?(\tau) \cap
\JH_{G(\Fp)} W(\lambda)$ by Proposition~\ref{prop: sufficiently
  generic tame inertial L-param}, as required. Thus from now on we may
assume $\lambda'=\lambda\not\in X_1(\uT)$ or $\lambda'\neq\lambda$.

We will first find $\lambda'' \in X(\uT)_+$ such that $F(\lambda'')$
is a $\uG$-constituent of $W(\lambda)$, and such that $\lambda''\neq
\lambda$ if $\lambda'\neq \lambda$. If $\lambda' = \lambda \not\in
X_1(\uT)$, we take $\lambda'' := \lambda$. If however $\lambda' \neq
\lambda$, 
choose $\lambda''\ne \lambda$ maximal such that $\lambda''$ is dominant and
$\lambda'\uparrow\lambda''\uparrow\lambda$. By Corollary~\ref{cor:ye-wang} there exists an affine
reflection $s_{\beta,np}\in W_p$ ($\beta\in\Phi^+,n\in\Z$) such
that $s_{\beta,np}\cdot\lambda=\lambda''$. As $\lambda''$ is dominant, $\langle
\lambda+\eta,\beta^\vee\rangle>np>0$. Jantzen's sum
formula~\cite[II.8.19]{MR2015057} says that for a certain descending filtration
$(V(\lambda)^i)_{i\ge 0}$ on the Weyl module $V(\lambda)$ we have\[\sum_{i>0}\ch
V(\lambda)^i=\sum_{\alpha\in\Phi^+}\sum_{0<mp<\langle
  \lambda+\eta,\alpha^\vee\rangle} \nu_p(mp)\sgn(w_{\alpha,m})\ch
W(w_{\alpha,m}s_{\alpha,mp}\cdot\lambda),\] where $w_{\alpha,m}\in W$ is chosen
such that $w_{\alpha,m}s_{\alpha,mp}\cdot\lambda$ is dominant. Note that the $p$-adic valuation
$\nu_p(mp)$ is positive, as $m > 0$. By~\cite[II.6.8]{MR2015057}, for
each term in the sum, $w_{\alpha,m}s_{\alpha,mp}\cdot\lambda\uparrow\lambda$ and equality
does not hold. Also, as $\lambda$ is $p$-regular, all
$w_{\alpha,m}s_{\alpha,mp}\cdot\lambda$ that occur in this sum are
distinct. (See also~\cite[II.8.19, Rk.~3]{MR2015057}.) Note that
$w_{\beta,n}=1$ by the previous paragraph. Therefore, by the maximality of
$\lambda''$ and by the strong linkage principle, $F(\lambda'')$ is a
$\uG$-constituent of $W(\lambda)$, as claimed. (It occurs once in
$W(w_{\beta,n}s_{\beta,np}\cdot\lambda)$, but cannot occur in any other term.)

Suppose first that $\lambda'' \in X_1(\uT)$, so $\lambda'' \neq
\lambda$. Then $F(\lambda'') \in W^?(\tau)$ by Proposition~\ref{prop:
  sufficiently generic tame inertial L-param} and
\eqref{eq:taufromlambda'}, so $F(\lambda'') \in W^?(\tau) \cap
\JH_{G(\Fp)} W(\lambda) $ and $d(\lambda'') < d(\lambda)$, as
required.

Alternatively, if $\lambda''\notin X_1(\uT)$, then
$\lambda''=\lambda_0''+p\lambda_1''$, where $\lambda_0''\in X_1(\uT)$ and
$\lambda_1''\in X(\uT)_+-X^0(\uT)$. By Lemma~\ref{lem: lambda mu relation on
  bounded epsilon}, as $\lambda'\uparrow\lambda''$ we can write
$\lambda'=\sigma\cdot(\mu+p\varepsilon)$ for some $\sigma\in W$, some dominant
$\mu\uparrow\lambda_0''$, and some $\varepsilon\in X(\uT)$ such that
$w'\varepsilon\le\lambda_1''$ for all $w'\in W$. As
$(w,\lambda'+\eta)=(w,\sigma(\mu+p\varepsilon+\eta))$ is in the same $X(\uT)\rtimes
W$-orbit as $(w',\mu+\pi\varepsilon'+\eta)$, where
$w':=\sigma^{-1}w\pi\sigma\pi^{-1}$ and
$\varepsilon':=\pi^{-1}w'\pi\varepsilon$, we have
\begin{equation}
  \label{eq:tauprimeexpression}
  \tau\cong\tau(w',\mu+\pi\varepsilon'+\eta).
\end{equation}
By genericity, we may assume that $p$ is large enough such that $\lambda_1''\in
C_0$ and that $\mu+\pi\varepsilon'$ lies in the same alcove as $\mu$ for any
possible $\lambda_1''$ and $\varepsilon'$. Then
$\mu+\pi\varepsilon'\uparrow\lambda_0''+\pi\varepsilon''$ for some $\varepsilon''\in
W\varepsilon'=W\varepsilon$. Note that $\varepsilon''$ is a weight of
$F(\lambda''_1)=W(\lambda''_1)$. Hence (as in the proof of~\cite[Prop.~9.1]{herzigthesis}), $F(\lambda_0''+\pi\varepsilon'')$ is a $G(\Fp)$-constituent of
$F(\lambda'')\cong F(\lambda_0'')\otimes F(\lambda_1'')^{(\pi)}$,
hence by the above
a $G(\Fp)$-constituent of $W(\lambda)$. By Proposition~\ref{prop: sufficiently
  generic tame inertial L-param} and~(\ref{eq:tauprimeexpression}),
$F(\lambda_0''+\pi\varepsilon'')\in \W^?(\tau)$, 
hence $F(\lambda''_0+\pi\varepsilon'') \in W^?(\tau) \cap \JH_{G(\Fp)}
W(\lambda) $ and $d(\lambda_0''+\pi \varepsilon'') = d(\lambda_0'') <
d(\lambda)$, as required.
\end{proof}

\begin{thm}\label{thm:main-result}
  For sufficiently generic tame inertial $L$-parameters $\tau$ we have $W^?(\tau) =
  \Wexpl(\tau) = \cC(\Wobv(\tau))$.
\end{thm}

\begin{proof}
 It suffices to show $\Wexpl(\tau) \subset W^?(\tau) \subset
 \cC(\Wobv(\tau))$. To see that $\Wexpl(\tau) \subset W^?(\tau)$,
 first note that $\Wobv(\tau) \subset W^?(\tau)$ by
 Propositions~\ref{prop: obvious weights of tau} and~\ref{prop:
   sufficiently generic tame inertial L-param}. Suppose there is a
 Levi $M \subset G$ containing $T$ such that $\tau$ factors via
 $\tau^M : I_{\Qp} \to \Mhat(\Fpbar)$. Then $\tau^M$ is as generic as
 we like by Lemma~\ref{lem:fh-jan15-lemma2}. It remains to check that
 if $\nu \in X_1(\uT)$ and $w\in W$ are such
 that $w \cdot \nu \in X(\uT)_{+,M}$, then $W^?(\tau^M)
 \cap \JH_{M(\Fp)} (W^M(w\cdot \nu)) \ne \varnothing$ implies $F(\nu) \in
 W^?(\tau)$. Noting that $\|w\cdot \nu\|_M \le \| \nu\|$, we get from
 Proposition~\ref{prop:fh-jan15-4} that $\tau^M \cong
 \tau^M(w',\lambda'+\eta)$ for some $M$-dominant $\lambda' \uparrow_M
 w\cdot \nu$ and some $w'\in W_M$. By Lemma~\ref{lem:fh-jan15-lemma3}
 we have $\sigma \cdot \lambda' \uparrow \nu$, where $\sigma \in W$
 such that $\sigma(\lambda'+\eta) \in X(\uT)_+$. Hence $\tau \cong
 \tau(w',\lambda'+\eta)\cong \tau(\sigma w' \pi \sigma^{-1} \pi^{-1},
 \sigma \cdot \lambda' + \eta)$. From this we deduce as in
 Lemma~\ref{lem:fh-jan15-one}(i) that $\sigma\cdot \lambda'\in
 X(\uT)_+$, as it is as deep as we like in its alcove, hence that
 $F(\nu) \in W^?(\tau)$ by Proposition~\ref{prop: sufficiently generic
   tame inertial L-param}.

To show $W^?(\tau) \subset \cC(\Wobv(\tau))$, we show 
$$ F(\lambda) \in W^?(\tau) \implies F(\lambda) \in \cC(\Wobv(\tau))
\qquad \text{for all}\ \lambda\in X_1(\uT)$$
by induction on $d(\lambda)$. (Note that $d(\lambda)$ is bounded,
independent of $p$.) As $F(\lambda)\in W^?(\tau)$,
Proposition~\ref{prop: sufficiently generic tame inertial L-param}
implies that $\tau \cong \tau(w,\lambda'+\eta)$ for some dominant
$\lambda'\uparrow \lambda$ and some $w\in W$. If $\lambda' =\lambda$,
then $F(\lambda) \in \Wobv(\tau)$ by Proposition~\ref{prop: obvious
  weights of tau}. Otherwise, by Proposition~\ref{prop:fh-jan15-4}
there exists $\nu\in X_1(\uT)$ such that $F(\nu) \in W^?(\tau) \cap
\JH_{G(\Fp)} W(\lambda)$ and $d(\nu) < d(\lambda)$. By induction, we
have $F(\nu) \in \cC(\Wobv(\tau)) \cap \JH_{G(\Fp)} W(\lambda)$, hence
by definition of $\cC$ we get $F(\lambda) \in \cC(\Wobv(\tau))$. 
\end{proof}

\begin{rem}
  In principle the implied constant in this theorem (as well as in all other results in Section~\ref{sec:herzig-comparison})
  can be made explicit. We also remark that none of the results we use
  depend on Lusztig's conjecture.
\end{rem}

\subsection{The proof of Lemma~\ref{lem:correction to EG 4.1.1.}}
\label{sec:proof-of-correction-to-EG}

In this section we prove Lemma~\ref{lem:correction to EG 4.1.1.},
which we restate here (using once again the notation of Section~\ref{sec:BM}).
 
\begin{lem}
If $\lambda$ is a lift of $a\in (\Xn)^{S_k}$, then
  $L_\lambda\otimes_{\O}\F$ has socle $F_a$, and every other Jordan--H\"older factor
  of $L_\lambda\otimes_{\O}\F$ is of the form $F_b$ with $b\in (X_1^{(n)})^{S_k}$ and $\|b\|<\|a\|$.
\end{lem}

\begin{proof}
It suffices to prove the analogous claim over $\Fpbar$.
For this we work in the following more general setting:\ let $G$
denote a connected reductive group over $\Fp$ such that $G^\der$ is
simply connected\footnote{That is,  $G$ is  the special
  fibre of one of the groups that we considered
in Section~\ref{sec:unramified groups}, except we don't assume that
$Z(G)$ is connected or that $G$ has a local twisting element.}. 
 We also let 
$B$ be a Borel subgroup of $G$ with Levi subgroup $T$, $(\uG, \uB, \uT) := (G,B,T) \times \Fpbar$,
and let $F:\uG \to \uG$ denote the relative Frobenius. Let $\pi$ be the finite order automorphism
of $(\uG, \uB, \uT)$ as in the proof of Lemma~\ref{lem:bijection of weights}; in particular, $F = p \pi^{-1}$
on $X(\uT)$.   For the moment we work with the definition of $\|.\|$ given in
Section~\ref{subsec: the main result}, and check at the end of the
proof that it agrees with Definition~\ref{defn: norm on weights}.

We will show that for $a\in X_1(\uT)$, $W(a)$ has $\uG^F$-socle
$F(a)$, and that every other Jordan--H\"older factor is of the form
$F(b)$, $b\in X_1(\uT)$, $\|b\|<\|a\|$.

We first leave aside the socle and show by induction on $\|a\|$ that if
$V$ is a $\uG$-module with unique highest weight $a\in X_1(\uT)$, and $\dim
V_a=1$, then $[V:F(a)]_{\uG^F}=1$, and every other Jordan--H\"older
factor of the $\uG^F$-representation $V$ is of the form $F(b)$, $b\in X_1(\uT)$,
$\|b\|<\|a\|$.

Any irreducible $\uG$-constituent of $V$ is of the form $F(b)$ with $b\le
a$. Hence it is enough to show that if $[F(b):F(c)]_{\uG^F} >0$ ($c\in X_1(\uT)$) then
$\|c\|\le\|b\|$ and that
$[F(b):F(a)]_{\uG^F}=\delta_{ab}$.

If $b\in X_1(\uT)$, then $c\equiv b\pmod{X^0(\uT)}$ by Lemma~\ref{lem:bijection of weights} and we are
done. Otherwise, $b=b_0+pb_1$ with $b_0\in X_1(\uT)$ and $b_1\in
X(\uT)_+-X^0(\uT)$. Then \[F(b)\cong F(b_0)\otimes F(pb_1)\cong
F(b_0)\otimes F(\pi(b_1))\] as $\uG^F$-representations, and the latter
$\uG$-module has unique highest weight $b_0+\pi(b_1)$. As
$\|b_0+\pi(b_1)\|=\|b\|-(p-1)\|b_1\|<\|b\|$, we get by induction
that $\|c\|\le\|b_0+\pi(b_1)\|<\|b\|$.

The claim about the socle follows by dualising the statement
of \cite[Thm.\ 5.9]{bib:Hum}. (In the proof replace $\sigma$ by any element of $X(\uT)$ that pairs
to $p-1$ with any simple coroot and $\le_\Q$
by $\|\cdot\|\le \|\cdot\|$, keeping in mind the above result about
Jordan--H\"older factors.)

To deduce the lemma, apply the above with $G = \Res_{k/\Fp} \GL_n$ as
in Section~\ref{sec:gln}. We
have canonical identifications $\uG^F \cong \GL_n(k)$ and $X(\uT)
\cong (\Z^n)^{\Sk}$. In the notation of Section~\ref{subsec:Serre weights} and the
proof of Proposition~\ref{prop:Wexpl} we get:
$$  L_\lambda \otimes \Fpbar \cong  \prod_{\sigma \in \Sk}
                             \left(M_{a_{\sigma}} \otimes_{k,\sigma}
                             \Fpbar\right)  \cong W(a), $$
$$  F_a \otimes \Fpbar \cong  \prod_{\sigma \in \Sk}
                             \left(N_{a_{\sigma}} \otimes_{k,\sigma}
                             \Fpbar\right)  \cong F(a). $$
To recover Definition~\ref{defn: norm on weights}, note that
$\left(\sum \alpha^{\vee} \right)_{\sigma} = (n-1,n-3,\ldots,-n+1) \in
\Z^n_{+}$ for any $\sigma \in \Sk$.
\end{proof}

\subsection{Comparison with  \texorpdfstring{\cite{bib:ADP}}{[ADP02].}}
\label{sec:adp-comp}

Let $\rbar: G_{\Q} \to \GL_n(\Fp)$ be odd and
irreducible. In this section we prove
Proposition~\ref{prop:adp-comparison}, i.e.\ we show that if 
$\rbar|_{I_{\Qp}}$ is
semisimple and sufficiently generic then the Serre weights
predicted in \cite{bib:ADP} are a subset of
$\Wexpl(\rbar|_{G_{\Qp}})$.

Suppose that $F(\lambda)$, with $\lambda \in X_1^{(n)}$ sufficiently
deep in its alcove, is predicted for $\rbar$ by
\cite[Conj.~3.1]{bib:ADP}. Then according to Definition~2.23 of
\emph{loc.\ cit.}, but using our terminology, there exist integers
$n_i$,  an
$\eta$-partition $(\lambda^{(i)})$ of $\lambda$ with $\lambda^{(i)}
\in \Z^{n_i}_+$, weights $\mu^{(i)} \in X_1^{(n_i)}$, and  $n_i$-cycles
  $w_i \in S_{n_i}$ such that:
  \begin{itemize}
  \item $\lambda^{(i)} \equiv \mu^{(i)} \pmod{(p-1)\Z^{n_i}}$ for all
      $i$,
  \item $\rbar|_{I_{\Qp}} \cong \oplus_i \tau(w_i,\mu^{(i)} +
      \eta_{n_i})$, where each summand is irreducible, and
\item $\mu_1^{(i)} - \mu_{n_i}^{(i)} \le p-1$
  for all $i$.  
\end{itemize}
(In fact, in what follows we make no use of the final condition in the 
above list, nor of the irreducibility of the summands, nor of the fact that $w_i$ is an $n_i$-cycle.)

Write $\lambda^{(i)} = \mu^{(i)} + (p-1)\nu^{(i)}$ with
$\nu^{(i)}\in \Z_+^{n_i}$. Then by~\eqref{eq:extendedaffineweylaction}
we have 
$$ \tau(w_i,\mu^{(i)} +
      \eta_{n_i}) \cong \tau(\sigma w_i \sigma^{-1}, \sigma\cdot
      ((\lambda^{(i)} - p \nu^{(i)}) + p w_i^{-1} \nu^{(i)}) +
      \eta_{n_i})$$ for all $\sigma \in S_{n_i}$. By Lemma~\ref{lem:
        lambda mu relation on bounded epsilon} we have 
$$ \sigma_i \cdot ((\lambda^{(i)} - p\nu^{(i)}) + p w_i^{-1}
\nu^{(i)}) \uparrow (\lambda^{(i)} - p\nu^{(i)}) + p\nu^{(i)} =
\lambda^{(i)},$$
where $\sigma_i$ is chosen so that the left-hand side of this equation
is dominant. Proposition~\ref{prop:fh-jan15-4} then gives
$$W^{?}(\tau(w_i,\mu^{(i)} + \eta_{n_i})) \cap \JH_{\GL_{n_i}(\Fp)}
W(\lambda^{(i)}) \neq \emptyset,$$
and so by Proposition~\ref{prop:Wexpl}, Theorem~\ref{thm:main-result},
and Definition~\ref{defn:explicit-weights-ss} we deduce that
$F(\lambda) \in \Wexpl(\rbar|_{G_{\Qp}})$.

\subsection{Beyond unramified groups}\label{subsec: beyond unramified groups}It
is at present unclear how to formulate versions of the various conjectures of
this paper for general ramified groups, where crystalline representations are
not available. It seems reasonable to expect that at least for inner forms of
$\GL_n$, it should be possible to use the
Breuil--M\'ezard formalism; indeed, this is carried out for quaternion algebras
in the papers~\cite{gee-savitt-quaternion,gee-geraghty-quaternion}. For more
general groups the absence of a local Langlands correspondence and a mature
theory of types at present mean that it is unclear whether to expect the
Breuil--M\'ezard formalism to extend in the necessary fashion.

\appendix

\section{Wang's result on the \texorpdfstring{$\uparrow$}{uparrow}-ordering of alcoves}\label{sec:wangs-result-uparrow}

We give Wang's proof of the following theorem on the geometry of alcoves (see \cite{ye}, \cite{wang}). The following treatment is
based on Chuangxun (Allen) Cheng's translation of parts of \cite{wang}.

\subsection{Ye and Wang's result}\label{sec:ye-wangs-result}

Let $\uG$ denote a connected reductive group over $\Fpbar$, and let
$\uB$ be a Borel subgroup of $\uG$ with Levi subgroup $\uT$. We then keep the same notation as in Sections~\ref{sec:unramified groups}--\ref{sec:herzig-comparison},
for example we have $\Phi$, $\Phi^+$, $W_p$, $\ua$.
However, for convenience, in this section $C_0$ and $w_0$ do not have their usual meaning.

\begin{thm}[Ye, Wang]\label{thm:dominant-sequence}
  Suppose $C$, $C'$ are dominant alcoves such that $C \ua C'$. Then there exists a sequence of dominant
alcoves $C = C_0 \ua C_1 \ua \cdots \ua C_k = C'$ such that $d(C_i)-d(C_{i-1}) = 1$ for all $i$.
\end{thm}

The proof of this theorem will be discussed below. We first deduce a corollary.
Let $\rho := \frac 12\sum_{\alpha \in \Phi^+} \alpha$.
We say that $\lambda\in X(\uT)$ is \emph{$\rho$-dominant} if $\langle \lambda + \rho,\alpha^\vee\rangle \ge 0$ for all
$\alpha \in \Delta$.

\begin{cor}\label{cor:ye-wang}  \leavevmode
  \begin{enumerate}
  \item Suppose $C$, $C'$ are dominant alcoves such that $C \ua C'$. Then there exists a sequence of dominant alcoves $C =
    C_0 \ua C_1 \ua \cdots \ua C_k = C'$ and reflections $s_i \in W_p$ such that $s_i \cdot C_{i-1} = C_i$ for all $i$.
  \item Suppose $\lambda$, $\lambda' \in X(\uT)$ are $\rho$-dominant such that $\lambda \ua
    \lambda'$. Then there exists a sequence of $\rho$-dominant weights $\lambda = \lambda_0 \ua \lambda_1 \ua \cdots \ua \lambda_k = \lambda'$ and reflections
    $s_i \in W_p$ such that $s_i \cdot \lambda_{i-1} = \lambda_i$ for all $i$.
  \end{enumerate}
\end{cor}

\begin{proof}
  Part (i) follows from Theorem~\ref{thm:dominant-sequence} and the definition of $\uparrow$, as
  $d(C) < d(C')$ whenever $C \uparrow C'$ with $C \ne C'$ (and this in fact implies part (ii) in
  case $\lambda$ and $\lambda'$ are $p$-regular).

  For part (ii), let $F$ (resp.\ $F'$) be the facet containing $\lambda$ (resp.\ $\lambda'$). Let $C$ be the unique
  maximal alcove with respect to $\uparrow$ that contains $\lambda$, or equivalently $F$, in its
  closure. It exists by \cite[II.6.11(5)]{MR2015057}, taking $C = C^+(F)$ in the notation used
  there.  Similarly we let $C'$ be the unique maximal alcove such that $\lambda' \in \o{C'}$. As
  $\lambda$, $\lambda'$ are $\rho$-dominant, we see from \cite[II.6.11]{MR2015057} that $C$, $C'$
  are dominant alcoves.

  We claim that $C \uparrow C'$. An argument exactly as in \cite[II.6.11(4)]{MR2015057} (reflecting
  up from $C$ rather than down from $w\cdot C^-$) shows that $C \uparrow C''$ for some alcove $C''$
  such that $\lambda' \in \o{C''}$, i.e.\ $F' \subset \o{C''}$. By the maximality of $C'$ we deduce that $C \uparrow C''
  \uparrow C'$.

  Applying part (i) we get a sequence of dominant alcoves $C = C_0 \ua C_1 \ua \cdots
  \ua C_k = C'$ and reflections $s_i \in W_p$ such that $s_i \cdot C_{i-1} = C_i$ for all $i$. For
  each $i$ let $\lambda_i$ be the unique $W_p$-translate of $\lambda$ in $\o{C_i}$.  Then
  $s_i\cdot\lambda_{i-1} = \lambda_i$ for all $i$, $\lambda_k = \lambda'$, and $\lambda_i$ is
  $\rho$-dominant as $C_i$ is dominant.
\end{proof}


In the following, let $\AA$ denote the set of alcoves and $\AA^+$ the subset of dominant alcoves.  Let $\CH$ denote the set
of all hyperplanes
\begin{equation}
  \label{eq:4app}
  H_{\alpha,np} = \{ \lambda \in X(\uT) \otimes \R : \langle \lambda+\rho, \alpha\dual\rangle = np \}
\end{equation}
for $\alpha \in \Phi^+$, $n \in \Z$.  
For each hyperplane $H = H_{\alpha,np} \in \CH$, let $s_H \in W_p$ be the reflection in $H$.
It is denoted by $s_{\alpha,np}$ in \cite{MR2015057}, II.6.1.
We will loosely say $H$ is a wall of an alcove $C$ if $H$ contains a facet of $C$ of codimension one.

Given a hyperplane $H = H_{\alpha,np} \in \CH$, we let $H^-$ (resp., $H^+$) denote the half-space obtained by replacing
``$=$'' by ``$<$'' (resp., ``$>$'') in \eqref{eq:4app}. Recall that $C \ua C'$ if there exists a sequence of alcoves $C = C_0$, $C_1$, \dots, $C_k = C'$ such that
$C_i = s_{H_i}\cdot C_{i-1}$ and $C_i \subset H_i^+$ for all $i$ (\cite{MR2015057}, II.6.5). We say that $C \uua C'$ if there is such a
sequence satisfying moreover that $-\rho \in \o{H_i^-}$ for all $i$. (This was considered, for example, in \cite{andersen}.)
We will see in Corollary~\ref{cor:two-uparrows} that the two partial
orders agree on the set of dominant alcoves.

Let $C^+$ denote the lowest alcove and $D^+ = \{ \lambda : \langle
\lambda + \rho, \alpha\dual \rangle > 0  \textrm{ for all }  \alpha
\in \Phi^+ \}$ the ($\rho$-shifted) dominant Weyl chamber.

\begin{lem}\label{lem:reflection}
  If $C \in \AA$ and $H \in \CH$ then $d(C) \ne d(s_H\cdot C)$, and $C \ua s_H \cdot C$ if and only if $d(C) < d(s_H\cdot C)$.
  In particular if $d(s_H \cdot C) = d(C) + 1$, then $H = H_{\beta,(n_\beta+1)p}$ for some $\beta \in \Phi^+$ and $n_\beta$ as in~\eqref{eq:3app}
  \upl with $\rho$ replacing $\eta$\upr .
  Therefore there are only finitely many alcoves $C'$ such that $C \ua C'$ and $d(C')-d(C) = 1$.
\end{lem}

\begin{proof}
  This follows easily from Lemma II.6.6 in \cite{MR2015057} and its proof.
\end{proof}

\begin{lem}\label{lem:basic}
  Suppose $C \in \AA$ and that $H$ is a wall of $C$. Let $s = s_H$.
  Suppose that $r$, $w \in W_p$ and that $r$ is a reflection. If $w\cdot C \ua rw \cdot C$ and $rws \cdot C \ua ws\cdot C$, then
  $rw = ws$.
\end{lem}

\begin{proof}
  Let $H_1 \in \CH$ be the hyperplane fixed by $r$. Then $w\cdot C \subset H_1^-$ and $ws\cdot C \subset H_1^+$.
  But $w\cdot H$ is the \emph{unique} hyperplane separating alcoves $w\cdot C$ and $ws\cdot C$. Thus $H_1 = w\cdot H$
  and $r = s_{w\cdot H} = w s w^{-1}$.
\end{proof}

\begin{prop}\label{prop:reflect-in-wall}
  Suppose that $C$, $H$, $s$ are as in Lemma~\ref{lem:basic}. If there exists a sequence $w_0$, \dots, $w_h \in W_p$ such that
  $w_i w_{i-1}^{-1}$ is a reflection for all $i$ and $w_0 \cdot C \ua \cdots \ua w_h \cdot C$, then one of the following is true:
  \begin{enumerate}
  \item $w_0s \cdot C \ua \cdots \ua w_hs \cdot C$.
  \item There are integers $j$, $k$ such that $1 \le j \le k \le h$ such that
    $w_0s\cdot C \ua \cdots \ua w_{j-1}s \cdot C = w_j \cdot C \ua \cdots \ua w_h \cdot C$ and
    $w_0 \cdot C \ua \cdots \ua w_{k-1} \cdot C = w_k s \cdot C \ua \cdots \ua w_h s\cdot C$.
  \end{enumerate}
\end{prop}

\begin{proof}
  Suppose (i) does not hold. Letting $j$ be the minimum and $k$ be the maximum of the non-empty set $\{ i: w_i s \cdot C \ua w_{i-1}s \cdot C\}$,
  the proposition easily follows from Lemma~\ref{lem:basic}.
\end{proof}

\begin{cor}\label{cor:reflect-in-wall}
  Suppose that $C$, $H$, $s$ are as in Lemma~\ref{lem:basic}. If $w$, $w' \in W_p$ such that $w\cdot C \ua w'\cdot C$
  and $w's\cdot C \ua w'\cdot C$, then $ws\cdot C \ua w'\cdot C$.
\end{cor}

\begin{proof}
  We can find $w_0$, \dots, $w_h \in W_p$ such that
  $w_i w_{i-1}^{-1}$ is a reflection for all $i$ and $w \cdot C = w_0 \cdot C \ua \cdots \ua w_h \cdot C = w'\cdot C$.
  By the proposition, $w\cdot C \ua w'\cdot C$ implies $ws\cdot C \ua w's\cdot C$ or $ws\cdot C \ua w'\cdot C$.
  We are done in the second case. In the first case use $w's\cdot C \ua w'\cdot C$ to conclude.
\end{proof}

\begin{lem}\label{lem:ua-uua}
  Suppose $C \in \AA^+$ or that $C$ has a wall $H$ such that $s_H \cdot C \in \AA^+$. Let $r \in W_p$ be a reflection.
  Then $C \ua r\cdot C \iff C \uua r\cdot C$.
\end{lem}

\begin{proof}
  Clearly if $C \uua r\cdot C$ then $C \ua r\cdot C$. Conversely, suppose that $C \ua r\cdot C$. Say $r = s_{H_1}$, where
  $H_1 = H_{\alpha,mp}$ with $\alpha \in \Phi^+$. We have $C \subset H_1^-$ and we want to show that $-\rho \in \o{H_1^-}$. If
  $D^+ \cap H_1^- \ne \varnothing$, then for any point $x$ in the intersection, $0 < \langle x+\rho, \alpha\dual\rangle < mp$,
  so $-\rho \in H_1^-$ and we are done.  If $C \in \AA^+$, then $C \subset D^+ \cap H_1^-$ and we are done.

  If $C \not\in \AA^+$, then $s_H\cdot C \in \AA^+$ for some wall $H$ of $C$.
  If $H \ne H_1$, then $C$ and $s_H\cdot C$ lie on the same side of $H_1$, so $s_H \cdot C \subset D^+ \cap H_1^-$ and we are done.
  If $H = H_1$, then $r\cdot C = s_H \cdot C \in \AA^+$, so $H_1$ is a wall of $D^+$ and thus $-\rho \in H_1 \subset \o{H_1^-}$.
\end{proof}

\begin{prop}\label{prop:uua-dominant-seq}
  Suppose $C$, $C' \in \AA$ with $C' \uua C$. If $w$, $w' \in W$ such that $w'\cdot C'$ and $w \cdot C$ are dominant, then there
  exists a sequence of dominant alcoves $w'\cdot C' = C_0 \uua \cdots \uua C_k = w\cdot C$ such that
  $d(C_i) - d(C_{i-1}) = 1$ for all $i$.
\end{prop}

\begin{proof}
  By the definition of $\uua$ we can reduce to the case when $C' = s_{\alpha,np}\cdot C$ for some $\alpha \in \Phi^+$, $-\rho \in
  \o{H_{\alpha,np}^-}$, and $C \subset H_{\alpha,np}^+$. Thus $\langle x+\rho,\alpha\dual\rangle > np \ge 0$ for all $x \in
  C$. If $\beta = w\alpha$, we see that $\langle y+\rho,\beta\dual\rangle > np \ge 0$ for all $y \in w\cdot C$. Since $w \cdot
  C$ is dominant, we have $\beta \in \Phi^+$. Thus $-\rho \in \o{H_{\beta,np}^-}$ and $w\cdot C \subset H_{\beta,np}^+$.
  Now we apply \cite{MR2015057}, II.6.8 to the dominant alcove $w\cdot C$ and the reflection $s_{\beta,np}$ to obtain a sequence of dominant alcoves
  \begin{equation}\label{eq:1app}
    w''s_{\beta,np}w\cdot C = C_0 \ua \cdots \ua C_k = w\cdot C
  \end{equation}
  such that $d(C_i) - d(C_{i-1}) = 1$ for all $i$, for some $w'' \in W$. Finally notice that $s_{\beta,np}w = w s_{\alpha,np}$
  and that $\ua$ can be replaced with $\uua$ in \eqref{eq:1app} by Lemma~\ref{lem:ua-uua}.
\end{proof}

Note that the translations of $X(\uT) \otimes \R$ that stabilise $\CH$ are precisely given by $pX(\uT)$.
The following lemma is obvious.

\begin{lem}\label{lem:translate}
  Suppose $t \in pX(\uT)$ and that $C$, $C'$ are alcoves.

  \begin{enumerate}
  \item $d(C')-d(C) = d(t\cdot C') - d(t\cdot C)$.
  \item $C' \ua C \iff t\cdot C' \ua t\cdot C$.
  \end{enumerate}
\end{lem}

\begin{prop}\label{prop:finiteness}
  Suppose $C$ is an alcove and $n \in \Z_{\ge 0}$. Then
  \[\AA(C,n) = \{ C' : \text{$C \ua C'$ and $d(C') - d(C) \le n$} \}\] is finite. If $t \in pX(\uT)$, then
  $\AA(t\cdot C,n) = t \cdot \AA(C,n)$.
\end{prop}

\begin{proof}
  By \cite{MR2015057}, II.6.10 the first claim is reduced to the case $n = 1$, which is covered by Lemma~\ref{lem:reflection}.
  The second claim is immediate from Lemma~\ref{lem:translate}.
\end{proof}

Given $n \in \Z_{\ge 0}$ and an alcove $C$, we say that $C$ is in \emph{general $n$-position} if
$\AA(C,n) \subset \AA^+$. The finiteness of $\AA(C,n)$ guarantees the existence of alcoves in 
general $n$-position.

\begin{lem}\label{lem:seq-to-general-pos}
  If $n \in \Z_{\ge 0}$ and $C \in \AA^+$, then there exists a sequence of dominant alcoves
  $C = \wt C_0 \ua \cdots \ua \wt C_h$ with $h \in \Z_{\ge 0}$ such that 
  $\wt C_{i-1}$ and $\wt C_i$  are adjacent for all $i$ \upl i.e.\
  there is only one hyperplane between them\textup) 
  and such that $\wt C_h$ is in general $n$-position.
\end{lem}

\begin{proof}
First consider the case when $C = C^+$. Take $C' = w\cdot C$ in
general $n$-position, where $w \in W_p$. Let $S$ be the set of
reflections in the walls of $C^+$. Pick a reduced expression $w =
s_1\cdots s_r$ ($s_i \in S$) in the Coxeter group $(W_p,S)$. Letting
$C_i := s_1\cdots s_i \cdot C^+$, it is clear that $C_{i-1}$, $C_i$
are adjacent for all $i$. We claim that the $C_i$ are dominant and
that $C^+ = C_0 \uparrow \cdots \uparrow C_r = C'$. If $w := s_1\cdots
s_{r-1}$ and $H \in \mathcal{H}$ denotes the hyperplane fixed by $ws_r
w^{-1}$, then $\ell(s_H w) > \ell(w)$ implies by
\cite[Thm.~V.3.2.1]{MR1890629} that $C_0 = C^+$ and $C_{r-1} = w\cdot
C^+$ lie on the same side of $H$. As $H$ is the common wall of
$C_{r-1}, C_r$, we see that $C_0,C_r$ lie on opposite sides of $H$. Since
$C_0 = C^+$ and $C_r$ are dominant, it follows that $C_0 \subset H^-$,
so $C_r \subset H^+$ and $C_{r-1} \uparrow C_r$. As $H$ is not a wall
of $D^+$, we deduce that $C_{r-1}$ is dominant. The claim follows by induction.

  If $C$ is general, write $C = w\cdot C^+$ for some $w \in W_p$ and write $w = \sigma + p\nu$ with $\sigma \in W$, $\nu \in X(\uT)$.
  Then it is easy to see that $\langle \nu , \alpha^\vee \rangle > -1$ for all $\alpha \in \Phi^+$, i.e.\ $\nu \in X(\uT)_+$.
  Hence $t := p\nu$ maps dominant alcoves to dominant alcoves and alcoves in general $n$-position to alcoves in general $n$-position.
  So the lemma is true if $C = t\cdot C^+$.

  If $C = w\cdot C^+ \ne t \cdot C^+$, we only have to find a sequence of dominant alcoves $C = \wt C_0 \ua \cdots \ua \wt
  C_h = t\cdot C^+$ such that $\wt C_{i-1}$ and $\wt C_i$ are adjacent for all $i$.  We use an induction on the number of
  hyperplanes between $C$ and $t\cdot C^+$. Let $H$ be a wall of $C$ that lies between $C$ and $t\cdot C^+$, so $H$ cannot
  be a wall of $D^+$. Let $\wt C_1 =
  s_H\cdot C = s_H w \cdot C^+$. Then $\wt C_1 \in \AA^+$ and we have $t\cdot (-\rho) = w\cdot (-\rho) \in H$, so $t\cdot
  (-\rho) = s_H w\cdot (-\rho)$.  Since $t\cdot C^+ \subset H^+$, it follows that $C \ua \wt C_1$. Moreover, the number of
  hyperplanes between $\wt C_1$ and $t\cdot C^+$ is one less than the number of hyperplanes between $C$ and $t\cdot C^+$.
\end{proof}

Given an alcove $C$ and $n \in \Z_{\ge 0}$, we let $h(C,n)$ be the minimum possible value $h$ occurring as
the length of the sequence in Lemma~\ref{lem:seq-to-general-pos}.

\begin{proof}[Proof of Theorem \ref{thm:dominant-sequence}]
  We prove this by induction on $d := d(C')-d(C) \ge 0$. When $d$ is fixed, we induct on $h(C,d)$. The cases $d \le 1$ are trivial.
  For any $d$, the case $h(C,d) = 0$ is trivial. Now for fixed $C$ and $d$ we have a sequence of dominant alcoves
  \begin{equation*}
    C = \wt C_0 \ua \cdots \ua \wt C_h
  \end{equation*}
  as in Lemma~\ref{lem:seq-to-general-pos} and such that $h = h(C,d)$.  If $\wt C_1 \ua C'$ then $d(C')-d(\wt C_1) =
  d(C')-d(C)-1$ and we are done by the induction hypothesis.  So we can assume from now on that $\wt C_1 \nua C'$. We can
  write $\wt C_1 = s_H \cdot C$ and $C' = w\cdot C$ for some wall $H$ of $C$ and some $w \in W_p$. Let $\wt C_1' = w\cdot \wt
  C_1 = w s_H \cdot C$. We claim that $C' \ua \wt C_1'$.  Otherwise $\wt C_1' \ua C'$. So $C \ua w\cdot C$ and $ws_H \cdot C \ua w\cdot C$. By
  Cor.~\ref{cor:reflect-in-wall} this implies that $\wt C_1 = s_H \cdot C \ua w\cdot C = C'$, a contradiction.

  Thus $C' \ua \wt C_1'$, in particular $C \ua \wt C_1'$. We apply Cor.~\ref{cor:reflect-in-wall} again to
  $C \ua ws_H\cdot C$, $w\cdot C \ua ws_H \cdot C$ and get that $\wt C_1 = s_H \cdot C \ua ws_H \cdot C = \wt C_1'$.
  Now note that $d(\wt C_1')-d(\wt C_1) = d(C')-d(C) = d$, but $h(\wt C_1,d) = h(C,d)-1$. By induction hypothesis we have
  a sequence of dominant alcoves
  \begin{equation*}
    \wt C_1 = w_0\cdot \wt C_1 \ua w_1 \cdot \wt C_1 \ua \cdots \ua w_d \cdot \wt C_1 = \wt C_1'
  \end{equation*}
  such that $d(w_i \cdot \wt C_1) - d(w_{i-1} \cdot \wt C_1) = 1$ and so $w_iw_{i-1}^{-1}$ is a reflection in $W_p$ for all $i$.
  Note that $w_0 = 1$ and $w_d = w$. Since $\wt C_1 \nua C' = ws_H \cdot \wt C_1$, by Prop.~\ref{prop:reflect-in-wall} we have
  \begin{equation}
    \label{eq:2app}
    C = w_0 \cdot C \ua w_1 \cdot C \ua \cdots \ua w_d \cdot C = C'.
  \end{equation}
  Since $d = d(C')-d(C)$, we have $d(w_i \cdot C) - d(w_{i-1} \cdot C) = 1$ for all $i$. As $w_i\cdot C$ and the dominant
  alcove $w_i \cdot \wt C_1$ are adjacent, we may replace $\ua$ by $\uua$ in \eqref{eq:2app} (by Lemma~\ref{lem:ua-uua}). In
  particular, we have $w_i \cdot C \uua C'$ for all $i$. If some $w_i \cdot C$ is not dominant, then $w_i \cdot H$ is a wall of $D^+$ and $w_i
  s_H w_i^{-1} \in W$. By Prop.~\ref{prop:uua-dominant-seq}, $w_i \cdot \wt C_1 = w_i s_H \cdot C = (w_i s_H w_i^{-1}) w_i \cdot C \uua C'$, so
  $\wt C_1 \ua w_i \cdot \wt C_1 \ua C'$, contradiction. Thus $w_i \cdot C \in \AA^+$ for all $i$ and \eqref{eq:2app} satisfies the
  condition in the theorem.
\end{proof}

\begin{cor}\label{cor:two-uparrows}
  If $C$, $C'$ are dominant alcoves, then $C \ua C'$ if and only if $C \uua C'$.
\end{cor}

\begin{proof}
  This follows from Corollary~\ref{cor:ye-wang} and Lemma~\ref{lem:ua-uua}.
\end{proof}

\section{\texorpdfstring{$\Wobv(\rhobar)$}{W\_obv(rhobar)} is non-empty}\label{sec:wobvrhbar-nonempty}

\setcounter{equation}{0}
\setcounter{subsection}{1}

The purpose of this appendix is to give a proof of the following
result, which was promised in Remark~\ref{rem:wobv-nonempty}.

\begin{thm}\label{thm:wobv-nonempty}
Suppose $K/\Qp$ is a finite extension, and let $\rhobar : G_K \to
\GL_n(\Fpbar)$ be a representation such that $\rhobar|_{I_K}$ is
semisimple. Then the set $\Wobv(\rhobar)$ of obvious weights for
$\rhobar$ is non-empty.
\end{thm}

A fortiori the same is true for $\W_{\expl}(\rhobar)$. Moreover, the proof
shows that $\Wcrisforall(\rhobar)$ is non-empty when $\rhobar$ is semisimple.

\begin{proof}
 For each $\sigma \in S_k$ we fix an element $\emb_{\sigma} \in S_K$
 lifting $\sigma$.  Throughout this proof,
 if we refer to \emph{the} lift of some Serre weight $F$, we mean any
 lift $\lambda$ of $F$ for which $\lambda_{\emb} =0$ if
 $\emb \not\in \{\emb_{\sigma}\}_{\sigma \in S_k}$ (cf.\
 Definition~\ref{defn:lift of a serre weight}). We will prove that $\rhobar$ has an obvious lift
 $\rho$ of Hodge type $\lambda$, where $\lambda$ is the lift of some
 Serre weight.

We may without loss of generality assume that $\rhobar$ itself is
semisimple. We begin by explaining how to reduce to the case where
$\rhobar$ is irreducible, by induction on the number of
Jordan--H\"older factors of $\rhobar$. 

Indeed, suppose that $\rhobar =
\rhobar' \oplus \rhobar''$, where
$\rhobar'$ has dimension $d' > 0$ and $\rhobar''$ is irreducible. By
induction $\rhobar'$ has an obvious lift $\rho'$ of Hodge type
$\lambda'$, the lift of some Serre weight. Similarly $\rhobar''
\otimes \varepsilonbar^{-d'}$ has an obvious lift $\rho''$ of Hodge
type $\lambda''$, the lift of some Serre weight. 

For each $\sigma \in
S_k$, let $H_{\sigma} = \max \HT_{\emb_{\sigma}}(\rho')$ and
$h_{\sigma} = \min \HT_{\emb_{\sigma}}(\rho'') + d'$. Also let $\Lambda
\subset \Z^{S_k}$ be the sublattice consisting of tuples
$(x_{\sigma})$ such that $\prod \omega_{\sigma}^{x_{\sigma}} = 1$.  It
is elementary to see that there exists $x = (x_{\sigma}) \in \Lambda$
such that $h_{\sigma} + x_{\sigma} \in [H_{\sigma}+1,H_{\sigma}+p]$
for all $\sigma \in S_k$. (This comes down to the fact that
$\Z^{S_k}/\Lambda \cong \Z/(p^f-1)\Z$, along with the fact that
integers have base $p$ representations.) Let $\chi$ be a crystalline
character whose Hodge type is the lift of $x$, and such that $\chibar$ is
trivial; such a character exists by
Lemma~\ref{lem:existenceofcrystallinechars}(i) and (ii). Define
$\rho:=\rho' \oplus( \rho''\otimes \varepsilon^{d'} \otimes \chi)$. Then one
checks (considering separately the sets $\HT_{\emb}(\rho)$ where
$\emb=\emb_{\sigma}$ for some $\sigma \in S_k$, and the sets
$\HT_\emb(\rho)$ where $\emb \not\in \{\emb_{\sigma}\}_{\sigma \in S_k}$)
that $\rho$ is
an obvious lift of $\rhobar$ whose Hodge type is the lift of some
Serre weight.

It remains to consider the case where $\rhobar$ is irreducible. Let $d
= \dim \rhobar$, and write $\rhobar \cong \Ind_{K_d}^{K} \psibar$ where
$K_d/K$ is the unramified extension of  degree $d$ and  $\psibar :
G_{K_d} \to \Fpbarx{\times}$ is a character. We wish to prove the
existence of $d$-tuples of integers $\{(h_{\sigma,0}, \ldots,
h_{\sigma,d-1})\}_{\sigma \in S_k}$ such that $0 < h_{\sigma,i} -
h_{\sigma,{i+1}} \le p$ for all $\sigma$ and $i$, and a crystalline
character $\psi$ lifting~$\psibar$ such that 
$$ \bigcup_{\genfrac{}{}{0pt}{}{\emb' \in S_{K_d}}{\emb'|_{K} = \emb_{\sigma}}}
\HT_{\emb'}(\psi) = \{h_{\sigma,0}, \ldots, 
h_{\sigma,d-1}\}$$ 
for each $\sigma$, and such that if $\emb \in S_K$ but $\emb \not\in
\{\emb_{\sigma}\}_{\sigma\in S_k}$ then
$$ \bigcup_{\genfrac{}{}{0pt}{}{\emb' \in S_{K_d}}{\emb'|_{K} = \emb}}
\HT_{\emb'}(\psi) = \{0,1,\dots,d-1\}.$$
Let $\chi$ be any crystalline character of $G_{K_d}$ such that
$\HT_{\emb'}(\chi)=\{0\}$ whenever $\emb'|_K \in
\{\emb_{\sigma}\}_{\sigma \in S_k}$, and such that if $\emb \in S_K$ but $\emb \not\in
\{\emb_{\sigma}\}_{\sigma\in S_k}$ then
$$ \bigcup_{\genfrac{}{}{0pt}{}{\emb' \in S_{K_d}}{\emb'|_{K} = \emb}}
\HT_{\emb'}(\chi) = \{0,1,\dots,d-1\}.$$
Then the theorem comes down to the existence of integers
$h_{\sigma,i}$ as above and a crystalline character $\chi'$ of $G_{K_{d}}$ such that 
$$ \bigcup_{\genfrac{}{}{0pt}{}{\emb' \in S_{K_d}}{\emb'|_{K} = \emb_{\sigma}}}
\HT_{\emb'}(\chi') = \{h_{\sigma,0}, \ldots, 
h_{\sigma,d-1}\}$$ 
for each $\sigma \in S_k$, such that $\HT_{\emb'}(\chi')=\{0\}$ if $\emb'|_K \not\in
\{\emb_\sigma\}_{\sigma \in S_k}$, and such that $\o{\chi'} = \psibar
\chibar^{-1}$ (for then one can take $\psi = \chi' \chi$). 

Unless
$(d,f) = (2,1)$, where $p^f = \#k$, the existence of $\chi'$ is an immediate consequence of
Proposition~\ref{prop:nonempty-combinatorics} below (in combination
with both parts of Lemma~\ref{lem:existenceofcrystallinechars}). When
$(d,f)=(2,1)$, the existence of $\chi'$ will follow in the same way provided
that the character $\psibar \chibar^{-1}$ of $G_{K_2}$ does
not extend to $G_K$. If also $e(K/\Qp)=1$ then $\chibar$ is unramified, and
since $\psibar$ does not extend to $G_K$, the same is true of
$\psibar \chibar^{-1}$. If instead $e(K/\Qp)>1$, it is possible that $\psibar \chibar^{-1}$
extends to $G_K$. In that case choose any $\kappa' \in S_{K_2}$ such
that $\HT_{\kappa'}(\chi) = \{1\}$. Let $\kappa'' \in S_{K_2}$ be the
other embedding such that $\kappa''|_K = \kappa'|_K$ (so that
$\HT_{\kappa''}(\chi) = \{0\}$). Let $\chi_0$ be a crystalline
character of $G_{K_2}$ with the same labeled Hodge--Tate weights as
$\chi$, except that $\HT_{\kappa'}(\chi_0) = \{0\}$ and
$\HT_{\kappa''}(\chi_0) = \{1\}$. We note that $\chibar \chibar_0^{-1}|_{I_{K_2}} =
\omega_{\embb''}^{p-1}$. Hence $\chibar \chibar_0^{-1}$ 
does not extend to $G_K$ (since $\omega_{\embb''}$
is a fundamental character of niveau $2$ and its exponent is not
a multiple of $p+1$); so neither does the character $\psibar \chibar_0^{-1}$, 
and the result follows from Proposition~\ref{prop:nonempty-combinatorics} using $\psibar \chibar_0^{-1}$ in
place of $\psibar \chibar^{-1}$.
\end{proof}

\begin{prop}\label{prop:nonempty-combinatorics}
Given positive integers $d$ and $f$, any residue class modulo
$p^{df}-1$ {\upshape(}with the exception of the residue classes congruent to $0$
modulo $p+1$ when $d=2$, $f=1${\upshape)} is of the form $\sum_{i=0}^{df-1} x_i p^i$, where for any
$i_0 \in \Z$ the set $\{x_i : i \equiv i_0 \pmod{f}\}$ is of the form
$\{h_0,\ldots,h_{d-1}\}$ with $0 < h_i - h_{i+1} \le p$ for all $i$.
\end{prop}

\begin{proof} Let $N$ be the representative in the interval
  $[0,p^{df}-1)$ of our given residue class modulo $p^{df}-1$, and let
  $x_0,\ldots,x_{df-1}$ be the digits in the base $p$ expansion of
  $N$, so that certainly $N \equiv \sum_{i=0}^{df-1} x_i
  p^i \pmod{p^{df}-1}$.  We will argue by altering the $x_i$'s, preserving this congruence, until
  the condition on the sets $\{x_i : i \equiv i_0 \pmod{f}\}$ is met.
  The typical alteration will be to add $\delta p$ to
  $x_i$ and $-\delta$ to $x_{i+1}$ (with $x_{df}$ taken to mean
  $x_0$). We break into cases depending on the value of $f$.

\vskip 0.3cm 

(1) We consider first the case where $f$ is even. Restrict our  attention to the
$x_j$'s with $j\equiv 0,1 \pmod{f}$. Relabel the pairs $\{ (x_{if},
x_{if+1}) : i = 0,\dots,d-1\}$ as pairs
$(a_0,b_0),\dots,(a_{d-1},b_{d-1})$, but \emph{not necessarily in the
  same order}; instead, we choose the labeling so that
\begin{enumerate}
\item $b_0 \ge \dots \ge b_{d-1}$, and
\item if $b_{i} = b_{i+1}$ then $a_{i+1} \le a_{i}$.
\end{enumerate}
There exist integers $\delta_i \in \Z$ such that 
\begin{equation}\label{eq:a-delta}
   (a_{i+1} + \delta_{i+1} p) - (a_i + \delta_i p)  \in (0,p]
\end{equation}
for all $i \in [0,d-1)$. As $a_{i+1}-a_i \in (-p,p)$, we have $\delta_{i+1} -
\delta_i \in [0,1]$. For each $i$ we define $(a'_i,b'_i) =
(a_i+\delta_i p,b_i - \delta_i)$, thereby also altering the
corresponding $x$'s.   We claim that $a'_{i+1}-a'_i$ and
$b'_{i}-b'_{i+1}$ both lie in $(0,p]$ for all $i$.  The first of these
claims is precisely \eqref{eq:a-delta}. For the second claim, write
$$ b'_i - b'_{i+1} = (b_i - b_{i+1}) + (\delta_{i+1} - \delta_i),$$
and observe that $b_i-b_{i+1} \in [0,p-1]$ by (i), while
$\delta_{i+1}-\delta_i \in [0,1]$. It remains to note that if
 $b_i = b_{i+1}$ then $a_{i+1} \le a_i$ by (ii), implying
 $\delta_{i+1}-\delta_i=1$ by \eqref{eq:a-delta}, and so $b'_i-b'_{i+1}
 > 0$ in all cases.

The two claims together show that after making these alterations, the
sets $\{x_j : j \equiv j_0 \pmod{f}\}$ for $j_0 = 0,1$ are both of
the desired form
$\{h_0,\ldots,h_{d-1}\}$ with $0 < h_i - h_{i+1} \le p$ for all
$i$.  Iterating the above procedure for the $x_j$'s with $j \equiv
2j_0,2j_0+1 \pmod{f}\}$ for each $j_0 \in [1, f/2)$ in turn, the proposition follows in
this case. 

\vskip 0.3cm

(2) Next we suppose that $f$ is odd and $f \ge 3$. It is enough to explain
how to alter the triples $\{(x_{if},x_{if+1},x_{if+2}) : i =
0,\dots,d-1\}$, for then we can deal with the remaining consecutive
pairs of residue classes as in the case where $f$ was even.  The truth
of the proposition is certainly unchanged under multiplication of the
given residue class by a power of $p$, or equivalently, under cyclic
permutation of the $x_j$'s. We observe (trivially) that it is possible to cyclically
permute the $x_j$'s so that it is \emph{not} the case that  the pairs
$(x_{if+1},x_{if+2})$ are all of the form $(p-1,p-1)$ or $(0,0)$, with
both occurring, and we make such a cyclic permutation.

Now rewrite
the triples $(x_{if},x_{if+1},x_{if+2}) $ as $(a_0,b_0,c_0),\dots,(a_{d-1},b_{d-1},c_{d-1})$ with
the labeling so that
\begin{enumerate}
\item $c_0 \ge \dots \ge c_{d-1}$, and
\item if $c_{i} = c_{i+1}$ then $b_{i+1} \le b_{i}$.
\end{enumerate}
Conditions (i) and (ii), together with the condition on the pairs
$(x_{if+1},x_{if+2})$ 
 from the previous paragraph, imply
\begin{equation}\label{caseA}
\text{there is no value of $i$ such that } (b_i-b_{i+1},c_i-c_{i+1})=(p-1,p-1). 
  \end{equation}

There exist $\delta_i \in \Z$
such that
\begin{equation}\label{Av}
 (a_{i+1}+\delta_{i+1}p) - (a_i + \delta_i p) \in (0,p]
  \end{equation}
for all $i \in [0,d-1)$, and also $\epsilon_i \in \Z$ such that
\begin{equation}\label{Aiii}
  \lambda_i := (b_{i+1} + \epsilon_{i+1} p) - (b_i + \epsilon_i p) \in (0,p+1]
\end{equation}
for all $i \in [0,d-1)$, with
\begin{equation}\label{Aiv}
  \begin{cases}
    \lambda_i = p+1 & \implies a_{i+1} \le a_{i},\\
\lambda_i = 1 & \implies a_{i+1} > a_{i}.
  \end{cases}
\end{equation}
As in case (1) we have $\delta_{i+1}-\delta_i \in [0,1]$ for all $i$;
similarly we have $\epsilon_{i+1}-\epsilon_i \in [0,2]$, with
\begin{equation}\label{Astar}
\epsilon_{i+1} - \epsilon_i = 2 \qquad \text{iff} \qquad (b_i,b_{i+1})
= (p-1,0)  \text{ and } a_{i+1} \le a_{i} .
\end{equation}

For each $i$ we define $$(a'_i,b'_i,c'_i) = (a_i+\delta_i p, b_i +
\epsilon_i p - \delta_i, c_i - \epsilon_i),$$ thereby altering the
corresponding $x$'s, and claim that we then have $a'_{i+1}-a'_i$,
$b'_{i+1}-b'_i$, $c'_i - c'_{i+1}$ in $(0,p]$ for all $i$. Case (2)
will be complete once we have proved this claim.

That $a'_{i+1}-a'_i \in (0,p]$ is immediate from \eqref{Av}. Next, we
have 
$$b'_{i+1}-b'_i = \lambda_i - (\delta_{i+1}-\delta_i) $$
with the first term on the right-hand side in $(0,p+1]$ and the second
term in $[0,1]$. If $\lambda_i = p+1$ then $a_{i+1} \le a_{i}$ by
\eqref{Aiv}, which implies $\delta_{i+1}-\delta_i = 1$ by \eqref{Av};
similarly if $\lambda_i = 1$ then $a_{i+1} > a_i$ and
$\delta_{i+1}-\delta_i=0$. Thus in all cases we have $b'_{i+1}-b'_i
\in (0,p]$, as desired.

Finally, $$ c'_i - c'_{i+1} = (c_i - c_{i+1}) +
(\epsilon_{i+1}-\epsilon_i) $$
with the first term on the right-hand side in $[0,p-1]$ and the second
term in $[0,2]$. If $c_i=c_{i+1}$ then $b_{i+1} \le b_i$, implying
$\epsilon_{i+1} - \epsilon_i > 0$; thus $c'_i - c'_{i+1} > 0$ in all
cases. Suppose on the other hand that $c_i - c_{i+1} = p-1$. Then by \eqref{caseA} we have $b_i - b_{i+1} \neq
p-1$, and so $\epsilon_{i+1}-\epsilon_i \neq 2$ by \eqref{Astar}; thus
$c'_i - c'_{i+1} \le p$ in all cases, and case (2) is complete.

\vskip 0.3cm

(3) Finally we turn to  the case $f=1$. As usual, we take the $x_i$'s
at the outset to be
the digits in the base $p$ expansion of $N$.
  As in part (2) we will make use of the fact that the truth of the proposition
  is unchanged when multiplying
  the given residue class by a power of $p$ (i.e.\ under cyclic
  permutation of the $x_i$'s), as well as when adding any
  multiple of $(p^d-1)/(p-1)$ to the residue class (i.e.\ adding the same constant to
  each $x_i$).

  We first dispense with the case where $d$ is even and $N$ is divisible by
  $(p^d-1)/(p-1)$. By hypothesis we have $d \ge 4$ (recall that in the
  case $d=2$, $f=1$ the residue classes divisible by $p+1$ are
  excluded from the statement of the proposition).  Subtracting the
  appropriate multiple of $(p^d-1)/(p-1)$ we may suppose that
  $N=0$. Then writing $d = 2m+2$ with $m \ge 1$, we alter the $x_i$'s
  by replacing them with
$$(x'_0,\ldots,x'_{d-1}) := (p,2p-1,p-2,-1,2p,-2,\ldots,mp,-m).$$

In the remaining cases, we can reduce to one of the following three
situations.
\begin{enumerate}
\item[(I)] $d$ is odd, each $x_i$ lies in $[0,p-1]$, and $x_{d-1} = \max_i x_i$.
\item[(II)] $d$ is even, each $x_i$ lies in $[0,p-1]$, $x_1 = \max_i
  x_i > 0$, and $x_i=0$ for all even $i$.
\item[(III)] $d$ is even, each $x_i$ lies in $[0,p-1]$ except $x_1 =
  p$, and $x_i$ is non-zero for some odd $i > 1$.
\end{enumerate}
To see this, argue as follows.  If $d$ is odd, cyclically permute to assume that $x_{d-1}$
is maximal to put ourselves in case (I). Now suppose $d$ is even, so
that $(p^d-1)/(p-1) \nmid N$ and not all the $x_i$'s are equal.
Subtracting $\min_i x_i$ from each $x_i$, we can further suppose that
some $x_i$ is $0$.  Since not all $x_i$'s are zero, we can suppose (after
cyclically permuting if necessary) that $x_1 = 0$ and $x_2 > 0$.   If
$x_i$ is non-zero for some odd $i > 1$, then we add $p$ to $x_1$ and
$-1$ to $x_2$ to put ourselves in case (III).  Otherwise $x_i=0$ for all
odd $i$ but not for all even $i$.  After cyclically permuting so that
$x_1 = \max_i x_i$, we have $x_{2j} = 0$ for all $j$ and are in case (II).  This
completes the reduction.

Write $d = 2m$ if
$d$ is even and $d = 2m+1$ if $d$ is odd.  
We relabel the pairs of variables $(x_0,x_1),\ldots,(x_{2m-2},x_{2m-1})$ as
$(a_0,b_0),\ldots,(a_{m-1},b_{m-1})$, ordered as usual so that
\begin{enumerate}
\item $b_0 \ge \cdots \ge b_{m-1}$ and
\item if $b_i = b_{i+1}$ then $a_{i+1} \le a_{i}$.
\end{enumerate}
Note that when $d$ is odd, $x_{d-1}$ is not relabeled.  When $d$ is
even, we can (and do) take $(a_0,b_0) = (x_0,x_1)$:\ in case
(II) this is a consequence of the fact that $a_i=0$ for all $i$,
whereas in case (III) it is automatic.

There exist unique integers $\delta_i \in \Z$ such that  $\delta_0 = 1$ and 
\begin{equation}\label{eq:3-delta}
   (a_{i+1} + \delta_{i+1} p) - (a_i + \delta_i p)  \in (0,p]
\end{equation}
for all $i \in [0,m-1)$.  For each $i$ we define $(a'_i,b'_i) = (a_i + \delta_i p, b_i - \delta_i)$, thereby
altering the corresponding $x$'s.  (Note when $d$ is odd that
$x_{d-1}$ is unchanged.) It follows almost exactly as in
(1) that we have $a'_{i+1}-a'_i$, $b'_i-b'_{i+1} \in (0,p]$ for
all $i$; the only modification required is to note that in case (III),
although $b_0=p$ we still have $b_0 - b_1 \in [0,p-1]$ because of the
condition that $x_i$ is non-zero for some odd $i > 1$.

To complete the proof,  it will suffice to show that
$$
\begin{cases}
  a'_0-b'_0 \in (0,p] & \text{if $d$ is even}, \\
  a'_0-x_{d-1},\ x_{d-1}-b'_0 \in (0,p] & \text{if $d$ is odd}.
\end{cases}
$$

First suppose that $d$ is even. Since $\delta_0=1$ we have
$$a'_0 - b'_0 = (x_0-x_1) + (p+1).$$
In case (II) we have $x_0-x_1 \in (-p,0)$, while in case (III) we have
$x_0-x_1 \in [-p,0)$; in either case $a'_0-b'_0 \in (0,p]$.

Finally suppose that $d$ is odd. We have $a'_0 = a_0 + p$ and
$b'_0=b_0-1$. Since $a_0 \le x_{d-1}$ and both are in the range
$[0,p-1]$, we have $a'_0 - x_{d-1} = p - (x_{d-1} - a_0) \in
(0,p]$. Similarly $b_0 \le x_{d-1}$ and both are in the range
$[0,p-1]$, so that $x_{d-1} - b'_0 = (x_{d-1}-b_0) + 1 \in
(0,p]$. This completes the proof.
\end{proof}

\bibliography{eghs}
\bibliographystyle{amsalpha} 
\end{document}